\documentclass[]{article}

\usepackage{graphicx}
\usepackage{amsmath}
\usepackage{amssymb}
\usepackage{amsthm}
\usepackage{pxfonts}
\usepackage{enumerate}
\usepackage{color}
\usepackage{mathdots}
\usepackage{sectsty}
\usepackage{tikz}
\usepackage{adjustbox}
\usepackage{enumitem}
\usepackage{caption}
\usepackage{bbold}
\usepackage{mathrsfs}
\usepackage[hidelinks]{hyperref}
\allowdisplaybreaks

\sectionfont{\scshape\centering\fontsize{11}{14}\selectfont}
\subsectionfont{\scshape\fontsize{11}{14}\selectfont}
\usepackage{fancyhdr}
\usepackage[nottoc,notlot,notlof]{tocbibind}

\newcommand{\sgn}{\textnormal{sgn}}

\newcommand{\bup}{\boldsymbol{\upsilon}}
\newcommand{\bUpsilon}{\boldsymbol{\Upsilon}}
\newcommand{\bxi}{\boldsymbol{\xi}}

\newcommand{\de}{\delta}

\newcommand{\ep}{\epsilon}

\newcommand{\cT}{{\mathcal T}}
\newcommand{\cA}{\mathcal A}

\newcommand{\ka}{\kappa}
\newcommand*{\defeq}{\overset{\mathrm{def}}{=}}

\newcommand*\dif{\mathrm{d}}
\newcommand{\prob}{\mathbb{P}}
\newcommand{\expt}{\mathbb{E}}
\usepackage{bbm}
\newcommand{\indi}{\mathbf{1}}

\newcommand{\eps}{\varepsilon}

\newcommand{\dist}{\operatorname{dist}}

\newcommand{\Log}{\operatorname{Log}}
\newcommand{\R}{\mathbb R}
\newcommand{\C}{\mathbb C}
\newcommand{\bXi}{\boldsymbol{\mathsf{\Xi}}}
\newcommand{\Rc}{\mathfrak R}
\newcommand{\K}{\mathfrak K}
\newcommand{\BL}{\mathfrak A}

\newcommand{\bx}{\mathbf{x}}
\newcommand{\bl}{\mathbf{l}}

\newcommand{\Sine}{\mathsf{Sine}}
\newcommand{\mi}{\mathrm{i}}
\newcommand{\mT}{\mathrm{T}}
\newcommand{\me}{\mathrm{e}}

\newcommand{\mE}{{\mathfrak{E}}}

\newcommand{\UN}{{(N)}}
\newcommand{\sfX}{\mathsf{X}}
\newcommand{\sfB}{\mathsf{B}}
\newcommand{\sfG}{\mathsf{G}}
\newcommand{\sfL}{\mathsf{L}}

\newcommand{\sfS}{\mathsf{S}}
\newcommand{\sfU}{\mathsf{U}}

\newcommand{\bsfL}{\boldsymbol{\mathsf{L}}}
\newcommand{\bsfJ}{\boldsymbol{\mathsf{J}}}
\newcommand{\sfY}{\mathsf{Y}}

\newcommand{\bXXi}{\bXi_{\bxi}}
\newcommand{\tbXXi}{\widetilde{\bXi_{\bxi}}}

\newcommand{\tsup}{\mathrm{sup}}
\newcommand{\tinf}{\mathrm{inf}}
\newcommand{\tsai}{\mathcal{L}_{\kappa,\rho}^p}
\newcommand{\tsaipro}{\mathcal{P}_{\kappa,\rho}^p}

\newcommand\shorttitle{Infinite-dimensional Dyson Brownian motion(s)}
\newcommand\authors{T. Assiotis and F. Li}

\newtheorem{theorem}{Theorem}[section]
\newtheorem{proposition}[theorem]{Proposition}
\newtheorem{cor}[theorem]{Corollary}
\newtheorem{lemma}[theorem]{Lemma}
\newtheorem{definition}[theorem]{Definition}
\newtheorem*{theorem*}{Theorem}
\newtheorem*{mainone}{Theorem A}
\newtheorem*{maintwo}{Theorem B}
\newtheorem{remark}[theorem]{Remark}
\DeclareMathOperator{\length}{length}
\title{\large \bf INFINITE-DIMENSIONAL DYSON BROWNIAN MOTION(S)}
\author{\small THEODOROS ASSIOTIS AND FENGYI LI}
\date{}

\begin{document}
\maketitle

\begin{abstract}
We study infinite-dimensional Dyson Brownian motions obtained as limits of finite systems without rescaling the actual stochastic dynamics. For $\beta=2$, we construct determinantal processes on an extended space of initial data and prove convergence of their finite-dimensional distributions under essentially optimal conditions. This extends the seminal results of Katori and Tanemura \cite{Katori_2009}. The additional parameters record information at infinity and enter through an associated Laguerre--P\'olya entire function. Moreover, for explicit classes of configurations, we establish convergence on path space and the Markov property. We prove rescaled long-time convergence, in finite-dimensional distributions, to the stationary extended $\Sine$ process from arbitrary symmetric initial configurations with power-law counting exponent $q\in(0,2)$. This extends the integer lattice relaxation result of Katori and Tanemura \cite{Katori_2009} which was the only such result for explicit deterministic initial conditions. For $\beta\geq1$, we prove convergence of finite particle systems from regular initial data to the unique strong solution of an infinite-dimensional stochastic differential equation in a certain rigid-path-regularity class. This extends seminal works of Tsai and Osada \cite{MR3568040,Osada2013Logarithmic,Osada1996Dirichlet,Osada2012RandomMatrices}. We finally derive a stochastic partial differential equation of Burgers-type for the Stieltjes transform of the dynamics.
\end{abstract}

\numberwithin{equation}{section}
\tableofcontents
\addtocontents{toc}{\setcounter{tocdepth}{1}} 

\clearpage
\section{Introduction}
\subsection{Setting and motivation}
Dyson Brownian motion, introduced by Dyson in 1962, describes the evolution of eigenvalues of matrices undergoing Brownian motion \cite{Dyson1962}. It is one of the most extensively studied objects in random matrix theory and a quintessential example of a system of stochastic differential equations with singular logarithmic interaction. Its local relaxation properties underpin a powerful approach to proving universality of local eigenvalue statistics \cite{ErdosYau2017} and it also gives rise, under suitable soft-edge scaling, to the Airy line ensemble, a fundamental object in the Kardar–Parisi–Zhang (KPZ) universality class \cite{Corwin2012,CorwinHammond2014}.

In this paper we study infinite-dimensional Dyson Brownian motions obtained by letting the
number of particles tend to infinity without rescaling the dynamics. For a
finite interval $\mathcal I_N\subset\mathbb Z$, distinct ordered initial
positions $(x_i^{(N)})_{i\in\mathcal I_N}$ and $\beta\geq1$, the equations
\begin{equation}\label{eq SDE DBM}
 \dif\sfX_i^{(N)}\left(t\right)
 =\dif\sfB_i\left(t\right)
 +\frac\beta2\sum_{\substack{j\in\mathcal I_N\\j\ne i}}
 \frac{\dif t}{\sfX_i^{(N)}\left(t\right)-\sfX_j^{(N)}\left(t\right)},
 \qquad \sfX_i^{(N)}\left(0\right)=x_i^{(N)},\quad i\in\mathcal I_N,
\end{equation}
have a unique strong noncolliding solution, where the $(\sfB_i)_{i\in\mathbb Z}$
are independent standard Brownian motions; see
\cite{anderson2010introduction}. The central question conceptually answered in this paper is which infinite-particle
limits arise from these finite systems, and what initial data determine them.

The upshot is that the local limiting initial configuration $(x_i)_{i\in \mathbb{Z}}$ does not contain all the required data and thus one cannot truly speak of the infinite-dimensional Brownian motion starting from configuration $(x_i)_{i\in \mathbb{Z}}$ without further qualifications.
Particles escaping to infinity disappear from the configuration but their
reciprocal sums can leave a common drift, and their reciprocal-square sums
can leave a further quadratic parameter. Two finite approximations of the
same configuration can therefore produce different limiting dynamics.
The answer naturally involves an extended state space. Its elements encode
an entire function: the zeros record the limiting particles, while the
exponential factors retain the information lost at infinity.
The information retained at infinity has an explicit dynamical effect:
it produces a common translation and, at $\beta=2$, a dilation and time change.

The equilibrium dynamics of infinite Dyson systems were first studied by
Spohn~\cite{Spohn1987Dyson,Spohn1987Tracer}. Osada, later in collaboration with Tanemura, developed a general theory,  based on Dirichlet forms, for diffusion
constructions and solutions to infinite systems of stochastic differential equations with logarithmic and other singular
interactions~\cite{Osada1996Dirichlet,Osada2013Logarithmic,OSADA2016186,OsadaOsada2023Ergodicity,OsadaTanemura2020Tail,OsadaOsadaCoulomb}.
Katori and Tanemura constructed determinantal infinite-particle dynamics
from explicit initial configurations at $\beta=2$, see
\cite{Katori_2009,MR2918121,MR3019667}. Tsai constructed unique strong
solutions of the labelled equations for $\beta\geq1$ from a regular class
of explicit initial configurations, and proved convergence of spatial truncations
\cite{MR3568040}. 
Finally, more recently Suzuki~\cite{Suzuki2025Curvature} developed a more geometric perspective and has proven Bakry--\'Emery estimates and
transport regularity for the $\Sine_\beta$-symmetric Dirichlet forms.

In this paper we extend the determinantal construction through its initial parameters
and  control of the reciprocal
tails encoded in the aforementioned entire function, under essentially optimal conditions. We moreover prove convergence on path-space and the Markov property of the limiting dynamics. Finally, we prove a rescaled long-time limit to the $\Sine$ point process
for an explicit class of initial configurations. The initial configurations
may have zero, finite positive, or infinite asymptotic density. Under only
the stated power asymptotics from the abstract, the rescaled several-time laws converge to
those of the stationary extended $\Sine$ process. From explicit initial conditions the only known result before our work was for the integer configuration \cite{Katori_2009} which allowed for a very explicit computation. It is worth noting that using Dirichlet form theory ergodicity for almost every $\Sine_2$ initial condition (we note that one cannot write down explicit configurations without additional arguments) was proved in
\cite{OsadaOsada2023Ergodicity,Suzuki2024Ergodicity} and it was moreover later shown that the corresponding
Dirichlet forms have no spectral gap~\cite{Suzuki2025NoGap}.

The extended-state viewpoint is motivated by earlier work on dynamics
on boundaries of graphs~\cite{Assiotis2020HuaPickrell} and the
first author's joint work with Mirsajjadi on logarithmically
interacting diffusions and evolving characteristic
polynomials~\cite{AssiotisMirsajjadi2024ISDE,AssiotisMirsajjadi2026GL}.
Consistency or intertwining relations between finite systems play a central role in
these constructions as they fit a powerful algebraic framework introduced by Borodin and Olshanski called the method of intertwiners~\cite{BorodinOlshanskiMarkov,OlshanskiICM,OlshanskiLectureNotes}. This framework does not seem to give non-trivial results directly for Dyson Brownian motion. Here, instead we develop the extended-state viewpoint directly for unscaled
Dyson Brownian motion, using the determinantal structure at $\beta=2$
and the comparison of ordered gaps at $\beta\geq1$. At $\beta=2$, the main object is a Laguerre--P\'olya entire function~\cite{levin1964distribution,MR4487977}
associated with the initial data. For general $\beta \ge 1$, we extend Tsai's approach which is based on the comparison of
ordered gaps. We also show that from this approach one can derive a stochastic partial differential equation of Burgers-type for the limiting Stieltjes transform. This is the bulk (instead of edge), non-equilibrium analogue of an equation for the Stieltjes transform for the Airy line ensemble obtained by Huang and Zhang \cite{HUANG2026111028}.

\subsection{Main results}\label{sec:intro-main-results}
We first introduce the data, which is unfortunately quite heavy, needed to state our two main results precisely. Theorem A includes all our main results special to the $\beta=2$ case coming from the determinantal structure and Theorem B all our main results valid for $\beta \ge 1$ related to stochastic equations.

Write
$\mathfrak M$ for the space of locally finite, nonnegative integer-valued
Radon measures on $\mathbb R$, with the vague topology; multiplicities are
allowed. We write $\bxi$ for deterministic configurations and $\bXi$
for random configurations. Write
$\bxi^{[N]}=\bxi|_{[-N,N]}$ for spatial restriction, reserving $(N)$ for a
general approximation. Finite-dimensional convergence means weak
convergence in $\mathfrak M^m$ at each finite collection of admissible
times. The space $\mathcal C([0,\infty)\to\mathfrak M)$ carries the topology
of locally uniform vague convergence. For a map $F$, the pushforward
$F_\#\bxi$ places each atom $x$ at $F(x)$, with its multiplicity.

\smallskip\noindent\textbf{Extended initial data.}
Fix $b>0$, which is a purely technical device as will be explained in more detail in the sequel: this auxiliary cutoff separates the roots near the origin
from those encoded by their reciprocals. The coordinates depend on $b$,
but the kernel is unchanged under the corresponding change of coordinates
in \eqref{eq:cutoff-change}. A cutoff is admissible for a configuration
if it has no atoms at either $b$ or $-b$.
An element of the space $\bUpsilon_b$ of
Definition~\ref{def parameter space} has the form
\[
 \bup=\left(\left(\boldsymbol\alpha^+,\boldsymbol\alpha^-\right),
 \left(\mathbf a^+,\mathbf a^-\right),\gamma_1,\delta\right).
\]
The nonnegative sequences $\boldsymbol\alpha^\pm$ are decreasing, have
entries smaller than $1/b$, and are square summable. Their nonzero entries
encode roots $1/\alpha_j^+$ and $-1/\alpha_j^-$. The inner data
$\mathbf a^\pm=((a_j^\pm)_{j\geq1},k^\pm)$ consist of decreasing sequences
in $[0,b)$ with finitely many positive entries, and integers $k^\pm\geq0$.
They encode roots $a_j^+$, $-a_j^-$ for positive entries, and
$k^++k^-$ roots at zero. We use the convention $k^-=0$. Thus, the root
configuration denoted by $\bxi$ has no atoms at $\pm b$. Finally, $\gamma_1\in\mathbb R$
and $\delta\in\mathbb R$ satisfies
$\delta\geq\sum_j((\alpha_j^+)^2+(\alpha_j^-)^2)$; set
\begin{equation}\label{eq:intro-defect}
 \gamma_2\left(\bup\right)\defeq
 \delta-\sum_{j\geq1}\left(\left(\alpha_j^+\right)^2+
 \left(\alpha_j^-\right)^2\right),
 \qquad
 \tau\left(\bup\right)\defeq\frac1{\gamma_2\left(\bup\right)},
 \qquad \frac10\defeq\infty.
\end{equation}
Convergence in $\bUpsilon_b$ means coordinatewise convergence of the
reciprocal and inner-root sequences, eventual equality of each total inner
degree $q(\mathbf a^\pm)=\#\{j:a_j^\pm>0\}+k^\pm$, and convergence of
$\gamma_1$ and $\delta$. In particular, it does not require separate
convergence of $\gamma_2$.

The entire function associated with these data is
\begin{equation}\label{eq:intro-entire}
 \mE_{\bup}\left(z\right)\defeq
 \me^{-\gamma_1z-\gamma_2z^2/2}
 \prod_{\substack{x\in\bxi\\|x|<b}}\left(z-x\right)
 \prod_{\substack{x\in\bxi\\|x|>b}}
 \left(1-\frac zx\right)\me^{z/x},
\end{equation}
where products count multiplicities. It belongs to the Laguerre--P\'olya
class, the nonzero locally uniform limits of real polynomials with only
real zeros; see \cite{levin1964distribution,MR4487977}.
The correlation kernel $\K_{\bup}$ is an explicit double integral involving
a ratio of these entire functions. We defer its precise and rather complicated contour integral expression to
Section~\ref{sec:preliminaries}.

For configurations with finite reciprocal-square tail, the natural
reciprocal parameter is the symmetric principal value. Define
\begin{equation}\label{eq:intro-reciprocal-data}
 s_b\left(\bxi\right)\defeq
 \int_{|x|>b}\frac{\bxi\left(\dif x\right)}{x^2},
 \qquad
 p_b\left(\bxi\right)\defeq
 \lim_{\substack{R\to\infty\\R\in\mathbb N}}
 \int_{b<|x|\leq R}\frac{\bxi\left(\dif x\right)}x,
\end{equation}
whenever the principal value exists. For a configuration with
$s_b(\bxi)<\infty$ and no atoms at $\pm b$, let $f_r^b(\bxi)$ denote
its parameter with $\gamma_1=r$ and $\gamma_2=0$. For finite $\bxi$,
$p_b(\bxi)=\sum_{x\in\bxi,\,|x|>b}1/x$, and the unshifted finite
Dyson process has kernel $\K_{f_{p_b(\bxi)}^b(\bxi)}$.
Repeated finite initial points at $\beta=2$ are understood through the
entrance laws of Proposition~\ref{finitedysonmodel}, obtained as distinct
finite initial points merge.
The space $\mathfrak N$ consists of configurations for which both
quantities are finite, for an admissible cutoff $b$. We write
$\bXi_{\bxi}$ for the process with kernel
$\K_{f_{p_b(\bxi)}^b(\bxi)}$. This choice is independent of the cutoff.

\smallskip\noindent\textbf{Initial conditions for paths and the Markov property.}
For path convergence, consider simple configurations $\bxi^{(N)}$ with
finite reciprocal-square tails, no atoms at $\pm b$, and
$f_{r^{(N)}}^b(\bxi^{(N)})\to\bup$ in $\bUpsilon_b$, whose limiting root
configuration is simple and infinite. We require a decreasing function
$\varepsilon_{\mathrm{sep}}:[0,\infty)\to[0,\infty)$ tending to zero such that
\begin{equation}\label{eq:intro-separation}
 \sum_{\substack{y\in\bxi^{(N)}\\y\ne x,-x}}
 \frac1{\left|y^2-x^2\right|}
 \leq\varepsilon_{\mathrm{sep}}\left(\left|x\right|\right),
 \qquad x\in\bxi^{(N)},\quad N\geq1.
\end{equation}
The uniformity over the approximating family is part of this condition.

For the Markov property we use two explicit classes in $\mathfrak N$.
The first consists of symmetric simple configurations with a positive
minimum spacing. The second consists of
$\bxi=\ep\delta_0+\sum_{j\geq1}(\delta_{a_j}+\delta_{-a_j})$, $\ep\in\{0,1\}$ and $0<a_1<a_2<\cdots\to\infty$, with the following local-count bounds for
some $q\in(1,2)$. Let $\mathcal N(u,s)$ count all roots in $[u-s,u+s]$
and $\mathcal N^+(u,s)$ count its positive roots. There are positive
constants $C_{\mathrm{up}},c_{\mathrm{low}},C_{\mathrm{low}},U_0$ such that
\begin{equation}\label{eq:intro-local-counts}
 \begin{aligned}
 \mathcal N\left(u,s\right)
 &\leq C_{\mathrm{up}}\left(1+s\left(1+u\right)^{q-1}+s^q\right),
 &&u\geq0,\ s>0,\\
 \mathcal N^+\left(u,s\right)&\geq c_{\mathrm{low}}s u^{q-1},
 &&u\geq U_0,\ C_{\mathrm{low}}u^{1-q}\leq s\leq u/4.
 \end{aligned}
\end{equation}

A more general criterion for the Markov property is given in
Proposition~\ref{prop markov from rho}.

\smallskip\noindent\textbf{Initial conditions and scales for relaxation.}
For the long-time limit we consider
\begin{equation}\label{eq:intro-relaxation-data}
 \begin{aligned}
 \bxi&=\eps\delta_0+\sum_{j\geq1}\left(\delta_{a_j}+\delta_{-a_j}\right),
 &&0<a_1<a_2<\cdots\to\infty,\quad\eps\in\left\{0,1\right\},\\
 a_j&\sim\left(j/c\right)^{1/q},
 &&j\longrightarrow\infty,
 \end{aligned}
\end{equation}
where $0<q<2$ and $c>0$. We note that no rate of convergence or additional spacing
assumption is imposed in the relaxation theorem. These configurations
belong to $\mathfrak N$. For all sufficiently large $T$, put
\begin{equation}\label{eq:intro-sine-scales}
 n_T\defeq\max\left\{n\geq1:a_n^2\leq Tn\right\},
 \qquad \sigma_T\defeq\sqrt{T/n_T},
 \qquad \rho_q\defeq\frac1\pi
 \left(\frac{q\pi}{\sin\left(\pi q/2\right)}\right)^{1/(2-q)}.
\end{equation}
Write $\bXi_{\Sine,\rho}$ for the stationary extended $\Sine$ process of
density $\rho$, defined by the extended kernel in
Section~\ref{sec:sine}. Its one-time law is the $\Sine$ point process with
kernel $\sin(\pi\rho(x-y))/(\pi(x-y))$, with value $\rho$ on the diagonal.

Our first result constructs the limiting configuration processes and
describes their dependence on the initial data. The path, Markov and
relaxation assertions use the respective hypotheses just stated.
\begin{mainone}
Let $\beta=2$.
\begin{enumerate}[label=\textnormal{(\roman*)},leftmargin=*]
\item Every $\bup\in\bUpsilon_b$ determines a determinantal process
$\bXi_{\bup}$ with kernel $\K_{\bup}$ on $(0,\tau(\bup))$, obtainable as
a finite-dimensional limit of finite unscaled Dyson Brownian motions.
If $\bup^{(N)}\to\bup$, their kernels converge locally uniformly in the
spatial variables at each fixed finite collection of positive times below
$\tau(\bup)$, and the corresponding processes converge in finite-dimensional
distributions.
\item With the root data $\bxi$ fixed, write $\bXi^{\gamma_1,\gamma_2}$ for the
process with the indicated scalar parameters. Put $A_t=1-\gamma_2t$ and
$\Phi_t(x)=A_tx-\gamma_1t$. Then
\begin{equation}\label{eq:intro-parameter-map}
 \left(\bXi^{\gamma_1,\gamma_2}\left(t\right)\right)_{0<t<1/\gamma_2}
 \overset{\mathrm{law}}{=}
 \left(\left(\Phi_t\right)_\#
 \bXi^{0,0}\left(t/A_t\right)\right)_{0<t<1/\gamma_2}.
\end{equation}
\item Under the path hypotheses above, including
\eqref{eq:intro-separation}, the limiting parameter has $\gamma_2=0$.
The processes have continuous versions with their prescribed configurations
at time zero, and convergence holds in
$\mathcal C([0,\infty)\to\mathfrak M)$.
\item For either of the two Markov classes above, $\bXi_{\bxi}$ has the
Markov property, with transition laws given by restarting the same
construction from the current configuration.
\item Under \eqref{eq:intro-relaxation-data}, as $T\to\infty$,
\begin{equation}\label{eq:intro-sine-limit}
 \left(x\mapsto x/\sigma_T\right)_\#
 \bXi_{\bxi}\left(T+\sigma_T^2t\right)
  \overset{\mathrm{f.d.d.}}{\longrightarrow}\ \bXi_{\Sine,\rho_q}\left(t\right)
\end{equation}
in finite-dimensional distributions at fixed real rescaled times. The difference between the rescaled kernel and the limiting extended sine
kernel tends to zero locally uniformly in all four variables.
\end{enumerate}
\end{mainone}
\begin{proof}
Proposition~\ref{well_defined process} gives the construction and finite-particle
realisation, and Theorem~\ref{MainTheorem} gives parameter continuity.
Proposition~\ref{gammadependence} proves (ii), and
Theorem~\ref{theorem path cts} proves (iii). The two assertions in (iv)
are Corollaries~\ref{cor:section5-markov} and~\ref{xi in L have M-p}.
Finally, (v) is Theorem~\ref{thm:main}.
\end{proof}
The rescaled long-time statistics in (v) are unchanged by
symmetry-preserving finite modifications and ordered perturbations
$\widetilde a_j/a_j\to1$ of the initial positions within the stated class.
For $\gamma_2>0$, the restriction $t<1/\gamma_2$ is sharp.
Proposition~\ref{prop:gamma2-horizon-sharpness} gives finite initial
configurations with convergent parameters whose first correlation function
at the origin diverges both at $t=1/\gamma_2$ and at every later fixed time.
Thus the kernel-convergence conclusion cannot, in general, be extended
to or beyond this horizon.

The transformation in \eqref{eq:intro-parameter-map} also suggests a
labelled equation. Suppose that $p_b(\bxi)$ exists and that the natural
zero-defect process admits a labelling solving the principal-value Dyson
equation. Formally, for $\beta=2$ and $0<t<1/\gamma_2$, the transformed
particles then solve
\begin{equation}\label{eq ISDE gamma_2}
 \dif\sfX_i\left(t\right)=\dif\widetilde{\sfB}_i\left(t\right)
 +\left\{\operatorname{PV}\sum_{j\ne i}
 \frac1{\sfX_i\left(t\right)-\sfX_j\left(t\right)}
 -\frac{\gamma_1-p_b\left(\bxi\right)+\gamma_2\sfX_i\left(t\right)}
 {1-\gamma_2t}\right\}\dif t.
\end{equation}
Here $\operatorname{PV}$ denotes the limit of sums over $0<|j-i|\leq K$,
and the time change produces independent standard Brownian motions
$(\widetilde{\sfB}_i)_i$. The constant drift involves the excess
$\gamma_1-p_b(\bxi)$, equal to $\gamma_1$ for balanced data with
$p_b(\bxi)=0$. We leave justifying \eqref{eq ISDE gamma_2} as an open problem.

\smallskip\noindent\textbf{Regular labelled data.}
For the second result, let $\mathcal W$ be the ordered sequences
$\bx=(x_i)_{i\in\mathbb Z}$ in $\mathbb R\cup\{\pm\infty\}$, with
$x_i<x_{i+1}$ unless both are the same infinity, equipped with coordinatewise
convergence. A finite configuration is padded by $-\infty$ to the left
and $+\infty$ to the right. Set $\mathbb V=\frac12+\mathbb Z$ and define
the gaps $l_v=\mathfrak q(\bx)_v=x_{v+1/2}-x_{v-1/2}$ when both endpoints
are real, and $l_v=\infty$ otherwise. For a finite nonempty
$A\subset\mathbb V$, write
$\overline\Sigma_A^r(\bl)=|A|^{-1}\sum_{v\in A}l_v^r$ and
$\overline\Sigma_A=\overline\Sigma_A^1$. For $0<\kappa<1$, $\rho>0$ and
$p>1$, Tsai's regular gap class is
\begin{equation}\label{eq:intro-regular-gaps}
 \begin{aligned}
 \mathcal L_{\kappa,\rho}^p
 &\defeq\left\{\bl\in\left(0,\infty\right)^{\mathbb V}:
 \mathfrak D_{\kappa,\rho}\left(\bl\right)<\infty,\quad
 \mathfrak M_p\left(\bl\right)<\infty\right\},\\
 \mathfrak D_{\kappa,\rho}\left(\bl\right)
 &\defeq\sup_{m\in\mathbb Z\setminus\left\{0\right\}}
 \left|m\right|^\kappa
 \left|\overline\Sigma_{\left(0,m\right)\cap\mathbb V}\left(\bl\right)-\rho\right|,
 \qquad
 \mathfrak M_p\left(\bl\right)\defeq
 \sup_{m\in\mathbb Z\setminus\left\{0\right\}}
 \overline\Sigma_{\left(0,m\right)\cap\mathbb V}^p\left(\bl\right).
 \end{aligned}
\end{equation}
Here $(0,m)$ denotes the interval between $0$ and $m$, in either order.
The corresponding path class $\mathcal P_{\kappa,\rho}^p$ consists of
positive continuous gap paths on which both control quantities are
bounded uniformly on each compact time interval. Regularity implies
$x_i=x_0+\rho i+O(|i|^{1-\kappa})$. The parameter $\rho$ is a mean gap,
whereas $\rho_q$ above is a point density.

Let $\bx^{(N)}\to\bx$ be finite ordered approximations, where
$\bl=\mathfrak q(\bx)\in\mathcal L_{\kappa,\rho}^p$, and put
$\mathcal I_N=\{i:x_i^{(N)}\in\mathbb R\}$,
$\bl^{(N)}=\mathfrak q(\bx^{(N)})$ and
$l_v^{\mathrm{inf},(N)}=\inf_{n\geq N}l_v^{(n)}$.
Our assumptions on the approximations are: for some $C<\infty$, $N_0$
and $\rho^{(N)}>0$,
\begin{equation}\label{eq:intro-gap-approximations}
 \begin{aligned}
 &\rho^{(N)}\longrightarrow\rho,\qquad
 \bl^{\mathrm{inf},(N)}\in\mathcal L_{\kappa,\rho^{(N)}}^p
 \quad\left(N\geq N_0\right),\\
 &\sup_{N\geq N_0}\mathfrak D_{\kappa,\rho^{(N)}}
 \left(\bl^{\mathrm{inf},(N)}\right)\leq C,
 \qquad
 \sup_N\sup_{v:\,l_v^{(N)}<\infty}l_v^{(N)}\leq C.
 \end{aligned}
\end{equation}
Choose $b>0$ avoiding the absolute values of every finite initial
coordinate, including those of the limit, and define the reciprocal
parameters and their discrepancy by
\begin{equation}\label{eq:intro-tail-discrepancy}
 \gamma_{1,b}\left(\bx\right)\defeq
 \lim_{K\to\infty}\sum_{\substack{|i|\leq K\\|x_i|>b}}\frac1{x_i},
 \qquad
 \gamma\defeq\lim_{N\to\infty}\gamma_{1,b}\left(\bx^{(N)}\right)
 -\gamma_{1,b}\left(\bx\right).
\end{equation}
We require the second limit to exist; $1/(\pm\infty)=0$. The discrepancy
$\gamma$ is independent of the admissible cutoff. For the associated
configuration $\bxi=\sum_i\delta_{x_i}$, the regularity estimate gives
$\gamma_{1,b}(\bx)=p_b(\bxi)$; see Section~\ref{sec:isde}.
Thus, at $\beta=2$, $\gamma$ is the excess linear parameter
$\gamma_1-p_b(\bxi)$ in \eqref{eq ISDE gamma_2}, with
$\gamma_1=\lim_N\gamma_{1,b}(\bx^{(N)})$.
For regular labelled configurations the drift is
\[
 \mu_i\left(\bx\right)\defeq
 \lim_{K\to\infty}\sum_{0<|j-i|\leq K}\frac1{x_i-x_j}.
\]
Finally, for a regular particle process and $z$ in the upper half-plane
$\mathbb H=\{z\in\mathbb C:\operatorname{Im}z>0\}$, set the dynamical Stieltjes transform:
\begin{equation}\label{eq:intro-stieltjes-definition}
 \sfS_t\left(z\right)\defeq
 \lim_{R\to\infty}\sum_{|i|\leq R}\frac1{\sfX_i\left(t\right)-z},
 \qquad
 \sfS_{t,\gamma}\left(z\right)\defeq\sfS_t\left(z\right)+\gamma.
\end{equation}
For the finite systems, $\sfS_t^{(N)}(z)$ is the corresponding sum over
$i\in\mathcal I_N$.

Our second result identifies the finite-particle limit through the labelled
equations for every $\beta\geq1$. The discrepancy in the initial reciprocal
tails survives as a common drift and as an additive constant in the
Stieltjes transform.
\begin{maintwo}
Let $\beta\geq1$ and suppose the regularity and approximation assumptions
above hold. Construct the finite Dyson systems from $\bx^{(N)}$ using
the common Brownian family $(\sfB_i)_{i\in\mathbb Z}$. Then, almost surely,
each fixed coordinate converges uniformly on compact time intervals to
the unique strong solution $\boldsymbol{\sfX}=(\sfX_i)_{i\in \mathbb{Z}}$, among processes with gap process in
$\mathcal P_{\kappa,\rho}^p$, of
\begin{equation}\label{eq:intro-shifted-isde}
 \dif\sfX_i\left(t\right)=\dif\sfB_i\left(t\right)
 +\frac\beta2\left\{\mu_i\left(\boldsymbol{\sfX}\left(t\right)\right)
 -\gamma\right\}\dif t,
 \qquad\sfX_i\left(0\right)=x_i,\quad i\in\mathbb Z.
\end{equation}
Pathwise uniqueness holds among adapted solutions in this regular class.
Moreover, $\sfS^{(N)}\to\sfS_{\cdot,\gamma}$ locally uniformly on
$[0,\infty)\times\mathbb H$ almost surely, and
\begin{equation}\label{eq:intro-stieltjes}
 \dif\sfS_{t,\gamma}\left(z\right)=\dif\sfU_t\left(z\right)
 +\left\{\frac\beta4\partial_z\left(\sfS_{t,\gamma}\left(z\right)^2\right)
 +\frac{2-\beta}{4}\partial_z^2\sfS_{t,\gamma}\left(z\right)\right\}\dif t,
\end{equation}
where $\sfU$ is a holomorphic local-martingale field with the covariations
specified in Proposition~\ref{prop:stieltjes-spde}.
\end{maintwo}
\begin{proof}
Theorem~\ref{theo sde theorem} gives the coupled particle convergence and
the strong existence and pathwise uniqueness statements.
Proposition~\ref{prop s transform convergence} identifies the limit of
the finite Stieltjes transforms, and Proposition~\ref{prop:stieltjes-spde}
gives its stochastic evolution and covariations.
\end{proof}
Conversely, motivated by the work of Huang and Zhang \cite{HUANG2026111028}, under the regularity and stopped continuation assumptions of
Proposition~\ref{prop:pole-recovery}, we show that  the poles of the Stieltjes field \eqref{eq:intro-stieltjes} recover
the shifted particle equation. This  recovery is proved there
by contour integration around the moving poles following \cite{HUANG2026111028}.

\smallskip\noindent\textbf{Extensions of previous results.}
At $\beta=2$, our determinantal construction extends the initial-condition
classes of Katori and Tanemura~\cite{Katori_2009} to configurations with
finite reciprocal-square tails, with additional limiting data retained
in $\bup$. The extension also reaches path convergence. For example,
the symmetric configuration with
\[
 \bxi\defeq\sum_{n\geq1}\left(\delta_{a_n}+\delta_{-a_n}\right),
 \qquad a_n\defeq\sqrt n\log\left(n+1\right)
\]
has finite reciprocal-square sum and square-separation bound
$O(1/\log n)$, while its reciprocal $\alpha$-moments diverge for every
$\alpha\in(1,2)$. Thus Theorem~\ref{theorem path cts} applies to its
symmetric truncations. We also establish the Markov property for
explicit non-equilibrium initial configurations and extend the lattice
relaxation theorem of~\cite{Katori_2009} to symmetric power-asymptotic
configurations with every exponent $q\in(0,2)$, including convergence
of the several-time laws.

For every $\beta\geq1$, Theorem~B extends the finite-approximation theory
within Tsai's regularity framework~\cite{MR3568040} to controlled arrays
carrying a nonzero reciprocal discrepancy. It identifies the resulting
common translation $-\beta\gamma t/2$, proves almost-sure, locally
uniform convergence of the Stieltjes transforms to $\sfS_{t,\gamma}$,
and derives the stochastic Burgers equation for this shifted field, which as far as we can tell are novel results.

\subsection{Strategy of proof}\label{sec:proof-strategy}
The starting point is the closure of normalized real-rooted polynomials
in the Laguerre--P\'olya class. The canonical function
\eqref{eq:intro-entire} retains information which vague convergence of
their zeros loses. For example, locally uniformly on $\mathbb C$,
\[
 \left(1-\frac zN\right)^N\longrightarrow\me^{-z},
 \qquad
 \left(1-\frac{z^2}{2N}\right)^N\longrightarrow\me^{-z^2/2}.
\]
All zeros escape every compact set, but their reciprocal sums and
reciprocal-square sums leave the exponential factors. These are the
roles of $\gamma_1$ and the quadratic defect $\gamma_2$. The convergent
coordinate is the total square parameter $\delta$; the defect need not
converge separately. Proposition~\ref{uniform convergence of entire}
turns convergence in $\bUpsilon_b$ into locally uniform convergence of
the canonical functions.

The finite-particle kernel makes this analytic convergence useful.
In the notation of Definition~\ref{def def of K}, with $\mi\mathbb R$
oriented upward, its contour part is
\[
\begin{aligned}
 &\BL_{\bup,\Gamma}\left(\left(s,x\right),\left(t,y\right)\right)=\frac1{\left(2\pi\mi\right)^2\sqrt{st}}
 \int_{\Gamma\times\mi\mathbb R}
 \frac{\me^{\left(w-y\right)^2/(2t)-\left(z-x\right)^2/(2s)}}{w-z}
 \frac{\mE_{\bup}\left(w\right)}{\mE_{\bup}\left(z\right)}
 \,\dif z\,\dif w.
\end{aligned}
\]
The full kernel also contains a residue correction and the time-ordered
heat term. Passing to the limit requires control of the unbounded
contours, the small denominator near the roots, and the singularity
where the contours intersect. At intersections the integral is interpreted
jointly in the two real contour parameters.

The decisive bounds are those of Proposition~\ref{bound}. For
$\bup^{(N)}\to\bup$ in $\bUpsilon_b$, write
$\mE^{(N)}=\mE_{\bup^{(N)}}$ and $\gamma_2=\gamma_2(\bup)$. For every
$\eta>0$ and all sufficiently large $N$,
\[
\begin{aligned}
 \left|\mE^{(N)}\left(z\right)\right|
 &\leq C_\eta\me^{\left(\gamma_2+\eta\right)\left|z\right|^2/2},&&z\in\mathbb C,\\
 \left|\mE^{(N)}\left(z\right)\right|^{-1}
 &\leq C_\eta\me^{\left(\gamma_2+\eta\right)\Re\left(z^2\right)/2},
 &&z\in\Gamma.
\end{aligned}
\]
We split the product into a finite core and a genus-one tail. Convergence
of $\delta$ controls the tail together with the approximating Gaussian
coefficient, while the contour geometry supplies the reciprocal bound.
The heat factors dominate below $1/\gamma_2$. At an intersection the
remaining singularity has the integrable form $(u^2+v^2)^{-1/2}$.
These bounds give kernel convergence and then finite-dimensional
convergence by Proposition~\ref{prop mgf convergence}.
Proposition~\ref{prop:gamma2-horizon-sharpness} gives examples of first-density
divergence at and beyond the horizon, making the restriction sharp for
kernel convergence.

For path convergence, the same entire functions provide interpolation weights.
Under the path hypotheses of Section~\ref{sec:paths}, write
$\mE=\mE_{f_r^b(\bxi)}$. Definition~\ref{residue function} and
Proposition~\ref{bound of res} give, for $a\in\bxi$,
\[
\begin{gathered}
 \mathfrak F_{\bxi,r}\left(a,z\right)
 =\frac{\mE\left(z\right)}{\left(z-a\right)\mE'\left(a\right)},
 \qquad\left|\mathfrak F_{\bxi,r}\left(a,z\right)\right|
 \leq C_\eta\me^{\eta\left(a^2+\left|z\right|^2\right)},\\
 \sup_N\expt\left[
 \left|\left\langle\phi,\bXi^{(N)}\left(t\right)
 -\bXi^{(N)}\left(s\right)\right\rangle\right|^Q\right]
 \leq C_{Q,T,\phi}\left|t-s\right|^{Q/2}.
\end{gathered}
\]
The quotient is continued at $z=a$, and the bounds are uniform over the
approximating family; here $Q\geq4$ is even, $\phi\in C_c^1(\mathbb R)$,
and $s,t\in[0,T]$. The difficulty is controlling derivatives at infinitely
many roots. Square separation supplies the interpolation estimate and,
with the infinite limiting configuration, forces $\gamma_2=0$.
The complex Brownian representation of Katori and Tanemura
\cite{MR3019667} converts it into the increment estimate
of Corollary~\ref{cor kol bd phi}. Standard tightness, compact containment and
the identified finite-dimensional laws then give path convergence.

For the Markov property, restarting introduces random reciprocal
parameters. Under the density and principal-value hypotheses
\eqref{eq rho bound}--\eqref{eq rho symmetry}, the finite processes
from spatial truncations satisfy, for fixed $t>0$ and $\eta>0$,
\[
 \lim_{R\to\infty}\sup_N\prob\left[
 \left|p_R\left(\bXi_{\bxi^{[N]}}\left(t\right)\right)\right|
 +s_R\left(\bXi_{\bxi^{[N]}}\left(t\right)\right)>\eta\right]=0.
\]
Sublinear one-point bounds, particle-count moments and cancellation of
the mean reciprocal tail yield this estimate and its limiting counterpart.
The density bounds follow from the entire-function kernel, by saddle
analysis for the dense class and regularized-product estimates for the
separated class. Together with local matching and the no-atom conditions,
the tail bounds give convergence of restarted parameters. Parameter
continuity and vague-Borel measurability then allow passage to the limit
in the finite Markov identity, see Proposition~\ref{prop markov from rho}.

For relaxation, the logarithm of the entire function determines the
asymptotic phase. Under \eqref{eq:intro-relaxation-data}, use the
symmetric normalization $\mE_{\bxi}$ of Section~\ref{sec:sine} and set
$L_T=\sqrt{Tn_T}$. With $\Log$ interpreted as the sum of the factor
logarithms specified in \eqref{eq:finite-phase},
\[
 F_T\left(z\right)\defeq\frac{z^2}{2}
 +\frac1{n_T}\Log\left(\frac{\mE_{\bxi}\left(L_Tz\right)}{L_T^\eps}\right)
 \longrightarrow\frac{z^2}{2}
 +\frac\pi{\sin\left(\pi q/2\right)}\left(-z^2\right)^{q/2}.
\]
This convergence holds with three derivatives locally off the real axis.
Counting-function convergence identifies the limit; logarithmic
small-root and reciprocal-square large-root estimates justify passage
through the infinite product. We use the exact finite saddles
$F_T'(\pm\mi w_T)=0$, with $w_T\to u_q=\pi\rho_q$; see
Lemmas~\ref{lem:phase-convergence} and~\ref{lem:exact-saddle}.
The qualitative power asymptotics give no convergence rate, so expansion
at a limiting saddle could leave an uncontrolled linear term after
multiplication by $n_T$. The exact saddles remove this difficulty.

The contour deformation then exposes the term that survives. Put
$T_s=T+\sigma_T^2s$ and write $\BL_{\bxi}$ for the corrected analytic
part of the kernel. Uniformly on compact sets of rescaled parameters,
\[
\begin{aligned}
 \sigma_T\BL_{\bxi}\left(\left(T_s,\sigma_TX\right),
                         \left(T_t,\sigma_TY\right)\right)
 &=R_{T;s,t}\left(X,Y\right)
   +O\left(n_T^{-1/2}\right)+O\left(L_T^{-1}\right),\\
 R_{T;s,t}\left(X,Y\right)
 &\longrightarrow\frac1{2\pi}\int_{-u_q}^{u_q}
 \me^{\left(t-s\right)k^2/2+\mi k\left(Y-X\right)}\,\dif k.
\end{aligned}
\]
Here $R_{T;s,t}$ is the crossed residue in Proposition~\ref{prop:multitime}.
Global quadratic phase bounds control the deformed double integral,
including its intersections, giving the first error; the fixed-contour
correction gives the second. Several-time perturbations have bounded
coefficients on compact parameter sets and are absorbed by these bounds.
The surviving residue gives the sine integral. The heat term scales
exactly, so the full-kernel difference converges even across the time
diagonal, yielding the several-time $\Sine$ limit.

For general $\beta\ge 1$, Tsai's idea of comparison of ordered gaps
\cite{MR3568040} supplies lower and upper envelope processes. Their
limiting gaps leave a common translation undetermined. Reciprocal-tail
estimates, uniform in $N$ and on compact time intervals, identify the
zeroth particle and the missing drift. Theorem~\ref{theo sde theorem}
and Proposition~\ref{prop s transform convergence} give, with $\sfY$
the regular unshifted solution driven by the same Brownian family,
\[
\begin{gathered}
 \sfX_i\left(t\right)=\sfY_i\left(t\right)-\frac{\beta\gamma t}{2},
 \qquad\sfS_t^{(N)}\left(z\right)\longrightarrow\sfS_t\left(z\right)+\gamma.
\end{gathered}
\]
Thus the same reciprocal-tail discrepancy produces both a common translation of the particles and an additive constant in the limiting Stieltjes transform.
Analytic convergence also controls derivatives,
but passing to martingale identities and covariations requires a common
fourth-moment localization. This gives \eqref{eq:intro-stieltjes}.
Finally, the stopped-continuation hypotheses of
Proposition~\ref{prop:pole-recovery} permit contour integration around
moving poles and recover the shifted particle equation.

\subsection{Future directions}
The construction at $\beta=2$ leaves open the problem of identifying labelled
equations for very dense initial configurations and for $\gamma_2>0$.
This requires an appropriate summation or renormalisation of the
interaction. A related question is to find an intrinsic analytic class
for the Stieltjes equation and prove initial-value uniqueness which
also identifies its evolving poles, without assuming the regular
particle representation and continuation properties of
Proposition~\ref{prop:pole-recovery}. Huang and Zhang~\cite{HUANG2026111028}
characterise the Airy$_\beta$ line ensemble through a Stieltjes-transform
equation and its pole evolution under Airy-type asymptotic assumptions.
Their result provides a model for such a characterisation, with a
different initial-data problem from the one considered here.

We expect that the extended-state and entire-function approach has
analogues for Airy \cite{Corwin2012,CorwinHammond2014,KatoriTanemura2009Airy} and Bessel \cite{KatoriTanemura2011Bessel,Wu2023Bessel} dynamics at $\beta=2$. The question is
which prescribed initial configurations and additional limiting data
can be included in such constructions at the soft and hard edges.
On the other hand, the present general-$\beta$ SDE argument uses two-sided bulk geometry and
a finite positive asymptotic mean gap, so an edge analogue requires a
different comparison and control of the reciprocal tails. It would be interesting to develop this and we leave this as an open problem.

In higher dimensions, Osada's Ginibre dynamics
\cite{Osada2012RandomMatrices,OsadaTanemura2020Tail} and the recent
Coulomb constructions of Osada-Osada~\cite{OsadaOsadaCoulomb} provide related infinite-particle
equations. Existing strong-solution results include fixed initial
configurations in specified sets of full equilibrium measure. It
would be very interesting but probably difficult to describe explicit classes of prescribed
deterministic initial configurations and determine whether their
finite approximations require additional data at infinity. As far as we can tell, the
ordered-gap comparison and the one-variable entire functions used
here have no immediate counterparts in this setting.

\subsection{Organisation of the paper}
Section~\ref{sec:preliminaries} introduces the extended parameter space
and its canonical entire functions, and constructs the determinantal
processes from the finite Dyson kernels.
Section~\ref{sec:kernels} proves continuity of the kernels and
finite-dimensional distributions, identifies the effects of the two
additional parameters, and gives the sharpness example for the time
horizon.
Section~\ref{sec:paths} establishes continuous versions and convergence
on path space under the square-separation condition, using the complex
Brownian representation.
Section~\ref{sec:density} obtains one-point density estimates for dense
symmetric initial configurations.
Section~\ref{sec:markov} gives a criterion for the Markov property and
applies it to the dense and uniformly separated classes.
Section~\ref{sec:sine} proves the rescaled long-time convergence to the
extended $\Sine$ process for the stated regularly varying configurations.
Section~\ref{sec:isde} proves the general-$\beta$ particle convergence
and identifies its shifted equations. It then derives the Stieltjes
equation and the conditional recovery of the particle dynamics from
its poles.

\paragraph{Acknowledgements} Fengyi Li is grateful to the School of Mathematics of the University of Edinburgh for funding through a PhD studentship.

\paragraph{AI disclosure} All ideas and proofs were initially generated by the human authors in the course of the past three years as part of the second author's PhD thesis. During the final stages of preparation we made use of AI tools (mainly ChatGPT 5.5, 5.6) for technical help with rigorous details in the asymptotic analysis of contour integrals in Section \ref{sec:sine} and the proof of Lemma \ref{lem:uniform-density-error}. We verified all mathematical arguments. In the past weeks we made use of AI tools (ChatGPT 6) for proofreading and general editorial editing, including language editing, of the manuscript. Any remaining mistakes are our own responsibility.

\section{Preliminaries}\label{sec:preliminaries}
Let us say a word on the notation. Boldface denotes vectors, matrices and configuration measures. Random quantities are normally written in sans serif, with bold sans serif for random vectors, matrices and configurations.

We recall the configuration space and its vague topology for the process construction.
\begin{definition}
We denote by $\mathfrak{M}$ the space of nonnegative, locally finite,
integer-valued Radon measures on $\mathbb{R}$. Thus every
$\bxi\in\mathfrak{M}$ can be written as
\begin{equation*}
    \bxi=\sum_{i\in I}\delta_{x_i},
\end{equation*}
where $I$ is at most countable, repetitions among the $x_i$ are allowed,
and $\bxi(K)<\infty$ for every compact $K\subset\mathbb{R}$. We say that
$\bxi^{(N)}$ converges vaguely to $\bxi$, and write
$\bxi^{(N)}\xrightarrow{\mathrm{vg}}\bxi$, if
\begin{equation*}
    \lim_{N\to\infty}\int_{\mathbb{R}}f\,\dif\bxi^{\left(N\right)}
    =\int_{\mathbb{R}}f\,\dif\bxi
    \qquad\text{for every }f\in\mathcal{C}_c\left(\mathbb{R}\right).
\end{equation*}
Equipped with the vague topology, $\mathfrak{M}$ is a Polish space
(see, for example, \cite{article}).
\end{definition}

Fix a complete metric $d_{\mathrm{vg}}$ compatible with the vague topology on
$\mathfrak{M}$.
The continuous-path space will carry the convergence proved in Section~4.
\begin{definition}
For $T>0$, let $\mathcal C([0,T]\to\mathfrak M)$ be the space of continuous
paths, equipped with the metric
\begin{equation*}
    d_{\mathcal C,T}\left(\phi,\psi\right)
    \defeq\sup_{t\in\left[0,T\right]}d_{\mathrm{vg}}\left(\phi\left(t\right),\psi\left(t\right)\right).
\end{equation*}
Then $\mathcal{C}([0,T]\to\mathfrak{M})$ is Polish; see
\cite[Theorem~4.19]{Kechris1995}.
\end{definition}
\begin{remark}
The topology of locally uniform convergence on
$\mathcal{C}([0,\infty)\to\mathfrak{M})$ is induced by
\begin{equation*}
    d_{\mathcal C}\left(\phi,\psi\right)
    \defeq\sum_{n=1}^{\infty}2^{-n}
      \min\left\{1,\sup_{t\in\left[0,n\right]}d_{\mathrm{vg}}\left(\phi\left(t\right),\psi\left(t\right)\right)\right\}.
\end{equation*}
Convergence for this metric is equivalent to convergence in
$\mathcal{C}([0,T]\to\mathfrak{M})$ for every $T>0$.
\end{remark}
An $\mathfrak{M}$-valued process $\bXi(\cdot)$ is called
continuous if it has a version whose paths belong to
$\mathcal{C}([0,\infty)\to\mathfrak{M})$.
\subsection{Determinantal processes and extended kernels}
We briefly recall the point-process facts used below. Let $\Lambda$ be a
locally compact Polish space and let $\lambda$ be a locally finite Radon
measure on $\Lambda$. We write $\mathfrak{M}_\Lambda$ for the corresponding
configuration space. A point process on $\Lambda$ is a probability measure
on $\mathfrak{M}_\Lambda$.

Given $\bxi=\sum_{i\in I}\delta_{x_i}$ and $n\in\mathbb{N}$, put
\[
    I_{\neq}^{n}
    \defeq\left\{\left(i_1,\ldots,i_n\right)\in I^n:i_p\neq i_q\text{ whenever }p\neq q\right\},
\]
and define the $n$-th factorial measure by
\begin{equation}\label{Xi_n}
    \bxi_n
    \defeq\sum_{\left(i_1,\ldots,i_n\right)\in I_{\neq}^{n}}
      \delta_{\left(x_{i_1},\ldots,x_{i_n}\right)}.
\end{equation}
Thus distinct indices, rather than distinct spatial locations, are used,
and multiplicities are retained. We set $\mathbb{R}_+=[0,\infty)$.

Correlation measures record factorial moments and will specify the determinantal laws.
\begin{definition}
Let $\bXi$ be a random configuration with the point-process law under consideration,
and let $\bXi_n$ be its factorial measure as in \eqref{Xi_n}.
The $n$-th correlation measure is
\[
    M_n\left(A\right)\defeq\expt\left[\bXi_n\left(A\right)\right],\qquad A\in\mathcal{B}\left(\Lambda^n\right).
\]
If $M_n$ is absolutely continuous with respect to $\lambda^{\otimes n}$,
we denote its density by $\rho_n$:
\[
    M_n\left(A\right)=\int_A\rho_n\left(x_1,\ldots,x_n\right)
    \prod_{j=1}^{n}\lambda\left(\dif x_j\right).
\]
The function $\rho_n$ is called the $n$-th correlation function.
\end{definition}
The following local growth criterion of Lenard \cite{Lenard1973} ensures that the correlation functions determine the law. We use it to obtain uniqueness from a locally bounded determinantal kernel.
\begin{proposition}\label{prop criterion rho determine bxi}
For every relatively compact Borel set $B\subset\Lambda$, put
\[
    m_n\left(B\right)\defeq\frac{1}{n!}\int_{B^n}\rho_n\left(x_1,\ldots,x_n\right)
    \prod_{j=1}^{n}\lambda\left(\dif x_j\right).
\]
If, for every such $B$, there exists $C_B<\infty$ such that
\[
    m_n\left(B\right)\leq C_B^n n^n,\qquad n\in\mathbb{N},
\]
then the correlation functions determine the point process uniquely.
In particular, it is sufficient that
\[
    \sup_{\left(x_1,\ldots,x_n\right)\in B^n}\left|\rho_n\left(x_1,\ldots,x_n\right)\right|
    \leq C_B^n n^{2n},\qquad n\in\mathbb{N}.
\]
\end{proposition}

We describe the class of point processes whose correlation functions are determinants of a kernel.
\begin{definition}\label{determinantal}
A point process on $\Lambda$ is called determinantal if there exists a
measurable, locally bounded kernel
$\K:\Lambda\times\Lambda\longrightarrow\mathbb{C}$ such that
\begin{equation}\label{eq rho and kernel}
    \rho_n\left(x_1,\ldots,x_n\right)
    =\det\left[\K\left(x_i,x_j\right)\right]_{i,j=1}^{n},\qquad n\in\mathbb{N}.
\end{equation}
\end{definition}

\begin{remark}\label{remark bounded K determines process}
If $B\subset\Lambda$ is relatively compact and
\[
    \left\|\K\right\|_B\defeq\sup_{x,y\in B}\left|\K\left(x,y\right)\right|,
\]
then Hadamard's inequality gives
\[
    \sup_{\left(x_1,\ldots,x_n\right)\in B^n}\left|\rho_n\left(x_1,\ldots,x_n\right)\right|
    \leq\left\|\K\right\|_B^n n^{n/2}.
\]
Thus a locally bounded determinantal kernel determines the process
uniquely by Proposition~\ref{prop criterion rho determine bxi}.
\end{remark}

Extended kernels encode the joint configurations at any finite collection of times.
\begin{definition}\label{def time-dependent process}
Let $\mathcal I\subset\mathbb{R}$. An $\mathfrak{M}$-valued process
$\bXi=(\bXi(t))_{t\in\mathcal I}$ is called
determinantal with extended kernel $\K$ if, for every finite set
$\mathcal A=\{t_1,\ldots,t_M\}\subset\mathcal I$, the configuration
\[
    \sum_{m=1}^{M}\delta_{t_m}\otimes\bXi\left(t_m\right)
\]
on $\mathcal A\times\mathbb{R}$ is determinantal, with reference measure
equal to counting measure on $\mathcal A$ times Lebesgue measure on
$\mathbb{R}$, and with kernel equal to the restriction of $\K$.
\end{definition}

We specify the finite-dimensional convergence notion used in the kernel criterion.
\begin{definition}\label{def weakly finite dim}
For $\mathfrak M$-valued processes $\bXi^{(N)}$ and $\bXi$ on $\mathcal I$,
we write $\bXi^{(N)}(\cdot)\xrightarrow{\mathrm{f.d.d.}}\bXi(\cdot)$ if
their joint laws at every finite collection of times in $\mathcal I$
converge weakly in the corresponding finite product of $\mathfrak M$.
\end{definition}

We record the determinantal convergence criterion used throughout the paper. It turns local uniform kernel convergence on each finite set of times into convergence of the corresponding finite-dimensional distributions; see \cite[Proposition~2.18]{https://doi.org/10.1112/jlms.70482}.
\begin{proposition}\label{prop mgf convergence}
Let $\mathcal I\subset\mathbb{R}$ and let $\bXi^{(N)}$ be
determinantal processes with extended kernels $\K^{(N)}$. Suppose that
$\K((s,\cdot),(t,\cdot))$ is continuous on $\mathbb{R}^2$ for every
$s,t\in\mathcal I$, and that, for every finite
$\mathcal A\subset\mathcal I$ and every compact $B\subset\mathbb{R}$,
\[
 \lim_{N\to\infty}
 \max_{s,t\in\mathcal A}\sup_{x,y\in B}
 \left|\K^{\left(N\right)}\left(\left(s,x\right),\left(t,y\right)\right)-\K\left(\left(s,x\right),\left(t,y\right)\right)\right|=0.
\]
Then $\K$ is the extended kernel of a determinantal process
$\bXi$, and
\[
    \bXi^{\left(N\right)}\left(\cdot\right)
    \xrightarrow{\mathrm{f.d.d.}}\bXi\left(\cdot\right).
\]
\end{proposition}

\subsection{Parameter space}
Reciprocal tails are recorded by decreasing square-summable sequences.
\begin{definition}\label{def al^pm}
For $b>0$, define
    \begin{equation*}
    \begin{aligned}
        \mathbb{W}_b \defeq \left\{\boldsymbol{\alpha}=\left(\alpha_1,\alpha_2,\dots\right)\in \left[0,\infty\right)^{\mathbb{N}} :\frac{1}{b}>\alpha_1\geq\alpha_2\geq\dots\geq0, \sum_{i=1}^\infty \alpha_i^2 <\infty\right\}.
    \end{aligned}
    \end{equation*}
We equip $\mathbb{W}_b$ with the metric
\[
 d_{\mathbb W}\left(\boldsymbol{\alpha},\boldsymbol{\beta}\right)
 \defeq\sum_{j=1}^{\infty}2^{-j}
  \min\left\{1,\left|\alpha_j-\beta_j\right|\right\}.
\]
This metric induces coordinatewise convergence. We also define
\[
 s\left(\boldsymbol{\alpha}\right)\defeq\sum_{j=1}^{\infty}\alpha_j^2.
\]
\end{definition}
We write $\mathbb Z_+=\mathbb N\cup\{0\}$.
The finite inner roots and their total multiplicity require a separate parameter space.
\begin{definition}
For $b>0$, define
    \begin{equation*}
    \begin{aligned}
        \mathbb{U}_b \defeq\left\{\mathbf{a}=\left( \left(a_1,a_2,\dots\right),k\right)\in \left[0,\infty\right)^{\mathbb{N}}\times\mathbb{Z}_+ :b>a_1\geq a_2\geq\dots\geq0, \sum_{i=1}^\infty \indi_{\left\{a_i>0\right\}}<\infty \right\}.
    \end{aligned}
    \end{equation*}
For $\mathbf a=((a_1,a_2,\ldots),k)\in\mathbb U_b$, define
\[
 m\left(\mathbf a\right)\defeq\sum_{j=1}^{\infty}\indi_{\left\{a_j>0\right\}},
 \qquad q\left(\mathbf a\right)\defeq m\left(\mathbf a\right)+k.
\]
We equip $\mathbb U_b$ with the metric
\[
 d_{\mathbb U}\left(\mathbf a,\mathbf a'\right)
 \defeq\indi_{\left\{q\left(\mathbf a\right)\neq q\left(\mathbf a'\right)\right\}}
  +\sum_{j=1}^{\infty}2^{-j}\min\left\{1,\left|a_j-a'_j\right|\right\}.
\]
Equivalently, $\mathbf a^{(N)}\to\mathbf a$ if its coordinates converge
and $q(\mathbf a^{(N)})=q(\mathbf a)$ for all sufficiently large $N$.
\end{definition}

We combine the root parameters with the linear and quadratic data of the canonical product.
\begin{definition}\label{def parameter space}
Let $b>0$. Define
\[
\bUpsilon_b\defeq\left\{\begin{aligned}
&
\left(\left(\boldsymbol{\alpha}^+,\boldsymbol{\alpha}^-\right),
 \left(\mathbf a^+,\mathbf a^-\right),\gamma_1,\delta\right)
 \in\left(\mathbb W_b\times\mathbb W_b\right)
 \times\left(\mathbb U_b\times\mathbb U_b\right)\times\mathbb R\times\mathbb R_+:\\
&\mathbf a^\pm=\left(\left(a_1^\pm,a_2^\pm,\ldots\right),k^\pm\right),\qquad
k^-=0,\qquad
\delta\geq s\left(\boldsymbol{\alpha}^+\right)+s\left(\boldsymbol{\alpha}^-\right)
\end{aligned}\right\}.
\]
We equip $\bUpsilon_b$ with the corresponding product topology.
We assign all zero roots to the positive inner coordinate, so $k^-=0$;
this removes the redundancy in the representation of $\mathfrak E_{\bup}$.
\end{definition}
Each parameter tuple determines a canonical entire function and its positive-time horizon.
\begin{definition}
    Let $b>0$ and let $\bup = \left((\boldsymbol{\alpha^+},\boldsymbol{\alpha^-}),(\mathbf{a}^+,\mathbf{a}^-),\gamma_1,\delta\right)\in\bUpsilon_b$. Denote
    \begin{equation}\label{gamma2}
        \gamma_2=\gamma_2\left(\bup\right)
        \defeq \delta-s\left(\boldsymbol{\alpha}^+\right)
                       -s\left(\boldsymbol{\alpha}^-\right)\geq0.
    \end{equation}
    We also define the time horizon
    \[
    \tau\left(\bup\right)\defeq
    \begin{cases}
      \gamma_2\left(\bup\right)^{-1},&\gamma_2\left(\bup\right)>0,\\
      \infty,&\gamma_2\left(\bup\right)=0.
    \end{cases}
    \]
    We define $\mathfrak{E}_{\bup}\colon\mathbb{C}\longrightarrow\mathbb{C},$ as
    \begin{equation}\label{kernel function}
    \begin{aligned}
        \mathfrak{E}_{\bup}\left(z\right) \defeq {\me}^{-\gamma_1 z-\frac{\gamma_2}{2}z^2}z^{k^++k^-}\prod_{i=1}^{m\left(\mathbf{a}^+\right)}\left(z-a_i^+\right)\prod_{i=1}^{m\left(\mathbf{a}^-\right)}\left(z+a_i^-\right)\prod_{i=1}^\infty {\me}^{z\alpha_i^+} \left(1-z\alpha_i^+\right)\prod_{i=1}^\infty {\me}^{-z\alpha_i^-} \left(1+z\alpha_i^-\right).
    \end{aligned}
    \end{equation}
\end{definition}
\paragraph*{Change of cutoff.}
The cutoff only specifies which roots enter the finite factors.
Let $\bup\in\bUpsilon_b$ have root configuration $\bxi$, and let $b'>b$
with $\bxi(\{\pm b'\})=0$. Write $A=\{x\in\bxi:b<|x|<b'\}$,
counting multiplicities, and $S_j=\sum_{x\in A}x^{-j}$ for $j=1,2$.
Move these finitely many roots from the reciprocal sequences to the inner
coordinates and set $\gamma_1'=\gamma_1-S_1$ and $\delta'=\delta-S_2$.
The resulting tuple $\bup'\in\bUpsilon_{b'}$ satisfies
$\gamma_2(\bup')=\gamma_2(\bup)$, $\tau(\bup')=\tau(\bup)$, and
\begin{equation}\label{eq:cutoff-change}
 \mathfrak E_{\bup'}\left(z\right)
 =\left(\prod_{x\in A}\left(-x\right)\right)\mathfrak E_{\bup}\left(z\right).
\end{equation}
Indeed, each transferred factor satisfies
$(1-z/x)\me^{z/x}=(-1/x)(z-x)\me^{z/x}$, and the change in
$\gamma_1$ cancels the linear exponentials. The constant in
\eqref{eq:cutoff-change} cancels from the ratio defining the kernel below,
so the kernel and process law are unchanged. For the embeddings and tail
sums defined below, this gives
\[
 p_{b'}\left(\bxi\right)=p_b\left(\bxi\right)-S_1,\qquad
 s_{b'}\left(\bxi\right)=s_b\left(\bxi\right)-S_2,\qquad
 \K_{f_{r-S_1}^{b'}\left(\bxi\right)}=\K_{f_r^b\left(\bxi\right)}.
\]
The identity for $p_b$ holds whenever its principal value exists.
Thus $\gamma_1-p_b(\bxi)$ is invariant when defined, whereas keeping a freely
prescribed $r$ fixed when changing $b$ need not preserve the law.
On the overlap where both cutoffs avoid the roots, this change of
coordinates is continuous: local root matching makes the transferred
multiset eventually have fixed size and convergent roots, and hence
convergent reciprocal sums. Decreasing the cutoff gives the inverse
change.
We separate the finite factors from the genus products. For $\sigma\in\{-1,1\}$,
$\boldsymbol\alpha\in\mathbb W_b$, and $\mathbf a=((a_j)_{j\geq1},k)\in\mathbb U_b$, define
\[
 \mathfrak G^{\sigma}_{\boldsymbol\alpha}\left(z\right)
 \defeq\prod_{j=1}^{\infty}\me^{\sigma\alpha_jz}\left(1-\sigma\alpha_jz\right),
 \qquad
 \mathfrak P^{\sigma}_{\mathbf a}\left(z\right)
 \defeq z^k\prod_{j=1}^{m\left(\mathbf a\right)}\left(z-\sigma a_j\right).
\]
Then
\begin{equation}\label{Decomposition}
 \mathfrak E_{\bup}\left(z\right)
 =\me^{-\gamma_1z-\gamma_2z^2/2}
  \mathfrak P^+_{\mathbf a^+}\left(z\right)\mathfrak P^-_{\mathbf a^-}\left(z\right)
  \mathfrak G^+_{\boldsymbol\alpha^+}\left(z\right)\mathfrak G^-_{\boldsymbol\alpha^-}\left(z\right).
\end{equation}
For either sign, the genus product $\mathfrak G^{\sigma}_{\boldsymbol\alpha}$ is entire
of order at most two; see \cite{levin1964distribution}. Moreover, for every
$\epsilon>0$ there is $L_\epsilon>0$ such that
\begin{equation}\label{boundofGenus}
 \left|\mathfrak G^{\sigma}_{\boldsymbol\alpha}\left(z\right)\right|\leq\me^{\epsilon\left|z\right|^2},
 \qquad \left|z\right|\geq L_\epsilon,\quad \sigma\in\left\{-1,1\right\}.
\end{equation}
The canonical entire function varies locally uniformly with its parameters. This supplies the compact-convergence input for the kernel limits in Section~3; see \cite{MR4487977}.
\begin{proposition}\label{uniform convergence of entire}
    Let $b>0$ and $({\bup}^{(N)})_{N=1}^\infty \subset \bUpsilon_b$ be a sequence, and ${\bup}\in\bUpsilon_b$. If ${\bup}^{(N)} \longrightarrow {\bup}$, then one has $\mathfrak{E}_{{\bup}^{(N)}}\longrightarrow\mathfrak{E}_{{\bup}}$ locally uniformly on $\mathbb{C}$.
\end{proposition}
\begin{proof}
The cited Laguerre--Pólya continuity result gives locally uniform
convergence of the combined Gaussian and genus-product factor under
coordinatewise convergence of $\boldsymbol{\alpha}^{\pm}$ and convergence
of $\gamma_1$ and $\delta$; the individual values $\gamma_2^{(N)}$ need
not converge. For the finite factors, convergence in $\mathbb U_b$ fixes
the total degree
$m(\mathbf a)+k$ eventually. Any positive root which disappears in the
limit converges to zero and is absorbed by the factor $z^k$. Hence
\[
 z^{k^{\left(N\right)}}\prod_{j=1}^{m\left(\mathbf a^{\left(N\right)}\right)}\left(z-a_j^{\left(N\right)}\right)
 \longrightarrow
 z^k\prod_{j=1}^{m\left(\mathbf a\right)}\left(z-a_j\right)
\]
locally uniformly. Applying this separately to the positive and negative
finite factors proves the proposition.
\end{proof}
We define the fixed contour, its correction and the full kernel used in the construction.
\begin{definition}\label{def def of K}
Let $b>0$ and $\bup\in\bUpsilon_b$. Define
\begin{equation}\label{Gamma}
\begin{aligned}
    &\Gamma_1^+ \defeq \left\{x+ \frac{\mi}{\sqrt{3}} x : x \geq \sqrt{3} \right\} \bigcup \left\{x- \frac{\mi}{\sqrt{3}} x : x \geq \sqrt{3} \right\},\\
    &\Gamma_1^- \defeq \left\{x+ \frac{\mi}{\sqrt{3}} x : x \leq -\sqrt{3} \right\} \bigcup \left\{x- \frac{\mi}{\sqrt{3}} x : x \leq -\sqrt{3} \right\},\\
    &\Gamma_2 \defeq \left\{x+\mi: -\sqrt{3}\leq x\leq \sqrt{3}\right\}\bigcup\left\{x-\mi: -\sqrt{3}\leq x \leq \sqrt{3}\right\},\\ 
    &\Gamma_1 \defeq \Gamma_1^+ \bigcup \Gamma_1^-,\quad\quad \Gamma \defeq \Gamma_1 \bigcup \Gamma_2.
\end{aligned}
\end{equation}
Orient the upper component of $\Gamma$ from $+\infty$ to $-\infty$
and the lower component from $-\infty$ to $+\infty$. For
$L\geq\sqrt{3}$, let $\Gamma_L$ be the positively oriented closed contour
obtained from $\Gamma\cap\{|\Re z|\leq L\}$ by adding the two vertical
end segments. For $w\notin\Gamma$, define
\[
 \operatorname{Ind}_{\Gamma}\left(w\right)
 \defeq\lim_{L\to\infty}\frac{1}{2\pi\mi}
   \oint_{\Gamma_L}\frac{\dif z}{z-w}.
\]
In particular,
\[
 \operatorname{Ind}_{\Gamma}\left(\mi u\right)=\indi_{\left\{\left|u\right|<1\right\}}
\]
for Lebesgue-almost every $u\in\mathbb R$.

For $s,t\in(0,\tau(\bup))$, $x,y\in\mathbb R$, and an oriented
sub-contour $\Theta\subset\Gamma$, set
\begin{equation}\label{kernel_omega}
\BL_{\bup,\Theta}\left(\left(s,x\right),\left(t,y\right)\right)
\defeq\frac{1}{2\pi\mi}\int_{\Theta\times\mathbb R}
 \mathfrak h_s\left(x,z\right)\mathfrak h_t\left(-\mi y,u\right)
 \frac{1}{\mi u-z}
 \frac{\mathfrak E_{\bup}\left(\mi u\right)}{\mathfrak E_{\bup}\left(z\right)}
 \,\dif z\,\dif u .
\end{equation}
When $\Theta$ contains either crossing point, this is understood as a
joint two-dimensional improper integral, not as an iterated
principal-value integral. Here
\[
 \mathfrak h_s\left(x,z\right)\defeq\frac{1}{\sqrt{2\pi s}}
 \me^{-\left(x-z\right)^2/\left(2s\right)}.
\]
Define the parameter-independent residue correction
\[
\begin{aligned}
\Rc_{\Gamma}\left(\left(s,x\right),\left(t,y\right)\right)
&\defeq\int_{\mathbb R}\operatorname{Ind}_{\Gamma}\left(\mi u\right)
  \mathfrak h_s\left(x,\mi u\right)\mathfrak h_t\left(-\mi y,u\right)\,\dif u\\
&=\int_{-1}^{1}
  \mathfrak h_s\left(x,\mi u\right)\mathfrak h_t\left(-\mi y,u\right)\,\dif u ,
\end{aligned}
\]
the analytic part
\[
 \BL_{\bup}\defeq\BL_{\bup,\Gamma}+\Rc_{\Gamma},
\]
and the full kernel
\[
 \K_{\bup}\left(\left(s,x\right),\left(t,y\right)\right)
 \defeq\BL_{\bup}\left(\left(s,x\right),\left(t,y\right)\right)
  -\indi_{\left\{s>t\right\}}\mathfrak h_{s-t}\left(x,y\right).
\]
See Figure~\ref{Gammagraph} for the graph of $\Gamma$.
\end{definition}
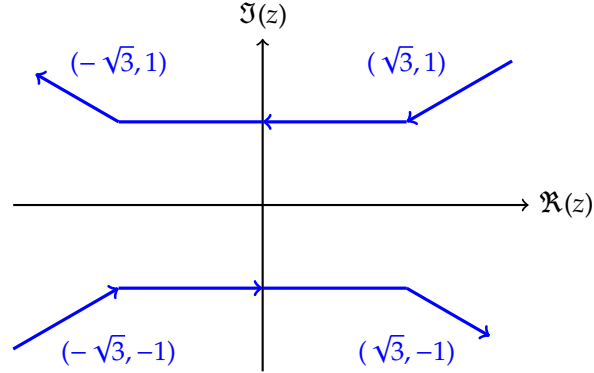
\begin{figure}[ht]
\centering
\begin{tikzpicture}[scale=1.1]
  \def\a{1.732}
  \def\b{0.57735}
  \def\xmax{3}  

  \draw[->, thick] (-\xmax, 0) -- (\xmax+0.2, 0) node[right] {$\Re(z)$};
  \draw[->, thick] (0, -2) -- (0, 2) node[above] {$\Im(z)$};

  \draw[very thick, blue, ->, domain=\xmax:\a, samples=100, variable=\x]
    plot ({\x}, {\b*(\x)});
  \draw[very thick, blue, ->, domain=-\a:-\a-1, samples=100, variable = \x]
    plot ({\x}, {-\b*(\x)});
  \draw[very thick, blue, ->, domain=\a:\a+1, samples=100, variable=\x]
    plot ({\x}, {-\b*(\x)});
  \draw[very thick, blue, ->, domain = -\xmax:-\a, samples=100, variable = \x]
   plot ({\x}, {\b*(\x)});
  \draw[very thick, blue,->] (\a, 1) -- (0, 1);
  \draw[very thick, blue] (0, 1) -- (-\a, 1);
  \draw[very thick, blue, ->] (-\a, -1) -- (0, -1);
  \draw[very thick, blue] (0, -1) -- (\a, -1);
  \node[blue, above] at (\a,\a-0.3) {$(\sqrt{3},1)$};
  \node[blue, below] at (\a,-\a+0.3) {$(\sqrt{3},-1)$};
  \node[blue, above] at (-\a,\a-0.3) {$(-\sqrt{3},1)$};
  \node[blue, below] at (-\a,-\a+0.3) {$(-\sqrt{3},-1)$};
\end{tikzpicture}
\captionsetup{justification=raggedright,singlelinecheck=false}
\caption{Contour $\Gamma$, with sloped parts $\Gamma_1$ and horizontal parts $\Gamma_2$.}
\label{Gammagraph}
\end{figure}

\begin{remark}
Below we prove that $\K_{\bup}$ is the kernel of a determinantal process
on $(0,\tau(\bup))$ and is obtained as a finite-dimensional Dyson limit.
The estimates establishing convergence of the joint contour integral are
proved in the next section.
\end{remark}
To embed configurations in the parameter space, we separate their inner roots from their reciprocal tails.
\begin{definition}\label{reorder xi}
Let $b>0$ and $\bxi\in\mathfrak M$ with $\bxi(\{0\})=k$ and
$\bxi(\{\pm b\})=0$. Put
$m_{\bxi,b}^+=\bxi((0,b))$, $m_{\bxi,b}^-=\bxi((-b,0))$,
$n_{\bxi,b}^+=\bxi((b,\infty))$, and
$n_{\bxi,b}^-=\bxi((-\infty,-b))$. Then
\begin{equation}\label{eq ordered bxi}
 \bxi=k\delta_0+\sum_{i=1}^{m_{\bxi,b}^+}\delta_{x_{i,<}^+}
 +\sum_{i=1}^{m_{\bxi,b}^-}\delta_{x_{i,<}^-}
 +\sum_{i=1}^{n_{\bxi,b}^+}\delta_{x_{i,>}^+}
 +\sum_{i=1}^{n_{\bxi,b}^-}\delta_{x_{i,>}^-}.
\end{equation}
Order the nonempty particle lists so that
    \begin{equation*}
    \begin{aligned}
        &0<x^+_{m^+_{\bxi,b},<}\leq \dots\leq x^+_{1,<}<b,\quad
        -b<x^-_{1,<}\leq\dots\leq x^-_{m^-_{\bxi,b},<}<0,\\
        &b<x^+_{1,>} \leq x_{2,>}^+\leq \dots,\quad -b>x^-_{1,>}\geq x^-_{2,>}\geq \dots.
    \end{aligned}
    \end{equation*}
    We next define
    \begin{equation}\label{M_b}
        \mathfrak{M}_b \defeq \left\{{\bxi}\in\mathfrak{M}: {\bxi}\left(\left\{\pm b\right\}\right)=0,\quad \sum_{i=1}^{n_{{\bxi},b}^+}\left(\frac{1}{x_{i,>}^+}\right)^2 + \sum_{i=1}^{n_{{\bxi},b}^-}\left(\frac{1}{x_{i,>}^-}\right)^2 <\infty\right\}. 
    \end{equation}
\end{definition}
The following embedding records these root data together with a prescribed reciprocal parameter.
\begin{definition}
    For $b>0$ and $r\in\mathbb R$, define the parameter map
    \begin{equation*}
f_r^b\colon\mathfrak{M}_b\longrightarrow\bUpsilon_b,\quad {\bxi}
    \mapsto \left(\left(\boldsymbol{\alpha^+}\left(\bxi\right),\boldsymbol{\alpha^-}\left(\bxi\right)\right),
    \left(\mathbf{a}^+\left(\bxi\right),\mathbf{a}^-\left(\bxi\right)\right),r,\de\left(\bxi\right)\right),
    \end{equation*}
    by
    \begin{equation}\label{eq f_r}
    \begin{aligned}
        &\alpha_j^+\left(\bxi\right)\defeq
        \begin{cases}1/x_{j,>}^+,&j\leq n_{\bxi,b}^+,\\0,&j>n_{\bxi,b}^+,\end{cases}
        \qquad
        \alpha_j^-\left(\bxi\right)\defeq
        \begin{cases}-1/x_{j,>}^-,&j\leq n_{\bxi,b}^-,\\0,&j>n_{\bxi,b}^-,\end{cases}\\ 
        &\mathbf{a}^+\left(\bxi\right) \defeq \left(\left(x_{1,<}^+,\dots,x_{m_{{\bxi},b}^+,<}^+,0,0\dots\right), k\right),\quad \mathbf{a}^-\left(\bxi\right) \defeq \left(\left(-x_{1,<}^-,\dots,-x_{m_{{\bxi},b}^-,<}^-,0,0,\dots\right), 0\right),\\
        &\gamma_1=r,\quad \delta\left(\bxi\right)
        \defeq s\left(\boldsymbol{\alpha}^+\right)+s\left(\boldsymbol{\alpha}^-\right).
    \end{aligned}
    \end{equation}
The dependence of the root coordinates on $b$ is suppressed in the notation.
\end{definition}
Thus $f_r^b(\bxi)$ belongs to the canonical slice $k^-=0$, and
\[
    \gamma_2\left(f_r^b\left(\bxi\right)\right)=0.
\]
For example, let ${\bxi} = \sum_{i=-N}^N \delta_i$. One has $\mE_{f_0^{1/2}({\bxi})}(z) = z\prod_{i=1}^N(1-\frac{z}{i})\prod_{i=1}^N(1+\frac{z}{i})$, which is a polynomial.
For $b>0$ and $\bxi\in\mathfrak M$, write
\begin{equation}\label{eq:tail-sums}
 p_b\left(\bxi\right)\defeq
 \lim_{\substack{R\to\infty\\R\in\mathbb N}}
 \int_{\{b<|x|\leq R\}}\frac{\bxi(\dif x)}{x},\qquad
 s_b\left(\bxi\right)\defeq
 \int_{\{|x|>b\}}\frac{\bxi(\dif x)}{x^2},
\end{equation}
where $p_b$ is defined when the displayed limit is finite, and $s_b$ may
equal $+\infty$. For finite $\bxi$, the first expression is
simply the sum $\sum_{|x_i|>b}x_i^{-1}$; it selects the parameter of
the unshifted Dyson kernel.
Deforming a separated finite-root contour to the fixed contour creates the following residue. This identifies the correction needed to recover the finite Dyson kernel.
\begin{lemma}\label{lemma finite contour correction}
Let $\bxi$ be a finite configuration with $\bxi(\{0\})=0$, choose
$b>0$ with $\bxi(\{\pm b\})=0$, and let $C_{\bxi}$ be a positively
oriented union of small loops enclosing the support points of $\bxi$
and disjoint from $\mi\mathbb R$. Products over $a\in\bxi$ count
multiplicity. Then, for $s,t>0$ and $x,y\in\mathbb R$,
\[
\begin{aligned}
&\frac{1}{2\pi\mi}
 \int_{C_{\bxi}\times\mathbb R}
 \mathfrak h_s\left(x,z\right)\mathfrak h_t\left(-\mi y,u\right)
 \frac{1}{\mi u-z}
 \prod_{a\in\bxi}\frac{a-\mi u}{a-z}
 \,\dif z\,\dif u\\
&\hspace{25mm}
=\BL_{f_{p_b\left(\bxi\right)}^b\left(\bxi\right),\Gamma}\left(\left(s,x\right),\left(t,y\right)\right)
 +\Rc_{\Gamma}\left(\left(s,x\right),\left(t,y\right)\right).
\end{aligned}
\]
\end{lemma}
\begin{proof}
First truncate $\Gamma$ to $\Gamma_L$, with $L$ larger than every
$|a|$, and fix $u$ away from the two crossing values. During the
deformation of $C_{\bxi}$ to $\Gamma_L$, the only additional pole is
$z=\mi u$. Its residue is
\[
 -\mathfrak h_s\left(x,\mi u\right)\mathfrak h_t\left(-\mi y,u\right).
\]
The contour $\Gamma_L$ encloses this pole precisely when $|u|<1$.
Thus the residue theorem gives, for the analytic integrand with the
$u$-dependent heat factor included,
\[
 \frac{1}{2\pi\mi}\int_{\Gamma_L}\left(\cdots\right)\,\dif z
 =\frac{1}{2\pi\mi}\int_{C_{\bxi}}\left(\cdots\right)\,\dif z
  -\indi_{\left\{\left|u\right|<1\right\}}
   \mathfrak h_s\left(x,\mi u\right)\mathfrak h_t\left(-\mi y,u\right).
\]
Rearranging and integrating in $u$ gives the displayed identity.
For a closing segment $z=\pm L+\mi v$, $|v|\leq L/\sqrt3$,
\[
 \Re\!\left[-\frac{\left(x-z\right)^2}{2s}\right]
 \leq -\frac{L^2}{3s}+\frac{\left|x\right|L}{s}.
\]
The remaining finite-configuration factors have at most polynomial
growth in $L$ and $u$, whereas $\mathfrak h_t(-\mi y,u)$ is Gaussian
in $u$. Consequently the closing-segment integrals vanish as
$L\to\infty$, and dominated convergence yields the integral over
$\Gamma$. The singularities at $(z,u)=(\mi,1)$ and $(-\mi,-1)$ are
locally integrable in the joint two-dimensional sense.
\end{proof}

We identify the full kernel with finite Dyson Brownian motion, including initial configurations with multiple points. These finite systems will supply the approximations for the general parameter construction.
\begin{proposition}\label{finitedysonmodel}
Let $\bxi\in\mathfrak M$ have finite mass $|\bxi|$, and choose $b>0$
with $\bxi(\{\pm b\})=0$. The $\beta=2$ Dyson Brownian motion started
from $\bxi$, understood as its entrance-law extension when $\bxi$ has
multiple points, is determinantal at positive times. Its extended kernel
is $\K_{f_{p_b(\bxi)}^b(\bxi)}$, independently of the admissible choice
of $b$.
\end{proposition}
\begin{proof}
For a simple configuration with no particle at zero,
Proposition~2.1 of \cite{Katori_2009} gives
\begin{equation}\label{eq kernel finite dim}
\begin{aligned}
\K_{\bxi}\left(\left(s,x\right),\left(t,y\right)\right)
={}&\frac{1}{2\pi\mi}
 \int_{C_{\bxi}\times\mathbb R}
 \mathfrak h_s\left(x,z\right)\mathfrak h_t\left(-\mi y,u\right)
 \frac{1}{\mi u-z}
 \prod_{a\in\bxi}\frac{a-\mi u}{a-z}
 \,\dif z\,\dif u\\
&-\indi_{\left\{s>t\right\}}\mathfrak h_{s-t}\left(x,y\right).
\end{aligned}
\end{equation}
Directly from the definition of $f_{p_b(\bxi)}^b$,
\[
 \frac{\mathfrak E_{f_{p_b\left(\bxi\right)}^b\left(\bxi\right)}\left(\mi u\right)}
      {\mathfrak E_{f_{p_b\left(\bxi\right)}^b\left(\bxi\right)}\left(z\right)}
 =\prod_{a\in\bxi}\frac{a-\mi u}{a-z}.
\]
Lemma~\ref{lemma finite contour correction} identifies
\eqref{eq kernel finite dim} with $\K_{f_{p_b(\bxi)}^b(\bxi)}$.

For multiple nonzero roots, choose simple real perturbations
$\bxi^{(m)}\to\bxi$, preserving the number of particles and avoiding
$\{0,\pm b\}$. The finite-particle continuity statement immediately
preceding Proposition~2.1 in \cite{Katori_2009} shows that Dyson
Brownian motion started from $\bxi$, including multiplicities, is the
positive-time weak limit of these noncolliding systems. Moreover,
$f_{p_b(\bxi^{(m)})}^b(\bxi^{(m)})\to
f_{p_b(\bxi)}^b(\bxi)$, so Proposition \ref{Uniformconvergence} gives locally uniform convergence of the kernels on each finite set of positive times. Proposition \ref{prop mgf convergence} show that the limiting kernel is determinantal with kernel $\K_{f_{p_b(\bxi)}^b(\bxi)}$, independently of the chosen perturbation.  Equivalently, the loops in
\eqref{eq kernel finite dim} coalesce into loops around the corresponding
higher-order poles.  If $\bxi(\{0\})=k>0$, use distinct positive real
perturbations of the $k$ zero atoms as well; the same argument gives the
shrinking-loop prescription at zero. Finally, the displayed ratio is
intrinsic to $\bxi$, so the result does not depend on $b$.
\end{proof}
Every admissible parameter can be approximated by finite Dyson configurations. Their limits construct the determinantal process associated with the canonical kernel.
\begin{proposition}\label{well_defined process}
Let $\bup\in\bUpsilon_b$. Then $\K_{\bup}$ is the extended kernel of a
determinantal process
$(\bXi_{\bup}(t))_{t\in(0,\tau(\bup))}$. Moreover, there
exists a sequence of finite configurations $(\bxi^{(n)})_{n\geq1}$
whose $\beta=2$ Dyson models converge to $\bXi_{\bup}$ in
finite-dimensional distributions on every compact subinterval of
$(0,\tau(\bup))$.
\end{proposition}
\begin{proof}
Write
\[
 \bup=\left(\left(\boldsymbol{\alpha}^+,\boldsymbol{\alpha}^-\right),
       \left(\mathbf a^+,\mathbf a^-\right),\gamma_1,\delta\right),
 \qquad k^-=0,
\]
and set
\[
 c_n\defeq\gamma_1-
 \left(\sum_{j=1}^{n}\alpha_j^+
       -\sum_{j=1}^{n}\alpha_j^-\right),
 \qquad
 M_n\defeq\left\lceil n\left(1+\left|c_n\right|+c_n^2\right)\right\rceil,
 \qquad
 v_n\defeq\sqrt{\frac{\gamma_2}{2n}}.
\]
For $a\geq0$ and $m\in\mathbb Z_+$, let $[a]_m$ denote the multiset of
$m$ copies of $a$, with copies of zero omitted. Define finite multisets
\[
\begin{aligned}
 A_n^+
 &\defeq\left\{\alpha_j^+:1\leq j\leq n,\ \alpha_j^+>0\right\}
   \uplus\left[\frac{\max\left(c_n,0\right)}{M_n}\right]_{M_n}
   \uplus\left[v_n\right]_n,\\
 A_n^-
 &\defeq\left\{\alpha_j^-:1\leq j\leq n,\ \alpha_j^->0\right\}
   \uplus\left[\frac{\max\left(-c_n,0\right)}{M_n}\right]_{M_n}
   \uplus\left[v_n\right]_n.
\end{aligned}
\]
Rearrange each multiset in decreasing order. After discarding finitely
many $n$, all its elements lie in $(0,1/b)$. Define
\[
\begin{aligned}
\bxi^{\left(n\right)}
\defeq{}&\sum_{j=1}^{m\left(\mathbf a^+\right)}\delta_{a_j^+}
 +\sum_{j=1}^{m\left(\mathbf a^-\right)}\delta_{-a_j^-}
 +\left(k^++k^-\right)\delta_0+\sum_{a\in A_n^+}\delta_{a^{-1}}
 +\sum_{a\in A_n^-}\delta_{-a^{-1}},
\end{aligned}
\]
where multiset elements are counted with multiplicity. The asymmetric
block gives
\[
 \sum_{a\in A_n^+}a-\sum_{a\in A_n^-}a=\gamma_1,
\]
and hence $p_b(\bxi^{(n)})=\gamma_1$. Moreover,
\[
\begin{aligned}
 \sum_{a\in A_n^+\uplus A_n^-}a^2
={}&\sum_{j=1}^{n}
 \left(\left(\alpha_j^+\right)^2+\left(\alpha_j^-\right)^2\right)
 +\frac{c_n^2}{M_n}+\gamma_2
 \longrightarrow\delta.
\end{aligned}
\]
Indeed,
\[
 \frac{\left|c_n\right|}{M_n}\leq\frac1n,\qquad
 \frac{c_n^2}{M_n}\leq\frac1n,\qquad
 v_n\longrightarrow0.
\]
Set $d_n=\max\{|c_n|/M_n,v_n\}$, so that $d_n\to0$. If
$\alpha_j^\pm>0$, then for all sufficiently large $n$ the added entries
are smaller than $\alpha_j^\pm/2$ and the first $j$ target entries are
present. If $\alpha_j^\pm=0$, the $j$-th rearranged entry is at most
$d_n$. Thus decreasing rearrangement preserves convergence of every
fixed coordinate, and
\[
 \bup^{\left(n\right)}
 \defeq f_{p_b\left(\bxi^{\left(n\right)}\right)}^b\left(\bxi^{\left(n\right)}\right)
 \longrightarrow\bup
 \qquad\text{in }\bUpsilon_b.
\]
By Proposition~\ref{finitedysonmodel}, $\K_{\bup^{(n)}}$ is the finite
Dyson kernel. Proposition~\ref{Uniformconvergence}, together with
Proposition~\ref{prop mgf convergence}, yields a determinantal process
with kernel $\K_{\bup}$ and
\[
 \bXi_{\bup^{\left(n\right)}}\left(\cdot\right)
 \xrightarrow{\mathrm{f.d.d.}}
 \bXi_{\bup}\left(\cdot\right)
\]
on $(0,\tau(\bup))$, since Proposition~\ref{prop mgf convergence}
applies to every finite set of times in this interval.
\end{proof}
\section{Correlation kernels}\label{sec:kernels}
Recall the definitions of $\gamma_2$ and $\tau(\bup)$ from
\eqref{gamma2}, and the contour integral $\BL_{\bup,\Gamma}$, the
residue correction $\Rc_\Gamma$, and the full kernel
$\K_{\bup}$ from Definition~\ref{def def of K}.

The next theorem gives continuity of the determinantal laws under convergence in the parameter space. Its kernel convergence will also identify the limits of the approximations used later.
\begin{theorem}\label{MainTheorem}
Let $b>0$ and suppose that ${\bup}^{\UN}\longrightarrow\bup$ in
$\bUpsilon_b$. Put $\tau_N=\tau({\bup}^{\UN})$ and
$\tau=\tau(\bup)$. For every finite set $\mathcal A\subset(0,\tau)$ and
every compact set $B\subset\mathbb R$, one has
\begin{equation}\label{eq main local kernel convergence}
 \lim_{N\to\infty}\max_{s,t\in\mathcal A}\sup_{x,y\in B}
 \left|\K_{{\bup}^{\UN}}\left(\left(s,x\right),\left(t,y\right)\right)
       -\K_{\bup}\left(\left(s,x\right),\left(t,y\right)\right)\right|=0.
\end{equation}
Moreover, $\mathcal A\subset(0,\tau_N)$ for all sufficiently large $N$.
Consequently, the associated determinantal processes satisfy
\[
 {\bXi}^{\UN}\left(\cdot\right)
 \xrightarrow[]{\mathrm{f.d.d.}}{\bXi}\left(\cdot\right)
 \qquad\hbox{on }\left(0,\tau\right).
\]
\end{theorem}

\begin{remark}
If $\gamma_2(\bup)=0$, then $\tau(\bup)=\infty$, and hence the theorem
holds on every finite time interval.
\end{remark}

\subsection{Bounds for the entire functions}\label{section bd for function}

We first bound the canonical entire functions and their reciprocals on the integration contour. Combined with the heat factors, these bounds provide uniform domination for kernel convergence.
\begin{proposition}\label{bound}
Suppose that ${\bup}^{\UN}\longrightarrow\bup$ in $\bUpsilon_b$, and
write $\gamma_2=\gamma_2(\bup)$. For every $\epsilon>0$ there are
$C_\epsilon<\infty$ and $N_\epsilon\in\mathbb N$ such that
\begin{equation}\label{bdinequality}
\begin{aligned}
 \sup_{N\geq N_\epsilon}\left|\mathfrak E_{{\bup}^{\UN}}\left(z\right)\right|
 &\leq C_\epsilon
       \exp\left(\frac{\gamma_2+\epsilon}{2}\left|z\right|^2\right),
 &&z\in\mathbb C,\\
 \sup_{N\geq N_\epsilon}
 \left|\frac{1}{\mathfrak E_{{\bup}^{\UN}}\left(z\right)}\right|
 &\leq C_\epsilon
       \exp\left(\frac{\gamma_2+\epsilon}{2}\Re\left(z^2\right)\right),
 &&z\in\Gamma .
\end{aligned}
\end{equation}
The same bounds, after increasing $C_\epsilon$, hold for
$\mathfrak E_{\bup}$.
\end{proposition}

Write $\gamma_2^{\UN}=\gamma_2(\bup^{\UN})$.

The reciprocal-square tails and the Gaussian coefficient admit a common upper bound under parameter convergence. This controls their combined contribution even when the Gaussian coefficients do not converge separately.
\begin{lemma}\label{lm bound for sum}
For every $\epsilon>0$ there are $N_\epsilon,M_\epsilon\in\mathbb N$
such that
\[
 \sup_{N\geq N_\epsilon}
 \left\{
 \sum_{j=M_\epsilon}^{\infty}\left(\alpha_j^{+,\UN}\right)^2+
 \sum_{j=M_\epsilon}^{\infty}\left(\alpha_j^{-,\UN}\right)^2+
 \gamma_2^{\UN}\right\}
 \leq\gamma_2+\epsilon .
\]
\end{lemma}

\begin{proof}
Choose $M_\epsilon$ so that
\[
 \sum_{j=M_\epsilon}^{\infty}\left(\alpha_j^+\right)^2+
 \sum_{j=M_\epsilon}^{\infty}\left(\alpha_j^-\right)^2\leq\epsilon/2.
\]
For every $N$ the expression inside braces equals
\[
 \delta^{\UN}
 -\sum_{j=1}^{M_\epsilon-1}\left(\alpha_j^{+,\UN}\right)^2
 -\sum_{j=1}^{M_\epsilon-1}\left(\alpha_j^{-,\UN}\right)^2.
\]
Coordinatewise convergence of the displayed finite sums and of
$\delta^{\UN}$ makes this at most
\[
 \delta
 -\sum_{j=1}^{M_\epsilon-1}\left(\alpha_j^+\right)^2
 -\sum_{j=1}^{M_\epsilon-1}\left(\alpha_j^-\right)^2+\epsilon/2
 \leq\gamma_2+\epsilon
\]
for all sufficiently large $N$.
\end{proof}

We use the elementary inequalities
\begin{equation}\label{naive}
 \left|\left(1+z\right)\me^{-z}\right|\leq \me^{\left|z\right|^2/2},
 \qquad \left|\me^z\right|\leq \me^{\left|z\right|}
\end{equation}

On the sloped rays, a single genus-one factor satisfies a useful reciprocal bound. Applying it to the tail product gives the denominator estimate below.
\begin{lemma}\label{denolemma}
If $z\in\mathbb C$ satisfies $|\Re z|=\sqrt3|\Im z|$, then
\[
 \left|\me^z\left(1-z\right)\right|^{-1}\leq
 \exp\left(\frac{\Re\left(z^2\right)}{2}\right).
\]
\end{lemma}

\begin{proof}
It suffices to write $z=a(1+\mi/\sqrt3)$; the other ray is identical.
Then
\[
 \ell\left(a\right)\defeq\log\left|\me^z\left(1-z\right)\right|+\frac{\Re\left(z^2\right)}2
 =a+\frac{a^2}{3}
  +\frac12\log\left(1-2a+\frac43a^2\right)
\]
has derivative
\[
 \ell'\left(a\right)=\frac{8a^3}{3\left(4a^2-6a+3\right)}.
\]
The denominator is strictly positive, so $\ell$ has its global minimum
$\ell(0)=0$ at $a=0$.
\end{proof}

\begin{proof}[Proof of Proposition \ref{bound}]
Apply Lemma~\ref{lm bound for sum} with error $\epsilon/2$, and put
\[
 \mathfrak U_\epsilon^{\UN}\left(z\right)\defeq
 \me^{-\gamma_2^{\UN}z^2/2}
 \prod_{j=M_{\epsilon/2}}^\infty
 \me^{z\alpha_j^{+,\UN}}\left(1-z\alpha_j^{+,\UN}\right)
 \me^{-z\alpha_j^{-,\UN}}\left(1+z\alpha_j^{-,\UN}\right).
\]
By \eqref{naive} and Lemma~\ref{denolemma}, for all sufficiently large $N$,
\[
 \left|\mathfrak U_\epsilon^{\UN}\left(z\right)\right|
 \leq\me^{(\gamma_2+\epsilon/2)|z|^2/2}\quad(z\in\mathbb C),\qquad
 \left|\mathfrak U_\epsilon^{\UN}\left(z\right)\right|^{-1}
 \leq\me^{(\gamma_2+\epsilon/2)\Re(z^2)/2}\quad(z\in\Gamma_1).
\]
The remaining factor
$\mathfrak E_{{\bup}^{\UN}}/\mathfrak U_\epsilon^{\UN}$ has uniformly
bounded polynomial degree, bounded finite-product parameters, and bounded
linear exponential parameter $\gamma_1^{\UN}$. Consequently, after
increasing $N_\epsilon$ and changing $C_\epsilon$,
\begin{align}
 \sup_{N\geq N_\epsilon}
 \left|\frac{\mathfrak E_{{\bup}^{\UN}}\left(z\right)}
              {\mathfrak U_\epsilon^{\UN}\left(z\right)}\right|
 &\leq C_\epsilon \me^{\epsilon\left|z\right|^2/4},
 &&z\in\mathbb C,\label{eq finite core numerator}\\
 \sup_{N\geq N_\epsilon}
 \left|\frac{\mathfrak U_\epsilon^{\UN}\left(z\right)}
              {\mathfrak E_{{\bup}^{\UN}}\left(z\right)}\right|
 &\leq C_\epsilon \me^{\epsilon\Re\left(z^2\right)/4},
 &&z\in\Gamma_1.\label{eq finite core denominator}
\end{align}
Indeed, the polynomial factors have uniformly bounded degree, all
remaining exponential factors have uniformly bounded linear type, and
polynomial and linear-exponential growth can be absorbed into an
arbitrarily small quadratic exponential. On the sloped rays the
finite canonical factors are uniformly separated from zero; for
example,
\[
 \left|1-\alpha a\left(1\pm\mi/\sqrt3\right)\right|^2
 =\left(1-\alpha a\right)^2+\frac{\alpha^2a^2}{3}\geq\frac14 .
\]
Equations \eqref{eq finite core numerator} and
\eqref{eq finite core denominator}, with the error split among the
constituent estimates, prove \eqref{bdinequality} on
$\mathbb C$ and $\Gamma_1$. Finally, $\Gamma_2$ is compact and is
disjoint from the real zero set. Proposition
\ref{uniform convergence of entire} therefore gives
\[
 \sup_{N\geq N_\epsilon}\sup_{z\in\Gamma_2}
 \left|\mathfrak E_{{\bup}^{\UN}}\left(z\right)\right|^{-1}<\infty.
\]
The limiting estimates follow in the same way, or by taking limits.
\end{proof}

\subsection{Pointwise and locally uniform convergence of the kernel}

For $\sigma\in\{-1,1\}$ write
\[
 \Gamma_2^\sigma\defeq\left\{q+\mi\sigma:-\sqrt3\leq q\leq\sqrt3\right\}.
\]
The integrals over $\Gamma_2^\sigma\times\mathbb R$ below are joint
Lebesgue integrals. Values of a one-dimensional slice at the single
crossing point are immaterial.

At a contour crossing, the singularity is integrable in the two real integration variables. The quantitative bound below makes small crossing neighbourhoods uniformly negligible.
\begin{lemma}\label{integralbiliylemma}
For every $L>0$ and $\sigma\in\{-1,1\}$,
\[
 \int_{-\sqrt3}^{\sqrt3}\int_{-L}^{L}
 \frac{\dif u\,\dif q}{\sqrt{q^2+\left(u-\sigma\right)^2}}<\infty.
\]
More precisely, for $0<\rho<1$,
\begin{equation}\label{eq crossing neighborhood}
 \int_{\left|q\right|<\rho}\int_{\left|u-\sigma\right|<\rho}
 \frac{\dif u\,\dif q}{\sqrt{q^2+\left(u-\sigma\right)^2}}
 \leq2\pi\sqrt2\,\rho .
\end{equation}
\end{lemma}

\begin{proof}
The integrand is $1/r$ in the two real coordinates
$(q,u-\sigma)$. The square in \eqref{eq crossing neighborhood} is
contained in the disc of radius $\sqrt2\rho$, and polar integration
gives the stated bound. The remainder of the bounded rectangle is
bounded away from the singular point.
\end{proof}

We now combine the product bounds with the heat factors to dominate the kernel integrand and its first spatial derivatives. These estimates control both the unbounded tails and the contour crossings.
\begin{lemma}\label{kernel majorant}
Let $0<a<T<\tau(\bup)$ and let $B\subset\mathbb R$ be compact. Set
$Q=[a,T]^2\times B^2$, and define
\[
 \mathcal J_N\left(s,x,t,y;z,u\right)
 \defeq
 \mathfrak h_s\left(x,z\right)\mathfrak h_t\left(-\mi y,u\right)
 \frac{\mathfrak E_{{\bup}^{\UN}}\left(\mi u\right)}
      {\left(\mi u-z\right)\mathfrak E_{{\bup}^{\UN}}\left(z\right)}.
\]
There exist $c,C>0$ and $N_0\in\mathbb N$ such that, uniformly on $Q$,
for $N\geq N_0$ and
$D\in\{1,\partial_x,\partial_y\}$,
\begin{align}
 \left|D\mathcal J_N\left(s,x,t,y;z,u\right)\right|
 &\leq C\left(1+\left|u\right|+\left|z\right|\right)^2
 \me^{-cu^2-c\left|z\right|^2+C\left|z\right|},
 &&z\in\Gamma_1,\label{eq majorant sloped}\\
 \left|D\mathcal J_N\left(s,x,t,y;q+\mi\sigma,u\right)\right|
 &\leq
 \frac{C\left(1+\left|u\right|\right)^2\me^{-cu^2}}
 {\sqrt{q^2+\left(u-\sigma\right)^2}},
 &&q+\mi\sigma\in\Gamma_2^\sigma.
 \label{eq majorant crossing}
\end{align}
The same estimates hold for the limiting integrand.
\end{lemma}

\begin{proof}
Choose $\epsilon>0$ so small that
$c_0=1/T-\gamma_2-\epsilon>0$, and then use Proposition \ref{bound}.
For the inner heat kernel,
\[
 \left|\mathfrak h_t\left(-\mi y,u\right)\right|
 =\left(2\pi t\right)^{-1/2}
 \exp\left(\frac{y^2}{2t}-\frac{u^2}{2t}\right),
\]
and hence its product with
$|\mathfrak E_{{\bup}^{\UN}}(\mi u)|$ is bounded by a constant times
$\me^{-c_0u^2/2}$ on $Q$. If
$z=r\pm\mi r/\sqrt3\in\Gamma_1$, then
\[
 \Re\left(z^2\right)=\frac23r^2,\qquad
 \left|\mathfrak h_s\left(x,z\right)\right|
 =\left(2\pi s\right)^{-1/2}
 \exp\left(-\frac{x^2-2xr+2r^2/3}{2s}\right).
\]
The reciprocal estimate in \eqref{bdinequality} therefore leaves a
negative quadratic in $r$, uniformly for $(s,x)\in[a,T]\times B$.
Moreover $|\mi u-z|\geq|\Re z|=|r|\geq\sqrt3$. This proves
\eqref{eq majorant sloped}. On $\Gamma_2$, compact convergence and
the fact that all zeros are real give a uniform bound for
$|\mathfrak E_{{\bup}^{\UN}}(z)|^{-1}$, while
\[
 \left|\mi u-\left(q+\mi\sigma\right)\right|=\sqrt{q^2+\left(u-\sigma\right)^2}.
\]
This proves \eqref{eq majorant crossing}. Finally,
\[
 \partial_x\mathfrak h_s\left(x,z\right)
 =-\frac{x-z}{s}\mathfrak h_s\left(x,z\right),\qquad
 \partial_y\mathfrak h_t\left(-\mi y,u\right)
 =\frac{y-\mi u}{t}\mathfrak h_t\left(-\mi y,u\right),
\]
so differentiation contributes only the polynomial factors already
allowed in the two bounds.
\end{proof}

The preceding bounds give local uniform convergence, including first
spatial derivatives of the analytic part. This supplies the kernel
convergence needed in Theorem~\ref{MainTheorem}.
\begin{proposition}\label{Uniformconvergence}
Let $0<a<T<\tau(\bup)$ and let $B\subset\mathbb R$ be compact. For
$D\in\{1,\partial_x,\partial_y\}$,
\begin{equation}\label{eq uniform analytic convergence}
 \lim_{N\to\infty}
 \sup_{\substack{a\leq s,t\leq T\\x,y\in B}}
 \left|
 D\left\{
 \BL_{{\bup}^{\UN}}\left(\left(s,x\right),\left(t,y\right)\right)
 -\BL_{\bup}\left(\left(s,x\right),\left(t,y\right)\right)
 \right\}\right|=0.
\end{equation}
In particular, for every finite $\mathcal A\subset(0,\tau(\bup))$,
\begin{equation}\label{eq uniform full kernel convergence}
 \lim_{N\to\infty}
 \max_{s,t\in\mathcal A}\sup_{x,y\in B}
 \left|
 \K_{{\bup}^{\UN}}\left(\left(s,x\right),\left(t,y\right)\right)
 -\K_{\bup}\left(\left(s,x\right),\left(t,y\right)\right)
 \right|=0.
\end{equation}
\end{proposition}

\begin{proof}
Choose $\epsilon>0$ so that $\gamma_2(\bup)+\epsilon<1/T$. Lemma
\ref{lm bound for sum} gives
$\gamma_2(\bup^{\UN})\leq\gamma_2(\bup)+\epsilon$ for all
sufficiently large $N$. Hence $T<\tau(\bup^{\UN})$, and the
approximating kernels below are defined on $[a,T]^2\times B^2$ for all
such $N$.
On a compact truncation of $\Gamma\times\mathbb R$ away from the
two crossings, Proposition~\ref{uniform convergence of entire} and
compact convergence of the reciprocals give uniform convergence of
the integrands and their first spatial derivatives, uniformly on
$[a,T]^2\times B^2$. By \eqref{eq majorant crossing} and
\eqref{eq crossing neighborhood}, crossing neighbourhoods of radius
$\rho$ contribute $O(\rho)$ uniformly in $N$ and the external variables.
The tails of the joint integral are uniformly negligible by
\eqref{eq majorant sloped} and the Gaussian factor in
\eqref{eq majorant crossing}. Letting $N\to\infty$, then
$\rho\downarrow0$ and the truncation radius tend to infinity proves
\eqref{eq uniform analytic convergence} for the contour integrals. The residue correction is the same for every $N$, so it
cancels in the displayed difference. The heat-kernel term also
cancels, giving \eqref{eq uniform full kernel convergence}.
\end{proof}

\begin{proof}[Proof of Theorem \ref{MainTheorem}]
For every finite $\mathcal A\subset(0,\tau(\bup))$, Lemma~\ref{lm bound for sum}
gives $\mathcal A\subset(0,\tau_N)$ for all sufficiently large $N$.
Proposition~\ref{Uniformconvergence} proves
\eqref{eq main local kernel convergence}.
Applying Proposition~\ref{prop mgf convergence} on each finite time set
$\mathcal A$ gives the asserted finite-dimensional convergence.
\end{proof}

\subsection{Dependence on the gamma parameters}

Fix $b>0$, $\boldsymbol\alpha^\pm\in\mathbb W_b$,
$\mathbf a^\pm\in\mathbb U_b$ with $k^-=0$, $\gamma_1\in\mathbb R$,
and $\gamma_2\geq0$. Set
\begin{equation}\label{eq mu and mu^gamma}
\begin{aligned}
 \bup&\defeq\left(\left(\boldsymbol\alpha^+,\boldsymbol\alpha^-\right),
             \left(\mathbf a^+,\mathbf a^-\right),0,
             s\left(\boldsymbol\alpha^+\right)+s\left(\boldsymbol\alpha^-\right)\right),\\
 \bup^{\gamma_1,\gamma_2}
 &\defeq\left(\left(\boldsymbol\alpha^+,\boldsymbol\alpha^-\right),
         \left(\mathbf a^+,\mathbf a^-\right),\gamma_1,
         s\left(\boldsymbol\alpha^+\right)+s\left(\boldsymbol\alpha^-\right)+\gamma_2\right).
\end{aligned}
\end{equation}
Then $\mathfrak E_{\bup^{\gamma_1,\gamma_2}}(z)
=\me^{-\gamma_1z-\gamma_2z^2/2}\mathfrak E_{\bup}(z)$.
Write $\tau_\gamma=\tau(\bup^{\gamma_1,\gamma_2})$ and, for
$0\leq t<\tau_\gamma$, put
\[
\begin{aligned}
 A_t&\defeq1-\gamma_2t,\qquad
 \vartheta_{\gamma_2}\left(t\right)\defeq\frac{t}{A_t},\qquad
 \chi_t\left(x\right)\defeq\frac{x+\gamma_1t}{A_t},\\
 \mathfrak g\left(t,x\right)
 &\defeq\exp\left(-\gamma_1x-\frac{\gamma_1^2t}{2}
       -\frac{\gamma_2\left(x+\gamma_1t\right)^2}{2A_t}\right).
\end{aligned}
\]

\begin{remark}\label{equivalentkernel}
For a locally bounded determinantal kernel $\K$, multiplication by a
gauge factor $\mathfrak g(t,y)/\mathfrak g(s,x)$ leaves every correlation
determinant unchanged: the row and column factors cancel.
Thus the process law is unchanged whenever $\mathfrak g$ and its
reciprocal are measurable and locally bounded.
\end{remark}

The gamma parameters give a time change and an affine spatial transformation.
The next identity also identifies the common drift used in Section~\ref{sec:paths}.
\begin{proposition}\label{gammadependence}
For $0<s,t<\tau_\gamma$,
\begin{equation}\label{Integralterm}
 \K_{\bup^{\gamma_1,\gamma_2}}
 \left(\left(s,x\right),\left(t,y\right)\right)
 =
 \frac{\mathfrak g\left(t,y\right)}
      {\mathfrak g\left(s,x\right)\sqrt{A_sA_t}}\,
 \K_{\bup}
 \left(\left(\vartheta_{\gamma_2}\left(s\right),\chi_s\left(x\right)\right),
       \left(\vartheta_{\gamma_2}\left(t\right),\chi_t\left(y\right)\right)\right).
\end{equation}
Consequently, if $\bXi$ has kernel $\K_{\bup}$ and
$\Phi_t(q)=A_tq-\gamma_1t$, then
\[
 \bXi^{\gamma_1,\gamma_2}\left(t\right)
 \defeq\left(\Phi_t\right)_\#\bXi\left(\vartheta_{\gamma_2}\left(t\right)\right),
 \qquad 0<t<\tau_\gamma,
\]
has kernel $\K_{\bup^{\gamma_1,\gamma_2}}$.
\end{proposition}

\begin{proof}
Completing the two squares gives
\[
\begin{aligned}
 \mathfrak h_s\left(x,z\right)\me^{\gamma_1z+\gamma_2z^2/2}
 &=\frac{\mathfrak h_{\vartheta_{\gamma_2}(s)}
            \left(\chi_s\left(x\right),z\right)}
          {\mathfrak g\left(s,x\right)\sqrt{A_s}},\\
 \mathfrak h_t\left(-\mi y,u\right)\me^{-\gamma_1\mi u+\gamma_2u^2/2}
 &=\frac{\mathfrak g\left(t,y\right)}{\sqrt{A_t}}\,
   \mathfrak h_{\vartheta_{\gamma_2}(t)}
            \left(-\mi\chi_t\left(y\right),u\right).
\end{aligned}
\]
Substituting in \eqref{kernel_omega} proves \eqref{Integralterm} for
the contour integrals. At $z=\mi u$ the parameter-dependent
exponentials cancel, so the same identities give
\[
 \Rc_\Gamma\left(\left(s,x\right),\left(t,y\right)\right)
 =
 \frac{\mathfrak g\left(t,y\right)}
      {\mathfrak g\left(s,x\right)\sqrt{A_sA_t}}\,
 \Rc_\Gamma\left(
  \left(\vartheta_{\gamma_2}\left(s\right),\chi_s\left(x\right)\right),
  \left(\vartheta_{\gamma_2}\left(t\right),\chi_t\left(y\right)\right)\right).
\]
Thus the analytic parts obey the same identity. For $s>t$,
a further completion of squares gives
\begin{equation}\label{heatkernelterm}
 \mathfrak h_{s-t}\left(x,y\right)
 =
 \frac{\mathfrak g\left(t,y\right)}
      {\mathfrak g\left(s,x\right)\sqrt{A_sA_t}}\,
 \mathfrak h_{\vartheta_{\gamma_2}(s)-\vartheta_{\gamma_2}(t)}
       \left(\chi_s\left(x\right),\chi_t\left(y\right)\right).
\end{equation}
Since $\vartheta_{\gamma_2}$ is strictly increasing, the time-order
indicator is unchanged. Subtracting \eqref{heatkernelterm} proves
\eqref{Integralterm}.

The change of variables $x=\Phi_t(q)$ gives the spatial Jacobian
$1/\sqrt{A_sA_t}$ in the transformed kernel. The remaining factor
is removed by Remark~\ref{equivalentkernel}, proving the process statement.
\end{proof}
For $\gamma_2=0$ this is translation by $-\gamma_1t$.
For $\gamma_1=0$ it is dilation by $1-\gamma_2t$ combined with the
time change $t\mapsto t/(1-\gamma_2t)$.

\subsection{Sharpness of the time horizon}
The restriction to times below $\tau(\bup)$ cannot in general be relaxed in
Theorem~\ref{MainTheorem}. The following symmetric clouds have convergent
parameters, but their first correlation functions diverge at the origin both
at and after the limiting time horizon.

\begin{proposition}\label{prop:gamma2-horizon-sharpness}
Fix $\gamma>0$ and set
\[
 t_c\defeq\gamma^{-1},\qquad
 a_N\defeq\sqrt{2N/\gamma},\qquad
 \bxi^{(N)}\defeq N\left(\delta_{a_N}+\delta_{-a_N}\right).
\]
For any fixed $b>0$, the parameters $f_0^b(\bxi^{(N)})$, defined for all
sufficiently large $N$, converge to a parameter $\bup_\gamma$ with
$\gamma_2(\bup_\gamma)=\gamma$. Let $\rho_t^{(N)}$ be the first correlation
function of the $\beta=2$ Dyson Brownian motion from $\bxi^{(N)}$, in the
entrance-law sense of Proposition~\ref{finitedysonmodel}. Then
\[
 \rho_{t_c}^{\left(N\right)}\left(0\right)
 \sim\frac{\Gamma\left(3/4\right)\sqrt\gamma}
 {2^{1/4}\pi^{3/2}}N^{1/4},\qquad
 \rho_t^{\left(N\right)}\left(0\right)
 \sim\frac{a_N}{\pi t}\sqrt{\gamma t-1}\quad\left(t>t_c\right),
\]
where $\Gamma$ denotes Euler's gamma function. In particular, the first
correlation functions cannot converge locally uniformly to a finite
function at any fixed $t\geq t_c$.
\end{proposition}
\begin{proof}
For large $N$, all roots lie outside $[-b,b]$. Both reciprocal sequences
consist of $N$ copies of $a_N^{-1}$ followed by zeros, and their combined
square sum is $2N/a_N^2=\gamma$. The reciprocal sum is zero and the inner
root data are empty. Thus the parameters converge coordinatewise with
$\gamma_1=0$, $\delta=\gamma$ and both limiting reciprocal sequences zero.
This gives $\gamma_2(\bup_\gamma)=\gamma$.

We first compute the density for the finite cloud
$N(\delta_a+\delta_{-a})$, with arbitrary $a,t>0$. Write it as
$\rho_t^{N,a}$ and let $\mathsf G_N$ be a sum of $N$ independent
exponential random variables of mean one. We claim that
\begin{equation}\label{eq:gamma2-gamma-density}
 \rho_t^{N,a}\left(0\right)
 =\frac{1}{\pi t}\expt\left[
       \left(2t\mathsf G_N-a^2\right)_+^{1/2}\right].
\end{equation}
Here $q_+=\max\{q,0\}$. To prove the identity, take a positively oriented
small circle $D$ about $a^2$, disjoint from $(-\infty,0]$, and let $C$ be
its two inverse images under $z\mapsto z^2$. The finite kernel formula
\eqref{eq kernel finite dim}, at $s=t$ and $x=y=0$, gives
\[
 \rho_t^{N,a}\left(0\right)
 =\frac{1}{2\pi t}\int_{\mathbb R}\me^{-u^2/\left(2t\right)}
 \frac{1}{2\pi\mi}\oint_C
 \frac{\me^{-z^2/\left(2t\right)}}{\mi u-z}
 \left(\frac{a^2+u^2}{a^2-z^2}\right)^N\,\dif z\,\dif u.
\]
These integrals are absolutely convergent. Averaging the integrand over
$u$ and $-u$ replaces $(\mi u-z)^{-1}$ by $-z/(z^2+u^2)$.
Set $c=a^2+u^2$ and change variables to $w=z^2$ on each component of
$C$. The two Jacobian factors of $1/2$ add, and expansion at $w=a^2$
evaluates the resulting contour integral as
\[
 -\frac{c^N}{2\pi\mi}\oint_D
 \frac{\me^{-w/\left(2t\right)}}{\left(a^2-w\right)^N\left(w+u^2\right)}\,\dif w
 =\me^{-a^2/\left(2t\right)}\sum_{k=0}^{N-1}\frac{\left(c/\left(2t\right)\right)^k}{k!}.
\]
Since $\prob(\mathsf G_N>v)=\me^{-v}\sum_{k=0}^{N-1}v^k/k!$ for $v\geq0$,
Tonelli's theorem now gives
\[
 \rho_t^{N,a}\left(0\right)
 =\frac{1}{2\pi t}\int_{\mathbb R}
   \prob\left(2t\mathsf G_N>a^2+u^2\right)\,\dif u
 =\frac{1}{\pi t}\expt\left[
   \left(2t\mathsf G_N-a^2\right)_+^{1/2}\right].
\]

For $a=a_N$ and $t>t_c$, divide \eqref{eq:gamma2-gamma-density} by
$\sqrt N$. Since $\mathsf G_N/N\to1$ in $L^2$ and
$|\sqrt{x_+}-\sqrt{y_+}|\leq\sqrt{|x-y|}$,
\[
 \frac{\rho_t^{\left(N\right)}\left(0\right)}{\sqrt N}
 \longrightarrow\frac{\sqrt{2t-2/\gamma}}{\pi t}.
\]
At $t=t_c$, put $\mathsf Z_N=(\mathsf G_N-N)/\sqrt N$. The central limit
theorem gives $\mathsf Z_N\xrightarrow{\mathrm{law}}\mathsf Z$, where $\mathsf Z$ is
standard normal. Moreover,
$\expt[(\mathsf Z_N)_+^2]\leq\expt[\mathsf Z_N^2]=1$, so
$((\mathsf Z_N)_+^{1/2})_{N\geq1}$ is uniformly integrable. Therefore
\[
 \frac{\rho_{t_c}^{\left(N\right)}\left(0\right)}{N^{1/4}}
 =\frac{\sqrt{2\gamma}}{\pi}
      \expt\left[\left(\mathsf Z_N\right)_+^{1/2}\right]
 \longrightarrow\frac{\sqrt{2\gamma}}{\pi\sqrt{2\pi}}
       \int_0^\infty v^{1/2}\me^{-v^2/2}\,\dif v
 =\frac{\Gamma\left(3/4\right)\sqrt\gamma}{2^{1/4}\pi^{3/2}}.\qedhere
\]
\end{proof}

\section{Convergence on path space}\label{sec:paths}

We shall use the following subspace of the configuration space:
\begin{equation*}
 \mathfrak{M}^0
 \defeq\left\{\bxi\in\mathfrak{M}:\bxi\left(\left\{y\right\}\right)\leq1\ \hbox{for every }y\in\mathbb R,
 \quad
 \sum_{\substack{x\in\operatorname{supp}\left(\bxi\right)\\ \left|x\right|\geq1}}\frac1{x^2}<\infty
 \right\}.
\end{equation*}
As before, \(x\in\bxi\) means \(x\in\operatorname{supp}(\bxi)\).
All processes in this section are extended to time \(0\) by their deterministic
initial configurations.

We say that
\(\bxi^\UN\) converges to \(\bxi\) rootwise if, for every \(R>0\) such that
\(\bxi(\{\pm R\})=0\), the atoms in \((-R,R)\) can, for all sufficiently
large \(N\), be put in bijection with those of \(\bxi\) in \((-R,R)\), and
the matched atoms converge.  In particular, rootwise convergence implies
vague convergence and uniform boundedness of the number of atoms in each
fixed compact interval.

The next theorem upgrades parameter convergence to convergence on path space. Its separation hypothesis also rules out a positive limiting Gaussian coefficient.
\begin{theorem}\label{theorem path cts}
 Let \(b>0\), let \(\bxi^\UN\in\mathfrak{M}^0\cap\mathfrak{M}_b\),
 and let \(r^\UN\to r\).  Put
 \[
       \bup^\UN\defeq f_{r^\UN}^b\left(\bxi^\UN\right),
 \]
 and assume that
 \[
       \bup^\UN\longrightarrow\bup\quad\hbox{in }\bUpsilon_b.
 \]
 Write
 \(\mathbf a^\pm=((a_i^\pm)_{i\geq1},k^\pm)\), and let \(\bxi\) be the
 root configuration encoded by \(\bup\), namely
 \[
 \bxi\defeq\left(k^++k^-\right)\delta_0
      +\sum_{i=1}^{m\left(\mathbf a^+\right)}\delta_{a_i^+}
      +\sum_{i=1}^{m\left(\mathbf a^-\right)}\delta_{-a_i^-}
      +\sum_{\alpha_i^+>0}\delta_{1/\alpha_i^+}
      +\sum_{\alpha_i^->0}\delta_{-1/\alpha_i^-}.
 \]
 Assume that \(\bxi\) is simple and infinite, and that there is a decreasing function
 \(\varepsilon_{\mathrm{sep}}\colon[0,\infty)\to[0,\infty)\), with
 \(\varepsilon_{\mathrm{sep}}(R)\to0\), such that
 \begin{equation}\label{extra condition}
  \text{for every \(N\) and every \(x\in\bxi^\UN\),}\qquad
  \sum_{\substack{y\in\bxi^\UN\\y\ne x,\,-x}}
       \frac1{\left|y^2-x^2\right|}
  \leq \varepsilon_{\mathrm{sep}}\left(\left|x\right|\right).
 \end{equation}
 Let \(\bXi^\UN\) be the process with kernel
 \(\K_{\bup^\UN}\).  Then
 \[
     \bup=f_r^b\left(\bxi\right),
 \]
 the processes \(\bXi^\UN\) and the process
 \(\bXi\) with kernel \(\K_{f_r^b(\bxi)}\) have continuous
 versions on \([0,\infty)\), with
 \(\bXi^\UN(0)=\bxi^\UN\) and
 \(\bXi(0)=\bxi\), and
 \begin{equation*}
   \bXi^\UN\left(\cdot\right)
   \xrightarrow{\mathrm{law}}
   \bXi\left(\cdot\right)
   \quad\text{in }\mathcal C\left(\left[0,\infty\right)\to\mathfrak M\right).
 \end{equation*}
\end{theorem}

The separation hypothesis forces the Gaussian coefficient to vanish. This identifies the limiting parameter with the embedding of the limiting configuration, as required in Theorem~\ref{theorem path cts}.
\begin{lemma}\label{gamma_2 is 0}
 Under the assumptions of Theorem~\ref{theorem path cts},
 $\gamma_2(\bup)=0$ and hence $\bup=f_r^b(\bxi)$.
\end{lemma}

\begin{proof}
Parameter convergence implies rootwise convergence of $\bxi^\UN$ to $\bxi$.
Indeed, eventual equality of $q(\mathbf a)=m(\mathbf a)+k$ and coordinate convergence in
$\mathbb U_b$ match the finite inner multisets, including roots absorbed
at zero. Outside $[-b,b]$, the decreasing reciprocal coordinates match
the finitely many entries exceeding $1/R$ whenever
$\bxi(\{\pm R\})=0$; their number is eventually constant. Restricting the
inner matching when $R\leq b$, and combining these matchings when $R>b$,
gives the required bijections. The $\gamma_1$-coordinate of $\bup$ is
$r$, so only the equality of its square coordinate with the root square
mass remains to be proved. Write $\gamma_2=\gamma_2(\bup)$.
 For every \(N\),
 \[
  \delta^\UN
   =s\left(\boldsymbol{\alpha}^{+,\UN}\right)
     +s\left(\boldsymbol{\alpha}^{-,\UN}\right)
   =\sum_{\substack{y\in\bxi^\UN\\\left|y\right|>b}}\frac1{y^2}.
 \]
 Choose \(L>b\) outside the countable set
 \(\{|x|:x\in\operatorname{supp}(\bxi)
          \cup\bigcup_N\operatorname{supp}(\bxi^\UN)\}\).
 Rootwise convergence on the compact annulus
 \(\{b<|y|<L\}\), together with \(\delta^\UN\to\delta\), gives
 \begin{equation}\label{escaped square mass}
 \begin{aligned}
  \lim_{N\to\infty}
  \sum_{\substack{y\in\bxi^\UN\\\left|y\right|\geq L}}\frac1{y^2}
   &=\delta-
     \sum_{\substack{y\in\bxi\\b<\left|y\right|<L}}\frac1{y^2}=\gamma _2+
     \sum_{\substack{y\in\bxi\\\left|y\right|\geq L}}\frac1{y^2}
   \geq\gamma _2 .
 \end{aligned}
 \end{equation}

 Suppose that \(\gamma _2>0\).  Since \(\bxi\) is infinite and locally
 finite, choose \(x\in\bxi\) so far from the origin that
 \(\varepsilon_{\mathrm{sep}}(|x|/2)<\gamma _2/2\), and let
 \(x^\UN\in\bxi^\UN\) be its matched atom.  Next choose \(L\), as above,
 with \(L>2\sup_{N\geq N_0}|x^\UN|\).  For all sufficiently large \(N\),
 every atom \(y\) in the sum below is different from \(x^\UN\) and
 \(-x^\UN\), and
 \[
  \sum_{\substack{y\in\bxi^\UN\\\left|y\right|\geq L}}\frac1{y^2}
  \leq
  \sum_{\substack{y\in\bxi^\UN\\\left|y\right|\geq L}}
       \frac1{\left|y^2-\left(x^\UN\right)^2\right|}
  \leq\varepsilon_{\mathrm{sep}}\left(\left|x^\UN\right|\right)
  \leq\varepsilon_{\mathrm{sep}}\left(\left|x\right|/2\right)<\frac{\gamma _2}{2}.
 \]
 This contradicts \eqref{escaped square mass}.  Hence \(\gamma _2=0\).
 Therefore \(\bup=f_r^b(\bxi)\).
\end{proof}

The following interpolation functions are the weights in the complex Brownian representation.
\begin{definition}\label{residue function}
 For \(r\in\mathbb R\), \(\bxi\in\mathfrak M_b\cap\mathfrak M^0\),
 \(x\in\bxi\), and \(z\in\mathbb C\), define
 \begin{equation*}
 \mathfrak F_{\bxi,r}\left(x,z\right)\defeq
 \frac{\mathfrak E_{f_r^b\left(\bxi\setminus\left\{x\right\}\right)}\left(z\right)}
      {\mathfrak E_{f_r^b\left(\bxi\setminus\left\{x\right\}\right)}\left(x\right)}
 \begin{cases}
  1,& \left|x\right|<b,\\
  \exp\left(z/x-1\right),& \left|x\right|>b.
 \end{cases}
 \end{equation*}
 Since \(\bxi\) is simple, the denominator is nonzero.
\end{definition}

For \(L>b\), write \(\bxi^{[L]}=\bxi|_{[-L,L]}\), choosing \(L\) so that
neither \(L\) nor \(-L\) is an atom of \(\bxi\) or any \(\bxi^\UN\).
We use the same truncation notation for \(\bxi^\UN\).

The interpolation functions satisfy a common subquadratic bound and converge locally uniformly at matched roots. These estimates provide the domination needed for the Brownian moment representation.
\begin{proposition}\label{bound of res}
 Under the assumptions of Theorem \ref{theorem path cts}, for every
 \(\eta>0\) there is a finite \(C_\eta\), independent of \(N,L,x,z\),
 such that
 \begin{equation}\label{eq bound of res}
 \left|\mathfrak F_{\boldsymbol\zeta,r_*}\left(x,z\right)\right|
 \leq C_\eta\exp\left(\eta\left(|x|^2+|z|^2\right)\right),
 \qquad x\in\boldsymbol\zeta,\quad z\in\mathbb C,
 \end{equation}
 for each of the pairs
 $(\boldsymbol\zeta,r_*)=(\bxi^\UN,r^\UN)$,
 $(\bxi^{\UN,[L]},r^\UN)$, $(\bxi,r)$, and $(\bxi^{[L]},r)$.
 Moreover, if
 \(x^\UN\to x\) are matched roots, then
 \begin{equation}\label{residue local convergence}
  \mathfrak F_{\bxi^\UN,r^\UN}\left(x^\UN,\cdot\right)
   \longrightarrow\mathfrak F_{\bxi,r}\left(x,\cdot\right)
 \end{equation}
 locally uniformly on \(\mathbb C\).  The analogous assertion holds as
 \(L\to\infty\) for spatial truncations.
\end{proposition}

\begin{proof}
 Lemma \ref{gamma_2 is 0}, coordinatewise convergence, and convergence of
 the \(\delta\)-coordinates imply the uniform square-tail estimate
 \begin{equation}\label{uniform square tail}
  \lim_{M\to\infty}\ \sup_N
  \left\{\sum_{j\geq M}\left(\alpha_j^{+,\UN}\right)^2+
         \sum_{j\geq M}\left(\alpha_j^{-,\UN}\right)^2\right\}=0.
 \end{equation}
 Indeed, subtract the first \(M-1\) convergent square coordinates from
 \(\delta^\UN\to
 s(\boldsymbol{\alpha}^+)+s(\boldsymbol{\alpha}^-)\).
 Deleting roots only decreases the left hand side.  The canonical-product
 inequality
 \[
       \left|\left(1-w\right)\me^w\right|\leq \me^{\left|w\right|^2/2}
 \]
 applied after separating finitely many root coordinates therefore yields,
 uniformly for \(\boldsymbol{\zeta}\) equal to any of
 \(\bxi^\UN,\bxi^{\UN,[L]},\bxi,\bxi^{[L]}\),
 \begin{equation}\label{residue numerator bound}
  \left|\mathfrak E_{f_{r_*}^b\left(\boldsymbol{\zeta}\setminus\left\{x\right\}\right)}\left(z\right)\right|
   \leq C_\eta \me^{\eta\left|z\right|^2},
 \end{equation}
 where \(r_*=r^\UN\) or \(r\), respectively.  Here the bounded
 \(r_*\)'s and the uniformly finite collection of roots in a compact
 interval are absorbed into \(C_\eta\).

 We spell out the denominator estimate, since it is where
 \eqref{extra condition} is needed.  For \(x^2\ne y^2\) and $y\neq0$,
 \begin{equation}\label{simple equality}
  \frac1{\left(1-x/y\right)\me^{x/y}}
  =\left(1+\frac{x^2}{y^2-x^2}\right)
    \left(1+x/y\right)\me^{-x/y}.
 \end{equation}
 Put
 \[
  I_{\boldsymbol\zeta,x}
    \defeq\left\{y\in\boldsymbol\zeta\setminus\left\{x\right\}:\left|y\right|<b\right\},
  \qquad
  O_{\boldsymbol\zeta,x}
    \defeq\left\{y\in\boldsymbol\zeta\setminus\left\{x\right\}:\left|y\right|>b\right\}.
 \]
 Since configurations in \(\mathfrak M_b\) have no roots at \(\pm b\),
 the definition of the canonical product gives the exact factorization
 \begin{equation}\label{denominator core split}
 \begin{aligned}
 &\left|
  \mathfrak E_{f_{r_*}^b\left(\boldsymbol\zeta\setminus\left\{x\right\}\right)}\left(x\right)
  \right|^{-1}={\me}^{r_* x}
  \prod_{y\in I_{\boldsymbol\zeta,x}}\left|x-y\right|^{-1}
  \prod_{y\in O_{\boldsymbol\zeta,x}}
       \left|\left(1-x/y\right)\me^{x/y}\right|^{-1}.
 \end{aligned}
 \end{equation}
 The number of inner roots is uniformly bounded.  If \(|x|\geq2b\),
 then \(|x-y|\geq|x|/2\) for every inner root.  If \(|x|<2b\),
 rootwise convergence to the simple limiting configuration makes the
 distinct roots in a fixed larger compact interval uniformly separated;
 the finitely many exceptional \(N\)'s are absorbed into the constant.
 The same statement holds for the truncations, which contain the same
 inner core.  Consequently, for every \(\eta>0\),
 \begin{equation}\label{inner core bound}
  \prod_{y\in I_{\boldsymbol\zeta,x}}\left|x-y\right|^{-1}
  \leq C_\eta \me^{\eta x^2}
 \end{equation}
 uniformly over the four families of configurations under consideration.

 Apply \eqref{simple equality} only to the outer roots.  If
 \(-x\in O_{\boldsymbol\zeta,x}\), then necessarily \(|x|>b\), and that
 one outer canonical factor contributes exactly \(\me/2\) to the inverse.
 For all remaining outer roots, the product of the final factors in
 \eqref{simple equality}, together with \(\me^{r_* x}\), is the canonical
 product at \(x\) of the reflected outer subconfiguration, with parameter
 \(-r_*\).  The proof of \eqref{residue numerator bound} applies to this
 product: reflection preserves square tails, deletion decreases them,
 and \(-r_*\) remains bounded.  Hence its absolute value is at most
 \(C_\eta \me^{\eta x^2}\).  Finally, using
 \(|1+a|\leq \me^{|a|}\), with \(a=x^2/(y^2-x^2)\), in
 \eqref{denominator core split}, and shrinking \(\eta\) at the start,
 gives
 \begin{equation}\label{residue denominator bound}
 \left|
  \mathfrak E_{f_{r_*}^b\left(\boldsymbol{\zeta}\setminus\left\{x\right\}\right)}\left(x\right)
 \right|^{-1}
 \leq C_\eta \me^{\eta x^2}
 \exp\left\{
   x^2\sum_{\substack{y\in\boldsymbol{\zeta}\\y\ne x,-x}}
          \frac1{\left|y^2-x^2\right|}
 \right\}.
 \end{equation}
 Truncation decreases the sum.  For \(\boldsymbol{\zeta}=\bxi^\UN\),
 \eqref{extra condition} bounds it by
 \(\varepsilon_{\mathrm{sep}}(|x|)\); for the limiting configuration, Fatou's
 lemma gives the same conclusion with
 \(\varepsilon_{\mathrm{sep}}(|x|/2)\).  Since this tends to zero, its product
 with \(x^2\) is absorbed into \(\me^{\eta x^2}\) for large \(|x|\), while
 bounded \(x\)'s are absorbed into \(C_\eta\).
 The extra factor \(\exp(z/x-1)\), when present, is also absorbed by
 \(2|z/x|\leq\eta|z|^2+\eta^{-1}x^{-2}\) and \(|x|>b\).
 Combining \eqref{residue numerator bound} and
 \eqref{residue denominator bound}, and replacing \(\eta\) by a smaller
 number at the start, proves \eqref{eq bound of res}.

 Finally, deleting a matched root from the parameter tuple preserves
 coordinatewise convergence (and changes the square coordinate by the
 corresponding reciprocal square when the root lies outside \([-b,b]\)).
 Proposition \ref{uniform convergence of entire}, together with the
 nonvanishing of the limiting denominator, then proves
 \eqref{residue local convergence}.  The proof for truncations is
 identical.
\end{proof}

Let \(\boldsymbol{\eta}=\sum_{i=1}^n\delta_{x_i}\) be finite and simple,
and set
\[
 \mathfrak P_{\boldsymbol{\eta}}^{x}\left(z\right)
 \defeq\prod_{a\in\boldsymbol{\eta}\setminus\left\{x\right\}}
       \frac{z-a}{x-a}.
\]
Let \(\mathsf Z_i=\mathsf X_i+\mi \mathsf Y_i\) be independent complex Brownian motions started
at \(x_i\), with all their real and imaginary parts independent.
We write $\expt_{\mathbf x}^{\mathrm{BM}}$ for this Brownian expectation
and $\expt_{r,\boldsymbol\eta}^{\mathrm{DBM}}$ for expectation under the
process with kernel $\K_{f_r^b(\boldsymbol\eta)}$, omitting the
superscripts when the law is clear.

We first obtain the complex Brownian representation for finite configurations and an arbitrary reciprocal parameter. This includes the common drift needed when applying the representation to spatial truncations.
\begin{theorem}\label{theorem complex repre}
 Fix \(b>0\), let
 \(\boldsymbol{\eta}\in\mathfrak M_b\) be finite and simple, let
 \(r\in\mathbb R\), and let \(T>0\). If \(G\) is a bounded
 cylinder functional of the configuration path up to time \(T\), then
 \begin{equation}\label{finite arbitrary r representation}
 \expt_{r,\boldsymbol{\eta}}^{\mathrm{DBM}}\left[G\left(\bXi\right)\right]
 =
 \expt_{\mathbf{x}}^{\mathrm{BM}}\left[
  G\left(\sum_{i=1}^n\delta_{\mathsf X_i\left(\cdot\right)}\right)
  \det\left[
    \mathfrak F_{\boldsymbol{\eta},r}\left(x_i,\mathsf Z_j\left(T\right)\right)
  \right]_{i,j=1}^n
 \right].
 \end{equation}
\end{theorem}

\begin{proof}
 Directly from Definition \ref{residue function},
 \begin{equation}\label{finite bridge factor}
  \mathfrak F_{\boldsymbol{\eta},r}\left(x,z\right)
   =\exp\left\{-\left(r-p_b\left(\boldsymbol{\eta}\right)\right)\left(z-x\right)\right\}
     \mathfrak P_{\boldsymbol{\eta}}^x\left(z\right).
 \end{equation}
 Put \(c=r-p_b(\boldsymbol{\eta})\).  The standard finite Dyson
 representation (the case \(c=0\)) is the determinantal-martingale
 representation of \cite{MR3019667}.  Factoring the exponential in
 \eqref{finite bridge factor} out of rows and columns gives
 \[
 \det\left[\mathfrak F_{\boldsymbol{\eta},r}\left(x_i,\mathsf Z_j\left(T\right)\right)\right]
 =
 \me^{-c\sum_j\left(\mathsf Z_j\left(T\right)-x_j\right)}
 \det\left[\mathfrak P_{\boldsymbol{\eta}}^{x_i}\left(\mathsf Z_j\left(T\right)\right)\right].
 \]
 For independent standard real Brownian motions \(\mathsf B,\mathsf Y\)
 and a polynomial \(H\), Gaussian integration gives
 \[
 \begin{aligned}
  \expt\left[\me^{-c\mathsf B\left(T\right)}F\left(\mathsf B\right)\right]
    &=\me^{c^2T/2}\expt\left[F\left(\mathsf B-c\,\cdot\right)\right],\\
  \expt\left[\me^{-\mi c\mathsf Y\left(T\right)}H\left(a+\mi \mathsf Y\left(T\right)\right)\right]
    &=\me^{-c^2T/2}\expt\left[H\left(a+cT+\mi \mathsf Y\left(T\right)\right)\right].
 \end{aligned}
 \]
 The first identity is Cameron--Martin and the second is the elementary
 Gaussian contour shift, justified here because the determinant is a
 polynomial in each complex endpoint.  The scalar factors cancel.  The
 real shift sends \(\mathsf X_j(t)\) to \(\mathsf X_j(t)-ct\), while the \(+cT\) in the
 imaginary identity cancels that shift inside
 \(\mathsf Z_j(T)=\mathsf X_j(T)+\mi \mathsf Y_j(T)\).  Hence
 \begin{align*}
 &\expt_{\mathbf{x}}\left[
  G\left(\sum_i\delta_{\mathsf X_i\left(\cdot\right)}\right)
  \me^{-c\sum_j\left(\mathsf Z_j\left(T\right)-x_j\right)}
  \det\left[\mathfrak P_{\boldsymbol{\eta}}^{x_i}\left(\mathsf Z_j\left(T\right)\right)\right]
 \right]=
 \expt_{\mathbf{x}}\left[
  G\left(\sum_i\delta_{\mathsf X_i\left(\cdot\right)-ct}\right)
  \det\left[\mathfrak P_{\boldsymbol{\eta}}^{x_i}\left(\mathsf Z_j\left(T\right)\right)\right]
 \right].
 \end{align*}
 This identity follows first for bounded cylinder functions by ordinary
 Gaussian integration and then by a monotone-class argument.  By
 Proposition \ref{gammadependence}, changing the \(\gamma _1\)-coordinate
 from \(p_b(\boldsymbol{\eta})\) to \(r\) shifts every particle by
 \(-ct=-(r-p_b(\boldsymbol{\eta}))t\).  Applying the standard representation proves
 \eqref{finite arbitrary r representation}.
\end{proof}

Expanding an increment moment and integrating out the unused Brownian
coordinates gives determinants of size at most the moment order.
The resulting formula will pass to infinite configurations.
\begin{cor}\label{complex representation}
Let $\phi\in\mathcal C_c^1(\mathbb R)$, $0\leq s<t\leq T$, and $Q\in\mathbb N$.
Then
 \begin{equation}\label{eq complex representation}
 \begin{aligned}
 &\expt_{r,\boldsymbol{\eta}}
 \left[
  \left\{\left\langle\phi,\bXi\left(t\right)\right\rangle
    -\left\langle\phi,\bXi\left(s\right)\right\rangle\right\}^{Q}
 \right]
 =
 \sum_{q=1}^{Q}
 \sum_{\substack{p_1,\ldots,p_q\geq1\\p_1+\cdots+p_q=Q}}
 \binom{Q}{p_1,\ldots,p_q}\\
 &\qquad\times
 \sum_{\substack{x_{i_1}>\cdots>x_{i_q}\\x_{i_j}\in\boldsymbol{\eta}}}
 \expt\left[
  \prod_{j=1}^q
   \left\{\phi\left(\mathsf X_{i_j}\left(t\right)\right)-\phi\left(\mathsf X_{i_j}\left(s\right)\right)\right\}^{p_j}
  \det\left[
   \mathfrak F_{\boldsymbol{\eta},r}
      \left(x_{i_j},\mathsf Z_{i_k}\left(T\right)\right)
  \right]_{j,k=1}^q
 \right].
 \end{aligned}
 \end{equation}
\end{cor}
\begin{proof}
Expand the $Q$-th power in \eqref{finite arbitrary r representation} and
group terms by their distinct particle indices and multiplicities.
For an unused index $j$, independence and the complex Brownian
martingale identity give
\[
 \expt\left[\mathfrak F_{\boldsymbol\eta,r}
                  \left(x_i,\mathsf Z_j\left(T\right)\right)\right]
 =\mathfrak F_{\boldsymbol\eta,r}\left(x_i,x_j\right)=\delta_{ij}.
\]
The identity is applicable because \eqref{finite bridge factor} is an
exponential times a polynomial. Integrating each unused column therefore
removes its corresponding row and column, yielding
\eqref{eq complex representation}, also when $s=0$.
\end{proof}

The reduced representation gives a uniform even-moment estimate for linear-statistic increments. This supplies both tightness and the uniform integrability needed to pass to infinite configurations.
\begin{lemma}\label{lemma for bound}
 Let \(\mathcal A\subset(\mathfrak M_b\cap\mathfrak M^0)\times\mathbb R\)
 be a family of pairs \((\boldsymbol\zeta,r)\) for which the residue estimate
 \eqref{eq bound of res} holds with common constants and
 \begin{equation}\label{uniform gaussian summability}
   \sup_{(\boldsymbol{\zeta},r)\in\mathcal A}
   \sum_{x\in\boldsymbol{\zeta}}\me^{-c x^2}<\infty
   \qquad\text{for every }c>0.
 \end{equation}
 For every even \(Q\geq4\), \(T<\infty\), and
 \(\phi\in\mathcal C_c^1(\mathbb R)\), there is
 \(C_{Q,T,\phi}<\infty\) such that, for every
 \((\boldsymbol\zeta,r)\in\mathcal A\) with \(\boldsymbol\zeta\) finite,
 and \(0\leq s,t\leq T\),
 \begin{equation}\label{even Q moment estimate}
 \expt_{r,\boldsymbol{\zeta}}\left[
 \left|
  \left\langle\phi,\bXi\left(t\right)\right\rangle
  -\left\langle\phi,\bXi\left(s\right)\right\rangle
 \right|^Q\right]
 \leq C_{Q,T,\phi}\left|t-s\right|^{Q/2}.
 \end{equation}
\end{lemma}

\begin{proof}
 Fix a term in \eqref{eq complex representation}, with
 \(p_1+\cdots+p_q=Q\), and let \(K=\operatorname{supp}\phi\).  Put
 \[
  \mathsf I_j\defeq\indi\left\{\mathsf X_{i_j}\left(s\right)\in K
                      \text{ or }\mathsf X_{i_j}\left(t\right)\in K\right\}.
 \]
 The corresponding difference of \(\phi\)'s vanishes unless \(\mathsf I_j=1\),
 and the mean-value theorem gives
 \[
 \left|\phi\left(\mathsf X_{i_j}\left(t\right)\right)-\phi\left(\mathsf X_{i_j}\left(s\right)\right)\right|^{p_j}
 \leq \left\|\phi'\right\|_\infty^{p_j}
      \left|\mathsf X_{i_j}\left(t\right)-\mathsf X_{i_j}\left(s\right)\right|^{p_j}\mathsf I_j .
 \]
 Hölder's inequality with exponents \(2,4,4\) bounds the absolute value
 of the expectation in this term by
 \begin{align}
 &\left\{\expt\left[\prod_{j=1}^q
   \left(\left\|\phi'\right\|_\infty
       \left|\mathsf X_{i_j}\left(t\right)-\mathsf X_{i_j}\left(s\right)\right|\right)^{2p_j}
  \right]\right\}^{1/2}
 \left\{\expt\left[\prod_{j=1}^q \mathsf I_j\right]\right\}^{1/4}\notag\\
 &\qquad\times
 \left\{\expt\left[\left|
  \det\left[\mathfrak F_{\boldsymbol{\zeta},r}
       \left(x_{i_j},\mathsf Z_{i_k}\left(T\right)\right)\right]_{j,k=1}^q
 \right|^4\right]\right\}^{1/4}.
 \label{holder 244}
 \end{align}
 The first factor in \eqref{holder 244} is
 \[
 C_{\mathbf{p}}\left\|\phi'\right\|_\infty^{Q}
 \left|t-s\right|^{Q/2};
 \]

 If \(K\subset[-R,R]\), the Gaussian transition density, uniformly for
 \(0\leq s,t\leq T\), gives
 \[
 \left\{\expt\left[\prod_{j=1}^q \mathsf I_j\right]\right\}^{1/4}
 \leq C_{q,T,R}\exp\left\{-c_{T,R}
                  \sum_{j=1}^q x_{i_j}^2\right\}.
 \]
 Expanding the determinant, applying \eqref{eq bound of res}, and then
 integrating the Gaussian variables gives, for every sufficiently small
 \(\eta>0\),
 \[
 \left\{\expt\left[\left|
  \det\left[\mathfrak F_{\boldsymbol{\zeta},r}
       \left(x_{i_j},\mathsf Z_{i_k}\left(T\right)\right)\right]_{j,k=1}^q
 \right|^4\right]\right\}^{1/4}
 \leq C_{q,T,\eta}
       \exp\left\{\eta C_{q,T}
                  \sum_{j=1}^q x_{i_j}^2\right\}.
 \]
 Choose \(\eta\) small enough that at least half of the Gaussian decay
 remains.  Summing
 \eqref{holder 244} over distinct roots is finite, uniformly in the
 configuration, by \eqref{uniform gaussian summability}.  Finally sum
 over the finitely many \(q\)'s and compositions of \(Q\).  This proves
 \eqref{even Q moment estimate}.
\end{proof}

\begin{remark}\label{remark sum is uniformly bounded}
 The family consisting of all \(\bxi^\UN,\bxi\) and all their spatial
 truncations satisfies \eqref{uniform gaussian summability}.  Indeed,
 parameter convergence bounds
 \[
 \sup_N\sum_{\substack{x\in\bxi^\UN\\\left|x\right|>b}}\frac1{x^2}<\infty,
\]
 and rootwise convergence bounds the number of roots in every compact
 interval.  Since \(\me^{-cx^2}\leq C_c/x^2\) for \(|x|>b\), the assertion
 follows; truncation only decreases the sum.
\end{remark}

We now pass the reduced moment formula from spatial truncations to an infinite configuration. The interpolation bounds control its series, while higher moments justify convergence of the expectation.
\begin{proposition}\label{condition for complex representation}
 Fix \(r\in\mathbb R\), let
 \(\boldsymbol{\zeta}\in\mathfrak M_b\cap\mathfrak M^0\) be infinite,
 and let \(\bXi\) have kernel \(\K_{f_r^b(\boldsymbol\zeta)}\).
 Assume the residue bound \eqref{eq bound of res}, with the same
 constants for \(\boldsymbol{\zeta}^{[L]}\).  Then
 \eqref{eq complex representation}, with the sum over all distinct roots
 of \(\boldsymbol{\zeta}\), is valid for every \(Q\in\mathbb N\),
 \(\phi\in\mathcal C_c^1(\mathbb R)\), and \(0\leq s<t\leq T\).
 Moreover the series is absolutely convergent.
\end{proposition}

\begin{proof}
 Apply Corollary \ref{complex representation} to
 \(\boldsymbol{\zeta}^{[L]}\).  For a fixed collection of matched roots,
 canonical-product convergence gives
 \[
  f_r^b\left(\boldsymbol\zeta^{\left[L\right]}\setminus\left\{x\right\}\right)
   \longrightarrow f_r^b\left(\boldsymbol\zeta\setminus\left\{x\right\}\right)
  \quad\hbox{in }\bUpsilon_b .
 \]
 Proposition \ref{uniform convergence of entire}, evaluated also at the
 nonzero limiting denominator, therefore gives locally uniform convergence
 of every entry of the reduced determinant.  For the majorant, use
 \[
  \left|\phi\left(\mathsf X_i\left(t\right)\right)-\phi\left(\mathsf X_i\left(s\right)\right)\right|^{p}
   \leq \left(2\left\|\phi\right\|_\infty\right)^p
      \indi\left\{\mathsf X_i\left(s\right)\in\operatorname{supp}\phi
                    \text{ or }\mathsf X_i\left(t\right)\in\operatorname{supp}\phi\right\}.
 \]
 The indicator and determinant estimates from the proof of Lemma
 \ref{lemma for bound} then give, uniformly in \(L\), the majorant
 \[
       C\exp\left\{-c\sum_jx_{i_j}^2\right\}.
 \]
 The Gaussian summability argument in
 Remark~\ref{remark sum is uniformly bounded} applies to
 $\boldsymbol\zeta$ and its truncations. Thus the majorant is summable
 uniformly in $L$, and the right-hand side converges term by term
 and in absolute value.

 It remains to justify convergence of the unbounded polynomial on the
 left.  Apply Lemma \ref{lemma for bound} with the even exponent
 \(R=2\max\{2,Q\}>Q\).  The resulting uniform higher-moment bound makes
 the \(Q\)-th powers uniformly integrable.
 Theorem \ref{MainTheorem} gives convergence in distribution at the
 positive times, while at time \(0\) the initial configurations converge
 vaguely.  Uniform integrability therefore gives convergence of the
 expectations on the left, proving the stated infinite representation.
\end{proof}

The preceding estimates give a uniform Kolmogorov bound for the original processes and their spatial truncations. This is the tightness input in the proof of Theorem~\ref{theorem path cts}.
\begin{cor}\label{cor kol bd phi}
 Under the assumptions of Theorem \ref{theorem path cts}, for every even
 \(Q\geq4\), \(T<\infty\), and
 \(\phi\in\mathcal C_c^1(\mathbb R)\), there is
 \(C_{Q,T,\phi}<\infty\) such that
 \begin{equation}\label{eq B_phi}
 \sup_N\expt_{r^\UN,\bxi^\UN}\left[
 \left|
  \left\langle\phi,\bXi^\UN\left(t\right)\right\rangle
  -\left\langle\phi,\bXi^\UN\left(s\right)\right\rangle
 \right|^Q\right]
 \leq C_{Q,T,\phi}\left|t-s\right|^{Q/2},
 \qquad 0\leq s,t\leq T.
 \end{equation}
 The same estimate holds for all spatial truncations and for the limiting
 configuration.
\end{cor}

\begin{proof}
 Proposition \ref{bound of res}, Remark
 \ref{remark sum is uniformly bounded}, Lemma \ref{lemma for bound}, and
 Proposition \ref{condition for complex representation} apply with
 constants uniform in \(N\) and in the truncation level.
\end{proof}

\begin{proof}[Proof of Theorem \ref{theorem path cts}]
 We first construct the asserted continuous versions.  For fixed \(N\),
 the finite systems associated with
 \(\bxi^{\UN,[L]}\) and parameter \(r^\UN\) are, by Theorem
 \ref{theorem complex repre}, common-drift transforms of finite Dyson
 Brownian motion and hence have continuous paths.  Proposition
 \ref{bound of res} and Corollary \ref{cor kol bd phi} give estimates
 uniform in \(L\).  Since
 \(f_{r^\UN}^b(\bxi^{\UN,[L]})\to
   f_{r^\UN}^b(\bxi^\UN)\),
 Theorem \ref{MainTheorem} identifies every positive-time
 finite-dimensional limit with the kernel \(\K_{\bup^\UN}\).
 The same argument applies to \(\bxi\).

 We give the tightness argument, including compact containment.  Fix
  \(T<\infty\).  Choose a countable convergence-determining family
  \((\phi_k)_{k\geq1}\subset\mathcal C_c^\infty(\mathbb R)\) for the
  vague topology, put
 \[
  A_k\defeq1+C_{4,T,\phi_k},\qquad
  \psi_k\defeq A_k^{-1/4}\phi_k,
 \]
  and use the resulting compatible vague metric
 \begin{equation*}
  d_T\left(\boldsymbol{\mu},\boldsymbol{\nu}\right)
  \defeq\sum_{k=1}^\infty2^{-k}
    \frac{\left|\left\langle\psi_k,\boldsymbol{\mu}
                        -\boldsymbol{\nu}\right\rangle\right|}
         {1+\left|\left\langle\psi_k,\boldsymbol{\mu}
                        -\boldsymbol{\nu}\right\rangle\right|}.
 \end{equation*}
 Jensen's inequality and \eqref{eq B_phi} give
 \begin{equation}\label{Holder}
  \sup_{N,L}
  \expt\!\left[
   d_T\left(\bXi^{\UN,\left[L\right]}\left(t\right),
               \bXi^{\UN,\left[L\right]}\left(s\right)\right)^4
  \right]
  \leq
  \sum_{k=1}^\infty2^{-k}A_k^{-1}
       C_{4,T,\phi_k}\left|t-s\right|^2
  \leq \left|t-s\right|^2.
 \end{equation}
 The same bound holds without truncation and for the limit.

 For compact containment, choose nonnegative
 \(\chi_j\in\mathcal C_c^\infty(\mathbb R)\) with
 \(\chi_j\geq1\) on \([-j,j]\).  Rootwise convergence gives
 \[
        a_j\defeq\sup_{N,L}
        \left\langle\chi_j,\bxi^{\UN,\left[L\right]}\right\rangle<\infty.
 \]
  The fourth-moment estimate and the standard
  Garsia--Rodemich--Rumsey consequence (with any H\"older exponent below
  \(1/4\)) imply
 \[
 \sup_{N,L}\expt\left[
   \sup_{0\leq t\leq T}
   \left|\left\langle\chi_j,\bXi^{\UN,\left[L\right]}\left(t\right)
                    -\bxi^{\UN,\left[L\right]}\right\rangle\right|^4
 \right]\leq D_{j,T}<\infty.
 \]
 Given \(\varepsilon>0\), choose \(M_j>a_j\) so that
 \(D_{j,T}(M_j-a_j)^{-4}\leq\varepsilon2^{-j}\), and set
 \[
  K_\varepsilon
   \defeq\left\{\boldsymbol{\mu}\in\mathfrak M:
       \left\langle\chi_j,\boldsymbol{\mu}\right\rangle\leq M_j
       \text{ for every }j\right\}.
 \]
 This set is compact for the vague topology: it is closed, and its
 defining inequalities give a uniform mass bound on every compact
 interval.  The union bound yields
 \[
 \inf_{N,L}
 \prob\left\{\bXi^{\UN,\left[L\right]}\left(t\right)\in K_\varepsilon
       \text{ for every }0\leq t\leq T\right\}
 \geq1-\varepsilon.
 \]
  Compact containment together with \eqref{Holder} and the standard
  tightness criterion based on a uniform modulus of continuity proves
  tightness in
 \(\mathcal C([0,T]\to\mathfrak M)\).  Letting \(L\to\infty\) constructs
 continuous versions of every \(\bXi^\UN\) and of
 \(\bXi\).  Their values at zero are the corresponding
 deterministic configurations; continuity at zero follows from
 \eqref{eq B_phi} with \(s=0\).

 The same estimates, now uniform in \(N\), prove tightness of
 \((\bXi^\UN)_{N\geq1}\) on every finite time interval.
 Lemma \ref{gamma_2 is 0} identifies the limiting parameter as
 \(f_r^b(\bxi)\).  Hence Theorem \ref{MainTheorem} identifies the
 finite-dimensional distributions of every subsequential path limit at
 strictly positive times with those of \(\bXi\).  At time
 zero, rootwise convergence gives
 \(\bXi^\UN(0)=\bxi^\UN\to\bxi\) vaguely.  Continuity then
 identifies mixed finite-dimensional distributions which include time
 \(0\), by letting their positive time arguments decrease to zero.
 Since the Borel sigma-field on path space is generated by evaluations at
 a dense set of times, the subsequential limit is unique.  Thus the whole
 sequence converges on \([0,T]\).  Finally let \(T\) range over the
 positive integers to obtain convergence in
 \(\mathcal C([0,\infty)\to\mathfrak M)\).
\end{proof}
\section{Bounds for the first correlation function}\label{sec:density}

\subsection{Assumptions and main bound}

Fix
\begin{equation*}
 q\in\left(1,2\right),\qquad \theta\defeq q-1\in\left(0,1\right),\qquad\ep\in\{0,1\}.
\end{equation*}
Let $0<a_1<a_2<\cdots$ with $a_n\to\infty$, and fix $b>0$ such that
$b\ne a_n$ for every $n$.  Consider the symmetric configuration and its
finite truncations
\begin{equation}\label{eq:config}
 \bxi\defeq \ep\delta_0+\sum_{n=1}^{\infty}\left(\delta_{a_n}+\delta_{-a_n}\right),
 \quad
 \bxi^\UN\defeq \ep\delta_0+\sum_{n=1}^{N}\left(\delta_{a_n}+\delta_{-a_n}\right).
\end{equation}
Let $\rho_{\bxi^\UN,t}$ be the one-point density at time $t>0$ of the
finite Dyson model started from $\bxi^\UN$, and put
\begin{equation*}
 \rho_{N,t}\defeq \rho_{\bxi^\UN,t},
 \quad
 A_N\defeq \left\{\pm a_1,\ldots,\pm a_N\right\}.
\end{equation*}
For $u\ge0$ and $s>0$, define the local interval counts
\begin{equation*}
 \mathcal N\left(u,s\right)\defeq\#\left\{a\in\left\{\pm a_n:n\ge1\right\}:\left|a-u\right|\le s\right\},
 \qquad
 \mathcal N^+\left(u,s\right)\defeq\#\left\{n\ge1:\left|a_n-u\right|\le s\right\}.
\end{equation*}

We impose the following assumptions.

\emph{($\mathbf{UC}$) Upper local count.}  There is $C_{\mathrm{up}}<\infty$ such
that
\begin{equation*}
 \mathcal N\left(u,s\right)
 \le C_{\mathrm{up}}\left[1+s\left(1+u\right)^\theta+s^q\right],
 \quad u\ge0,\ s>0.
\end{equation*}

\emph{($\mathbf{LC}$) Lower local count.}  There are
$c_{\mathrm{low}}>0$, $C_{\mathrm{low}}<\infty$, and $U_0<\infty$ such
that
\begin{equation}\label{eq:LC}
 \mathcal N^+\left(u,s\right)\ge c_{\mathrm{low}}s u^\theta
 \quad\text{whenever}\quad
 u\ge U_0,\quad C_{\mathrm{low}}u^{-\theta}\le s\le \frac u4.
\end{equation}
If $U_0\le u\le a_N/4$ and $s\le u/4$, then
$u+s\le5a_N/16<a_N$.  Thus every positive root counted in
\eqref{eq:LC} already belongs to $A_N$ throughout the saddle range used
below.

The upper counting assumption also gives
\begin{equation}\label{eq:count-square-tail}
 \sum_{a_n>R}\frac1{a_n^2}\le C R^{q-2},\qquad R\ge1.
\end{equation}
Indeed, the roots in $2^jR<a_n\le2^{j+1}R$ contribute at most
$C(2^jR)^{q-2}$, and these bounds are summable because $q<2$.
Thus $\sum_n a_n^{-2}<\infty$. The ordered outer reciprocal coordinates
of $f_0^b(\bxi^\UN)$ converge coordinatewise, their squared sums converge
monotonically, the inner coordinates eventually stabilize, and the linear
coordinate $\gamma_1$ is zero. Consequently,
$f_0^b(\bxi^\UN)\to f_0^b(\bxi)$ in $\bUpsilon_b$.

We first bound the one-point density uniformly over finite truncations. The sublinear growth bound will imply the Markov property in Section~\ref{sec:markov}.
\begin{theorem}\label{thm bound}
Let $\bxi$ and $\bxi^\UN$ be given by \eqref{eq:config}.  Assume \emph{($\mathbf{UC}$)} and \emph{($\mathbf{LC}$)}.  For every $t>0$, there are constants
$R_t,C_t<\infty$, independent of $N$, such that
\begin{equation}\label{eq:interior-main}
 \rho_{N,t}\left(x\right)\le C_t\left(1+\left|x\right|^\theta\right),
 \quad R_t\le \left|x\right|\le \frac{a_N}{16}.
\end{equation}
Moreover,
\begin{equation}\label{eq:global-main}
 \sup_{N\ge1}\rho_{N,t}\left(x\right)
 \le C_t\left(1+\left|x\right|^{q/2}\right),
 \quad x\in\R.
\end{equation}
\end{theorem}

Since $q<2$, one has $\theta=q-1<q/2<1$.  Thus
\eqref{eq:global-main} has a uniformly sublinear exponent, while
\eqref{eq:interior-main} gives the sharper interior exponent.

\begin{remark}\label{rem:limit-density}
Here $\gamma_2=0$, so Proposition~\ref{well_defined process} supplies
the process $\bXXi$ for every positive time. Theorem~\ref{MainTheorem}
gives $\rho_{N,t}\to\rho_t$ locally uniformly, where
$\rho_t(x)=\K_{f_0^b(\bxi)}((t,x),(t,x))$.
For each fixed $|x|\ge R_t$, one has
$|x|\le a_N/16$ for all sufficiently large $N$, so
\eqref{eq:interior-main} passes to the limit. The density is bounded on
the remaining compact interval by local uniform convergence. Thus,
after enlarging $C_t$ if necessary,
\begin{equation*}
 \rho_t\left(x\right)\le C_t\left(1+\left|x\right|^{q-1}\right),\quad x\in\R.
\end{equation*}
\end{remark}

The proof is divided into three spatial regions.  A saddle analysis treats the region
$R_t<|x|\le a_N/16$; locally uniform convergence treats bounded $x$; and
a dimension-uniform deformed-GUE estimate treats $|x|>a_N/16$.

In Section 5.2-5.7, we prove Theorem \ref{thm bound} for $\ep=0$. Throughout these subsections, we impose this restriction. The case $\ep=1$ will be proved in Section 5.8 using the same auxiliary phase and contours.

\subsection{Correlation function and phase}

For all $N\in\mathbb{N}$, we define the function
\begin{equation*}
 Q_N\left(z\right)\defeq\prod_{n=1}^{N}\left(1-\frac{z^2}{a_n^2}\right).
\end{equation*}
Since $A_N=-A_N$, one has $p_b(\bxi^\UN)=0$, and directly from
\eqref{eq f_r},
\begin{equation*}
 \frac{\mathfrak E_{f_0^b\left(\bxi^\UN\right)}\left(w\right)}
      {\mathfrak E_{f_0^b\left(\bxi^\UN\right)}\left(z\right)}
 =\frac{Q_N\left(w\right)}{Q_N\left(z\right)}.
\end{equation*}
Proposition~\ref{finitedysonmodel} and \eqref{eq kernel finite dim},
specialized to equal times and followed by the change of variables
$w=\mi u$, therefore give
\begin{equation}\label{eq:kernel}
 \rho_{N,t}\left(x\right)
 =\frac{1}{\left(2\pi \mi\right)^2t}
   \oint_{C_N}\,\dif  z\int_{\mi\R}\dif w\,
   \frac{\exp\!\left\{\frac{\left(w-x\right)^2-\left(z-x\right)^2}{2t}\right\}}
        {w-z}\,
   \frac{Q_N\left(w\right)}{Q_N\left(z\right)}.
\end{equation}
Here $C_N=C_{\bxi^\UN}$ may be any positively oriented union of small
loops enclosing the points of $A_N$ once and disjoint from $\mi\R$;
the line $\mi\R$ is oriented upward. The resulting double integral is
an ordinary absolutely convergent integral.

By Lemma~\ref{lemma finite contour correction}, the separated finite-root
representation \eqref{eq:kernel} already equals the full kernel at equal times
$\K_{f_0^b(\bxi^\UN)}$. The residue $v/(\pi t)$ in
\eqref{eq:decomposition} is created by the subsequent deformation of the
$w$-contour and is distinct from the fixed-contour correction
$\Rc_\Gamma$ of Definition~\ref{def def of K}.

Introduce the phase
\begin{equation*}
 \Phi_{N,x}\left(z\right)\defeq\frac{\left(z-x\right)^2}{2t}+\log Q_N\left(z\right).
\end{equation*}
Only phase differences and derivatives occur, so local logarithm branches
suffice.  Formula \eqref{eq:kernel} becomes
\begin{equation*}
 \rho_{N,t}\left(x\right)
 =\frac{1}{\left(2\pi \mi\right)^2t}
   \oint_{C_N}\,\dif  z\int_{\mi\R}\dif w\,
   \frac{\me^{\Phi_{N,x}\left(w\right)-\Phi_{N,x}\left(z\right)}}{w-z}.
\end{equation*}
For fixed $z$, the $w$-integrand has only the simple pole $w=z$; the
Gaussian controls vertical contour movements at imaginary infinity.

Writing $z=u+\mi v$, the saddle equation $\Phi'_{N,x}(z)=0$ is equivalent to
\begin{equation}\label{eq:saddle-system}
 x=u+t\sum_{a\in A_N}\frac{u-a}{\left(u-a\right)^2+v^2},
 \quad
 v\left(\frac1t-\sum_{a\in A_N}\frac1{\left(u-a\right)^2+v^2}\right)=0.
\end{equation}

\subsection{Consequences of the counting assumptions}

Applying \emph{($\mathbf{UC}$)} with $u=0$ and $s=a_N$ gives
\begin{equation}\label{eq:edge-size}
 N\le C\left(1+a_N^q\right).
\end{equation}
We shall use the reciprocal-square tail bound
\eqref{eq:count-square-tail} established above.

We shall also use two standard consequences of this estimate.  The products
$Q_N$ converge locally uniformly to
\begin{equation*}
 Q_\infty\left(z\right)\defeq \prod_{n=1}^{\infty}\left(1-\frac{z^2}{a_n^2}\right),
\end{equation*}
and, uniformly in $N$,
\begin{equation*}
 \log\left|Q_N\left(z\right)\right|\le C\left(1+\left|z\right|^q\right),\quad z\in\C.
\end{equation*}
For completeness, take $R=1+|z|$.  The factors with $a_n>2R$ are bounded
using $\log(1+s)\le s$ and the reciprocal-square tail.  For
$a_n\le2R$, dyadically group the roots according to their size; \emph{($\mathbf{UC}$)}
at the origin bounds the number in the $j$th block by
$C[1+(2^{-j}R)^q]$, while its logarithmic weight is $O(j+1)$.  Summing
until the fixed first root is reached gives $O(1+R^q)$.  The same tail
estimate is the Weierstrass criterion for local uniform convergence.

For $k>0$, define
\begin{equation*}
 \sigma_{N,k}\left(u,v\right)
 \defeq\sum_{a\in A_N}\frac1{\left(\left(u-a\right)^2+v^2\right)^{k/2}},
 \quad u\ge0,\ v>0,
\end{equation*}
and write $\sigma_N=\sigma_{N,2}$.

The upper counting assumption controls the resolvent sums appearing in the saddle equation. We record a bound uniform in the truncation.
\begin{lemma}\label{lem:resolvent}
If $k>q$, then, uniformly in $N$, $u\ge0$, and $v>0$,
\begin{equation*}
 \sigma_{N,k}\left(u,v\right)
 \le C_k\left[v^{-k}+\left(1+u\right)^\theta v^{1-k}+v^{q-k}\right].
\end{equation*}
\end{lemma}

\begin{proof}
Set
\begin{equation*}
 B_0\defeq \left\{a\in A_N:\left|a-u\right|\le v\right\},\quad
 B_j\defeq \left\{a\in A_N:2^{j-1}v<\left|a-u\right|\le2^jv\right\},\quad j\ge1.
\end{equation*}
These sets cover $A_N$.  The contribution of $B_0$ is bounded directly by
\emph{($\mathbf{UC}$)}.  If $j\ge1$, then
\begin{equation*}
 \sum_{a\in B_j}\frac1{\left(\left(u-a\right)^2+v^2\right)^{k/2}}
 \le \#B_j\,\left(2^{j-1}v\right)^{-k}.
\end{equation*}
Moreover,
$\#B_j\le C[1+2^jv(1+u)^\theta+(2^jv)^q]$.  Hence
\begin{equation*}
 \sum_{a\in B_j}\frac1{\left(\left(u-a\right)^2+v^2\right)^{k/2}}
 \le C_k\!\left[
  2^{-jk}v^{-k}
  +2^{-j\left(k-1\right)}\left(1+u\right)^\theta v^{1-k}
  +2^{-j\left(k-q\right)}v^{q-k}\right].
\end{equation*}
All three geometric series converge because $k>q>1$.
\end{proof}

In particular, taking $k=2$ gives
\begin{equation}\label{eq:resolvent-two}
 \sup_{N\ge1}\sigma_N\left(u,v\right)
 \le C\left[v^{-2}+\frac{\left(1+u\right)^\theta}{v}+v^{q-2}\right].
\end{equation}

The lower counting assumption places a root close to every sufficiently large point. This ensures that the saddle equation has positive height in the interior region.
\begin{lemma}\label{lem:nearest}
There are $C,U<\infty$ such that
\begin{equation*}
 \inf_{n\ge1}\left|u-a_n\right|\le Cu^{-\theta},\quad u\ge U.
\end{equation*}
\end{lemma}

\begin{proof}
Choose $K\ge C_{\mathrm{low}}$ with $c_{\mathrm{low}}K\ge1$ and apply
\emph{($\mathbf{LC}$)} with $s=Ku^{-\theta}$.  For large $u$ this scale is at most
$u/4$, and the corresponding interval contains a positive root.
\end{proof}

The same assumption forces the gaps between consecutive roots to vanish. We use this to construct a connected tail contour through the saddle.
\begin{lemma}\label{lem:gaps}
There are $K<\infty$ and $n_0$ such that, for every $n\ge n_0$,
\begin{equation}\label{eq:gaps}
 a_{n+1}-a_n
 \le K\left(\frac{a_n+a_{n+1}}2\right)^{-\theta}.
\end{equation}
In particular, $a_{n+1}-a_n\to0$.
\end{lemma}

\begin{proof}
Apply Lemma~\ref{lem:nearest} at $u_n=(a_n+a_{n+1})/2$.
The nearest positive root is an endpoint, so
$(a_{n+1}-a_n)/2\le C u_n^{-\theta}$ for all sufficiently large $n$.
This proves \eqref{eq:gaps}.
\end{proof}

Fix $t>0$.  By Lemma~\ref{lem:gaps}, all sufficiently late gaps are less
than $2\sqrt t$.  If $a<b$ are consecutive roots in this tail and
$r\in[a,b]$, then, with the convention $1/0=+\infty$,
\begin{equation}\label{eq:real-superlevel}
 \frac1{\left(r-a\right)^2}+\frac1{\left(b-r\right)^2}
 \ge\frac8{\left(b-a\right)^2}>\frac1t.
\end{equation}
Thus every sufficiently late positive gap lies in the real trace of the
symmetric superlevel domain introduced below.

\subsection{The saddle curve and the two critical points}

We now construct the positive saddle height and identify its order of growth. This height determines the leading contribution to the density.
\begin{lemma}\label{lem:height}
There is $U_t<\infty$ such that, whenever
\begin{equation}\label{eq:height-range}
 U_t\le u\le \frac{a_N}{4},
\end{equation}
there is a unique $v_N(u)>0$ satisfying
$t\sigma_N(u,v_N(u))=1$.  Moreover,
\begin{equation}\label{eq:height-size}
 c_tu^\theta\le v_N\left(u\right)\le C_tu^\theta,
\end{equation}
with constants independent of $N$ and $u$.
\end{lemma}

\begin{proof}
Let $d_N(u)=\dist(u,A_N)$.  By Lemma~\ref{lem:nearest} and the
finite-truncation observation following \eqref{eq:LC},
$d_N(u)\le Cu^{-\theta}$ on \eqref{eq:height-range}, once $U_t$ is large.
Consequently,
\begin{equation*}
 \sigma_N\left(u,0\right)\ge d_N\left(u\right)^{-2}>1/t.
\end{equation*}
Here and below we use the convention
\begin{equation*}
 \sigma_N\left(u,0\right)\defeq \lim_{v\to 0}\sigma_N\left(u,v\right)\in\left(0,\infty\right].
\end{equation*}
Thus $\sigma_N(u,0)>1/t$.  On $(0,\infty)$, the function
$v\mapsto\sigma_N(u,v)$ is continuous and strictly decreasing, and it
tends to zero as $v\to\infty$.  This proves existence and uniqueness.

For the upper bound, insert $v=Ku^\theta$ into
\eqref{eq:resolvent-two}:
\begin{equation*}
 \sigma_N\left(u,Ku^\theta\right)
 \le C\left[K^{-2}u^{-2\theta}+K^{-1}+K^{q-2}u^{\theta\left(q-2\right)}\right].
\end{equation*}
Because $q-2<0$, this is less than $1/t$ when $K$ and then $U_t$ are
large.  Monotonicity gives $v_N(u)\le C_tu^\theta$.

For the lower bound, the nearest-root term and
$\sigma_N(u,v_N(u))=1/t$ imply
\begin{equation*}
 \frac1t\ge\frac1{d_N\left(u\right)^2+v_N\left(u\right)^2}.
\end{equation*}
After increasing $U_t$, $d_N(u)^2\le t/2$, and hence
$v_N(u)\ge\sqrt{t/2}$.  Thus $v_N(u)$ lies above the lower scale in
\emph{($\mathbf{LC}$)}.  The upper height bound and $\theta<1$ give
$v_N(u)\le u/4$ for large $u$, and the finite-truncation observation makes
the count below a count in $A_N$:
\begin{equation}\label{eq:height-count}
 \#\left\{a\in A_N:\left|a-u\right|\le v_N\left(u\right)\right\}
 \ge c\,v_N\left(u\right)u^\theta.
\end{equation}
Every term indexed by this set is at least
$[2v_N(u)^2]^{-1}$.  Therefore
\begin{equation*}
 \frac1t=\sigma_N\left(u,v_N\left(u\right)\right)
 \ge c\frac{u^\theta}{v_N\left(u\right)},
\end{equation*}
which proves the lower bound in \eqref{eq:height-size}.
\end{proof}

Define the real saddle map
\begin{equation}\label{eq:Psi}
 \Psi_N\left(u\right)
 \defeq u+t\sum_{a\in A_N}
       \frac{u-a}{\left(u-a\right)^2+v_N\left(u\right)^2}.
\end{equation}

Differentiating the saddle equation controls both the height and the real saddle map. These bounds give uniqueness of the saddle and quadratic decay along the contour.
\begin{lemma}\label{lem:derivatives}
For $u$ in \eqref{eq:height-range}, put
\begin{equation*}
 D_a\left(u\right)\defeq \left(u-a\right)^2+v_N\left(u\right)^2,
 \quad
 \mathfrak a_N\left(u\right)\defeq \sum_{a\in A_N}D_a\left(u\right)^{-2},
 \quad
 \mathfrak b_N\left(u\right)\defeq \sum_{a\in A_N}\left(u-a\right)D_a\left(u\right)^{-2}.
\end{equation*}
Then
\begin{equation}\label{eq:derivative-identities}
 v_N'\left(u\right)=-\frac{\mathfrak b_N\left(u\right)}{v_N\left(u\right)\mathfrak a_N\left(u\right)},
 \quad
 \Psi_N'\left(u\right)=2t\left[v_N\left(u\right)^2\mathfrak a_N\left(u\right)
                +\frac{\mathfrak b_N\left(u\right)^2}{\mathfrak a_N\left(u\right)}\right].
\end{equation}
There are constants $c_t,C_t>0$, independent of $N$ and $u$, such that
\begin{equation}\label{eq:derivative-bounds}
 c_t\le\Psi_N'\left(u\right)\le2,
 \quad
 \left|v_N'\left(u\right)\right|\le C_t.
\end{equation}
In particular, $\Psi_N$ is strictly increasing.
\end{lemma}

\begin{proof}
The implicit-function theorem applies because
\begin{equation*}
 \partial_v\sigma_N\left(u,v\right)
 =-2v\sum_{a\in A_N}\left(\left(u-a\right)^2+v^2\right)^{-2}<0.
\end{equation*}
Differentiating $\sum_aD_a(u)^{-1}=1/t$ gives the first identity in
\eqref{eq:derivative-identities}.  Direct differentiation of
\eqref{eq:Psi}, followed by this substitution, gives the second.

By \eqref{eq:height-count} and \eqref{eq:height-size}, the interval
$|a-u|\le v_N(u)$ contains at least $c_tv_N(u)^2$ roots.  For those roots,
$D_a(u)\le2v_N(u)^2$, so
\begin{equation}\label{eq:curvature}
 v_N\left(u\right)^2\mathfrak a_N\left(u\right)\ge c_t.
\end{equation}
This proves the lower bound for $\Psi_N'$.  Cauchy--Schwarz gives
\begin{equation*}
 \frac{\mathfrak b_N\left(u\right)^2}{\mathfrak a_N\left(u\right)}
 \le\sum_{a\in A_N}\frac{\left(u-a\right)^2}{D_a\left(u\right)^2}.
\end{equation*}
Consequently,
\begin{equation*}
 v_N\left(u\right)^2\mathfrak a_N\left(u\right)
 +\frac{\mathfrak b_N\left(u\right)^2}{\mathfrak a_N\left(u\right)}
 \le\sum_{a\in A_N}\frac1{D_a\left(u\right)}=\frac1t,
\end{equation*}
which proves $\Psi_N'\le2$.  The same Cauchy--Schwarz estimate gives
$|\mathfrak b_N|^2\le\mathfrak a_N/t$; combining this with
\eqref{eq:curvature} and the first identity in
\eqref{eq:derivative-identities} proves the bound for $v_N'$.
\end{proof}

The real saddle map differs from the identity by a sublinear term. This allows us to locate the saddle corresponding to a prescribed spatial point.
\begin{lemma}\label{lem:Psi-location}
Uniformly on \eqref{eq:height-range},
\begin{equation}\label{eq:Psi-location}
 \left|\Psi_N\left(u\right)-u\right|\le C_tu^\theta\log\left(2+u\right).
\end{equation}
\end{lemma}

\begin{proof}
Write $v=v_N(u)$.  We estimate the absolute value of
\begin{equation*}
 \sum_{a\in A_N}\frac{u-a}{\left(u-a\right)^2+v^2}
\end{equation*}
by separating roots with $a_n\le2u$ from paired roots with $a_n>2u$.

For the near roots, use the annuli
\begin{equation*}
 B_0\defeq \left\{a:\left|a-u\right|\le v\right\},
 \quad
 B_j\defeq \left\{a:2^{j-1}v<\left|a-u\right|\le2^jv\right\},\quad j\ge1,
\end{equation*}
through the smallest integer $J$ for which $2^Jv\ge3u$.  By minimality,
$2^Jv<6u$.  On $B_0$, \emph{($\mathbf{UC}$)} and
\eqref{eq:height-size} give
\begin{equation*}
 \sum_{a\in B_0}\frac{\left|u-a\right|}{\left(u-a\right)^2+v^2}
 \le C\left[v^{-1}+\left(1+u\right)^\theta+v^{q-1}\right]
 \le C_tu^\theta.
\end{equation*}
For $j\ge1$, the same count at radius $2^jv$ gives
\begin{equation*}
 \sum_{a\in B_j}\frac{\left|u-a\right|}{\left(u-a\right)^2+v^2}
 \le C\left[\left(2^jv\right)^{-1}+\left(1+u\right)^\theta+\left(2^jv\right)^{q-1}\right]
 \le C_tu^\theta.
\end{equation*}
The number of nonempty annuli is $O_t(\log(2+u))$, because
$v\ge c_tu^\theta$.  Hence
\begin{equation}\label{eq:Psi-near}
 \sum_{\substack{a\in A_N\\ \left|a\right|\le2u}}
 \frac{\left|u-a\right|}{\left(u-a\right)^2+v^2}
 \le C_tu^\theta\log\left(2+u\right).
\end{equation}

For $r>2u$, pair $r$ and $-r$:
\begin{equation*}
 \frac{u-r}{\left(u-r\right)^2+v^2}+\frac{u+r}{\left(u+r\right)^2+v^2}
 =\frac{2u\left(u^2+v^2-r^2\right)}
 {\left(\left(u-r\right)^2+v^2\right)\left(\left(u+r\right)^2+v^2\right)}.
\end{equation*}
For large $u$, \eqref{eq:height-size} gives $v\le u$; hence the last
display is $O(u/r^2)$.  The reciprocal-square tail therefore yields
\begin{equation}\label{eq:Psi-far}
 \left|\sum_{\substack{1\le n\le N\\a_n>2u}}
 \left[
 \frac{u-a_n}{\left(u-a_n\right)^2+v^2}
 +\frac{u+a_n}{\left(u+a_n\right)^2+v^2}
 \right]\right|
 \le Cu\left(2u\right)^{q-2}=Cu^\theta.
\end{equation}
Combining \eqref{eq:Psi-near} and \eqref{eq:Psi-far}, then multiplying
by $t$, proves \eqref{eq:Psi-location}.
\end{proof}

The preceding bounds give a conjugate pair of critical points throughout the interior region. Their heights have the same order as the density bound we seek.
\begin{proposition}\label{prop:saddles}
There is $R_t<\infty$, independent of $N$, such that, whenever
$R_t\le x\le a_N/16$, there is a unique
$u_N(x)\in[U_t,a_N/4]$ satisfying $\Psi_N(u_N(x))=x$.  Moreover,
\begin{equation*}
 \frac x2\le u_N\left(x\right)\le2x,
 \quad
 c_tx^\theta\le v_N\left(u_N\left(x\right)\right)\le C_tx^\theta,
\end{equation*}
and the phase has the conjugate critical points
\begin{equation}\label{eq:critical-points}
 z_\pm\defeq u_N\left(x\right)\pm\mi v_N\left(u_N\left(x\right)\right).
\end{equation}
\end{proposition}

\begin{proof}
Increase $R_t$ so that $R_t/2\ge U_t$ and
\begin{equation*}
 C_tr^\theta\log\left(2+r\right)<\frac r2
 \quad\text{for every }r\ge R_t/2.
\end{equation*}
Then
Lemma~\ref{lem:Psi-location} gives
\begin{equation*}
 \Psi_N\left(x/2\right)<x<\Psi_N\left(2x\right).
\end{equation*}
Because $x\le a_N/16$, the whole interval
$[x/2,2x]$ lies in $[U_t,a_N/4]$.  Continuity and strict monotonicity give
the unique solution.  The height estimate follows from
Lemma~\ref{lem:height}, and \eqref{eq:saddle-system} gives
\eqref{eq:critical-points}.
\end{proof}

\subsection{Steepest-descent geometry}

Define the symmetric superlevel domain
\begin{equation*}
 \Omega_{N,t}
 \defeq\left\{r+\mi s\in\C:
 t\sum_{a\in A_N}\frac1{\left(r-a\right)^2+s^2}>1\right\},
\end{equation*}
with the convention that the summand is $+\infty$ at a source.  Write
\begin{equation*}
 h_{N,t}\left(r,s\right)\defeq t\sum_{a\in A_N}\frac1{\left(r-a\right)^2+s^2}.
\end{equation*}

Choose $m=m(t)$ so large that $a_m\ge U_t$ and
\begin{equation*}
 a_{n+1}-a_n<2\sqrt t,\quad n\ge m.
\end{equation*}
Fix a non-source point
\begin{equation*}
 U\in\left(a_m,a_{m+1}\right).
\end{equation*}
Then, for every $N\ge m+1$, \eqref{eq:real-superlevel} shows that
\begin{equation*}
 h_{N,t}\left(r,0\right)>1,\quad U\le r\le a_N.
\end{equation*}
For $r>a_N$, the function $r\mapsto h_{N,t}(r,0)$ is strictly decreasing
from $+\infty$ to zero.  Hence there is a unique $R_{N,t}^+>a_N$ such that
$h_{N,t}(R_{N,t}^+,0)=1$, and
\begin{equation*}
 \left\{r\ge U:h_{N,t}\left(r,0\right)>1\right\}=\left[U,R_{N,t}^+\right).
\end{equation*}

For $U\le r<R_{N,t}^+$, the function $s\mapsto h_{N,t}(r,s)$ decreases
strictly on $(0,\infty)$.  Hence there is a unique global tail height
$\widehat v_N(r)>0$ satisfying $h_{N,t}(r,\widehat v_N(r))=1$.  It is
$C^1$ on $(U,R_{N,t}^+)$, is continuously extendible to $U$, and extends
continuously to the right endpoint by $\widehat v_N(R_{N,t}^+)=0$.  Put
\begin{equation*}
 \widehat z_N\left(r\right)\defeq r+\mi\widehat v_N\left(r\right),
 \quad U\le r\le R_{N,t}^+.
\end{equation*}
On $[U,a_N/4]$, uniqueness implies $\widehat v_N=v_N$.

\noindent
The cut $U$ and the index after which this construction works are independent
of $N$.  The truncated tail
\begin{equation*}
 \cT_{N,t}^+\defeq \overline{\Omega_{N,t}\cap\left\{\Re z>U\right\}}
\end{equation*}
has the exact description
\begin{equation*}
 \cT_{N,t}^+
 =\left\{r+\mi s:
 U\le r\le R_{N,t}^+,
 \quad \left|s\right|\le\widehat v_N\left(r\right)\right\}.
\end{equation*}
In particular, it is connected.  It is a truncated part of a component
of $\Omega_{N,t}$.

Introduce
\begin{equation*}
 H_N\left(z\right)\defeq z+t\sum_{a\in A_N}\frac1{z-a}.
\end{equation*}
Define, on the open global tail,
\begin{equation*}
 \widehat\Psi_N\left(r\right)
 \defeq r+t\sum_{a\in A_N}
 \frac{r-a}{\left(r-a\right)^2+\widehat v_N\left(r\right)^2}.
\end{equation*}
The boundary equation gives
\begin{equation*}
 H_N\left(\widehat z_N\left(r\right)\right)=\widehat\Psi_N\left(r\right)\in\R.
\end{equation*}
The algebraic derivative identity in Lemma~\ref{lem:derivatives}
applies to the global height and gives $\widehat\Psi_N'(r)>0$
for $U<r<R_{N,t}^+$.  Thus $\widehat\Psi_N$ is strictly increasing and
extends monotonically to the right endpoint.  Since $R_{N,t}^+>a_N$, all
denominators remain nonzero there.  Thus $\widehat\Psi_N$ and the global
boundary phase defined below extend continuously to $R_{N,t}^+$.

Fix $x$ in the range of Proposition~\ref{prop:saddles}, and abbreviate
\begin{equation*}
 u\defeq u_N\left(x\right),\quad v\defeq v_N\left(u\right),\quad z_+=u+\mi v.
\end{equation*}
On the local upper boundary put
\begin{equation*}
 G_N\left(r\right)\defeq \Re\Phi_{N,x}\left(r+\mi v_N\left(r\right)\right).
\end{equation*}
Since $\Phi_{N,x}'(z)=(H_N(z)-x)/t$, one has
\begin{equation*}
 G_N'\left(r\right)=\frac{\Psi_N\left(r\right)-x}{t}.
\end{equation*}
Using $x=\Psi_N(u)$ and \eqref{eq:derivative-bounds}, we obtain
\begin{equation}\label{eq:horizontal-gaussian}
 G_N\left(r\right)-G_N\left(u\right)\ge c_t\left(r-u\right)^2
 \quad\text{whenever the interval between $r$ and $u$ lies in }
 \left[U_t,a_N/4\right].
\end{equation}
For the global tail define
$\widehat G_N(r)=\Re\Phi_{N,x}(\widehat z_N(r))$.  Then
\begin{equation*}
 \widehat G_N'\left(r\right)=\frac{\widehat\Psi_N\left(r\right)-x}{t}
\end{equation*}
on the open tail, and its monotonicity statements extend to the endpoint by
continuity.

For the vertical line $L_u=\{u+\mi s:s\in\R\}$, use the branch-independent
real phase
\begin{equation*}
 F_N\left(s\right)
 \defeq \frac{\left(u-x\right)^2-s^2}{2t}+\log\left|Q_N\left(u+\mi s\right)\right|.
\end{equation*}
If $u\in A_N$, set $F_N(0)=-\infty$; otherwise the same formula applies at
$s=0$.  Thus real-part phase notation below always means this
single-valued expression.  Direct differentiation away from a possible
source at $s=0$ gives
\begin{equation}\label{eq:F-derivative}
 F_N'\left(s\right)=s\left(\sigma_N\left(u,\left|s\right|\right)-\frac1t\right),
 \quad s\ne0.
\end{equation}

Along the vertical contour, the real part of the phase decreases quadratically away from the two critical points. The resulting estimate controls the unbounded part of the double integral.
\begin{lemma}\label{lem:vertical}
There is $c_t>0$, independent of $N,u,s$, such that
\begin{equation}\label{eq:vertical-gaussian}
 F_N\left(s\right)-F_N\left(v\right)
 \le-c_t\dist\left(s,\left\{-v,v\right\}\right)^2,
 \quad s\in\R.
\end{equation}
\end{lemma}

\begin{proof}
The function $F_N$ is even, so take $s\ge0$.  If $u\in A_N$, the estimate
at $s=0$ follows by letting $s\downarrow0$.  Let
$E=\{a\in A_N:|a-u|\le v\}$.  By \eqref{eq:height-count} and
\eqref{eq:height-size}, $\#E\ge c_tv^2$.

If $0\le s\le v$, then
\begin{equation*}
 \sigma_N\left(u,s\right)-\sigma_N\left(u,v\right)
 =\sum_{a\in A_N}
 \frac{v^2-s^2}{\left(\left(u-a\right)^2+s^2\right)\left(\left(u-a\right)^2+v^2\right)}.
\end{equation*}
Restricting to $E$ yields
$\sigma_N(u,s)-1/t\ge c_t(v^2-s^2)/v^2$.  Integrating
\eqref{eq:F-derivative} gives
\begin{equation}\label{eq:vertical-one}
 F_N\left(v\right)-F_N\left(s\right)\ge c_t\left(v-s\right)^2,
 \quad 0\le s\le v.
\end{equation}
The same computation, now using
$((u-a)^2+s^2)\le5v^2$ for $a\in E$ and $v\le s\le2v$, gives
\begin{equation}\label{eq:vertical-two}
 F_N\left(v\right)-F_N\left(s\right)\ge c_t\left(s-v\right)^2,
 \quad v\le s\le2v.
\end{equation}
Finally,
\begin{equation*}
 \sigma_N\left(u,v\right)-\sigma_N\left(u,2v\right)\ge c_t,
\end{equation*}
so monotonicity of $\sigma_N(u,s)$ implies
$F_N'(s)\le-c_ts$ for $s\ge2v$.  Hence
\begin{equation}\label{eq:vertical-three}
 F_N\left(s\right)-F_N\left(2v\right)\le-c_t\left(s^2-4v^2\right),
 \quad s\ge2v.
\end{equation}
Combining this with \eqref{eq:vertical-two} at $2v$ gives
\begin{equation}\label{eq:vertical-four}
 F_N\left(v\right)-F_N\left(s\right)\ge c_t\left(s-v\right)^2,
 \quad s\ge2v.
\end{equation}
Equations \eqref{eq:vertical-one}--\eqref{eq:vertical-four}, followed by
evenness, prove \eqref{eq:vertical-gaussian}.
\end{proof}

\subsection{Global contour and the interior kernel estimate}

We also need a bound on the length of the global contour. Polynomial growth of its length will be dominated by the phase decay near the finite edge.
\begin{lemma}\label{lem:length}
For every fixed $t>0$,
\begin{equation}\label{eq:length}
 \length\left(\partial\Omega_{N,t}\right)
 \le C_tNa_N
 \le C_t\left(1+a_N^{q+1}\right).
\end{equation}
The constant is independent of $N$; it may depend on $t$, $q$, the constant
in \emph{($\mathbf{UC}$)}, and the fixed configuration through $a_1>0$.
\end{lemma}

\begin{proof}
Put $D_a(r,s)=(r-a)^2+s^2$.  Away from the sources, every boundary point
satisfies
\begin{equation}\label{eq:boundary-level}
 t\sum_{a\in A_N}\frac1{D_a\left(r,s\right)}=1.
\end{equation}
Clear denominators and define
\begin{equation*}
 P_N\left(r,s\right)
 \defeq t\sum_{a\in A_N}\prod_{\substack{b\in A_N\\b\ne a}}D_b\left(r,s\right)
   -\prod_{a\in A_N}D_a\left(r,s\right).
\end{equation*}
Away from the sources,
$P_N=(\prod_aD_a)(h_{N,t}-1)$, while at a source $P_N>0$.
Consequently, $\Omega_{N,t}=\{P_N>0\}$ is semialgebraic and
$\partial\Omega_{N,t}\subset\{P_N=0\}$.  Since $A_N$ has $2N$ points,
\begin{equation*}
 \deg P_N=4N,
 \quad
 \text{the leading homogeneous part is }-\left(r^2+s^2\right)^{2N}.
\end{equation*}
The restriction of $P_N$ to any real affine line with nonzero direction
$(\alpha,\beta)$ has leading coefficient
$-(\alpha^2+\beta^2)^{2N}\ne0$.  Thus no affine line is a component of
the algebraic curve.

We next bound the curve spatially.  At a boundary point with $s\ne0$,
\eqref{eq:boundary-level} gives
\begin{equation*}
 \left|s\right|\le\sqrt{2Nt};
\end{equation*}
the same inequality is trivial when $s=0$.  If
$d(r)=\dist(r,[-a_N,a_N])>0$, then
\eqref{eq:boundary-level} gives
\begin{equation*}
 d\left(r\right)\le\sqrt{2Nt};
\end{equation*}
when $d(r)=0$ there is nothing to prove.  By \eqref{eq:edge-size} and
$q/2<1$, the boundary is therefore contained in a disk of radius
\begin{equation}\label{eq:containing-disk}
 R_{N,t}^{\mathrm{disk}}\le C_t\left(1+a_N\right).
\end{equation}

This bounded semialgebraic boundary is a finite union of rectifiable
arcs and points, so the Cauchy--Crofton formula applies. Every affine
line meets the boundary at at most $4N$ points, because the restriction
of $P_N$ is a nonzero polynomial of degree at most $4N$. Only lines
meeting the disk in \eqref{eq:containing-disk} contribute. Hence
\begin{equation*}
 \length\left(\partial\Omega_{N,t}\right)
 \le C_tN\left(1+a_N\right)
 \le C_tNa_N.
\end{equation*}
Here the final constant may depend on $a_1>0$.  Combining this with
\eqref{eq:edge-size} proves \eqref{eq:length}.
\end{proof}

We now deform the finite-particle kernel through the critical points. The crossed residue gives the saddle height; the remaining double integral is uniformly bounded.
\begin{proposition}\label{prop:interior-bound}
For every fixed $t>0$, there are $R_t,C_t<\infty$, independent of $N$,
such that, whenever $R_t\le x\le a_N/16$,
\begin{equation*}
 \rho_{N,t}\left(x\right)
 \le C_t+\frac{v_N\left(u_N\left(x\right)\right)}{\pi t}
 \le C_t\left(1+x^\theta\right).
\end{equation*}
By symmetry, the same estimate holds for negative $x$.
\end{proposition}

\begin{proof}
We treat $x>0$. We construct the contour and bound its pieces before
performing the deformation.

Use the non-source cut $U$ fixed above.  Choose
$N_t\ge\max\{m+1,2\}$ so large
that, for every $N\ge N_t$, the truncated tail $\cT_{N,t}^+$ is available, $a_N/4>U$, and the finite core to
the left of $U$ has stabilized.  The positively oriented boundary
$\Gamma_N^+$ of $\cT_{N,t}^+$ starts at
$U-\mi\widehat v_N(U)$, follows the lower graph to $R_{N,t}^+$, returns on
the upper graph to $U+\mi\widehat v_N(U)$, and closes by the downward chord
at $U$.  Reflect through the origin to obtain the positively oriented
negative-tail contour $\Gamma_N^-=-\Gamma_N^+$.

Surround the finitely many positive sources to the left of $U$ by a fixed
union $\Gamma^{0,+}$ of small positively oriented loops contained in
$\{0<\Re z<U\}$, and put $\Gamma^{0,-}=-\Gamma^{0,+}$.  Define, once and
for all in the rest of the proof,
\begin{equation*}
 \Gamma_N
 \defeq \Gamma_N^+\cup\Gamma^{0,+}\cup\Gamma_N^-\cup\Gamma^{0,-}.
\end{equation*}
Every source is enclosed exactly once, and the two central unions stay on
their respective sides of the initial line $\mi\R$.

The chord heights are uniformly bounded.  Indeed,
\eqref{eq:resolvent-two}, at the fixed abscissae $\pm U$, gives
\begin{equation*}
 \sup_N\sum_{a\in A_N}\frac1{\left(U-a\right)^2+s^2}\longrightarrow0
 \quad \left(\left|s\right|\to\infty\right).
\end{equation*}
Let $K_\Gamma$ be the fixed compact union of the central loop traces
and vertical segments containing the chords, chosen disjoint from the
zeros of $Q_\infty$. Local uniform convergence of $Q_N$, followed by
enlargement of $N_t$, gives
\begin{equation}\label{eq:carrier-product}
 0<c_\Gamma\le\left|Q_N\left(z\right)\right|\le C_\Gamma<\infty,
 \quad z\in K_\Gamma,\quad N\ge N_t.
\end{equation}

Put $M=\sup_{z\in K_\Gamma}|\Re z|$.  Enlarge $R_t$ so that
Proposition~\ref{prop:saddles} applies,
\begin{equation*}
 R_t>2U,\quad R_t\ge4M,\quad
 R_t>\max_{N<N_t}\frac{a_N}{16}.
\end{equation*}
For $N<N_t$, the asserted range is empty.  Henceforth take $N\ge N_t$ and
$R_t\le x\le a_N/16$, and abbreviate
\begin{equation*}
 u=u_N\left(x\right),\quad v=v_N\left(u\right),\quad z_+=u+\mi v,\quad z_-=u-\mi v.
\end{equation*}
Then $x/2\le u\le2x$, $c_tx^\theta\le v\le C_tx^\theta$, and
$u>U$, $u\ge2M$.

Put $r_0=a_N/4$.  Since $u\le2x\le a_N/8$, one has
$r_0-u\ge a_N/8$.  On $[U,r_0]$, the global functions agree with the
local ones, so \eqref{eq:horizontal-gaussian} gives
\begin{equation*}
 \widehat G_N\left(r_0\right)-\widehat G_N\left(u\right)\ge c_ta_N^2.
\end{equation*}
Define the outer positive arc
\begin{equation*}
 \Gamma_{N,\mathrm{out}}^+
 \defeq \Gamma_N^+\cap\left\{z\in\C:\Re z\ge r_0\right\}.
\end{equation*}
Because $\widehat\Psi_N$ is strictly increasing,
$\widehat G_N'(r)=(\widehat\Psi_N(r)-x)/t>0$ for $r\ge r_0$.
Because $Q_N$ has real coefficients,
\begin{equation*}
 \Re\Phi_{N,x}\left(r-\mi\widehat v_N\left(r\right)\right)
 =\Re\Phi_{N,x}\left(r+\mi\widehat v_N\left(r\right)\right).
\end{equation*}
Thus every point $z$ on this arc satisfies
\begin{equation}\label{eq:outer-arc-phase-gap}
 \Re\Phi_{N,x}\left(z\right)-\Re\Phi_{N,x}\left(z_+\right)\ge c_ta_N^2,\quad\forall z\in \Gamma_{N,\mathrm{out}}^+.
\end{equation}
Also $|w-z|\ge r_0-u\ge a_N/8$ for $w\in L_u$.  Combining this phase
gap with Lemmas~\ref{lem:vertical} and~\ref{lem:length} gives
\begin{equation}\label{eq:outer-contribution}
 \frac1{\left(2\pi\right)^2t}
 \int_{\Gamma_{N,\mathrm{out}}^+}\left|\dif z\right|
 \int_{L_u}\left|\dif w\right|\,
 \left|\frac{\me^{\Phi_{N,x}\left(w\right)-\Phi_{N,x}\left(z\right)}}{w-z}\right|
 \le C_t\left(1+a_N^{q+1}\right)\me^{-c_ta_N^2}\le C_t.
\end{equation}

On the regular positive graphs $U\le r\le r_0$, the derivative estimate
in \eqref{eq:derivative-bounds} gives $|\,\dif (r\pm \mi v_N(r))|\le C_t\,\dif  r$.
Equations \eqref{eq:horizontal-gaussian} and
\eqref{eq:vertical-gaussian} yield
\begin{equation}\label{eq:regular-phase}
 \left|\me^{\Phi_{N,x}\left(u+\mi s\right)-\Phi_{N,x}\left(r\pm \mi v_N\left(r\right)\right)}\right|
 \le
 \exp\left\{-c_t\left[\left(r-u\right)^2+\dist\left(s,\left\{-v,v\right\}\right)^2\right]\right\}.
\end{equation}
On the upper graph and for $s\ge0$, the Lipschitz bound for $v_N$ gives
the shear estimate
\begin{equation*}
 \left|u+\mi s-\left(r+\mi v_N\left(r\right)\right)\right|
 \ge c_t\sqrt{\left(r-u\right)^2+\left(s-v\right)^2}.
\end{equation*}
For $s\le0$, the same denominator is at least
$v_N(r)\ge c_tU^\theta$.  The lower graph is treated by reflection.
Writing $\alpha=r-u$ and $\beta=s-v$ near the upper crossing
(and $\beta=s+v$ near the lower crossing), the local absolute majorant is
\begin{equation*}
 C_t\frac{\me^{-c_t\left(\alpha^2+\beta^2\right)}}
          {\sqrt{\alpha^2+\beta^2}},
\end{equation*}
and its two-dimensional integrability is explicit:
\begin{equation}\label{eq:crossing-dominating-integral}
 \int_{\R^2}
 \frac{\me^{-c_t\left(\alpha^2+\beta^2\right)}}
      {\sqrt{\alpha^2+\beta^2}}\,
 \dif\alpha\,\dif\beta
 =2\pi\int_0^\infty \me^{-c_t r^2}\,\dif r<\infty.
\end{equation}
Away from fixed neighborhoods of the two crossings, the denominator is
uniformly separated from zero and the Gaussian factors are integrable. Thus
the entire regular positive-tail contribution is bounded by $C_t$.

Evenness of $Q_N$ gives the exact identity
\begin{equation}\label{eq:negative-gain}
 \Re\Phi_{N,x}\left(-r+\mi\widehat v_N\left(r\right)\right)
 -\Re\Phi_{N,x}\left(r+\mi\widehat v_N\left(r\right)\right)
 =\frac{2rx}{t},
\end{equation}
and the same identity holds on the lower boundary.  On the regular negative
tail, $|u+\mi s-(-r\pm \mi v_N(r))|\ge u+r$, so
\begin{equation*}
 C_t\int_U^{r_0}\int_\R
 \frac{\me^{-c_t\left(r-u\right)^2-c_t\dist\left(s,\left\{-v,v\right\}\right)^2-2rx/t}}
      {u+r}\,\dif s\dif r
 \le C_t.
\end{equation*}
The outer negative arc has both the outer phase gap and the additional
gain in \eqref{eq:negative-gain}, so it satisfies the same bound as
\eqref{eq:outer-contribution}.

It remains to control the two chords and the central loops.  For
$z\in K_\Gamma$, \eqref{eq:carrier-product} gives
\begin{equation*}
 \Re\Phi_{N,x}\left(z\right)=\frac{x^2}{2t}+O_t\left(x\right)+O_t\left(1\right).
\end{equation*}
At the saddle, the uniform product-growth estimate and
Lemma~\ref{lem:Psi-location} give
\begin{equation}\label{eq:saddle-phase-upper}
 \Re\Phi_{N,x}\left(z_+\right)
 \le C_t\left[\left(u-x\right)^2+1+\left|z_+\right|^q\right]
 \le C_t\left[x^{2\theta}\log^2\left(2+x\right)+x^q+1\right]=o\left(x^2\right).
\end{equation}
After increasing $R_t$,
\begin{equation}\label{eq:fixed-gap}
 \Re\Phi_{N,x}\left(z\right)-\Re\Phi_{N,x}\left(z_+\right)
 \ge c_tx^2,
 \quad z\in K_\Gamma.
\end{equation}
Moreover, $|u+\mi s-z|\ge u-M\ge u/2$ on the carrier.  Its length is fixed,
so its total absolute contribution is at most $C_t\me^{-c_tx^2}$.

For fixed $z$, define
\begin{equation*}
 \mathcal H_N\left(z,w\right)
 \defeq\frac{\me^{\Phi_{N,x}\left(w\right)-\Phi_{N,x}\left(z\right)}}{w-z}.
\end{equation*}
Here the exponential is interpreted through the polynomial quotient.
If $L_c=\{c+\mi s:s\in\R\}$ is oriented upward and $r_-<r_+$, the residue
theorem on a large rectangle, followed by the Gaussian limit on its
horizontal sides, gives the fixed-$z$ identity
\begin{equation}\label{eq:vertical-rectangle}
 \int_{L_{r_+}}\mathcal H_N\left(z,w\right)\dif w
 =\int_{L_{r_-}}\mathcal H_N\left(z,w\right)\dif w
  +2\pi \mi\,\indi_{\left\{r_-<\Re z<r_+\right\}},
 \quad z\notin L_{r_-}\cup L_{r_+}.
\end{equation}
The residue is one because the exponential equals one at $w=z$.

For the fixed $N$ under consideration, choose
$B_N>\max\{R_{N,t}^+,\sup_{z\in C_N}\Re z\}$. Start
from \eqref{eq:kernel} with $z$ on $C_N$ and $w$ on $\mi\R$. Move the
$w$-line to $L_{B_N}$. Each positive source loop crossed during this
movement is complete, and its integrated residue is
$\frac1{2\pi\mi t}\oint\dif z=0$; the negative loops are not crossed.
With $w$ on $L_{B_N}$, deform $C_N$ into the final contour $\Gamma_N$;
no $w=z$ pole is crossed.

We now move the $w$-line back from $L_{B_N}$ to $L_u$. The two contours meet
at $z_+$ and $z_-$, so we justify the movement before writing the result.
For $\delta>0$, remove from the upper and lower regular graphs the two
short parameter arcs $|r-u|<\delta$, and call the remaining contour
$\Gamma_{N,\delta}$.  On this set, integrate
\eqref{eq:vertical-rectangle} with $r_-=u$ and $r_+=B_N$ over $z$. All integrals are
absolutely convergent because the contours are separated.

The removed mass tends to zero as $\delta\downarrow0$. Indeed, the shear
estimate and \eqref{eq:regular-phase} give, in local coordinates at either
crossing, the majorant whose integrability is recorded in
\eqref{eq:crossing-dominating-integral}. Absolute continuity of its integral
shows that its mass over
$|\alpha|<\delta$ tends to zero.  All other contour pieces were already
shown to be absolutely integrable previously.  We may therefore let
$\delta\downarrow0$ for the $L_u$-integral. For the $L_{B_N}$-integral,
$\mathcal H_N(z,w)$ is smooth in $z$ near $z_+$ and $z_-$, because $L_{B_N}$
lies strictly to the right of the complete $z$-contour.  Its contribution
over the removed arcs is $O(\delta)$.  We may therefore let
$\delta\downarrow0$ in the integrated rectangle identity.  The final double integral is therefore absolutely convergent.

The set of final $z$-points with real part greater than $u$ is precisely
the arc $\cA_N\subset\Gamma_N^+$ that begins at $z_-$, follows the lower
boundary to $R_{N,t}^+$, and returns on the upper boundary to $z_+$.  By
the chosen positive orientation and \eqref{eq:vertical-rectangle}, the
residue has positive sign and equals
\begin{equation*}
 \frac1{2\pi \mi t}\int_{\cA_N}\dif z
 =\frac1{2\pi \mi t}\left(z_+-z_-\right)
 =\frac{v}{\pi t}.
\end{equation*}

The exact post-deformation identity is
\begin{equation}\label{eq:decomposition}
 \rho_{N,t}\left(x\right)=\frac{v}{\pi t}+I_{N,t}\left(x\right),
 \quad
 I_{N,t}\left(x\right)
 \defeq \frac1{\left(2\pi \mi\right)^2t}
 \int_{\Gamma_N}\dif z\int_{L_u}\dif w\,
 \mathcal H_N\left(z,w\right).
\end{equation}
The bounds on the regular, outer, and fixed contour pieces above give
\begin{equation}\label{eq:remainder}
 \left|I_{N,t}\left(x\right)\right|\le C_t
\end{equation}
uniformly in $N$ and $R_t\le x\le a_N/16$.

Combining \eqref{eq:decomposition}, \eqref{eq:remainder}, and
$v\le C_tx^\theta$ gives
\begin{equation*}
 \rho_{N,t}\left(x\right)
 \le\frac{v}{\pi t}+\left|I_{N,t}\left(x\right)\right|
 \le C_t\left(1+x^\theta\right).
\end{equation*}
Since $A_N=-A_N$, the finite density is even in $x$.  Replacing $x$ by
$-x$ proves the negative-$x$ estimate and completes the proof.
\end{proof}

\subsection{Compact region, outer region, and completion}

Near and beyond the finite edge we use a density bound for Gaussian matrix perturbations. Its dependence on the number of particles is weak enough to give the required global estimate.
\begin{lemma}\label{lem:wegner}
Let $\boldsymbol\eta=\sum_{j=1}^m\delta_{b_j}$ be any deterministic finite
configuration, and let $\rho_{\boldsymbol\eta,t}$ be the first correlation function
at time $t>0$ of the $\beta=2$ Dyson model started from $\boldsymbol\eta$.  There is
a universal constant $C<\infty$ such that
\begin{equation*}
 \sup_{y\in\R}\rho_{\boldsymbol\eta,t}\left(y\right)
 \le C\sqrt{\frac mt}.
\end{equation*}
\end{lemma}

\begin{proof}
Let $\mathbf{D}=\operatorname{diag}(b_1,\ldots,b_m)$.  In the Hermitian Brownian
normalization corresponding to \eqref{eq:kernel}, the fixed-time matrix is
\begin{equation*}
 \mathbf{D}+\boldsymbol{\mathsf H}_m\left(t\right)
 \ \stackrel{\mathrm{law}}=\
 \mathbf{D}+\sqrt{mt}\,\boldsymbol{\mathsf V}_m,
\end{equation*}
where $\boldsymbol{\mathsf V}_m$ is normalized GUE: its law is proportional to
\begin{equation*}
 \exp\!\left\{-\frac m2\operatorname{Tr}\left(\mathbf V^2\right)\right\}\,\dif \mathbf V.
\end{equation*}
The mean-density estimate for Gaussian perturbations states, uniformly
in the deterministic Hermitian matrix $\mathbf B$,
\begin{equation*}
 \expt\,\left[\#\left\{\operatorname{Spec}\left(\mathbf B+\boldsymbol{\mathsf V}_m\right)\cap J\right\}\right]
 \le Cm\left|J\right|
\end{equation*}
for every interval $J$.  This is Theorem 1(iii) of \cite{c685048cba3543e29220f47cf0728c56}.

Apply the estimate to $\mathbf B=\mathbf D/\sqrt{mt}$ and
$J=I/\sqrt{mt}$.  Then
\begin{equation*}
 \expt\left[\#\left\{\operatorname{Spec}\left(\mathbf D+\boldsymbol{\mathsf H}_m\left(t\right)\right)\cap I\right\}\right]
 \le C\sqrt{\frac mt}\,\left|I\right|.
\end{equation*}
The fixed-time eigenvalue law is the Dyson model started from $\boldsymbol\eta$.
Dividing by $|I|$ and shrinking $I$ gives the asserted bound almost
everywhere; continuity of the finite-particle correlation function, which
follows from the standard finite-particle contour kernel, upgrades it to
every $y$.
\end{proof}

\begin{proof}[Proof of Theorem~\ref{thm bound}]
Let $R_t$ be the constant in Proposition~\ref{prop:interior-bound}, enlarged
if necessary. By the parameter convergence established above, taking
$\mathcal A=\{t\}$ in Theorem~\ref{MainTheorem} gives locally uniform
convergence $\rho_{N,t}\to\rho_t$. Hence the sequence is uniformly bounded
on $[-R_t,R_t]$ for all sufficiently large $N$.
Each of the finitely many remaining $\rho_{N,t}$ is continuous, so
\begin{equation}\label{eq:compact-bound}
 \sup_{N\ge1}\sup_{\left|x\right|\le R_t}\rho_{N,t}\left(x\right)\le C_t.
\end{equation}

For the outer region, apply Lemma~\ref{lem:wegner} with $m=2N$.  By
\eqref{eq:edge-size},
\begin{equation*}
 \rho_{N,t}\left(x\right)\le C_t\sqrt N
 \le C_t\left(1+a_N^{q/2}\right).
\end{equation*}
If $|x|>a_N/16$, this gives
\begin{equation}\label{eq:outer-bound}
 \rho_{N,t}\left(x\right)\le C_t\left(1+\left|x\right|^{q/2}\right).
\end{equation}

Proposition~\ref{prop:interior-bound} proves
\eqref{eq:interior-main}. The compact bound
\eqref{eq:compact-bound}, the interior bound, and
\eqref{eq:outer-bound} cover every pair $(N,x)$, including the indices
with $a_N/16<R_t$. Since $\theta=q-1<q/2$ for $q<2$, they give
\eqref{eq:global-main}.
\end{proof}

\subsection{Addition of an atom at zero}
We complete the proof of Theorem \ref{thm bound} for the configurations
\begin{equation*}
    \xi_\epsilon
    =\epsilon\delta_0+\sum_{n\geq1}(\delta_{a_n}+\delta_{-a_n}),
    \quad
    \xi_\epsilon^{(N)}
    =\epsilon\delta_0+\sum_{n=1}^{N}(\delta_{a_n}+\delta_{-a_n}),
    \quad \epsilon\in\{0,1\}.
\end{equation*}
Assume \emph{($\mathbf{UC}$)} and \emph{($\mathbf{LC}$)}, with $q\in(1,2)$ and $\theta=q-1$ as in
Section 5.1. Write $\rho_{N,t}^{(\epsilon)}$ for the one-point density
started from $\xi_\epsilon^{(N)}$. Sections 5.2--5.7 establish the theorem
for $\epsilon=0$. The polynomial $Q_N$, phase $\Phi_{N,x}$, saddle parameters, and contours used below are those of that case.
\begin{proof}[Completion of the proof of Theorem \ref{thm bound}]
Fix $t>0$. Constants may depend on $t$ and the fixed initial configuration,
but are independent of $N$ and $x$. We first consider
$R_t\leq x\leq a_N/16$, enlarging $R_t$ as in Proposition \ref{prop:interior-bound} so that
all the contours used there are available whenever this range is nonempty. Put
\begin{equation*}
    u=u_N(x),\quad v=v_N(u),\quad z_+=u+iv,\quad
    L_u=u+i\mathbb R,
\end{equation*}
where vertical lines are oriented upward. In particular,
$x/2\leq u\leq2x$ and $c_tx^\theta\leq v\leq C_tx^\theta$.
Choose a fixed sufficiently small $r_*>0$ such that the positively
oriented circle $C_0=\{|z|=r_*\}$ encloses no pair root and its closed
disk is disjoint from the contours $\Gamma_N$ of Proposition \ref{prop:interior-bound}.
Such a choice is possible: the tail contours satisfy $|\Re z|\geq U$,
and the finitely many central loops lie strictly in the two open
half-planes. Set $\widehat\Gamma_N=\Gamma_N\cup C_0$.

We first obtain an exact comparison of the two densities. For each fixed
$N$, choose $B_N$ to the right of all the root contours and of
$\widehat\Gamma_N$, and write $L_{B_N}=B_N+i\mathbb R$. The finite-particle
contour formula gives
\begin{equation*}
    \rho_{N,t}^{(\epsilon)}(x)
    =\frac{1}{(2\pi i)^2t}
    \int_{\widehat\Gamma_N}\dif z\int_{L_{B_N}}\dif w\,
    \mathrm{e}^{\Phi_{N,x}(w)-\Phi_{N,x}(z)}
    \frac{(w/z)^\epsilon}{w-z}.
\end{equation*}
Here and below the exponential is interpreted as
\begin{equation*}
    \mathrm{e}^{\Phi_{N,x}(w)-\Phi_{N,x}(z)}
    =\exp\!\left\{\frac{(w-x)^2-(z-x)^2}{2t}\right\}
      \frac{Q_N(w)}{Q_N(z)}.
\end{equation*}
For completeness, the formula for $\epsilon=0$ follows by moving the
$w$-line in (5.6) to $L_{B_N}$ and then deforming the root loops.
Each complete source loop crossed in the first movement contributes
zero, since the residue at $w=z$ is one and the integral of $\dif z$
around a closed loop is zero. The added circle $C_0$ contributes zero
in this case. For $\epsilon=1$, first replace the origin by
$\eta\in(0,r_*)$, use a small loop around $\eta$ disjoint from
$i\mathbb R$, and make the same movement of the $w$-line. With the line
on $L_{B_N}$, deform the new loop to $C_0$ and let $\eta\downarrow0$.
Proposition \ref{finitedysonmodel} identifies the limiting finite-particle density.
The limit under the integral is justified, for fixed $N$, by the
separation of the contours and Gaussian decay in $\Im w$; the polynomial
ratio is $(w-\eta)Q_N(w)/((z-\eta)Q_N(z))$.

Subtracting the two formulas and using
\begin{equation*}
    \frac{w}{z(w-z)}-\frac1{w-z}=\frac1z
\end{equation*}
removes the pole $w=z$. The remaining integrand is entire in $w$.
For fixed $N$, Gaussian decay on the horizontal sides of a rectangle
therefore permits the shift from $L_{B_N}$ to $L_u$ without a residue.
Consequently,
\begin{equation*}
    \rho_{N,t}^{(1)}(x)-\rho_{N,t}^{(0)}(x)=J_{N,t}(x),
    \quad
    J_{N,t}(x)=\frac{1}{(2\pi i)^2t}
    \int_{\widehat\Gamma_N}\frac{\dif z}{z}
    \int_{L_u}\dif w\,
    \mathrm{e}^{\Phi_{N,x}(w)-\Phi_{N,x}(z)}.
\end{equation*}
These integrals are absolutely convergent for fixed $N$. In particular,
the comparison integral has no singularity where the two contours meet.

We next bound $J_{N,t}$ uniformly. By \eqref{eq:vertical-gaussian},
\begin{equation*}
    \int_{\mathbb R}
    \mathrm{e}^{\Re\Phi_{N,x}(u+is)-\Re\Phi_{N,x}(z_+)}\dif s
    \leq
    \int_{\mathbb R}\mathrm{e}^{-c_t\operatorname{dist}(s,\{-v,v\})^2}\dif s
    \leq C_t.
\end{equation*}
Thus it suffices to bound
\begin{equation*}
    \int_{\widehat\Gamma_N}
    \frac{|\dif z|}{|z|}\,
    \mathrm{e}^{\Re\Phi_{N,x}(z_+)-\Re\Phi_{N,x}(z)}.
\end{equation*}
The contours stay a fixed positive distance from zero. On the regular
positive graphs, parametrized by $U\leq r\leq a_N/4$, \eqref{eq:derivative-bounds} gives
$|\dif z|\leq C_t\dif r$, and \eqref{eq:horizontal-gaussian} gives a phase gap at least
$c_t(r-u)^2$. The reflected graphs have the additional nonnegative gap
$2rx/t$ by \eqref{eq:negative-gain}. Their total contribution is therefore at most
\begin{equation*}
    C_t\int_U^{a_N/4}\mathrm{e}^{-c_t(r-u)^2}\dif r\leq C_t.
\end{equation*}
On the outer arcs, by \eqref{eq:outer-arc-phase-gap} the phase gap is at least
$c_ta_N^2$. Together with the length bound \eqref{eq:length}, this gives a contribution
at most $C_t(1+a_N^{q+1})e^{-c_ta_N^2}\leq C_t$.

It remains to treat the fixed central loops, the two closing chords,
and $C_0$. Include $C_0$ in the compact carrier $K_\Gamma$ of
Proposition \ref{prop:interior-bound}. The enlarged carrier is disjoint from the zeros of
$Q_\infty$, so local uniform convergence of $Q_N$ yields uniform upper
and positive lower bounds for $|Q_N|$ there for all sufficiently large $N$.
Every remaining $Q_N$ is also nonzero on this carrier, so enlarging the
constants gives the same bounds for all $N$. Hence, uniformly on this carrier,
\begin{equation*}
    \Re\Phi_{N,x}(z)=\frac{x^2}{2t}+O_t(x)+O_t(1).
\end{equation*}
The saddle estimate \eqref{eq:saddle-phase-upper} gives
\begin{equation*}
    \Re\Phi_{N,x}(z_+)
    \leq C_t\bigl(x^{2\theta}\log^2(2+x)+x^q+1\bigr)=o(x^2).
\end{equation*}
Since $\theta<1$ and $q<2$, enlarging $R_t$ gives a phase gap at least
$c_tx^2$. These contour pieces have uniformly bounded length and stay
away from zero, so their contribution is at most $C_t \mathrm{e}^{-c_tx^2}$.
Combining the three bounds proves $|J_{N,t}(x)|\leq C_t$ throughout the
interior range. The case $\epsilon=0$ of \eqref{eq:interior-main} now implies
\begin{equation*}
    \rho_{N,t}^{(1)}(x)\leq C_t(1+x^\theta),
    \quad R_t\leq x\leq a_N/16.
\end{equation*}
Reflection invariance of the finite Dyson model gives the corresponding bound for negative $x$.

For the compact region, \eqref{eq:count-square-tail} gives $\sum_n a_n^{-2}<\infty$.
Adding the fixed inner root at zero does not alter the reciprocal tails,
and the parameter-convergence argument of Section~5.1 gives
$f_0^b(\xi_1^{(N)})\to f_0^b(\xi_1)$, with no Gaussian component.
Theorem \ref{MainTheorem} therefore gives locally uniform convergence of the
one-point densities at every $t>0$. As in \eqref{eq:compact-bound}, local uniform
convergence and continuity of the finitely many remaining densities yield
\begin{equation*}
    \sup_{N\geq1}\sup_{|x|\leq R_t}\rho_{N,t}^{(1)}(x)\leq C_t.
\end{equation*}
For $|x|>a_N/16$, Lemma \ref{lem:wegner} with $2N+1$ particles, followed by \eqref{eq:edge-size}, gives
\begin{equation*}
    \rho_{N,t}^{(1)}(x)
    \leq C_t\sqrt{2N+1}
    \leq C_t(1+a_N^{q/2})
    \leq C_t'(1+|x|^{q/2}).
\end{equation*}
The compact, interior, and outer regions cover every pair $(N,x)$.
Since $\theta=q-1<q/2$, they prove \eqref{eq:global-main}, and the interior estimate proves
\eqref{eq:interior-main}. Taking the larger constants from the two cases makes both bounds
uniform over $\epsilon\in\{0,1\}$. This completes the proof of
Theorem \ref{thm bound}. The limiting-density conclusion of Remark \ref{rem:limit-density} follows for
both cases by the same local uniform convergence.
\end{proof}

\section{The Markov property}\label{sec:markov}
\subsection{A criterion for the Markov property}
Finite-particle Dyson Brownian motion is Markov by uniqueness for its finite-dimensional SDE. In this section, we pass that property to the infinite-particle process constructed in Sections 2--3.

For a Borel set $A\subset\mathbb R$, write $\bxi|_A$ for the
restriction of $\bxi$ to $A$. Recall the reciprocal statistics $p_b$
and $s_b$ from Section~\ref{sec:preliminaries}: the principal value
in $p_b$ uses integer outer cutoffs, and $s_b$ is allowed to be infinite.

The state space consists of configurations with a finite reciprocal-square tail and a convergent symmetric principal value.
\begin{definition}\label{process start from bxi}
Define
\[
 \mathfrak N\defeq\bigcup_{b>0}
 \left\{\bxi\in\mathfrak M_b:p_b\left(\bxi\right)
 \text{ exists and is finite}\right\}.
\]
Here membership in $\mathfrak M_b$ includes $s_b(\bxi)<\infty$.
\end{definition}

Spatial truncation converges in the parameter space when the principal value exists. This identifies the infinite-particle process as a limit of finite Dyson models.
\begin{lemma}\label{lem:truncation-parameter-bridge}
Let $\boldsymbol\eta\in\mathfrak N$, and choose $b>0$ such that
$\boldsymbol\eta\in\mathfrak M_b$ and $p_b(\boldsymbol\eta)$ is finite.
For each integer $N>b$, put
$\boldsymbol\eta^{[N]}=\boldsymbol\eta|_{[-N,N]}$. Then
\begin{equation}\label{eq:truncation-parameter-bridge}
 f_{p_b\left(\boldsymbol\eta^{\left[N\right]}\right)}^b\left(\boldsymbol\eta^{\left[N\right]}\right)
 \longrightarrow f_{p_b\left(\boldsymbol\eta\right)}^b\left(\boldsymbol\eta\right)
 \qquad\text{in }\bUpsilon_b.
\end{equation}
\end{lemma}
\begin{proof}
The $\gamma_1$-coordinates converge because
\[
 p_b\left(\boldsymbol\eta^{\left[N\right]}\right)
 =\int_{\left\{b<\left|x\right|\le N\right\}}\frac{\boldsymbol\eta\left(\dif x\right)}x
 \longrightarrow p_b\left(\boldsymbol\eta\right).
\]
The $\delta$-coordinates converge monotonically because
\[
 \int_{\left\{b<\left|x\right|\le N\right\}}\frac{\boldsymbol\eta\left(\dif x\right)}{x^2}
 \longrightarrow s_b\left(\boldsymbol\eta\right).
\]
For $N>b$ all inner roots are already present, so the two
$\mathbb U_b$ coordinates are fixed. The ordered outer reciprocal-root
coordinates converge coordinatewise, with roots not yet present represented
by zero. These are exactly the remaining coordinates of $\bUpsilon_b$.
\end{proof}

For each configuration in this state space, the parameter construction defines a process at every positive time.
\begin{definition}
Let $\bxi\in\mathfrak N$ and choose an admissible cutoff $b$. Define
$\bXi_{\bxi}(\cdot)$ to be the determinantal process with kernel
$\K_{f_{p_b(\bxi)}^b(\bxi)}$. Put
$\bxi^{[N]}=\bxi|_{[-N,N]}$ and let $\bXi_{\bxi}^{[N]}(\cdot)$ be
the finite Dyson model started from $\bxi^{[N]}$.
\end{definition}
Since $\gamma_2(f_{p_b(\bxi)}^b(\bxi))=0$, this process is defined for
all $t>0$. Proposition~\ref{finitedysonmodel},
Lemma~\ref{lem:truncation-parameter-bridge}, and
Theorem~\ref{MainTheorem} give
\[
 \bXi_{\bxi}^{\left[N\right]}\left(\cdot\right)
 \xrightarrow{\mathrm{f.d.d.}}\bXi_{\bxi}\left(\cdot\right).
\]
The process is independent of the admissible cutoff: moving roots across
the cutoff changes the associated entire function only by a nonzero
constant, which cancels from the kernel ratio.

Only positive-time finite-dimensional convergence is used below; the
path-space theorem of Section 4 is not invoked. We write
$\prob_{\bxi},\expt_{\bxi}$ and
$\prob_{\bxi^{[N]}},\expt_{\bxi^{[N]}}$ for the corresponding laws and
expectations. Denote the one-point correlation functions at time $t$ by
$\rho_{\bxi,t}$ and $\rho_{\bxi^{[N]},t}$.

We first give conditions ensuring that the configuration remains in the state space at every fixed positive time. This makes the transition operators available at the intermediate times in the Markov identity.
\begin{proposition}\label{prop in N for all time}
Suppose that, for every $t>0$, there are constants
$0<C_{1,t},C_{2,t}<\infty$, $0\leq\alpha<1$, and $\nu>0$ such that
\begin{equation}\label{eq rho bound}
 \sup_{N\in\mathbb N}\rho_{\bxi^{\left[N\right]},t}\left(x\right)
 \leq C_{1,t}\left(1+\left|x\right|^\alpha\right),
 \quad x\in\mathbb R,
\end{equation}
and, for every $L_0\geq1$,
\begin{equation}\label{eq rho symmetry}
\begin{aligned}
 \sup_{N\in\mathbb N}
 \left|\lim_{L_1\to\infty}
 \int_{\left\{L_0\leq\left|x\right|\leq L_1\right\}}
       \frac{\rho_{\bxi^{\left[N\right]},t}\left(x\right)}x\dif x\right|
 &\leq C_{2,t}L_0^{-\nu},\\
 \left|\lim_{L_1\to\infty}
 \int_{\left\{L_0\leq\left|x\right|\leq L_1\right\}}
       \frac{\rho_{\bxi,t}\left(x\right)}x\dif x\right|
 &\leq C_{2,t}L_0^{-\nu}.
\end{aligned}
\end{equation}
Then, for every fixed $t>0$, $\bXXi(t)\in\mathfrak N$ almost surely.
\end{proposition}
By Theorem~\ref{MainTheorem}, the density bound \eqref{eq rho bound}
also holds for the limiting density $\rho_{\bxi,t}$.
For $\widetilde{\mathfrak M}\subseteq\mathfrak M$, set
\[
 \mathcal F_{\mathrm{bd}}\left(\widetilde{\mathfrak M}\right)
 \defeq\left\{\phi:\widetilde{\mathfrak M}\longrightarrow\mathbb R:
 \sup_{\bxi\in\widetilde{\mathfrak M}}\left|\phi\left(\bxi\right)\right|<\infty\right\},
\]
and
\[
 \mathcal C_{\mathrm{bd}}\left(\widetilde{\mathfrak M}\right)
 \defeq\left\{\phi\in\mathcal F_{\mathrm{bd}}\left(\widetilde{\mathfrak M}\right):
 \phi\text{ is vaguely continuous on }\widetilde{\mathfrak M}\right\}.
\]

The transition argument needs convergence that also controls the reciprocal parameters.
\begin{definition}
Let $\bxi^\UN,\bxi\in\mathfrak N$. We write
$\bxi^\UN\xrightarrow{\ \bUpsilon\ }\bxi$ if there exists a common
$b>0$ such that $\bxi^\UN,\bxi\in\mathfrak M_b$ for every $N$ and
\[
 f_{p_b\left(\bxi^\UN\right)}^b\left(\bxi^\UN\right)
 \longrightarrow f_{p_b\left(\bxi\right)}^b\left(\bxi\right)
 \qquad\text{in }\bUpsilon_b.
\]
Define
\[
 \mathcal C_{\bUpsilon}\left(\mathfrak N\right)
 \defeq\left\{\phi\in\mathcal F_{\mathrm{bd}}\left(\mathfrak N\right):
 \phi\left(\bxi^\UN\right)\longrightarrow\phi\left(\bxi\right)
 \text{ whenever }\bxi^\UN\xrightarrow{\ \bUpsilon\ }\bxi\right\}.
\]
\end{definition}
Since $\bUpsilon$-convergence implies vague convergence,
$\mathcal C_{\mathrm{bd}}(\mathfrak N)
\subset\mathcal C_{\bUpsilon}(\mathfrak N)$.

We associate a transition operator with the process started from each admissible configuration.
\begin{definition}
For $t>0$, define
\[
 \mT_t:\mathcal C_{\mathrm{bd}}\left(\mathfrak M\right)
 \longrightarrow\mathcal F_{\mathrm{bd}}\left(\mathfrak N\right),
 \qquad
 \left(\mT_t\phi\right)\left(\bxi\right)
 \defeq \expt_{\bxi}\!\left[\phi\left(\bXi_{\bxi}\left(t\right)\right)\right].
\]
\end{definition}

\begin{remark}
For a finite configuration $\boldsymbol\eta$, this process agrees with
the finite Dyson model started from $\boldsymbol\eta$, by
Proposition~\ref{finitedysonmodel}. Thus the same operator $\mT_t$ may
be used in the finite-particle Markov identity.

\label{remark upsilon continuous}
Theorem~\ref{MainTheorem} shows that $\bUpsilon$-convergence gives
finite-dimensional convergence. Hence
\[
 \mT_t\phi\in\mathcal C_{\bUpsilon}\left(\mathfrak N\right),
 \qquad \phi\in\mathcal C_{\mathrm{bd}}\left(\mathfrak M\right),\quad t>0.
\]

\label{rem:transition-measurable}
The function $\mT_t\phi$ is measurable on $\mathfrak N$ for the trace
of the vague Borel $\sigma$-field. Indeed, for $R\in\mathbb N$ put
$\boldsymbol\eta^{[R]}=\boldsymbol\eta|_{[-R,R]}$ and
\[
 \left(\mT_t^{\left[R\right]}\phi\right)\left(\boldsymbol\eta\right)
 \defeq \expt_{\boldsymbol\eta^{\left[R\right]}}
   \left[\phi\left(\bXi_{\boldsymbol\eta^{\left[R\right]}}\left(t\right)\right)\right].
\]
The restriction map is vague-Borel measurable, and the finite-particle
transition kernel is Borel measurable in its initial configuration, so
$\mT_t^{[R]}\phi$ is vague-Borel measurable. For each fixed
$\boldsymbol\eta\in\mathfrak N$,
Lemma~\ref{lem:truncation-parameter-bridge}, Theorem~\ref{MainTheorem},
and bounded vague continuity give
\[
 \left(\mT_t\phi\right)\left(\boldsymbol\eta\right)
 =\lim_{R\to\infty}\left(\mT_t^{\left[R\right]}\phi\right)\left(\boldsymbol\eta\right).
\]
Thus $\mT_t\phi$ is measurable as a pointwise limit.
\end{remark}

We use the following finite-dimensional formulation of the Markov property.
\begin{definition}
We say that $\bXi_{\bxi}(\cdot)$ has the Markov property with
transition operators $(\mT_t)_{t>0}$ if, for every $M\ge2$,
$0<t_1<\cdots<t_M<\infty$, and
$\phi_1,\ldots,\phi_M\in\mathcal C_{\mathrm{bd}}(\mathfrak M)$,
\[
\begin{aligned}
 &\expt_{\bxi}\left[\prod_{i=1}^M\phi_i\left(\bXXi\left(t_i\right)\right)\right]=\expt_{\bxi}\left[
 \left(\prod_{i=1}^{M-1}\phi_i\left(\bXXi\left(t_i\right)\right)\right)
 \left(\mT_{t_M-t_{M-1}}\phi_M\right)\left(\bXXi\left(t_{M-1}\right)\right)\right].
\end{aligned}
\]
Throughout this section, ``Markov property'' refers to this identity.
\end{definition}

The same density and principal-value bounds allow the finite-particle Markov identity to pass to the limit. We state the resulting criterion before proving the tail estimates on which it relies.
\begin{proposition}\label{prop markov from rho}
Suppose that, for every $t>0$, there are constants as in
Proposition~\ref{prop in N for all time} for which
\eqref{eq rho bound} and \eqref{eq rho symmetry} hold. Then
$\bXi_{\bxi}(\cdot)$ has the Markov property.
\end{proposition}

If every $\rho_{\bxi^{[N]},t}$ is even, so is its locally uniform
limit, and \eqref{eq rho symmetry} holds with both left-hand sides zero.
Thus the main work in symmetric applications is the uniform
finite-particle density estimate.
\subsection{Proof of the criterion}

We separate each reciprocal statistic into a bounded part and a tail.
\begin{definition}
Let $b>0$, $L>b$, and $\bxi\in\mathfrak M$. Recall that
$p_L(\bxi)$ and $s_L(\bxi)$ are the reciprocal statistics at cutoff $L$;
the principal value defining $p_L$ uses integer outer cutoffs. Define
\[
 p_b^L\left(\bxi\right)\defeq\int_{\left\{b<\left|x\right|\le L\right\}}\frac{\bxi\left(\dif x\right)}x,
 \qquad
 s_b^L\left(\bxi\right)\defeq\int_{\left\{b<\left|x\right|\le L\right\}}\frac{\bxi\left(\dif x\right)}{x^2}.
\]
\end{definition}
For each $N\in\mathbb N$, let
$\bXXi^\UN(t)=\sum_{j=1}^{n_N}\delta_{\mathsf X_j^{(N)}(t)}$ be a
finite-dimensional Dyson Brownian motion with $\beta=2$, started from a
deterministic configuration, and let $\rho_{\bxi^\UN,t}$ be its first
correlation function. For $L\ge1$, define
\begin{equation}\label{eq def of D+,D-}
\begin{aligned}
    &\mathsf D_{+,t}^{\left(N\right)}\left(L\right)\defeq \bXXi^{\left(N\right)}\left(t\right)\left(\left(0,L\right]\right)
    - \int_0^L \rho_{\bxi^\UN,t}\left(x\right)\dif x,\\
    &\mathsf D_{-,t}^{\left(N\right)}\left(L\right)\defeq \bXXi^{\left(N\right)}\left(t\right)\left(\left[-L,0\right)\right)
    - \int_{-L}^0 \rho_{\bxi^\UN,t}\left(x\right)\dif x.
\end{aligned}
\end{equation}
For the spatial truncations, we write $\mathsf D_{\pm,t}^{[N]}$ for the
same centered counts with $\bxi^{[N]}$ in place of $\bxi^\UN$.

We begin by controlling the fluctuations of one-sided particle counts. These estimates will give uniform control of the principal-value tails.
\begin{lemma}\label{lem Count-moment}
Suppose that, for some $0\le\alpha<1$,
\[
 \sup_N\rho_{\bxi^\UN,t}\left(x\right)
 \le C_t\left(1+\left|x\right|^\alpha\right),\qquad x\in\mathbb R.
\]
Then, for every $k\in\mathbb N$, there is a deterministic
$C_{k,t}<\infty$, independent of $N$ and $L$, such that
\begin{equation*}
    \sup_{N\in\mathbb{N}}
    \left\{\expt_{\bxi^\UN}\left[\left|\mathsf D_{+,t}^{\left(N\right)}\left(L\right)\right|^{2k}\right]
    +\expt_{\bxi^\UN}\left[\left|\mathsf D_{-,t}^{\left(N\right)}\left(L\right)\right|^{2k}\right]
    \right\}\leq C_{k,t}\left(1+L^{2\alpha k}\right),\qquad L\ge1.
\end{equation*}
\end{lemma}
Since $L\ge1$ and $2\alpha\le1+\alpha$, this also gives
\begin{equation}\label{eq:one-sided-count-moment-condition}
    \sup_{N\in\mathbb{N}}
    \left\{\expt_{\bxi^\UN}\left[\left|\mathsf D_{+,t}^{\left(N\right)}\left(L\right)\right|^{2k}\right]+\expt_{\bxi^\UN}\left[\left|\mathsf D_{-,t}^{\left(N\right)}\left(L\right)\right|^{2k}\right]\right\}\leq C_{k,t}L^{k\left(1+\alpha\right)}.
\end{equation}
\begin{remark}\label{rem vague imply bd}
    Suppose additionally that
\(\bXXi^{(N)}(t)\) converges vaguely in distribution to
\(\bXXi(t)\), and that the limiting process has no particles at fixed deterministic points almost surely. Then the same estimates hold for the corresponding centered counts of \(\bXXi(t)\).
\end{remark}
We prove the lemma first, then justify the passage to the limiting counts in Remark~\ref{rem vague imply bd}.
\begin{proof}[Proof of Lemma \ref{lem Count-moment}]
Let $\mathbf A_N$ be the deterministic diagonal $n_N\times n_N$
Hermitian matrix whose eigenvalues are the points of the initial
configuration. The $\beta=2$ Dyson Brownian motion started from this
configuration has the same law as the eigenvalue process of
$\mathbf A_N+\boldsymbol{\mathsf H}_{n_N}(t)$, where $\boldsymbol{\mathsf H}_{n_N}$ is Hermitian
Brownian motion. Write the ordered eigenvalues as
\begin{equation*}
    \mathsf X_1^{\left(N\right)}\left(t\right) \leq\cdots\leq \mathsf X_{n_N}^{\left(N\right)}\left(t\right),
\end{equation*}
and set
$m_j^{(N)}=\expt_{\bxi^\UN}[\mathsf X_j^{(N)}(t)]$. Equip the real vector
space of $n_N\times n_N$ Hermitian matrices with the Hilbert--Schmidt
norm. For Hermitian matrices $\mathbf H$ and $\widetilde{\mathbf H}$,
Weyl's inequality gives
\begin{equation*}
    \left|
    \lambda_j\left(\mathbf{A}_N+\mathbf{H}\right)
    -\lambda_j\left(\mathbf{A}_N+\widetilde {\mathbf{H}}\right)\right|\leq\left\|\mathbf{H}-\widetilde {\mathbf{H}}\right\|_{\mathrm{op}}\leq\left\|\mathbf{H}-\widetilde {\mathbf{H}}\right\|_{\mathrm{HS}}.
\end{equation*}
Under the present normalization, the covariance operator of the real
Gaussian coordinate vector of $\boldsymbol{\mathsf H}_{n_N}(t)$ is bounded by
$c_0t\mathbf I$, independently of $n_N$. The preceding display says that the
$j$-th ordered eigenvalue is $1$-Lipschitz in the Hilbert--Schmidt norm.
Gaussian concentration therefore gives a constant $c>0$, independent of
$N$, $n_N$, $j$, and $\mathbf A_N$, such that
\begin{equation}\label{eq:individual-eigenvalue-concentration}
    \prob_{\bxi^\UN}
    \left(\left|\mathsf X_j^{\left(N\right)}\left(t\right)-m_j^{\left(N\right)}\right|\geq r\right)\leq2\exp\left(-\frac{r^2}{ct}\right),\quad r\geq0.
\end{equation}
Choose $R_t<\infty$ such that
\begin{equation*}
    2\exp\left(-{R_t^2}/{\left(ct\right)}\right)\leq 1/2.
\end{equation*}
For \(u\in\mathbb{R}\), put
\begin{equation*}
    M_N\left(u\right)\defeq \#\left\{j:m_j^{\left(N\right)}\in\left[u,u+1\right]\right\}.
\end{equation*}
If \(m_j^{(N)}\in[u,u+1]\), then $|\mathsf X_j^{(N)}(t)-m_j^{(N)}|\leq R_t$ implies $\mathsf X_j^{(N)}(t)\in[u-R_t,u+1+R_t].$ It follows from \eqref{eq:individual-eigenvalue-concentration} that
\begin{equation*}
\begin{aligned}
    \frac12 M_N\left(u\right)
    &\leq\sum_{j=1}^{n_N}\prob_{\bxi^\UN}
    \left(\mathsf X_j^{\left(N\right)}\left(t\right)\in\left[u-R_t,u+1+R_t\right]\right) \\
    &=\expt_{\bxi^\UN}\left[\bXXi^\UN\left(t\right)
    \left(\left[u-R_t,u+1+R_t\right]\right)\right]=
    \int_{u-R_t}^{u+1+R_t}\rho_{\bxi^\UN,t}\left(y\right)\dif y.
\end{aligned}
\end{equation*}
Using the assumed density bound, we obtain
\begin{equation*}
    M_N\left(u\right)\leq C_t\left(1+\left|u\right|^\alpha\right),\quad u\in\mathbb{R},
\end{equation*}
where \(C_t\) is independent of \(N\). Covering an arbitrary interval by unit intervals and using
\begin{equation*}
    \left|x+y\right|^\alpha\leq \left|x\right|^\alpha+\left|y\right|^\alpha,\quad 0\leq\alpha\leq1,
\end{equation*}
we obtain, for every \(x\in\mathbb{R}\) and \(d\geq0\),
\begin{equation}\label{eq:mean-location-interval-bound}
\max\left\{
    \#\left\{j:m_j^{\left(N\right)}\in\left[x,x+d\right]\right\},\#\left\{j:m_j^{\left(N\right)}\in\left[x-d,x\right]\right\}\right\}\leq C_t \left(1+\left|x\right|^\alpha\right) \left(d+1\right)^{1+\alpha}.
\end{equation}
Define
\begin{equation*}
    \mathsf F_N\left(x\right)\defeq \bXXi^\UN\left(t\right)\left(\left(-\infty,x\right]\right),\quad J_N\left(x\right)\defeq\#\left\{j:m_j^{\left(N\right)}\leq x\right\}.
\end{equation*}
Also put $ A_x=1+|x|^\alpha.$ Suppose that \(J_N(x)+r\leq n_N\), and define $d_{+,N}(x,r)= m_{J_N(x)+r}^{(N)}-x.$
The interval $(x,x+d_{+,N}(x,r)]$ contains at least \(r\) mean locations. Hence
\eqref{eq:mean-location-interval-bound} gives
\begin{equation}\label{eq:mean-location-distance-bound}
    d_{+,N}\left(x,r\right)\geq c_t\left(\frac{r}{A_x}\right)^{1/\left(1+\alpha\right)}-1.
\end{equation}
The analogous estimate holds for the distance from \(x\) to the \(r\)-th mean location on its left.

If $\mathsf F_N(x)\geq J_N(x)+r$, then
$\mathsf X_{J_N(x)+r}^{(N)}(t)\leq x$. Combining
\eqref{eq:individual-eigenvalue-concentration} and
\eqref{eq:mean-location-distance-bound} controls the upper tail of
$\mathsf F_N(x)-J_N(x)$.

Similarly, if $\mathsf F_N(x)\leq J_N(x)-r,$ then $\mathsf X_{J_N(x)-r+1}^{(N)}(t)>x,$ whereas $m_{J_N(x)-r+1}^{(N)}\leq x.$ The same argument controls the lower tail. Events involving indices outside
\(\{1,\ldots,n_N\}\) are empty. After adjusting the constants to include
small values of \(r\), we obtain
\begin{equation}\label{eq:cumulative-count-tail}
    \prob_{\bxi^\UN}
    \left( \left|\mathsf F_N\left(x\right)-J_N\left(x\right)\right| \geq r \right) \leq C_t \exp\left[ -c_t\left(\frac{r}{A_x}\right)^{2/\left(1+\alpha\right)}\right].
\end{equation}
For every \(p\geq1\), integration of
\eqref{eq:cumulative-count-tail} gives
\begin{equation*}
\expt_{\bxi^\UN}
    \left[ \left|\mathsf F_N\left(x\right)-J_N\left(x\right)\right|^p
    \right] = p\int_0^\infty r^{p-1} \prob_{\bxi^\UN}
    \left(  \left|\mathsf F_N\left(x\right)-J_N\left(x\right)\right|  \geq r  \right)\dif r \leq  C_{p,t}A_x^p.
\end{equation*}
Since \(J_N(x)\) is deterministic,
\begin{equation*}
\begin{aligned}
    \left\|
    \mathsf F_N\left(x\right)-\expt_{\bxi^\UN}\left[\mathsf F_N\left(x\right)\right]
    \right\|_{L^p}
    &\leq\left\|\mathsf F_N\left(x\right)-J_N\left(x\right)\right\|_{L^p} +\left|\expt_{\bxi^\UN}\left[\mathsf F_N\left(x\right)\right]-J_N\left(x\right)
    \right| \\
    &\leq2\left\|\mathsf F_N\left(x\right)-J_N\left(x\right)\right\|_{L^p}.
\end{aligned}
\end{equation*}
Consequently, we have
\begin{equation}\label{eq:centered-cumulative-count-moment}
    \sup_{N\in\mathbb{N}}
    \expt_{\bxi^\UN}\left[\left|
    \mathsf F_N\left(x\right)-\expt_{\bxi^\UN}\left[\mathsf F_N\left(x\right)\right]\right|^p\right]\leq C_{p,t}
    \left(1+\left|x\right|^\alpha\right)^p.
\end{equation}
Because the one-point distributions have densities, endpoint conventions
are immaterial. Each one-sided centered count is the difference of the
centered cumulative counts at its two endpoints: $0,L$ for the positive
side and $-L,0$ for the negative side. Minkowski's inequality and
\eqref{eq:centered-cumulative-count-moment}, with $p=2k$, give both
asserted moment bounds. Since $L\ge1$ and $2\alpha\le1+\alpha$, this
also proves \eqref{eq:one-sided-count-moment-condition}.
\end{proof}
\begin{proof}[Proof of Remark \ref{rem vague imply bd}]
Let $B=(0,L]$ or $B=[-L,0)$. Since the limiting process has no
particles on $\partial B$ almost surely, vague convergence in distribution
gives
\[
 \bXXi^\UN\left(t\right)\left(B\right)\xrightarrow{\mathrm{law}}\bXXi\left(t\right)\left(B\right).
\]
The second-moment case of the preceding estimate, together with
the assumed finite-density bound, gives a uniform second-moment bound for these local
counts and hence uniform integrability. Their expectations therefore
converge, so the centered counts also converge in distribution. The
Portmanteau theorem gives
\[
\begin{aligned}
 &\expt_{\bxi}\left[\left|
 \bXXi\left(t\right)\left(B\right)-\expt_{\bxi}\left[\bXXi\left(t\right)\left(B\right)\right]\right|^{2k}\right]\leq\liminf_{N\to\infty}\expt_{\bxi^\UN}\left[\left|
 \bXXi^\UN\left(t\right)\left(B\right)-\expt_{\bxi^\UN}\left[\bXXi^\UN\left(t\right)\left(B\right)\right]
 \right|^{2k}\right].
\end{aligned}
\]
Thus the same bounds hold for the limiting centered counts.
\end{proof}

The next deterministic estimate converts bounds on counting discrepancies into bounds on reciprocal tails. It will be applied to a random configuration and its one-point density.
\begin{lemma}\label{lem katori 3.1,3.2}
Let $\bxi\in\mathfrak M$ be a locally finite configuration on $\mathbb R$, and let
$\lambda:\mathbb{R}\rightarrow[0,\infty)$ be locally integrable. Suppose that there
exist $L_0>0$, $C_0>0$, and $0<\delta<1$ such that, for every $R\geq L_0$,
\begin{equation*}
    \left|\bxi\left(\left(0,R\right]\right)-\int_0^R\lambda\left(x\right)\dif x\right|\leq
    C_0R^\delta,\quad \left|\bxi\left(\left[-R,0\right)\right) - \int_{-R}^0\lambda\left(x\right)\dif x \right|\leq
    C_0R^\delta.
\end{equation*}
For $K>L\geq L_0$, use the preceding notation $p_L^K$ and define
\begin{equation*}
 \bar p_L^K\left(\lambda\right)\defeq
 \int_{\left\{L<\left|x\right|\leq K\right\}}
 \frac{\lambda\left(x\right)}x\dif x.
\end{equation*}
Suppose that $\bar p_L(\lambda)=\lim_{K\rightarrow\infty}\bar p_L^K(\lambda)$ exists for every $L\geq L_0$. Then $p_L(\bxi) = \lim_{K\rightarrow\infty}p_L^K(\bxi)$ also exists for every $L\geq L_0$, and
\begin{equation}\label{eq:deterministic-tail-comparison}
    \left|p_L\left(\bxi\right)-\bar p_L\left(\lambda\right)\right|\leq\frac{2C_0\left(2-\delta\right)}{1-\delta}L^{\delta-1}.
\end{equation}
Consequently, if there exist $C_2>0$ and $\nu>0$ such that
\begin{equation*}
    \left|\bar p_L\left(\lambda\right)\right|\leq C_2L^{-\nu},\quad L\geq L_0,
\end{equation*}
then
\begin{equation}\label{eq:deterministic-principal-value-tail}
    \left|p_L\left(\bxi\right)\right|\leq C_2L^{-\nu}+\frac{2C_0\left(2-\delta\right)}{1-\delta}L^{\delta-1},\quad L\geq L_0.
\end{equation}
\end{lemma}
\begin{proof}
    Define the signed measure
\begin{equation*}
    \mu\left(\dif x\right)\defeq\bxi\left(\dif x\right)-\lambda\left(x\right)\,\dif x.
\end{equation*}
For a Borel set $A\subset(0,\infty)$, let $-A=\{-x:x\in A\}$
and define a signed measure $\vartheta$ on $(0,\infty)$ by $\vartheta(A)=\mu(A)-\mu(-A).$ We set
\begin{equation*}
H\left(R\right)\defeq\vartheta\left(\left(0,R\right]\right)=
    \left(\bxi\left(\left(0,R\right]\right)-\int_0^R\lambda\left(x\right)\dif x\right)-
    \left(\bxi\left(\left[-R,0\right)\right)-\int_{-R}^0\lambda\left(x\right)\dif x\right).
\end{equation*}
The assumptions of the lemma imply that
\begin{equation*}
    \left|H\left(R\right)\right|\leq2C_0R^\delta,\quad R\geq L_0.
\end{equation*}
For $K>L\geq L_0$, the chosen half-open interval conventions give
\begin{equation*}
p_L^K\left(\bxi\right)-\bar p_L^K\left(\lambda\right)=\int_{\left\{L<\left|x\right|\leq K\right\}}
    \frac{\mu\left(\dif x\right)}{x} =\int_{\left(L,K\right]}\frac{\vartheta\left(\dif r\right)}{r}.
\end{equation*}
Stieltjes integration by parts yields
\begin{equation}\label{eq:stieltjes-reciprocal-tail-identity}
    \int_{\left(L,K\right]}\frac{\vartheta\left(\dif r\right)}{r}=\frac{H\left(K\right)}{K}-\frac{H\left(L\right)}{L}+\int_L^K\frac{H\left(r\right)}{r^2} \dif r.
\end{equation}
Since $0<\delta<1$, $\left| {H(K)}/{K}\right|\leq2C_0K^{\delta-1}
    \longrightarrow 0$ as $K\rightarrow\infty$. Moreover,
\begin{equation*}
\int_L^\infty
    {\left|H\left(r\right)\right|}/{r^2}\dif r\leq2C_0\int_L^\infty r^{\delta-2}\dif r=\frac{2C_0}{1-\delta}L^{\delta-1}<\infty.
\end{equation*}
Therefore, by \eqref{eq:stieltjes-reciprocal-tail-identity}, we have
\begin{equation*}
\lim_{K\rightarrow\infty} \left( p_L^K\left(\bxi\right)-\bar p_L^K\left(\lambda\right) \right)=-\frac{H\left(L\right)}{L}+\int_L^\infty\frac{H\left(r\right)}{r^2}\dif r.
\end{equation*}
Its absolute value is bounded by
\begin{equation*}
\begin{aligned}
    \left|-\frac{H\left(L\right)}{L}+\int_L^\infty\frac{H\left(r\right)}{r^2}\dif r\right|
    &\leq\frac{\left|H\left(L\right)\right|}{L}+\int_L^\infty\frac{\left|H\left(r\right)\right|}{r^2}\dif r 
    \leq2C_0L^{\delta-1}+\frac{2C_0}{1-\delta}L^{\delta-1} \\
    &=\frac{2C_0\left(2-\delta\right)}{1-\delta}L^{\delta-1}.
\end{aligned}
\end{equation*}
Since $\bar p_L^K(\lambda)\rightarrow \bar p_L(\lambda)$, it follows that $p_L^K(\bxi)$ also converges as $K\rightarrow\infty$. This proves
\eqref{eq:deterministic-tail-comparison}. The triangle inequality gives
\eqref{eq:deterministic-principal-value-tail}.
\end{proof}
We now assume that the first correlation functions of the spatial truncations satisfy \eqref{eq rho bound} and \eqref{eq rho symmetry}.

Combining the count estimates with the deterministic comparison gives uniform control of the principal-value tails. This will yield convergence of the first reciprocal parameter under the coupling used below.
\begin{lemma}\label{lm gamma_1 converges}
Suppose that \eqref{eq rho bound} and \eqref{eq rho symmetry} hold.
For every $t>0$, there are $C_{3,t},C_t<\infty$ and
$\delta_\alpha\in((1+\alpha)/2,1)$ such that, for every $L_0\ge1$,
\[
\begin{aligned}
&\inf_{N\in\mathbb N}\prob_{\bxi^{\left[N\right]}}\left(
 \begin{array}{c}
 p_L\left(\bXXi^{\left[N\right]}\left(t\right)\right)\text{ exists and}\\[-1mm]
 \left|p_L\left(\bXXi^{\left[N\right]}\left(t\right)\right)\right|
 \le C_{2,t}L^{-\nu}+C_{3,t}L^{\delta_\alpha-1}
 \quad\forall L\ge L_0
 \end{array}\right)\ge1-C_tL_0^{-1},\\
&\prob_{\bxi}\left(
 \begin{array}{c}
 p_L\left(\bXXi\left(t\right)\right)\text{ exists and}\\[-1mm]
 \left|p_L\left(\bXXi\left(t\right)\right)\right|
 \le C_{2,t}L^{-\nu}+C_{3,t}L^{\delta_\alpha-1}
 \quad\forall L\ge L_0
 \end{array}\right)\ge1-C_tL_0^{-1}.
\end{aligned}
\]
\end{lemma}
\begin{proof}
Choose $k_\alpha\in\mathbb N$ so that
$k_\alpha(2\delta_\alpha-1-\alpha)>2$. Lemma~\ref{lem Count-moment}
and Chebyshev's inequality give, for $K\in\mathbb N$,
\[
 \sup_{N\in\mathbb N}\prob_{\bxi^{\left[N\right]}}
 \left(\left|\mathsf D_{\pm,t}^{\left[N\right]}\left(K\right)\right|>K^{\delta_\alpha}\right)
 \le C_tK^{-2}.
\]
Summing over the integers $K\ge\lfloor L_0\rfloor$ and taking a union
bound over both signs gives, after changing $C_t$,
\[
 \inf_{N\in\mathbb N}\prob_{\bxi^{\left[N\right]}}
 \left(
 \left|\mathsf D_{+,t}^{\left[N\right]}\left(K\right)\right|\vee\left|\mathsf D_{-,t}^{\left[N\right]}\left(K\right)\right|
 \le K^{\delta_\alpha}
 \quad\forall K\ge\left\lfloor L_0\right\rfloor
 \right)
 \ge1-C_tL_0^{-1}.
\]
If $K\le L<K+1$, monotonicity of the count and \eqref{eq rho bound}
show that
\[
 \left|\mathsf D_{\pm,t}^{\left[N\right]}\left(L\right)\right|
 \le \max\left\{\left|\mathsf D_{\pm,t}^{\left[N\right]}\left(K\right)\right|,
             \left|\mathsf D_{\pm,t}^{\left[N\right]}\left(K+1\right)\right|\right\}
     +\int_{K<\left|x\right|\le K+1}\rho_{\bxi^{\left[N\right]},t}\left(x\right)\dif x.
\]
Because $\delta_\alpha>\alpha$, the last integral is at most
$C_tL^{\delta_\alpha}$. Thus, for some $A_t<\infty$, the joint event
\[
 E_{N,t}\left(L_0\right)\defeq \left\{
 \left|\mathsf D_{+,t}^{\left[N\right]}\left(L\right)\right|\vee\left|\mathsf D_{-,t}^{\left[N\right]}\left(L\right)\right|
 \le A_tL^{\delta_\alpha}\quad\forall L\ge L_0
 \right\}
\]
has probability at least $1-C_tL_0^{-1}$, uniformly in $N$.
On this intersection both discrepancy hypotheses of
Lemma~\ref{lem katori 3.1,3.2} hold with
$\lambda=\rho_{\bxi^{[N]},t}$ and $C_0=A_t$. Combining that lemma with
\eqref{eq rho symmetry} proves the finite-$N$ assertion, with
\[
 C_{3,t}\defeq \frac{2A_t\left(2-\delta_\alpha\right)}{1-\delta_\alpha}.
\]
Remark~\ref{rem vague imply bd} supplies the same two moment estimates
for the limiting process. Repeating the joint-event argument proves the
second assertion.
\end{proof}

The density bound also controls the reciprocal-square tails. Together with the preceding lemma, this provides the tail estimates needed for convergence in the parameter space.
\begin{lemma}\label{lm gamma_2 converges}
Suppose that \eqref{eq rho bound} holds, and let
$0<\kappa<1-\alpha$. For every $t>0$, there is $C_t<\infty$ such that,
for every $L\ge1$,
\[
\begin{aligned}
 \inf_{N\in\mathbb N}\prob_{\bxi^{\left[N\right]}}
 \left(s_L\left(\bXXi^{\left[N\right]}\left(t\right)\right)\le L^{-\kappa}\right)
 &\ge1-C_tL^{\alpha+\kappa-1},\\
 \prob_{\bxi}\left(s_L\left(\bXXi\left(t\right)\right)\le L^{-\kappa}\right)
 &\ge1-C_tL^{\alpha+\kappa-1}.
\end{aligned}
\]
\end{lemma}
\begin{proof}
By \eqref{eq rho bound}, uniformly in $N$,
\[
 \expt_{\bxi^{\left[N\right]}}\left[s_L\left(\bXXi^{\left[N\right]}\left(t\right)\right)\right]
 \le 2C_{1,t}\int_{\left|x\right|>L}\left|x\right|^{\alpha-2}\dif x
 \le C_tL^{\alpha-1}.
\]
The same estimate holds for $\bXXi(t)$. Markov's inequality gives both
claims.
\end{proof}

\begin{proof}[Proof of Proposition \ref{prop in N for all time}]
Fix the deterministic cutoff $b=1$. Since the one-time law has a
correlation density,
\[
 \prob_{\bxi}\left(\bXXi\left(t\right)\left(\left\{-1,1\right\}\right)>0\right)=0.
\]
Moreover,
\[
\begin{aligned}
 \expt_{\bxi}\left[s_1\left(\bXXi\left(t\right)\right)\right]
 &=\int_{\left|x\right|>1}\frac{\rho_{\bxi,t}\left(x\right)}{x^2}\dif x\le C_{1,t}\int_{\left|x\right|>1}\frac{1+\left|x\right|^\alpha}{x^2}\dif x<\infty.
\end{aligned}
\]
Thus $s_1(\bXXi(t))<\infty$ almost surely. For $m\in\mathbb N$, let
\[
 E_m\defeq \left\{
 \begin{array}{c}
 p_L\left(\bXXi\left(t\right)\right)\text{ exists and}\\[-1mm]
 \left|p_L\left(\bXXi\left(t\right)\right)\right|
 \le C_{2,t}L^{-\nu}+C_{3,t}L^{\delta_\alpha-1}
 \quad\forall L\ge m
 \end{array}\right\}.
\]
They increase with $m$ and satisfy $\prob_{\bxi}(E_m)\ge1-C_tm^{-1}$
by Lemma~\ref{lm gamma_1 converges}. Hence
$\prob_{\bxi}(\bigcup_{m\ge1}E_m)=1$. On this union choose an integer
$L>\max\{m,1\}$. For $K>L$, local finiteness gives
\[
 p_1^K\left(\bXXi\left(t\right)\right)=p_1^L\left(\bXXi\left(t\right)\right)+p_L^K\left(\bXXi\left(t\right)\right).
\]
Letting $K\to\infty$ proves that $p_1(\bXXi(t))$ exists and equals
$p_1^L(\bXXi(t))+p_L(\bXXi(t))$. It is finite almost surely, proving
$\bXXi(t)\in\mathfrak N$.
\end{proof}

We next translate vague convergence and convergence of the two reciprocal statistics into convergence in the parameter space. This is the continuity statement needed for the transition factors.
\begin{lemma}\label{lem:vague-pv-to-upsilon}
Let $b>0$, and let
$\boldsymbol\eta^\UN,\boldsymbol\eta\in\mathfrak M_b$. Suppose that
\[
 \boldsymbol\eta^\UN\xrightarrow{\mathrm{vg}}\boldsymbol\eta,
 \qquad \boldsymbol\eta\left(\left\{-b,0,b\right\}\right)=0,
\]
and
\[
 p_b\left(\boldsymbol\eta^\UN\right)\longrightarrow p_b\left(\boldsymbol\eta\right),
 \qquad
 s_b\left(\boldsymbol\eta^\UN\right)\longrightarrow s_b\left(\boldsymbol\eta\right).
\]
Then
\[
 f_{p_b\left(\boldsymbol\eta^\UN\right)}^b\left(\boldsymbol\eta^\UN\right)
 \longrightarrow f_{p_b\left(\boldsymbol\eta\right)}^b\left(\boldsymbol\eta\right)
 \qquad\text{in }\bUpsilon_b.
\]
\end{lemma}
\begin{proof}
Because the limit has no atoms at $-b$, $0$, and $b$, the continuity-set
integer counts
\[
 \boldsymbol\eta^\UN\left(\left[0,b\right)\right)\longrightarrow\boldsymbol\eta\left(\left[0,b\right)\right),
 \qquad
 \boldsymbol\eta^\UN\left(\left(-b,0\right)\right)\longrightarrow\boldsymbol\eta\left(\left(-b,0\right)\right)
\]
are eventually constant. These are precisely the total-degree coordinates
$q$ in the positive and negative copies of $\mathbb U_b$. The ordered inner
locations converge as well. No inner root can tend to zero, since that would
create a zero atom in the vague limit.

For the outer roots, vague convergence gives convergence of every ordered
reciprocal coordinate whose limit is nonzero. If a fixed surplus outer rank
stayed bounded, a subsequence would either approach $\pm b$, which is
excluded, or create an additional outer atom in a compact continuity set,
contradicting convergence of the corresponding counts. Thus every surplus
outer root escapes to infinity, and its reciprocal coordinate tends to zero.
Thus both $\mathbb W_b$ coordinates converge. The assumed convergence of
$p_b$ and $s_b$ is exactly convergence of the $\gamma_1$- and
$\delta$-coordinates, respectively.
\end{proof}

We can now pass a transition factor to the limit along a subsequence. This is the remaining step in passing the finite-particle Markov identity to the limiting process.
\begin{lemma}\label{prop limit of markov once}
Assume \eqref{eq rho bound} and \eqref{eq rho symmetry}. Let $M\ge2$,
let $0<t_1<\cdots<t_M<\infty$, and let
$\phi_1,\ldots,\phi_M\in\mathcal C_{\mathrm{bd}}(\mathfrak M)$.
There exists a subsequence $(N_k)_{k\ge1}$ such that
\begin{equation}\label{eq convergence of conditional expt}
\begin{aligned}
&\expt_{\bxi^{\left[N_k\right]}}\left[
 \left(\prod_{i=1}^{M-1}\phi_i\left(\bXXi^{\left[N_k\right]}\left(t_i\right)\right)\right)
 \left(\mT_{t_M-t_{M-1}}\phi_M\right)\left(\bXXi^{\left[N_k\right]}\left(t_{M-1}\right)\right)
 \right]\\
&\quad\longrightarrow
\expt_{\bxi}\left[
 \left(\prod_{i=1}^{M-1}\phi_i\left(\bXXi\left(t_i\right)\right)\right)
 \left(\mT_{t_M-t_{M-1}}\phi_M\right)\left(\bXXi\left(t_{M-1}\right)\right)
 \right].
\end{aligned}
\end{equation}
\end{lemma}
\begin{proof}
Put $t_*=t_{M-1}$. Finite-dimensional convergence gives joint weak
convergence of the vectors
\[
 \left(\bXXi^{\left[N\right]}\left(t_1\right),\ldots,\bXXi^{\left[N\right]}\left(t_{M-1}\right)\right)
 \quad\text{to}\quad
 \left(\bXXi\left(t_1\right),\ldots,\bXXi\left(t_{M-1}\right)\right)
\]
in $\mathfrak M^{M-1}$. By the Skorohod representation theorem there
are copies, denoted by tildes, for which
\begin{equation}\label{eq copy almost sure vague convergence}
 \tbXXi^{\left[N\right]}\left(t_i\right)\xrightarrow{\mathrm{vg}}\tbXXi\left(t_i\right),
 \qquad i=1,\ldots,M-1,
\end{equation}
almost surely.
Fix $b=1$ and use the deterministic integers $L\ge2$. All the one-time
laws involved have correlation densities. After removing one null event,
none of the countably many coupled configurations has an atom at
$-1$, $0$, $1$, or $\pm L$, $L\ge2$. For each fixed such $L$,
\eqref{eq copy almost sure vague convergence} gives
\[
\begin{aligned}
 p_1^L\left(\tbXXi^{\left[N\right]}\left(t_*\right)\right)&\longrightarrow
 p_1^L\left(\tbXXi\left(t_*\right)\right),\\
 s_1^L\left(\tbXXi^{\left[N\right]}\left(t_*\right)\right)&\longrightarrow
 s_1^L\left(\tbXXi\left(t_*\right)\right)
\end{aligned}
\]
almost surely. Since $p_1=p_1^L+p_L$, Lemma~\ref{lm gamma_1 converges}
makes the two $p_L$ tails uniformly small in probability as
$L\to\infty$. Consequently,
\[
 p_1\left(\tbXXi^{\left[N\right]}\left(t_*\right)\right)\longrightarrow p_1\left(\tbXXi\left(t_*\right)\right)
 \qquad\text{in probability}.
\]
Similarly, $s_1=s_1^L+s_L$ and Lemma~\ref{lm gamma_2 converges} give
\[
 s_1\left(\tbXXi^{\left[N\right]}\left(t_*\right)\right)\longrightarrow s_1\left(\tbXXi\left(t_*\right)\right)
 \qquad\text{in probability}.
\]
Extract a common subsequence $(N_k)$ along which both scalar convergences
hold almost surely. Every finite coupled configuration belongs to
$\mathfrak M_1$. For the limit, the $s_1$ estimate in the proof of
Proposition~\ref{prop in N for all time}, together with the no-atom event at
$\pm1$, gives membership in $\mathfrak M_1$. Together with
\eqref{eq copy almost sure vague convergence},
Lemma~\ref{lem:vague-pv-to-upsilon} yields
\[
 \tbXXi^{\left[N_k\right]}\left(t_*\right)\xrightarrow{\ \bUpsilon\ }\tbXXi\left(t_*\right)
 \qquad\text{almost surely}.
\]
Remark~\ref{remark upsilon continuous} now gives convergence of the
$\mT_{t_M-t_{M-1}}\phi_M$ factors, while vague continuity gives
convergence of the other factors. All factors are bounded, so bounded
convergence proves \eqref{eq convergence of conditional expt}.
\end{proof}

\begin{proof}[Proof of Proposition \ref{prop markov from rho}]
Proposition~\ref{prop in N for all time} ensures that every transition
factor below is defined almost surely. Fix $M\ge2$,
$0<t_1<\cdots<t_M$, and
$\phi_1,\ldots,\phi_M\in\mathcal C_{\mathrm{bd}}(\mathfrak M)$.
Finite-dimensional convergence gives
\begin{equation}\label{eq all converge}
 \lim_{N\to\infty}\expt_{\bxi^{\left[N\right]}}
 \left[\prod_{i=1}^M\phi_i\left(\bXXi^{\left[N\right]}\left(t_i\right)\right)\right]
 =\expt_{\bxi}\left[\prod_{i=1}^M\phi_i\left(\bXXi\left(t_i\right)\right)\right].
\end{equation}
Each finite Dyson model satisfies the Markov identity with the same
operators $\mT_t$, by their finite-configuration agreement.
Along the subsequence supplied by
Lemma~\ref{prop limit of markov once}, its transition-factor expression
converges to the corresponding infinite-particle expectation. Combining this with
\eqref{eq all converge} proves the Markov identity.
\end{proof}

The symmetric configurations treated in Theorem~\ref{thm bound} satisfy the criterion just proved. Their density bound therefore yields the Markov property.
\begin{cor}\label{cor:section5-markov}
Under the hypotheses of Theorem~\ref{thm bound}, the process
$\bXi_{\bxi}(\cdot)$ has the Markov property.
\end{cor}
\begin{proof}
Let $\bxi^{[N]}=\bxi|_{[-N,N]}$ be the spatial truncations used in
this section. Each $\bxi^{[N]}$ is either $\epsilon\delta_0$ or exactly one of
the symmetric pair truncations in \eqref{eq:config}. Hence
\eqref{eq:global-main} gives \eqref{eq rho bound} for the spatial
truncations with exponent $\alpha=q/2<1$; the $\epsilon\delta_0$ case has bounded density obviously.
Remark~\ref{rem:limit-density} gives the corresponding limiting-density
bound, since $q-1<q/2$. The finite densities are even by reflection invariance, and the
limiting density is even by local uniform convergence. Thus both terms in
\eqref{eq rho symmetry} vanish. Finally, symmetry gives $p_b(\bxi)=0$,
so the process defined in this section is precisely the process with kernel
$\K_{f_0^b(\bxi)}$. Proposition~\ref{prop markov from rho} applies.
\end{proof}

\subsection{A class of Markov initial configurations}
We now give a concrete class for which the density criterion holds.
For a finite simple configuration, residue evaluation of the separated
finite-root kernel, followed by a Gaussian contour shift, gives
\cite{Katori_2009}
\begin{equation}\label{eq formula of rho sum}
 \rho_{\bxi,t}\left(x\right)
 =\frac1{2\pi t}\sum_{a\in\bxi}\me^{-\left(x-a\right)^2/\left(2t\right)}
 \int_{\mathbb R}\me^{-y^2/\left(2t\right)}
 \prod_{\substack{c\in\bxi\\c\ne a}}
 \left(1-\frac{\mi y+x-a}{c-a}\right)\dif y.
\end{equation}
The shift is justified because the integrand is a polynomial times a
Gaussian. This formula already includes the residue correction of
Section~\ref{sec:preliminaries}. If there is an atom at zero, perturbing
it off zero and using finite-particle continuity gives the same formula,
as in Proposition~\ref{finitedysonmodel}.

We isolate the shifted centres and regularized products appearing after completion of the square.
\begin{definition}
For a finite $\bxi\in\mathfrak M^0$, $a\in\bxi$, and $t>0$, define
\[
 H_{\bxi,a}\defeq \sum_{\substack{c\in\bxi\\c\ne a}}\frac1{c-a},
 \qquad c_{\bxi,a}\left(t\right)\defeq a-tH_{\bxi,a},
\]
and
\[
 E_1\left(z\right)\defeq \left(1-z\right)\me^z,
 \qquad
 \mathfrak C_{\bxi,a}\left(z\right)
 \defeq \prod_{\substack{c\in\bxi\\c\ne a}}
   E_1\left(\frac z{c-a}\right).
\]
\end{definition}

The residue formula reduces a uniform density estimate to bounds on the regularized products and the number of shifted centers in a unit interval. We state this criterion in a form suitable for uniformly separated configurations.
\begin{proposition}\label{prop bd 1st coro by F number}
Let $(\bxi^\UN)_{N\ge1}$ be finite configurations in $\mathfrak M^0$.
Fix $t>0$ and suppose that, for some $A_t<\infty$ and
$0\le\theta_t<1/(2t)$,
\begin{equation}\label{eq bound of F xi a}
 \sup_{N\in\mathbb N}\sup_{a\in\bxi^\UN}
 \left|\mathfrak C_{\bxi^\UN,a}\left(z\right)\right|
 \le A_t\me^{\theta_t\left|z\right|^2},\qquad z\in\mathbb C.
\end{equation}
Suppose also that, for some $0\le\alpha<1$ and $B_t<\infty$,
\begin{equation}\label{eq number is bounded}
 \sup_{N\in\mathbb N}
 \#\left\{a\in\bxi^\UN:c_{\bxi^\UN,a}\left(t\right)\in\left[u,u+1\right)\right\}
 \le B_t\left(1+\left|u\right|^\alpha\right),\qquad u\in\mathbb R.
\end{equation}
Then there is $C_t<\infty$ such that
\[
 \sup_{N\in\mathbb N}\rho_{\bxi^\UN,t}\left(x\right)
 \le C_t\left(1+\left|x\right|^\alpha\right),\qquad x\in\mathbb R.
\]
\end{proposition}
\begin{proof}
For $a\in\bxi^\UN$, put $u_a=x-a$. Since
\[
 \prod_{c\ne a}\left(1-\frac z{c-a}\right)
 =\me^{-H_{\bxi^\UN,a}z}\mathfrak C_{\bxi^\UN,a}\left(z\right),
\]
formula \eqref{eq formula of rho sum} becomes
\begin{equation}\label{eq rho 1st}
 \rho_{\bxi^\UN,t}\left(x\right)
 =\frac1{2\pi t}\sum_{a\in\bxi^\UN}\me^{-u_a^2/\left(2t\right)}
 \int_{\mathbb R}
 \me^{-y^2/\left(2t\right)-H_{\bxi^\UN,a}\left(u_a+\mi y\right)}
 \mathfrak C_{\bxi^\UN,a}\left(u_a+\mi y\right)\dif y.
\end{equation}
Put $q_a=x-c_{\bxi^\UN,a}(t)=u_a+tH_{\bxi^\UN,a}$.
Completing the square with $w=y+\mi tH_{\bxi^\UN,a}$ gives
$u_a+\mi y=q_a+\mi w$ and
$-(u_a^2+y^2)/(2t)-H_{\bxi^\UN,a}(u_a+\mi y)=-(q_a^2+w^2)/(2t)$.
The integrand is entire, and \eqref{eq bound of F xi a} with
$\theta_t<1/(2t)$ makes the vertical sides vanish. Moving the shifted
line back to $\mathbb R$ therefore yields
\[
 \rho_{\bxi^\UN,t}\left(x\right)
 =\frac1{2\pi t}\sum_{a\in\bxi^\UN}\me^{-q_a^2/\left(2t\right)}
 \int_{\mathbb R}\me^{-y^2/\left(2t\right)}
 \mathfrak C_{\bxi^\UN,a}\left(q_a+\mi y\right)\dif y.
\]
Put
\[
 \delta_t\defeq\frac1{2t}-\theta_t>0.
\]
Since the density is nonnegative, the triangle inequality and
\eqref{eq bound of F xi a} yield
\[
\begin{aligned}
 \rho_{\bxi^\UN,t}\left(x\right)
 &\le\frac{A_t}{2\pi t}\sum_{a\in\bxi^\UN}
 \me^{-\delta_t\left(x-c_{\bxi^\UN,a}\left(t\right)\right)^2}
 \int_{\mathbb R}\me^{-\delta_ty^2}\dif y\le C_t'\sum_{a\in\bxi^\UN}
 \me^{-\delta_t\left(x-c_{\bxi^\UN,a}\left(t\right)\right)^2}.
\end{aligned}
\]
Partition $\mathbb R$ into
\[
 I_{k,x}\defeq \left[x+k,x+k+1\right),\qquad k\in\mathbb Z.
\]
If $c_{\bxi^\UN,a}(t)\in I_{k,x}$, then
$|x-c_{\bxi^\UN,a}(t)|\ge\max\{|k|-1,0\}$. Hence
\[
 \sum_{a\in\bxi^\UN}
 \me^{-\delta_t\left(x-c_{\bxi^\UN,a}\left(t\right)\right)^2}
 \le B_t\sum_{k\in\mathbb Z}\left(1+\left|x+k\right|^\alpha\right)
 \me^{-\delta_t\max\left\{\left|k\right|-1,0\right\}^2}.
\]
If $\alpha=0$, Gaussian summability gives the asserted uniform bound
directly. If $0<\alpha<1$, then
$|x+k|^\alpha\le|x|^\alpha+|k|^\alpha$, and Gaussian summability again
gives $C_t''(1+|x|^\alpha)$, uniformly in $N$.
\end{proof}

Uniform separation of the initial particles bounds the number of shifted centers in a unit interval. This verifies the counting hypothesis of the preceding proposition.
\begin{lemma}\label{lm number is bounded}
Let $\bxi$ be a finite simple configuration with
$\inf_{a,c\in\bxi,\,a\ne c}|a-c|\ge d>0$. For every $t>0$ there is
$B_{t,d}<\infty$ such that
\[
 \#\left\{a\in\bxi:c_{\bxi,a}\left(t\right)\in\left[u,u+1\right)\right\}
 \le B_{t,d},\qquad u\in\mathbb R.
\]
\end{lemma}
\begin{proof}
Write $\bxi=\sum_{i=1}^m\delta_{x_i}$ with $x_{i+1}>x_i$, and put
\[
 h_i\defeq -H_{\bxi,x_i},\qquad
 c_i\defeq c_{\bxi,x_i}\left(t\right)=x_i+th_i.
\]
Fix $i<j$ and set $L=x_j-x_i$. Then
    \begin{equation*}
    \begin{aligned}
    h_j-h_i = &\frac{2}{L}+\sum_{i<k<j}\left(\frac{1}{x_j-x_k}+\frac{1}{x_k-x_i}\right)-L\sum_{k<i}\frac{1}{\left(x_i-x_k\right)\left(x_j-x_k\right)}
    -L\sum_{k>j}\frac{1}{\left(x_k-x_i\right)\left(x_k-x_j\right)}\\
    \geq & -L\sum_{k<i}\frac{1}{\left(x_i-x_k\right)\left(x_j-x_k\right)}
    -L\sum_{k>j}\frac{1}{\left(x_k-x_i\right)\left(x_k-x_j\right)}\\
    \geq & -\frac{2}{d}\sum_{n=1}^\infty\frac{L/d}{n\left(n+L/d\right)}.
    \end{aligned}
    \end{equation*}
    Here, the last inequality holds since for $n\in\mathbb{N}\setminus\{0\}$, one has
    \begin{equation*}
    \begin{aligned}
    &x_i-x_k\geq nd,\ x_j-x_k \geq L+nd\quad \textnormal{if } k=i-n,\\
    &x_k-x_j\geq nd,\ x_k-x_i \geq L+nd\quad\textnormal{if } k= j+n.
    \end{aligned}
    \end{equation*}
    For $R>0$, one has
    $\sum_{n=1}^\infty R/[n(n+R)]\leq3+\log(1+R)$. Thus
    \begin{equation*}
        h_j-h_i \geq -\frac{2}{d}\left(3+\log\left(1+\frac{L}{d}\right)\right).
    \end{equation*}
    This implies that
    \begin{equation*}
        c_j-c_i \geq L-\frac{2t}{d}\left(3+\log\left(1+\frac{L}{d}\right)\right).
    \end{equation*}
Choose $L_{t,d}<\infty$ so that the last lower bound exceeds $1$ whenever
$L>L_{t,d}$. Therefore any collection of shifted centers in one unit
interval comes from original roots contained in an interval of length at
most $L_{t,d}$, and
\[
 \#\left\{a\in\bxi:c_{\bxi,a}\left(t\right)\in\left[u,u+1\right)\right\}
 \le1+\left\lceil\frac{L_{t,d}}d\right\rceil.
\]
\end{proof}

We introduce the class of admissible simple configurations with a positive minimum spacing.
\begin{definition}
Define, with the convention $\inf\varnothing=+\infty$,
\[
 \mathfrak L\defeq
 \left\{\bxi\in\mathfrak N\cap\mathfrak M^0:
 \inf_{a,c\in\bxi,\,a\ne c}\left|a-c\right|>0\right\}.
\]
\end{definition}

Uniform separation also controls the regularized products. The two estimates together give a density bound uniform in space and in the truncation.
\begin{proposition}\label{prop:uniform-density-separated}
Let $\bxi\in\mathfrak L$, put
\[
 d\defeq \min\left\{1,\inf_{a,c\in\bxi,\,a\ne c}\left|a-c\right|\right\}>0,
 \qquad \bxi^{\left[N\right]}\defeq \bxi|_{\left[-N,N\right]},
\]
and let $\rho_{\bxi^{[N]},t}$ be the one-point correlation function of the
finite Dyson model. For every $t>0$ there is $C_{t,d}<\infty$ such that
\[
 \sup_{N\in\mathbb N}\sup_{x\in\mathbb R}
 \rho_{\bxi^{\left[N\right]},t}\left(x\right)\le C_{t,d}.
\]
\end{proposition}
\begin{proof}
Empty truncations have zero density and may be omitted.
By Proposition~\ref{prop bd 1st coro by F number}, it is enough to verify
\eqref{eq bound of F xi a} and \eqref{eq number is bounded} with
$\alpha=0$. We first prove the following bound for every finite
configuration $\boldsymbol\eta$ whose minimum spacing is at least $d$:
\begin{equation}\label{eq log bd for F}
 \sup_{a\in\boldsymbol\eta}
 \log\left|\mathfrak C_{\boldsymbol\eta,a}\left(z\right)\right|
 \le C\left(1+\frac{\left|z\right|}{d}\right)
          \log\left(2+\frac{\left|z\right|}{d}\right),
 \qquad z\in\mathbb C,
\end{equation}
where $C$ is universal. Put $r=|z|$. The assertion is immediate when
$r=0$, so assume $r>0$ and split the roots according as
$0<|c-a|\le2r$ or $|c-a|>2r$. Since
\[
 \log\left|E_1\left(w\right)\right|
 \le \left|w\right|+\log\left(1+\left|w\right|\right)\le2\left|w\right|,
\]
the spacing condition gives
\begin{equation}\label{eq l 2r bd}
\begin{aligned}
 &\sum_{\substack{c\in\boldsymbol\eta\\0<\left|c-a\right|\le2r}}
 \log\left|E_1\left(\frac z{c-a}\right)\right|\le4r\sum_{k=1}^{\left\lfloor2r/d\right\rfloor}\frac1{kd}
 \le C\frac rd\log\left(2+\frac rd\right).
\end{aligned}
\end{equation}
For $|w|<1/2$, the power series for $\log E_1(w)$ gives
$\log|E_1(w)|\le|w|^2$. Hence
\begin{equation}\label{eq g 2r bd}
\begin{aligned}
 &\sum_{\substack{c\in\boldsymbol\eta\\\left|c-a\right|>2r}}
 \log\left|E_1\left(\frac z{c-a}\right)\right|\le r^2\sum_{\substack{c\in\boldsymbol\eta\\\left|c-a\right|>2r}}
          \frac1{\left(c-a\right)^2}
 \le2r^2\sum_{k=0}^{\infty}\frac1{\left(2r+kd\right)^2}
 \le C\left(1+\frac rd\right).
\end{aligned}
\end{equation}
Equations \eqref{eq l 2r bd} and \eqref{eq g 2r bd} prove
\eqref{eq log bd for F}. Fix $t>0$ and take
$\theta_t=1/(4t)$. Since
\[
 C\left(1+r/d\right)\log\left(2+r/d\right)=o\left(r^2\right),
\]
there is $A_{t,d}<\infty$ such that the right-hand side of
\eqref{eq log bd for F} is at most
$\log A_{t,d}+\theta_t r^2$ for every $r\ge0$. Every truncation
$\bxi^{[N]}$ has minimum spacing at least $d$, so this proves
\eqref{eq bound of F xi a}, uniformly in $N$. Lemma~\ref{lm number is bounded}
proves \eqref{eq number is bounded}, again uniformly in $N$, with
$\alpha=0$. The result follows from
Proposition~\ref{prop bd 1st coro by F number}.
\end{proof}

For symmetric configurations, the reciprocal-density condition vanishes. The uniform density bound therefore gives the Markov property for the separated class.
\begin{cor}\label{xi in L have M-p}
Let $\bxi\in\mathfrak L$ be symmetric, in the sense that
$\bxi(\{a\})=\bxi(\{-a\})$ for every $a\in\mathbb R$. Then
$\bXi_{\bxi}(\cdot)$ has the Markov property.
\end{cor}
\begin{proof}
Each spatial truncation $\bxi^{[N]}$ is symmetric. Reflection invariance of
finite-particle Dyson Brownian motion therefore gives
$\rho_{\bxi^{[N]},t}(x)=\rho_{\bxi^{[N]},t}(-x)$. The locally uniform kernel
convergence in Theorem~\ref{MainTheorem} shows that the limiting density is
also even. Proposition~\ref{prop:uniform-density-separated} supplies the
finite-density bound in \eqref{eq rho bound}, and the limiting-density bound
then follows as noted after Proposition~\ref{prop in N for all time}. Thus
\eqref{eq rho symmetry} holds with both left-hand sides equal to zero.
Moreover, symmetry gives $p_b(\bxi)=0$ for every admissible cutoff $b$.
Proposition~\ref{prop markov from rho} now applies.
\end{proof}

As an example, for $\alpha>0$ define the power deformation
\[
 \mathbb Z_\alpha\defeq
 \sum_{j\in\mathbb Z}\delta_{\operatorname{sgn}\left(j\right)\left|j\right|^\alpha},
 \qquad\operatorname{sgn}\left(0\right)=0.
\]

For every $\alpha\ge1$, the configuration $\mathbb Z_\alpha$ is symmetric
and has minimum spacing at least one. In addition,
$p_{1/2}(\mathbb Z_\alpha)=0$ by symmetry and
$s_{1/2}(\mathbb Z_\alpha)<\infty$, since
$\sum_{j\ge1}j^{-2\alpha}<\infty$. Hence
$\mathbb Z_\alpha\in\mathfrak L$, and Corollary~\ref{xi in L have M-p}
shows that $\bXi_{\mathbb Z_\alpha}(\cdot)$ has the Markov
property. This includes the integer configuration $\mathbb Z=\mathbb Z_1$.

For $1/2<\alpha<1$, we first present the following lemma.
\begin{lemma}\label{lem:power-configuration-counts}
    Let $\alpha\in(1/2,1)$ and $\mathbb{Z}_{\alpha}$ be the corresponding power configuration. We have $\mathbb{Z}_{\alpha}$ satisfies the hypotheses of Theorem \ref{thm bound}, \emph{($\mathbf{UC}$)} and \emph{($\mathbf{LC}$)},  with
    \begin{equation*}
    \epsilon=1,\quad a_n= n^{\alpha},\quad q=\frac{1}{\alpha}\in(1,2),\quad \theta=q-1\in(0,1).
    \end{equation*}
\end{lemma}
\begin{proof}
Write $r_+=\max\{r,0\}$. The positive-root counting function satisfies
\begin{equation*}
    \#\{n\geq1:a_n\leq r\}=\lfloor r^q\rfloor,
    \quad r\geq0.
\end{equation*}
Hence, for $u\geq0$ and $s>0$, counting the positive roots, negative
roots, and the possible contribution of the origin gives
\begin{equation*}
    Z_\alpha([u-s,u+s])
    \leq
    2+(u+s)^q-(u-s)_+^q+(s-u)_+^q.
\end{equation*}
Since we have
\begin{equation*}
    (u+s)^q-(u-s)_+^q\leq 2qs(u+s)^{q-1}\leq 2q\left(su^\theta+s^q\right),
\end{equation*}
where the last inequality uses $0<\theta<1$, we obtain
\begin{equation*}
    Z_\alpha([u-s,u+s])
    \leq C_q\left[1+s(1+u)^\theta+s^q\right].
\end{equation*}
This proves \emph{($\mathbf{UC}$)}, including the atom at zero.

For \emph{($\mathbf{LC}$)}, let $u\geq1$ and $0<s\leq u/4$. Integer counting and the mean value theorem give
\begin{equation*}
    \begin{aligned}
    \#\{n\geq1:|a_n-u|\leq s\}
    \geq (u+s)^q-(u-s)^q-2\geq 2q(3/4)^{q-1}su^\theta-2.
    \end{aligned}
\end{equation*}
Put $c_q=q(3/4)^{q-1}$. If
$s\geq(2/c_q)u^{-\theta}$, the last expression is at least
$c_qsu^\theta$. Thus \emph{($\mathbf{LC}$)} holds with
\begin{equation*}
    U_0=1, \quad C_{\mathrm{low}}=\frac{2}{c_q},
    \quad c_{\mathrm{low}}=c_q.
\end{equation*}
\end{proof}
Combining above lemma and Corollary \ref{cor:section5-markov}, we see that $\mathbb{Z}_\alpha$ generates a Markov process with $\alpha\in(1/2,1)$.

In conclusion, $\mathbb{Z}_\alpha$ generates a infinite particles Markov process for all $\alpha>1/2$.
\section{The long-time sine-kernel limit}\label{sec:sine}
\subsection{Setting and motivation}
For $\rho>0$, define the
extended sine kernel
\begin{equation}\label{eq extended sine}
    \K_{\Sine,\rho}\left(\left(s,x\right),\left(t,y\right)\right) \defeq \frac{1}{2\pi}\int_{-\rho\pi}^{\rho\pi} \me^{u^2\left(t-s\right)/2+\mi u\left(y-x\right)}\dif u -\indi_{s>t}\mathfrak{h}_{s-t}\left(x,y\right).
\end{equation}
This kernel defines a determinantal point process, which we denote by
$\bXi_{\Sine,\rho}$; see \cite{NAGAO199842}.
Katori and Tanemura \cite{Katori_2009} prove that the process started
from $\mathbb Z$ relaxes to the extended $\Sine$ process. Writing
$\K_{\mathbb Z}$ for its kernel, they show that
\begin{equation*}
    \lim_{T\to\infty}\K_{\mathbb{Z}}\left(\left(T+s,x\right),\left(T+t,y\right)\right) = \K_{\Sine,1}\left(\left(s,x\right),\left(t,y\right)\right).
\end{equation*}
We extend this long-time limit to symmetric configurations with regularly varying counting function. Let
\begin{equation}\label{eq:configuration}
  \bxi\defeq \eps\delta_0+\sum_{j\geq1}\left(\delta_{a_j}+\delta_{-a_j}\right),
  \qquad \eps\in\left\{0,1\right\},\qquad 0<a_1<a_2<\cdots .
\end{equation}
Write $\mathcal N_+(R)=\#\{j:a_j\le R\}$ for the cumulative
positive-root count. We assume that, for some $c>0$ and $q\in(0,2)$,
\begin{equation}\label{eq:assumptions}
 a_j\sim\left(j/c\right)^{1/q},\qquad j\to\infty.
\end{equation}
In particular, $\mathcal N_+(R)\sim cR^q$.

Since $q<2$, the canonical product
\[
 \mathfrak E_{\bxi}\left(z\right)\defeq z^\eps\prod_{j\geq1}\left(1-\frac{z^2}{a_j^2}\right)
\]
converges locally uniformly because $\sum_j a_j^{-2}<\infty$.
Symmetry gives $\bxi\in\mathfrak N$, with $p_b(\bxi)=0$ at every
admissible cutoff. The process $\bXi_{\bxi}$ is therefore defined for
all positive times, with parameter $f_0^b(\bxi)$ and $\gamma_2=0$.
Its kernel is determined by this symmetric normalization, since
\[
 \mathfrak E_{f_0^b\left(\bxi\right)}\left(z\right)
 =C_b\mathfrak E_{\bxi}\left(z\right)
\]
for a nonzero constant $C_b$, which cancels from kernel quotients.

For $S,U>0$ and $x,y\in\R$, let $\Gamma$ be the contour defined in
\eqref{Gamma}, with the orientation fixed in
Definition~\ref{def def of K}, and orient $\mi\R$ upward. The contour
integral is
\begin{equation}\label{eq:kernel-input}
 \BL_{\bxi,\Gamma}\left(\left(S,x\right),\left(U,y\right)\right)
 \defeq \frac{1}{\left(2\pi\mi\right)^2\sqrt{SU}}
 \int_\Gamma\!\dif z\int_{\mi\R}\!\dif w\,
 \frac{1}{w-z}
 \exp\left\{\frac{\left(w-y\right)^2}{2U}-\frac{\left(z-x\right)^2}{2S}\right\}
 \frac{\mathfrak E_{\bxi}\left(w\right)}{\mathfrak E_{\bxi}\left(z\right)}.
\end{equation}
At a transverse contour intersection, the singularity is integrable
jointly in the two contour parameters; we use the convention of
Definition~\ref{def def of K}. The residue correction and the analytic
part are
\begin{equation}\label{eq:section7-corrected-analytic}
\begin{aligned}
 \Rc_\Gamma\left(\left(S,x\right),\left(U,y\right)\right)
 &=\int_{-1}^{1}
   \mathfrak h_S\left(x,\mi u\right)\mathfrak h_U\left(-\mi y,u\right)\,\dif u,\\
 \BL_{\bxi}&\defeq \BL_{\bxi,\Gamma}+\Rc_\Gamma.
\end{aligned}
\end{equation}
The full kernel is
\begin{equation}\label{eq:full-kernel}
\K_{\bxi}\left(\left(S,x\right),\left(U,y\right)\right)
\defeq \BL_{\bxi}\left(\left(S,x\right),\left(U,y\right)\right)-\indi_{S>U} \mathfrak{h}_{S-U}\left(x,y\right).
\end{equation}

\subsection{Main result}
For all sufficiently large $T$, we set
\begin{equation}\label{eq:canonical-scale}
 n_T\defeq \max\left\{n\geq1:a_n^2\leq Tn\right\},\qquad L_T\defeq \left(Tn_T\right)^{1/2},\qquad
 \sigma_T\defeq \frac{L_T}{n_T}=\frac{T}{L_T}.
\end{equation}
Define 
\begin{equation*}
 A_q\defeq \frac{q\pi}{\sin\left(\pi q/2\right)},\qquad
 u_q\defeq A_q^{1/\left(2-q\right)},\qquad
 \rho_q\defeq \frac{u_q}{\pi}.
\end{equation*}
A subscript $T$ denotes a finite-time quantity and a subscript $q$ its
limit. In particular, $w_T$ will denote the positive saddle height,
with $w_T\to u_q$ and saddle points $\pm\mi w_T$.
For $\tau>-n_T$, define
\[
 \begin{aligned}
 T_\tau&\defeq T+\sigma_T^2\tau,\\
 \bXi_{\bxi,T}\left(\tau\right)
 &\defeq\left(x\mapsto x/\sigma_T\right)_\#
 \bXi_{\bxi}\left(T_\tau\right).
 \end{aligned}
\]
Every fixed finite collection of rescaled times is admissible for all
sufficiently large $T$.

We now state the long-time limit at the scales defined above. Both the extended kernel and the finite-dimensional distributions converge to those of the extended $\Sine$ process.
\begin{theorem}\label{thm:main}
Assume \eqref{eq:configuration} and \eqref{eq:assumptions}, and let
$\bXi_{\bxi}$ be the determinantal process described above, with full
kernel \eqref{eq:full-kernel}. Then the analytic parts satisfy
\begin{equation}\label{eq:analytic-convergence}
 \sigma_T\BL_{\bxi}\left(\left(T_s,\sigma_TX\right),\left(T_t,\sigma_TY\right)\right)
 \longrightarrow
 \frac{1}{2\pi}\int_{-u_q}^{u_q}
 \me^{k^2\left(t-s\right)/2+\mi k\left(Y-X\right)}\dif k
\end{equation}
locally uniformly in $(s,t,X,Y)\in\mathbb R^4$. Consequently,
\[
 \sigma_T\K_{\bxi}\left(\left(T_s,\sigma_TX\right),\left(T_t,\sigma_TY\right)\right)
 -\K_{\Sine,\rho_q}\left(\left(s,X\right),\left(t,Y\right)\right)\longrightarrow0
\]
locally uniformly in all four variables. The processes $\bXi_{\bxi,T}$ converge in finite-dimensional
distributions, for the vague topology on locally finite point measures,
to the extended $\Sine$ process of density $\rho_q$.
\end{theorem}
In particular,
\begin{equation}\label{eq:equal-time-limit}
 \sigma_T\K_{\bxi}\left(\left(T,\sigma_TX\right),\left(T,\sigma_TY\right)\right)
 \longrightarrow \frac{\sin\left\{u_q\left(X-Y\right)\right\}}{\pi\left(X-Y\right)},
\end{equation}
with the continuous value $\rho_q$ at $X=Y$.
The heat term in \eqref{eq extended sine} is singular as $s\downarrow t$
at $X=Y$, but it cancels exactly from the full-kernel difference by
\eqref{eq:heat-scaling}. Thus the locally uniform analytic convergence
\eqref{eq:analytic-convergence} gives the same convergence for that
difference across the time diagonal.

\subsection{Canonical scales}

We first verify that the canonical scales are well defined and determine their asymptotics. Their exact relations simplify the rescaled kernel.
\begin{lemma}\label{lem:scale}
For all sufficiently large $T$, the maximum in
\eqref{eq:canonical-scale} exists and
\begin{equation}\label{eq:exact-scale-relations}
 a_{n_T}\leq L_T<a_{n_T+1},\qquad
 \mathcal N_+\left(L_T\right)=n_T,\qquad L_T^2=Tn_T.
\end{equation}
Moreover,
\begin{equation}\label{eq:scale-asymptotics}
 L_T\sim\left(cT\right)^{1/\left(2-q\right)},\quad
 n_T\sim c^{2/\left(2-q\right)}T^{q/\left(2-q\right)},\quad
 \sigma_T\sim c^{-1/\left(2-q\right)}T^{\left(1-q\right)/\left(2-q\right)}.
\end{equation}
\end{lemma}
\begin{proof}
From \eqref{eq:assumptions},
$a_n^2/n\sim c^{-2/q}n^{2/q-1}\to\infty$ because $q<2$.
Thus the defining set is finite.  It is nonempty for all sufficiently large
$T$, and $n_T\to\infty$.  Maximality gives
\[
 a_{n_T}^2\leq Tn_T=L_T^2,\qquad
 a_{n_T+1}^2>T\left(n_T+1\right)>Tn_T=L_T^2,
\]
which proves \eqref{eq:exact-scale-relations}.  Since
$a_{n+1}/a_n\to1$, the squeeze
$1\leq L_T/a_{n_T}<a_{n_T+1}/a_{n_T}$ shows
$L_T\sim a_{n_T}\sim(n_T/c)^{1/q}$.  Combining this with
$L_T^2=Tn_T$ yields $L_T^{2-q}\sim cT$, and all three relations in
\eqref{eq:scale-asymptotics} follow.
\end{proof}

At the canonical scale, the fixed-contour residue correction vanishes.
It therefore suffices to analyze the contour integral.
\begin{lemma}\label{lem:section7-correction-negligible}
For every compact set $\mathcal B\subset\mathbb R^4$,
\begin{equation}\label{eq:section7-correction-bound}
 \sup_{\left(s,t,X,Y\right)\in\mathcal B}
 \sigma_T\left|
 \Rc_\Gamma
 \left(\left(T_s,\sigma_TX\right),\left(T_t,\sigma_TY\right)\right)
 \right|
 \leq \frac{C_{\mathcal B}}{L_T}
 \longrightarrow0.
\end{equation}
\end{lemma}
\begin{proof}
Lemma~\ref{lem:scale} gives $n_T\to\infty$, and hence
$T_\tau/T=1+\tau/n_T$ is bounded above and away from zero, uniformly for
$\tau$ in a compact set and all sufficiently large $T$. For $|u|\leq1$
and $(s,t,X,Y)\in\mathcal B$, the explicit heat-kernel formula and
$\sigma_T^2/T=1/n_T$ give
\[
 \left|
 \mathfrak h_{T_s}\left(\sigma_TX,\mi u\right)
 \mathfrak h_{T_t}\left(-\mi\sigma_TY,u\right)
 \right|
 \leq \frac{C_{\mathcal B}}{\sqrt{T_sT_t}}
 \leq \frac{C_{\mathcal B}}{T}.
\]
Integrating over $u\in[-1,1]$ and using
$\sigma_T/T=L_T^{-1}$ proves \eqref{eq:section7-correction-bound}.
\end{proof}

\subsection{The logarithmic phase}
Put $n=n_T$, $L=L_T$, and
\[
 b_{j,T}\defeq \frac{a_j}{L},\qquad
 \nu_T\defeq \frac1n\sum_{j\geq1}\delta_{b_{j,T}}.
\]
On the upper half-plane choose the logarithm so that
$\Log(1-z^2/b^2)$ is real for $z=\mi r$, $r>0$.  On the lower half-plane
take the conjugate branch, which is real at $z=-\mi r$, $r>0$.
For the central factor use
$\Log(\pm\mi r)=\log r\pm\mi\pi/2$.  Define
\begin{equation}\label{eq:finite-phase}
 \ell_T\left(z\right)\defeq \frac1n\left\{\eps\Log z+
 \sum_{j\geq1}\Log\left(1-\frac{z^2}{b_{j,T}^2}\right)\right\},
 \qquad F_T\left(z\right)\defeq \frac{z^2}{2}+\ell_T\left(z\right).
\end{equation}
The omitted constant $\eps\log L/n$ cancels from every quotient
$\mathfrak E_{\bxi}(Lw)/\mathfrak E_{\bxi}(Lz)$.  The term $\eps\Log z/n$, although asymptotically
small, is retained in $F_T$ because its derivative contributes to the exact
finite saddle.

The rescaled root measures converge to a power-law measure and have uniform bounds at zero and infinity. These estimates justify convergence of the logarithmic phase and its derivatives.
\begin{lemma}\label{lem:measure-bounds}
There is a constant $C$ such that, for all large $T$ and all $j\geq1$,
\begin{equation}\label{eq:bj-comparison}
 C^{-1}\left(j/n\right)^{1/q}\leq b_{j,T}\leq C\left(j/n\right)^{1/q}.
\end{equation}
Consequently, for $r>0$, $0<\delta<1<R$,
\begin{align}
 \nu_T\left(\left(0,r\right]\right)&\leq Cr^q,\label{eq:count-bound}\\
 \int_{\left(0,\delta\right]}\left(1+\left|\log u\right|\right)\nu_T\left(\dif u\right)
 &\leq C\delta^q\left(1+\left|\log\delta\right|\right),\label{eq:small-bound}\\
 \int_{\left[R,\infty\right)}u^{-2}\nu_T\left(\dif u\right)&\leq CR^{q-2}.
 \label{eq:tail-bound}
\end{align}
Also, for every fixed $r>0$,
\begin{equation}\label{eq:measure-convergence}
 \nu_T\left(\left(0,r\right]\right)=\frac{\mathcal N_+\left(rL\right)}{\mathcal N_+\left(L\right)}\longrightarrow r^q.
\end{equation}
\end{lemma}

\begin{proof}
By Lemma~\ref{lem:scale}, $L\asymp n^{1/q}$.  Assumption
\eqref{eq:assumptions}, enlarged over the finitely many small $j$, therefore
gives \eqref{eq:bj-comparison}.  It implies
$\#\{j:b_{j,T}\leq r\}\leq Cnr^q$, proving
\eqref{eq:count-bound}.  Stieltjes integration by parts applied to this
distribution bound gives \eqref{eq:small-bound}.  Similarly,
\[
 \frac1n\sum_{b_{j,T}\geq R}b_{j,T}^{-2}
 \leq C\int_{cR^q}^{\infty}x^{-2/q}\dif x
 \leq CR^{q-2},
\]
which is \eqref{eq:tail-bound}.  Finally,
\eqref{eq:measure-convergence} follows from
$\mathcal N_+(x)\sim cx^q$ and $L\to\infty$.
\end{proof}

We now identify the limiting logarithmic phase, including its first three derivatives. This determines the limiting saddles and controls the finite-time phase near them.
\begin{lemma}\label{lem:phase-convergence}
For every compact $K\subset\C\setminus\R$ and $k=0,1,2,3$,
\begin{equation}\label{eq:C3-convergence}
 \sup_{z\in K}\left|\ell_T^{\left(k\right)}\left(z\right)-g_q^{\left(k\right)}\left(z\right)\right|\longrightarrow0,
 \qquad
 g_q\left(z\right)\defeq \frac{A_q}{q}\left(-z^2\right)^{q/2}.
\end{equation}
Here $(-z^2)^{q/2}$ is analytic separately in the two
half-planes, with conjugate branches, and is positive
when $z=\pm\mi r$, $r>0$.
\end{lemma}
\begin{proof}
For $z\in K$ write
$f_z(u)=\Log(1-z^2/u^2)$.  On $(0,\delta]$,
$|f_z(u)|\leq C_K(1+|\log u|)$, while each of its first three
$z$-derivatives is bounded uniformly in $u$.  On $[R,\infty)$, $f_z$ and
its first three derivatives are $O_K(u^{-2})$.  Hence
\eqref{eq:small-bound} and \eqref{eq:tail-bound} make both ends uniformly
small.  On $[\delta,R]$, \eqref{eq:measure-convergence} gives weak
convergence to $qu^{q-1}\,\dif  u$; uniformity in $z\in K$ follows from a finite
net and equicontinuity.  The central term $\eps\Log z/n$ and its derivatives
vanish uniformly on $K$.  We have therefore proved, including the first
three derivatives,
\begin{equation}\label{eq:phase-integral}
\ell_T\left(z\right)\longrightarrow\int_0^\infty qu^{q-1}\Log\left(1-z^2/u^2\right)\dif u=g_q\left(z\right).
\end{equation}

It remains to evaluate the integral.  Differentiation, already justified by
the preceding bounds, gives
\[
 g_q'\left(z\right)=-2qz\int_0^\infty\frac{u^{q-1}}{u^2-z^2}\dif u.
\]
For $z=\mi r$, $r>0$, the substitution $u=rv$ and Euler's beta integral give
\[
 \int_0^\infty\frac{u^{q-1}}{u^2+r^2}\dif u
 =r^{q-2}\int_0^\infty\frac{v^{q-1}}{1+v^2}\dif v
 =\frac{\pi r^{q-2}}{2\sin\left(\pi q/2\right)}.
\]
Analytic continuation on each half-plane therefore yields
\[
 g_q'\left(z\right)=A_q\frac{\left(-z^2\right)^{q/2}}{z}.
\]
Integration gives $g_q(z)=A_q(-z^2)^{q/2}/q$.  More explicitly, the
substitution $u=rv$ in \eqref{eq:phase-integral} gives
$g_q(\mi r)=r^q g_q(\mi)\to0$ as $r\downarrow0$, fixing the additive
constant to zero.  This proves
\eqref{eq:C3-convergence}.
\end{proof}

\subsection{The exact finite-time saddles}
For $w>0$ define
\begin{equation*}
 \mathcal Q_T\left(w\right)\defeq \frac{\eps}{nw^2}+\frac2n\sum_{j\geq1}
 \frac{1}{b_{j,T}^2+w^2}.
\end{equation*}

The contours will pass through the exact finite-time saddles. We locate these saddles and show that their curvature converges to a positive limit.
\begin{lemma}\label{lem:exact-saddle}
For every sufficiently large $T$ there is a unique $w_T>0$ such that
$\mathcal Q_T(w_T)=1$.  It satisfies
\begin{equation*}
 F_T'\left(\pm\mi w_T\right)=0,\quad w_T\longrightarrow u_q,\quad
 F_T''\left(\pm\mi w_T\right)\longrightarrow2-q.
\end{equation*}
More precisely,
\begin{equation}\label{eq:finite-second-derivative}
 F_T''\left(\mi w_T\right)=-w_T\mathcal Q_T'\left(w_T\right)
 =\frac{2\eps}{nw_T^2}+\frac{4w_T^2}{n}
 \sum_{j\geq1}\frac{1}{\left(b_{j,T}^2+w_T^2\right)^2}>0.
\end{equation}
If $\eps=0$, $z=0$ is an additional critical point of $F_T$; it is
separated from the dominant saddles by a positive phase gap.
\end{lemma}

\begin{proof}
Direct differentiation of \eqref{eq:finite-phase} gives
\begin{equation}\label{eq:saddle-equation}
 F_T'\left(\mi w\right)=\mi w\left\{1-\mathcal Q_T\left(w\right)\right\}.
\end{equation}
The function $\mathcal Q_T$ is continuous and strictly decreasing, and tends to zero
as $w\to\infty$.  If $\eps=1$, it diverges at zero.  If $\eps=0$, then
$b_{j,T}\leq1$ for $j\leq n$ by \eqref{eq:exact-scale-relations}, so
$\mathcal Q_T(0+)\geq2$.  Thus a unique positive root exists.

Lemma~\ref{lem:phase-convergence}, differentiated once, implies locally
uniformly on $(0,\infty)$ that
\begin{equation}\label{eq:saddle-profile-limit}
 \mathcal Q_T\left(w\right)\longrightarrow\mathcal G_q\left(w\right),
 \qquad \mathcal G_q\left(w\right)\defeq A_qw^{q-2}.
\end{equation}
The latter is strictly decreasing, and the equation
$\mathcal G_q(w)=1$ has the unique solution $w=u_q$. Monotonicity and \eqref{eq:saddle-profile-limit} give
$w_T\to u_q$.  Differentiating \eqref{eq:saddle-equation} at the root
gives \eqref{eq:finite-second-derivative}; convergence to $2-q$ follows
from the $C^2$ part of Lemma~\ref{lem:phase-convergence}.  When $\eps=0$,
$F_T$ is analytic at zero and $F_T'(0)=0$, but
$F_T''(0)=1-\mathcal Q_T(0+)<0$.  The vertical phase increases from zero toward
$\mi w_T$, and the quantitative gap is proved in
Lemma~\ref{lem:global-geometry} below.
\end{proof}

\subsection{Uniform phase geometry}
Put
\begin{equation*}
 H_q\left(z\right)\defeq \frac{z^2}{2}+g_q\left(z\right)
 =\frac{z^2}{2}+\frac{A_q}{q}\left(-z^2\right)^{q/2},
 \quad m_q\defeq \Re H_q\left(\mi u_q\right)
 =\frac{2-q}{2q}\,u_q^2>0.
\end{equation*}

The limiting phase has strict maxima at the saddles on the vertical contour and strict minima there on the horizontal contours. These opposite inequalities give decay of the corresponding exponential factors and will control the finite-time contours.
\begin{lemma}\label{lem:limit-geometry}
Along the imaginary axis, $\Re H_q(\mi r)$ has exactly two global
maxima, at $r=\pm u_q$.  Along either horizontal line through a
saddle,
\begin{equation}\label{eq:horizontal-strict}
 \Re H_q\left(x\pm\mi u_q\right)>m_q\quad\left(x\neq0\right).
\end{equation}
The inequalities are quadratic near the saddle and on the tails.
\end{lemma}
\begin{proof}
On the imaginary axis,
\[
 \Re H_q\left(\mi r\right)=-\frac{r^2}{2}+\frac{A_q}{q}\left|r\right|^q.
\]
Its derivative on $(0,\infty)$ is
$r\{A_qr^{q-2}-1\}$, so it increases up to $u_q$ and then
decreases.  Evenness gives the second maximum.

For the horizontal direction put $v=x/u_q$.  Since
$A_qu_q^{q-2}=1$,
\begin{equation*}
 \Psi_q\left(v\right)\defeq\frac{v^2}{2}
 +\frac{\Re\left(1-\mi v\right)^q-1}{q}
 =\frac{\Re H_q\left(x+\mi u_q\right)-m_q}{u_q^2}.
\end{equation*}
The function is even.  For $v>0$, with $\varphi=\arctan v$,
\begin{equation*}
 \Psi_q'\left(v\right)=v-\left(1+v^2\right)^{\left(q-1\right)/2}
 \sin\left\{\left(q-1\right)\varphi\right\}.
\end{equation*}
If $q\leq1$, the second term displayed with the minus sign is
nonnegative.  If $1<q<2$, then
\[
 \left(1+v^2\right)^{\left(q-1\right)/2}\sin\left\{\left(q-1\right)\varphi\right\}
 <\sec\varphi\sin\varphi=v,
\]
because $0<(q-1)\varphi<\varphi<\pi/2$ and
$\sec^{q-1}\varphi<\sec\varphi$.  Hence $\Psi_q'(v)>0$ and
\eqref{eq:horizontal-strict} follows.  The quadratic behavior near zero
comes from $H_q''(\pm\mi u_q)=2-q$.  For large $|r|$ or $|x|$, the
quadratic term dominates the term of order $q<2$.
\end{proof}

Local phase convergence does not control the unbounded contour tails.
The following product bounds supply the required growth estimates. Put
\[
 \underline w\defeq u_q/2,\qquad \overline w\defeq 3u_q/2.
\]
\begin{lemma}\label{lem:product-bounds}
There are $T_0,C>0$ such that, for $T\ge T_0$,
\begin{equation}\label{eq:saddle-bracket}
 0<\underline{w}<w_T<\overline{w},
\end{equation}
and the following bounds hold.

\begin{enumerate}[label=(\roman*)]
\item For every $r\in\R\setminus\{0\}$,
\begin{equation}\label{eq:vertical-product}
 \Re\ell_T\left(\mi r\right)\leq C\left|r\right|^q.
\end{equation}

\item Uniformly for $\underline{w}\leq|v|\leq \overline{w}$ and $z=x+\mi v$,
\begin{equation*}
 \Re\ell_T\left(z\right)\geq-C\left(2+\left|x\right|\right)^q\log\left(2+\left|x\right|\right).
\end{equation*}

\item The same lower bound, with $|x|$ in place of $2+|x|$, holds for
$|x|$ large in the region used for the closing vertical segments
\begin{equation*}
 \underline{w}\leq\left|v\right|\leq \left|x\right|/\sqrt3.
\end{equation*}
\end{enumerate}
\end{lemma}
\begin{proof}
From \eqref{eq:bj-comparison}, and because the summand below is decreasing,
\begin{align*}
 \frac1n\sum_{j\geq1}\log\left(1+\frac{r^2}{b_{j,T}^2}\right)
 &\leq \frac1n\sum_{j\geq1}
 \log\left\{1+Cr^2\left(n/j\right)^{2/q}\right\}\leq \int_0^\infty\log\left(1+Cr^2x^{-2/q}\right)\dif x
 =C_q\left|r\right|^q.
\end{align*}
For $|r|\leq1$, the central contribution
$\eps\log|r|/n$ is nonpositive; for $|r|>1$ it is absorbed by
$C|r|^q$.  This proves \eqref{eq:vertical-product}.

Since $w_T\to u_q$, increasing $T_0$ if necessary gives
\eqref{eq:saddle-bracket} with the displayed choices of
$\underline{w}$ and $\overline{w}$.

For the lower bound put $U=2+|x|$.  In case (ii), $|z|\leq C_0U$; the
same is true in the wedge after replacing $U$ by $|x|$.  Split the product
at $b_{j,T}=4C_0U$.  For $b\leq4C_0U$, the imaginary part keeps both
factors away from zero and
\[
 \log\left|1-\frac{z^2}{b^2}\right|
 =\log\left|b-z\right|+\log\left|b+z\right|-2\log b\geq-C\log U.
\]
There are at most $CnU^q$ such factors by
\eqref{eq:count-bound}.  For $b>4C_0U$, $|z|/b\leq1/4$ and
\[
 \log\left|1-z^2/b^2\right|\geq-C\left|z\right|^2/b^2.
\]
Using \eqref{eq:tail-bound}, the normalized sum of these terms is bounded
below by $-CU^2U^{q-2}=-CU^q$.  Finally,
$\eps\log|z|/n$ is bounded below uniformly because $|\Im z|\geq \underline{w}$.
The two lower bounds follow.
\end{proof}
Let $ M_T=\Re F_T(\mi w_T)=\Re F_T(-\mi w_T).$ By Lemmas~\ref{lem:phase-convergence} and \ref{lem:exact-saddle},$M_T\to m_q$.

Combining the limiting geometry with the product estimates gives global quadratic bounds for the finite-time phase. These bounds control both the contour deformation and its remainder.
\begin{lemma}\label{lem:global-geometry}
There are $T_0$ and $c_0>0$ such that for $T\geq T_0$, all $r,x\in\R$,
and $\varsigma\in\{-1,1\}$,
\begin{align}
 &M_T-\Re F_T\left(\mi r\right)
 \geq c_0\dist\left(r,\left\{-w_T,w_T\right\}\right)^2,
 \label{eq:global-vertical}\\
 &\Re F_T\left(x+\mi\varsigma w_T\right)-M_T\geq c_0x^2.
 \label{eq:global-horizontal}
\end{align}
When $\eps=1$, the left side of \eqref{eq:global-vertical} is interpreted
as $+\infty$ at $r=0$.
\end{lemma}
\begin{proof}
Write $V_T(r)=\Re F_T(\mi r)$.  From
\eqref{eq:saddle-equation}, for $r>0$,
\begin{equation}\label{eq:vertical-derivative}
 V_T'\left(r\right)=r\left\{\mathcal Q_T\left(r\right)-1\right\}.
\end{equation}
Thus $V_T$ increases on $(0,w_T)$ and decreases on $(w_T,\infty)$;
the negative half-axis is symmetric.  The $C^3$ convergence near the two
saddles and \eqref{eq:finite-second-derivative} give, for fixed small
$\delta>0$,
\begin{equation*}
 M_T-V_T\left(\varsigma w_T+r\right)\geq c r^2,\quad
 \Re F_T\left(x+\mi\varsigma w_T\right)-M_T\geq cx^2
\end{equation*}
whenever $|r|,|x|\leq\delta$.

On a fixed compact set outside these neighborhoods, the vertical strictness
follows from \eqref{eq:vertical-derivative} and the convergence
$\mathcal Q_T\to \mathcal G_q$.  A small neighborhood of zero is not covered by compact
off-real convergence.  There, \eqref{eq:vertical-product} gives
$V_T(r)\leq-r^2/2+C|r|^q$; alternatively, monotonicity bounds it by its
value at the endpoint of that neighborhood.  Since $M_T\to m_q>0$, the
gap remains positive after the neighborhood is chosen sufficiently small.
The horizontal compact gap follows from Lemma~\ref{lem:limit-geometry},
Lemma~\ref{lem:phase-convergence}, and $w_T\to u_q$.
On these compact regions, the quantities
$\dist(r,\{-w_T,w_T\})^2$ and $x^2$ are bounded.  Dividing the uniform
positive gaps by their respective upper bounds therefore turns the compact
estimates into the quadratic form required in
\eqref{eq:global-vertical}--\eqref{eq:global-horizontal}.

For large $|r|$, \eqref{eq:vertical-product} yields
\[
 V_T\left(r\right)\leq-r^2/2+C\left|r\right|^q\leq-r^2/4.
\]
For $z=x+\mi\varsigma w_T$, Lemma~\ref{lem:product-bounds} and
$q<2$ give
\[
 \Re F_T\left(z\right)\geq\frac{x^2-w_T^2}{2}
 -C\left(2+\left|x\right|\right)^q\log\left(2+\left|x\right|\right)\geq x^2/4
\]
for large $|x|$.  Combining the local quadratic estimates, the positive
compact gaps, and these tail estimates, and decreasing the constant if
necessary, proves \eqref{eq:global-vertical}--\eqref{eq:global-horizontal}.
\end{proof}
\begin{remark}
Qualitative phase convergence only gives $F_T'(\mi u_q)=o(1)$.
A contour centered at the limiting saddle would require
$F_T'(\mi u_q)=o(n_T^{-1/2})$ to control its linear term on the
Gaussian scale $n_T^{-1/2}$. Passing through the exact saddle
$\mi w_T$, where the derivative vanishes, avoids this rate condition.
\end{remark}
\subsection{Equal-time contour deformation}
For fixed $X,Y\in\R$, scaling $z=L\zeta$, $w=L\omega$ in
\eqref{eq:kernel-input} with $S=U=T$ gives the exact identity
\begin{equation}\label{eq:scaled-equal-kernel}
 \sigma_T\BL_{\bxi,\Gamma}\left(\left(T,\sigma_TX\right),\left(T,\sigma_TY\right)\right)
 =\frac{\me^{c_T\left(X,Y\right)}}{\left(2\pi\mi\right)^2}
 \int_{\Gamma/L}\!\dif\zeta\int_{\mi\R}\!\dif\omega\,
 \frac{\me^{n\left\{F_T\left(\omega\right)-F_T\left(\zeta\right)\right\}-Y\omega+X\zeta}}
 {\omega-\zeta},
\end{equation}
where $ c_T(X,Y)=\frac{Y^2-X^2}{2n}$. Indeed, the prefactor is $\sigma_TL^2/(TL)=\sigma_TL/T=1$, while
$L^2/T=n$, $L\sigma_T/T=1$, and $\sigma_T^2/T=1/n$.

Let $\mathcal{C}_T^+=\R+\mi w_T$ and $\mathcal{C}_T^-=\R-\mi w_T,$ with $\mathcal{C}_T^+$ oriented from right to left and $\mathcal{C}_T^-$ from left to right.
Put $\delta_T=L^{-1}$ and, for large $T$ (so $\delta_T<w_T$),
\begin{equation*}
 I_T\defeq \left[-w_T,-\delta_T\right]\cup\left[\delta_T,w_T\right].
\end{equation*}

Both the original and deformed integrals use the convention of
Definition~\ref{def def of K} at transverse intersections.
We deform the equal-time kernel onto the horizontal lines through the exact saddles. This separates the crossed residue from a double integral that will vanish in the limit.
\begin{proposition}\label{prop:deformation}
The right side of \eqref{eq:scaled-equal-kernel} equals
\begin{equation*}
 R_T\left(X,Y\right)+J_T\left(X,Y\right),
\end{equation*}
where
\begin{align}
 R_T\left(X,Y\right)&\defeq \frac{\me^{c_T\left(X,Y\right)}}{2\pi}
 \int_{I_T}\me^{\mi k\left(X-Y\right)}\dif k,
 \label{eq:finite-residue}\\
 J_T\left(X,Y\right)&\defeq \frac{\me^{c_T\left(X,Y\right)}}{\left(2\pi\mi\right)^2}
 \sum_{\varsigma=\pm1}\int_{\mathcal{C}_T^\varsigma}\!\dif\zeta
 \int_{\mi\R}\!\dif\omega\,
 \frac{\me^{n\left\{F_T\left(\omega\right)-F_T\left(\zeta\right)\right\}-Y\omega+X\zeta}}
 {\omega-\zeta}.
 \label{eq:deformed-remainder}
\end{align}
\end{proposition}
\begin{proof}
The upper component of $\Gamma/L$ crosses the imaginary axis at
$\mi \delta_T$, and the lower component at $-\mi \delta_T$.  Truncate each old and
new component at $\Re z=\pm R$.  In the upper and lower half-planes,
respectively, join the endpoints by vertical segments.  No zero of
$\mathfrak E_{\bxi}(Lz)$ is crossed because all zeros are real.

We first justify sending $R$ to infinity.  On a joining segment,
$|\Re z|=R$ and $\underline{w}\leq|\Im z|\leq R/\sqrt3$, so
\[
 \Re\left(z^2/2\right)\geq R^2/3,\quad
 \Re\ell_T\left(z\right)\geq-CR^q\log R
\]
by Lemma~\ref{lem:product-bounds}.  Hence
$\Re F_T(z)\geq R^2/4$ for large $R$, uniformly in large $T$.  The same
estimate holds on the sloped tails of the old contour.  Together with
\eqref{eq:global-vertical}, it bounds the absolute value of the closing
double integrals by a constant times
\[
 \me^{-cnR^2+C R}\int_\R \me^{-cn\dist\left(r,\left\{-w_T,w_T\right\}\right)^2}\dif r,
\]
Here $M_T$ is uniformly bounded. After choosing $R$ so large that
$M_T-R^2/4\leq-cR^2$ uniformly in large $T$, the term $nM_T$ is absorbed
into the negative quadratic contribution. The
displayed bound tends to zero as $R\to\infty$.  This also proves convergence of all
unbounded contour integrals used in the deformation.

For fixed $\omega=\mi r$ with
$r\notin\{-w_T,-\delta_T,\delta_T,w_T\}$, compare the old and new contours
by their winding numbers; this does not require their unbounded enclosed
regions to be nested. The new winding number about $\omega$ is one precisely
when $|r|<w_T$, whereas the old winding number is one precisely when
$|r|<\delta_T$. Thus their winding-number difference is supported exactly on
$r\in I_T$, where
\[
 \operatorname{Ind}_{\mathrm{old}}\left(\omega\right)
 -\operatorname{Ind}_{\mathrm{new}}\left(\omega\right)=-1.
\]
The four excluded endpoint values form a Lebesgue-null set and do not affect
the residue integral.
Since
\[
 \operatorname*{Res}_{\zeta=\omega}\frac{1}{\omega-\zeta}=-1,
\]
the counterclockwise orientations give
``old = new + crossed residue.''  At the pole the large phase cancels and
the remaining exponential is $\me^{(X-Y)\omega+c_T}$.  Writing
$\omega=\mi k$ produces \eqref{eq:finite-residue} with a positive sign.

Finally, near a same-saddle intersection write
$\omega=\mi(w_T+r)$ and $\zeta=x+\mi w_T$.  Then
$|\omega-\zeta|=(x^2+r^2)^{1/2}$, whose reciprocal is integrable in two
dimensions.  This proves the stated improper-integral interpretation and
completes the deformation.
\end{proof}

The global phase bounds show that the deformed double integral is of order $n_T^{-1/2}$. The residue therefore determines the limiting kernel.
\begin{proposition}
\label{prop:remainder}
For every compact $B\subset\R^2$,
\begin{equation}\label{eq:remainder-rate}
 \sup_{\left(X,Y\right)\in B}\left|J_T\left(X,Y\right)\right|=O\left(n_T^{-1/2}\right).
\end{equation}
\end{proposition}

\begin{proof}
On the branch $\mathcal{C}_T^\varsigma$, write
\[
 \omega=\mi r,\quad \zeta=x+\mi\varsigma w_T,\quad
 \Delta_T^{\mathrm v}\left(r\right)\defeq M_T-\Re F_T\left(\mi r\right),\quad
 \Delta_{T,\varsigma}^{\mathrm h}\left(x\right)
 \defeq \Re F_T\left(x+\mi\varsigma w_T\right)-M_T.
\]
For $X,Y$ in a fixed compact set, the absolute value of the integrand in
\eqref{eq:deformed-remainder} is bounded by
\begin{equation*}
 C\frac{\exp\left\{-n\left[\Delta_T^{\mathrm v}\left(r\right)
 +\Delta_{T,\varsigma}^{\mathrm h}\left(x\right)\right]+C\left|x\right|\right\}}
 {\left\{x^2+\left(r-\varsigma w_T\right)^2\right\}^{1/2}}.
\end{equation*}
By Lemma~\ref{lem:global-geometry},
\begin{equation}\label{eq:AB-bound}
 \Delta_T^{\mathrm v}\left(r\right)\geq c\dist\left(r,\left\{-w_T,w_T\right\}\right)^2,
 \quad \Delta_{T,\varsigma}^{\mathrm h}\left(x\right)\geq cx^2.
\end{equation}

For each branch $\varsigma$, split the vertical variable into the
half-lines $\varsigma r\ge0$ and $\varsigma r<0$. On the first, the
nearest saddle is $\varsigma w_T$. With $s=r-\varsigma w_T$, its
contribution is bounded by
\[
 C\int_{\mathbb R^2}
 \frac{\me^{-cn\left(x^2+s^2\right)+C\left|x\right|}}
 {\sqrt{x^2+s^2}}\dif x\dif s
 =O\left(n^{-1/2}\right).
\]
Indeed, completing the square bounds the exponential by
$C\me^{-cn(x^2+s^2)/2}$, and polar coordinates give the stated order.
On the other half-line, $|r-\varsigma w_T|\ge w_T\ge\underline w$,
while the nearest saddle is $-\varsigma w_T$. The two Gaussian
integrations therefore give $O(n^{-1})$. Summing over the branches
proves \eqref{eq:remainder-rate}.
\end{proof}

The residue now gives the equal-time sine kernel. The vanishing correction transfers this limit to the full kernel.
\begin{proposition}\label{prop:equal-time}
The convergence \eqref{eq:equal-time-limit} holds locally uniformly in $(X,Y)\in\R^2$.
\end{proposition}

\begin{proof}
Propositions~\ref{prop:deformation} and \ref{prop:remainder} show that the
scaled contour integral differs locally uniformly from its residue by a
quantity tending to zero. If $d=X-Y$, \eqref{eq:finite-residue} is
\begin{equation*}
 R_T\left(X,Y\right)=\me^{c_T\left(X,Y\right)}
 \frac{\sin\left(w_Td\right)-\sin\left(\delta_Td\right)}{\pi d},
\end{equation*}
with continuous value
$\me^{c_T(X,X)}(w_T-\delta_T)/\pi$ at $d=0$.  Since
$w_T\to u_q$, $\delta_T=L_T^{-1}\to0$, and $c_T\to0$ locally uniformly,
the right side converges locally uniformly to the sine kernel in
\eqref{eq:equal-time-limit}. Thus the same limit holds for
$\sigma_T\BL_{\bxi,\Gamma}$. By
Lemma~\ref{lem:section7-correction-negligible}, replacing the contour
integral by the analytic part $\BL_{\bxi}$ changes the scaled kernel by a locally
uniform $o(1)$. At equal times the heat term is absent, so
$\K_{\bxi}=\BL_{\bxi}$, which proves the proposition.
\end{proof}

\subsection{Convergence at several times}
We now prove the analytic-kernel limit at several times. Fix $s,t\in\R$ and set
\begin{equation*}
 T_s=T+\sigma_T^2s,\quad T_t=T+\sigma_T^2t,\quad
 \lambda_{T,\tau}\defeq \frac{T}{T_\tau}
 =\left(1+\frac{\tau}{n}\right)^{-1}.
\end{equation*}
For $\tau$ in a compact set these quantities are well defined for all
large $T$ and $\lambda_{T,\tau}\to1$ uniformly.

We next write the scaled analytic kernel at several times as a perturbation of the equal-time phase. The exact formula isolates the time-dependent terms that must be controlled.
\begin{lemma}\label{lem:multitime-formula}
With $\lambda_s=\lambda_{T,s}$ and $\lambda_t=\lambda_{T,t}$,
\begin{equation*}\begin{aligned}
 &\sigma_T\BL_{\bxi,\Gamma}\left(\left(T_s,\sigma_TX\right),\left(T_t,\sigma_TY\right)\right)=\frac{\sqrt{\lambda_s\lambda_t}}{\left(2\pi\mi\right)^2}
 \exp\left\{\frac{\lambda_tY^2-\lambda_sX^2}{2n}\right\}
 \int_{\Gamma/L}\!\dif\zeta\int_{\mi\R}\!\dif\omega\,
 \frac{\me^{\mathcal S_T\left(\zeta,\omega\right)}}{\omega-\zeta}.
\end{aligned}\end{equation*}
where
\begin{equation}\label{eq:multitime-action}
 \mathcal S_T\left(\zeta,\omega\right)
 \defeq n\left\{F_T\left(\omega\right)-F_T\left(\zeta\right)\right\}
 -\frac{t\lambda_t}{2}\omega^2
 +\frac{s\lambda_s}{2}\zeta^2
 -\lambda_tY\omega+\lambda_sX\zeta.
\end{equation}
\end{lemma}
\begin{proof}
After $z=L\zeta$, $w=L\omega$, the prefactor in
\eqref{eq:kernel-input}, including the outer $\sigma_T$, is
\[
 \frac{\sigma_TL}{\sqrt{T_sT_t}}
 =\frac{T}{\sqrt{T_sT_t}}=\sqrt{\lambda_s\lambda_t}.
\]
The two quadratic exponents become
\begin{align*}
 \frac{\left(L\omega-\sigma_TY\right)^2}{2T_t}
 -\frac{\left(L\zeta-\sigma_TX\right)^2}{2T_s}
 &=\frac n2\left(\lambda_t\omega^2-\lambda_s\zeta^2\right)
 -\lambda_tY\omega+\lambda_sX\zeta+\frac{\lambda_tY^2-\lambda_sX^2}{2n}.
\end{align*}
Since
\begin{equation*}
 n\left(\lambda_{T,\tau}-1\right)=-\tau\lambda_{T,\tau},
\end{equation*}
combining this expression with
$n\{\ell_T(\omega)-\ell_T(\zeta)\}$ gives exactly
\eqref{eq:multitime-action}.
\end{proof}

On compact sets of rescaled times, these additional terms are absorbed by the phase estimates. The same deformation then gives the several-time analytic limit.
\begin{proposition}\label{prop:multitime}
The convergence \eqref{eq:analytic-convergence} holds locally uniformly in
$(s,t,X,Y)\in\mathbb{R}^4$.
\end{proposition}
\begin{proof}
Fix a compact set $\mathcal{B}\subset\mathbb{R}^{4}$ of parameters
$(s,t,X,Y)$. All constants below may depend on $\mathcal{B}$, but they
are independent of $T$ and of the point in $\mathcal{B}$.
Since $\lambda_{T,\tau}\to 1$ uniformly for $\tau$ in compact sets and
$w_T\to u_q$, after increasing $T_0$ we may assume
\[
\frac{1}{2}\leq \lambda_s,\lambda_t\leq 2,\quad\underline{w}\leq w_T\leq \overline{w},
\]
uniformly for $(s,t,X,Y)\in\mathcal{B}$.\\
Recall from Lemma~\ref{lem:multitime-formula} that the exponent in the scaled double integral is
\[
\mathcal S_T\left(\zeta,\omega\right)
=n\left\{F_T\left(\omega\right)-F_T\left(\zeta\right)\right\}+P_T\left(\zeta,\omega\right),
\]
where
\[
P_T\left(\zeta,\omega\right)
\defeq -\frac{t\lambda_t}{2}\omega^2+\frac{s\lambda_s}{2}\zeta^2-\lambda_tY\omega+\lambda_sX\zeta.
\]

We first show that the additional term $P_T$ does not destroy the
steepest-descent estimates. On the final contours, write
\[
\omega=\mi r,\quad\zeta=x+\mi \varsigma w_T,
\quad r,x\in\mathbb{R},\quad\varsigma\in\left\{-1,1\right\}.
\]
Put $d_T(r)=\dist(r,\{-w_T,w_T\})$. Since
$r^2\le2d_T(r)^2+2\overline w^2$ and $|x|\le1+x^2$, boundedness of
the parameters gives
\[
 \Re P_T\left(x+\mi\varsigma w_T,\mi r\right)
 \le C\left\{1+r^2+x^2+\left|x\right|\right\}
 \le C'\left\{1+d_T\left(r\right)^2+x^2\right\}.
\]
On the other hand, Lemma~\ref{lem:global-geometry} gives
\[
\Re F_T\left(\mi r\right)-\Re F_T\left(x+\mi \varsigma w_T\right)\leq -c_0\left\{d_T\left(r\right)^2+x^2\right\}.
\]
Consequently, for all sufficiently large $T$,
\begin{equation}\label{eq:multitime-coercivity}
\Re\mathcal S_T\left(x+\mi \varsigma w_T,\mi r\right)
\leq
C-\frac{c_0n}{2}\left\{d_T\left(r\right)^2+x^2\right\},
\end{equation}
uniformly for $(s,t,X,Y)\in\mathcal{B}$.

We next justify the contour deformation. On a closing connector
$\zeta=\pm R+\mi v$ in the wedge used in Proposition~\ref{prop:deformation}, Lemma~\ref{lem:product-bounds} gives,
for all sufficiently large $R$, $\Re F_T(\zeta)\geq {R^2}/{4}.$
Moreover, Lemma~\ref{lem:global-geometry} gives $\Re F_T(\mi r)\leq M_T-c_0d_T(r)^2.$ The sequence $M_T$ is bounded, and on the connector
\[
\Re P_T\left(\zeta,\mi r\right)\leq C\left\{1+R^2+r^2\right\}.
\]
Using $r^2\leq 2d_T(r)^2+2\overline{w}^{\,2},$ we see that, for sufficiently large $T$, the terms of order $R^2+r^2$
coming from $P_T$ are absorbed by the factor $n$ multiplying the
finite-time phase. After first choosing $R$ sufficiently large, the
real part of the total exponent on the closing connectors is bounded above by the quantity $-cn\{R^2+d_T(r)^2\}.$
The same argument applies to the sloped tails of the original contour.
Thus the closing integrals tend to zero as $R\to\infty$, uniformly for
$(s,t,X,Y)\in\mathcal{B}$.

Therefore, the winding-number deformation in Proposition~\ref{prop:deformation} remains
valid with $\mathcal S_T$ in place of the equal-time exponent. It gives the
decomposition
\[
\sigma_T
 \BL_{\bxi,\Gamma}\left(\left(T_s,\sigma_TX\right),\left(T_t,\sigma_TY\right)\right)
=
R_{T;s,t}\left(X,Y\right)+J_{T;s,t}\left(X,Y\right),
\]
where $R_{T;s,t}$ is the crossed residue and
\[
J_{T;s,t}\left(X,Y\right)
\defeq 
\frac{\sqrt{\lambda_s\lambda_t}}{\left(2\pi \mi\right)^2}
\exp\left\{
\frac{\lambda_tY^2-\lambda_sX^2}{2n}
\right\}
\sum_{\varsigma=\pm1}
\int_{\mathcal{C}_T^\varsigma}\dif \zeta
\int_{\mi\mathbb{R}}\dif\omega\,
 \frac{\me^{\mathcal S_T\left(\zeta,\omega\right)}}{\omega-\zeta}.
\]
As before, the integrals at the transverse intersections are understood
as two-dimensional improper integrals.

Equation~\eqref{eq:multitime-coercivity} gives the Gaussian bound
used in Proposition~\ref{prop:remainder}, uniformly on $\mathcal B$.
The same split into $\varsigma r\ge0$ and $\varsigma r<0$ therefore gives
\[
 \sup_{\left(s,t,X,Y\right)\in\mathcal B}
 \left|J_{T;s,t}\left(X,Y\right)\right|
 =O\left(n_T^{-1/2}\right).
\]

It remains to calculate the residue. At the pole
$\zeta=\omega=\mi k$, the dominant phase cancels, and
\[
P_T\left(\mi k,\mi k\right)
=
\frac{t\lambda_t-s\lambda_s}{2}k^2
+\mi\left(\lambda_sX-\lambda_tY\right)k.
\]
The identity $t\lambda_t-s\lambda_s
=
(t-s)\lambda_s\lambda_t$ is exact. Since $I_T$ is symmetric, replacing $k$ by $-k$ gives
\[
R_{T;s,t}\left(X,Y\right)
\defeq
\frac{\sqrt{\lambda_s\lambda_t}}{2\pi}\exp\left\{\frac{\lambda_tY^2-\lambda_sX^2}{2n}\right\}
\int_{I_T}\exp\left\{\frac{t-s}{2}\lambda_s\lambda_t k^2
+\mi\left(\lambda_tY-\lambda_sX\right)k
\right\}
\dif k.
\]
Uniformly for $(s,t,X,Y)\in\mathcal B$, the two factors
$\lambda_s,\lambda_t$ and the exponential prefactor tend to one.
Moreover, $w_T\to u_q$ and $\delta_T\to0$. Thus the residue integrands
converge uniformly on a common bounded interval, and the omitted interval
$[-\delta_T,\delta_T]$ has length tending to zero. Consequently,
\[
R_{T;s,t}\left(X,Y\right)
\longrightarrow
\frac{1}{2\pi}
\int_{-u_q}^{u_q}
\exp\left\{
\frac{t-s}{2}k^2+\mi k\left(Y-X\right)
\right\}
\dif k
\]
locally uniformly in $(s,t,X,Y)$.

Since the deformed double integral tends to zero uniformly on
$\mathcal{B}$, the contour integral $\sigma_T\BL_{\bxi,\Gamma}$ has the
required locally uniform limit. Lemma~\ref{lem:section7-correction-negligible}
shows that its difference from the analytic part
$\sigma_T\BL_{\bxi}$ tends to zero locally uniformly. This proves
\eqref{eq:analytic-convergence}.
\end{proof}
We now combine these estimates to prove Theorem~\ref{thm:main}.
\begin{proof}[Proof of Theorem~\ref{thm:main}]
Proposition~\ref{prop:multitime} gives the analytic convergence. For $s>t$,
Brownian scaling gives
\begin{equation}\label{eq:heat-scaling}
 \sigma_T\mathfrak{h}_{T_s-T_t}\left(\sigma_TX,\sigma_TY\right)
 =\sigma_T\mathfrak{h}_{\sigma_T^2\left(s-t\right)}\left(\sigma_TX,\sigma_TY\right)
 =\mathfrak{h}_{s-t}\left(X,Y\right).
\end{equation}
The indicators also agree because $T_s>T_t$ exactly when $s>t$.
Combining \eqref{eq:full-kernel}, \eqref{eq:analytic-convergence}, and \eqref{eq:heat-scaling} proves the
stated locally uniform convergence of the full-kernel difference on $\mathbb R^4$. 

Let $\mathcal A\subset\mathbb R$ be a finite set of rescaled times and define
\begin{equation}\label{eq:rescaled-long-time-kernel}
 \K_T\left(\left(s,X\right),\left(t,Y\right)\right)
 \defeq \sigma_T\K_{\bxi}
   \left(\left(T_s,\sigma_TX\right),\left(T_t,\sigma_TY\right)\right).
\end{equation}
By spatial change of variables, this is the extended kernel of the rescaled
process $\bXi_{\bxi,T}$; the factor $\sigma_T$ is the spatial
Jacobian. The full-kernel convergence just proved, restricted to the
finite set $\mathcal A^2$, gives, for every compact $B\subset\mathbb R$,
\[
 \max_{s,t\in\mathcal A}\sup_{X,Y\in B}
 \left|
 \K_T\left(\left(s,X\right),\left(t,Y\right)\right)
 -\K_{\Sine,\rho_q}\left(\left(s,X\right),\left(t,Y\right)\right)
 \right|
 \longrightarrow0.
\]
Proposition~\ref{prop mgf convergence}, applied along any sequence
$T\to\infty$, now gives the asserted finite-dimensional convergence in the
vague topology. This completes the proof.
\end{proof}
\section{Infinite-dimensional SDEs}\label{sec:isde}
\subsection{The ISDE}
Proposition \ref{gammadependence} shows that two finite approximations may
converge vaguely to the same configuration while producing limits that differ
by a deterministic linear translation.  In this section we identify the same
effect directly from the infinite-dimensional SDE, for every $\beta\geq1$.
We use a common labelled state space in which finite systems are padded by points at infinity.
\begin{definition}
Define
\begin{equation*}
    \mathcal{W}\defeq  \left\{\bx \in\left(\mathbb{R}\cup\left\{\pm\infty\right\}\right)^\mathbb{Z}:
    x_i<x_{i+1}\ \textnormal{or}\ x_i=x_{i+1}\in\left\{\pm\infty\right\},\ \forall i\in\mathbb{Z}\right\}.
\end{equation*}
We call $\bx$ finite if only finitely many coordinates are real.  We give
$\mathcal W$ the product topology; thus
$\bx^\UN\to\bx$ in $\mathcal W$ means $x_i^\UN\to x_i$ for every
$i\in\mathbb Z$.
\end{definition}
Finite Weyl chambers are embedded in $\mathcal W$ by adjoining $-\infty$
on the left and $+\infty$ on the right.  We use the conventions
$1/(+\infty)=1/(-\infty)=0$. In cumulative-gap sums, $(a,b)$ denotes the open interval between the
endpoints, irrespective of their order.
The approximation-dependent translation is measured by a label-symmetric reciprocal tail.
\begin{definition}
Fix $\bx\in\mathcal W$ and $b>0$ with $x_i\neq\pm b$ for all $i$.
Whenever the following label-symmetric limit exists, define
\begin{equation}\label{eq def of gamma1b}
    \gamma_{1,b}\left(\bx\right)\defeq\lim_{k\to\infty}\sum_{b<\left|x_i\right|,\left|i\right|\leq k}1/x_i.
\end{equation}
\end{definition}
Gap coordinates remove common translations and carry the comparison order used below.
\begin{definition}
Set $\mathbb V=\frac12+\mathbb Z$.  Let $\mathcal L$ be the set of
$\bl\in(0,\infty]^{\mathbb V}$ whose finite-coordinate set is an interval
(possibly empty or infinite).  We call $\bl$ finite when this set is finite,
and order $\mathcal L$ coordinatewise in the extended order.  Define
$\mathfrak q\colon\mathcal W\to\mathcal L$ by
    \begin{equation}\label{eq label to gap}
    \mathfrak{q}\left(\bx\right)_v \defeq
    \begin{cases}
        \left(x_{v+1/2}-x_{v-1/2}\right),\quad &x_{v+1/2},x_{v-1/2}\in\mathbb{R},\\
        \infty,&\textnormal{otherwise}.
    \end{cases}
    \end{equation}
\end{definition}
For a nonempty finite $A\subset\mathbb V$, $r\geq1$, and a positive
gap vector $\mathbf z$, write
\begin{equation*}
 \overline\Sigma_A^r\left(\mathbf z\right)\defeq \frac1{\left|A\right|}\sum_{v\in A}z_v^r,
 \qquad \overline\Sigma_A\left(\mathbf z\right)\defeq \overline\Sigma_A^1\left(\mathbf z\right),
\end{equation*}
and put $\overline\Sigma_\emptyset^r=0$.
We use Tsai's regularity class to control the asymptotic mean gap and averaged moments.
\begin{definition}
Fix $\ka\in(0,1)$, $\rho>0$, and $p>1$.  Define
    \begin{equation*}
    \begin{aligned}
    \mathcal L_{\kappa,\rho}^p
    &\defeq \left\{\bl\in\left(0,\infty\right)^{\mathbb V}:
      \mathfrak D_{\kappa,\rho}\left(\bl\right)<\infty,\ \mathfrak M_p\left(\bl\right)<\infty\right\},\\
    \mathfrak D_{\kappa,\rho}\left(\bl\right)
    &\defeq \sup_{m\in\mathbb Z\setminus\left\{0\right\}}
      \left|m\right|^\kappa\left|\overline\Sigma_{\left(0,m\right)\cap\mathbb V}\left(\bl\right)-\rho\right|,\\
    \mathfrak M_p\left(\bl\right)
    &\defeq \sup_{m\in\mathbb Z\setminus\left\{0\right\}}
      \overline\Sigma_{\left(0,m\right)\cap\mathbb V}^p\left(\bl\right).
    \end{aligned}
    \end{equation*}
\end{definition}
The control quantities $\mathfrak D_{\kappa,\rho}$ and $\mathfrak M_p$ measure mean-gap deviation and averaged moments. The corresponding
asymptotic particle density is $1/\rho$. Thus $\mathcal L_{\kappa,\rho}^p$ consists of gap configurations with
asymptotic mean gap $\rho$ and a uniform averaged $p$-moment bound.
\begin{remark}\label{rem tsai linear bound}
    Let $\bx=(x_i)_{i\in\mathbb Z}\in\mathcal W$ with
    $\mathfrak q(\bx)\in\mathcal L_{\kappa,\rho}^p$.  Then
    \begin{equation*}
        x_i=x_0+\rho i+\ep\left(i\right),\qquad
        \left|\ep\left(i\right)\right|\leq\mathfrak D_{\kappa,\rho}\left(\mathfrak q\left(\bx\right)\right)\left|i\right|^{1-\kappa}.
    \end{equation*}
    Consequently, for all sufficiently large $i$,
    \begin{equation*}
        \left|\frac{1}{x_i}+\frac{1}{x_{-i}}\right|
        \leq C\frac{\left|x_0\right|+\mathfrak D_{\kappa,\rho}\left(\mathfrak q\left(\bx\right)\right)
        i^{1-\kappa}}{\rho^2i^2}.
    \end{equation*}
    Hence \eqref{eq def of gamma1b} exists for every admissible $b$.
    For the associated configuration $\bxi=\sum_i\delta_{x_i}$, it equals
    $p_b(\bxi)$. Indeed, put $K_R=\lfloor R/\rho\rfloor$. The preceding
    estimate places the endpoint labels of the spatial restriction
    $[-R,R]$ at $\pm R/\rho+O(R^{1-\kappa})$. Its reciprocal sum and the
    sum over $|i|\le K_R$, both restricted to $|x_i|>b$, therefore differ
    by $O(R^{1-\kappa})$ terms of size $O(R^{-1})$. Their difference is
    $O(R^{-\kappa})$, proving the identity along integer $R\to\infty$.
\end{remark}
We combine the regular infinite configurations with finite configurations to define the drift on a common state space.
\begin{definition}\label{def subset for function}
    Define $\mathcal{L}_0\subset\mathcal{L}$ and $\mathcal{W}_0\subset\mathcal{W}$ by
    \begin{equation*}
    \begin{aligned}
&\mathcal{L}_0\defeq \bigcup_{{p>1,\kappa\in\left(0,1\right),\rho>0}}\tsai\bigcup\left\{\bl\in\mathcal{L}: \bl\textnormal{ is finite}\right\},\\
&\mathcal{W}_0\defeq  \mathfrak{q}^{-1}\left(\mathcal{L}_0\right)=\left\{\bx\in\mathcal{W}: \mathfrak{q}\left(\bx\right)\in\mathcal{L}_0\right\}.
    \end{aligned}
    \end{equation*}
\end{definition}
For $a,b\in\mathbb R$ and $\bl\in\mathcal L$, define
 \begin{equation*}
     l_{\left(a,b\right)} \defeq \sum_{v\in\left(a,b\right)\cap\mathbb V}l_v,
 \end{equation*}
with the value $+\infty$ if one of the summands is infinite. For a gap process $\bsfL$, we write
\begin{equation*}
 \sfL_{\left(a,b\right)}\left(t\right)\defeq\sum_{v\in\left(a,b\right)\cap\mathbb V}\sfL_v\left(t\right),
\end{equation*}
and use the same scalar notation for cumulative gaps of the other processes.
 The particle drift is the label-symmetric difference of reciprocal cumulative gaps.
    \begin{definition}
    For $i\in\mathbb{Z}$, define $\mu_i\colon\mathcal{W}_0\longrightarrow\mathbb{R}$ by
        \begin{equation*}
            \mu_i\left(\bx\right)\defeq  \lim_{K\to\infty}\left(\sum_{k\in\mathbb{N},0<k\leq K}\frac{1}{\left[\mathfrak{q}\left(\bx\right)\right]_{\left(i,i-k\right)}}-\sum_{k\in\mathbb{N},0<k\leq K}\frac{1}{\left[\mathfrak{q}\left(\bx\right)\right]_{\left(i,i+k\right)}}\right).
        \end{equation*}
    \end{definition}
For finite $\bx$, $\mu_i(\bx)$ is zero or a finite sum. For solution paths, the same regularity bounds are required uniformly on compact time intervals.
\begin{definition}\label{def tsaipro}
Let $\mathcal C_+([0,\infty))$ be the space of continuous functions
$f\colon[0,\infty)\to(0,\infty)$.  Define
    \begin{equation*}
    \mathcal P_{\kappa,\rho}^p\defeq 
    \left\{\bl\left(\cdot\right)\in\mathcal C_+\left(\left[0,\infty\right)\right)^{\mathbb V}:
    \sup_{t\in\left[0,T\right]}\mathfrak D_{\kappa,\rho}\left(\bl\left(t\right)\right)<\infty,\ 
    \sup_{t\in\left[0,T\right]}\mathfrak M_p\left(\bl\left(t\right)\right)<\infty,\ \forall T>0\right\}.
    \end{equation*}
\end{definition}
We also apply $\mathfrak q$ pointwise to paths. We call a gap configuration \emph{regular} if it belongs to
$\mathcal L_{\kappa,\rho}^p$ for some $p>1$, $\kappa\in(0,1)$, and
$\rho>0$, and call a gap process regular if it belongs to the
corresponding $\mathcal P_{\kappa,\rho}^p$.
For later comparison, set
\begin{align*}
 \underline{\mathcal L}\left(\delta\right)
 &\defeq \left\{\mathbf z\in\left(0,\infty\right)^{\mathbb V}:
   \liminf_{\left|m\right|\to\infty}\overline\Sigma_{\left(0,m\right)\cap\mathbb V}\left(\mathbf z\right)
   \geq\delta\right\},\\
 \underline{\mathcal P}\left(\delta\right)
 &\defeq \left\{\mathbf z\left(\cdot\right)\in\mathcal C_+\left(\left[0,\infty\right)\right)^{\mathbb V}:
   \liminf_{\left|m\right|\to\infty}\inf_{0\leq s\leq T}
   \overline\Sigma_{\left(0,m\right)\cap\mathbb V}\left(\mathbf z\left(s\right)\right)\geq\delta,
   \ \forall T>0\right\},
\end{align*}
and put
\begin{equation*}
 \underline{\mathcal L}\defeq \bigcup_{\delta>0}\underline{\mathcal L}\left(\delta\right),
 \qquad
 \underline{\mathcal P}\defeq \bigcup_{\delta>0}\underline{\mathcal P}\left(\delta\right).
\end{equation*}

We use the standard notions of adapted solution, strong solution, and
pathwise uniqueness from \cite{1981v}.  More precisely, an adapted solution
is a continuous process satisfying the integral equation with respect to the
given Brownian family and filtration.  It is strong when it is adapted to the
usual augmentation of the Brownian filtration.  Pathwise uniqueness means
that any two adapted solutions on the same filtered probability space,
driven by the same Brownian family and with the same initial condition, are
indistinguishable.  For the Dyson interaction functions $\mu_i$, the particle
equation is
\begin{equation}\label{eq ISDE definition}
    \sfY_i\left(t\right) = x_i + \sfB_i\left(t\right)
    +\frac\beta2\int_0^t \mu_i\left(\boldsymbol{\sfY}\left(s\right)\right)\dif s,
    \quad\forall i\in\mathbb{Z}.
\end{equation}
Let $v\in\mathbb V$.  For $\bl\in
\mathcal L_0\cup\underline{\mathcal L}$ with $l_v<\infty$, define
\begin{equation*}
 \lambda_v\left(\bl\right)\defeq  \frac{2}{l_v}
 -\sum_{i\in\mathbb{Z},1<\left|i\right|}
 \left(\frac{1}{l_{\left(v,v+i\right)}}
 -\frac{1}{l_v+l_{\left(v,v+i\right)}}\right).
\end{equation*}
For an infinite configuration the series is absolutely convergent:
positive lower density gives
$l_{(v,v+i)}\geq c_v|i|$ for all sufficiently large $|i|$, and its
summand is at most $l_v/(c_v^2|i|^2)$.  For finite configurations it is
a finite sum under the convention $1/\infty=0$.  On $\mathcal L_0$ we
also define the symmetric principal value
\begin{equation*}
 \psi_0\left(\bl\right)\defeq  \lim_{K\to\infty}\left(
 \sum_{0<k\leq K}\frac{1}{l_{\left(0,-k\right)}}
 -\sum_{0<k\leq K}\frac{1}{l_{\left(0,k\right)}}\right).
\end{equation*}
Fix independent standard Brownian motions $(\sfB_i)_{i\in\mathbb Z}$ and set
\begin{equation*}
    \sfG_v\left(t\right) \defeq \sfB_{v+1/2}\left(t\right) - \sfB_{v-1/2}\left(t\right),\quad \forall t\geq 0, v\in\mathbb{V}.
\end{equation*}

For a regular initial configuration $\bx$ with gaps $\bl=\mathfrak q(\bx)$,
Tsai's canonical finite
approximation retains the particles in the open interval $(-n,n)$.
Write
\begin{equation*}
 i_n^+\defeq\max\left\{i:x_i<n\right\},\qquad
 i_n^-\defeq\min\left\{i:x_i>-n\right\},\qquad
 \mathcal I_n^{\mathrm{sp}}\defeq\left[i_n^-,i_n^+\right]\cap\mathbb Z,
\end{equation*}
and denote the resulting finite Dyson process, driven by the corresponding
Brownian motions, by $\widetilde{\boldsymbol{\sfY}}^{[n]}$.
For $\varepsilon>0$, let $\bsfL^{\vee\varepsilon}$ denote Tsai's greatest
$\underline{\mathcal P}$-valued solution from
$\bl\vee\varepsilon=(l_v\vee\varepsilon)_{v\in\mathbb V}$.
Here greatest means that it dominates every $\underline{\mathcal P}$-valued
adapted solution on the same filtered space with the same gap Brownian
motions and a smaller initial condition.

The following collects Tsai's existence, uniqueness and approximation
results used below; see Theorems 1.2 and 1.4 and Propositions 2.1, 2.5,
and 2.6 of \cite{MR3568040}.
\begin{proposition}\label{ISDE e and u}
    Let $\beta\geq1$, $\bx\in\mathcal W$, and
    $\bl=\mathfrak q(\bx)\in\mathcal L_{\kappa,\rho}^p$ for some
    $p>1$, $\kappa\in(0,1)$, and $\rho>0$. Then:
    \begin{enumerate}[label=\textnormal{(\roman*)}]
    \item There is a
    $\mathcal P_{\kappa,\rho}^p$-valued strong solution
    $\bsfL=(\sfL_v)_{v\in\mathbb V}$ and a continuous strong process
    $\sfY_0$ satisfying
    \begin{subequations}\label{eq gap,0th sde}
        \begin{equation}\label{eq gap sde}
            \sfL_v\left(t\right) = l_v+\sfG_v\left(t\right) + \frac{\beta}{2}\int_0^t \lambda_v\left(\bsfL\left(s\right)\right)\dif s,\quad v\in\mathbb{V}.
        \end{equation}
        \begin{equation}\label{eq 0th sde}
            \sfY_0\left(t\right) = x_0+\sfB_0\left(t\right) + \frac{\beta}{2}\int_0^t\psi_0\left(\bsfL\left(s\right)\right)\dif s.
        \end{equation}
    \end{subequations}
    If $\sfY_i=\sfY_0+\sfL_{(0,i)}$ for $i>0$ and
    $\sfY_i=\sfY_0-\sfL_{(0,i)}$ for $i<0$, then
    $\boldsymbol{\sfY}=(\sfY_i)_{i\in\mathbb Z}$ is a strong solution
    of \eqref{eq ISDE definition},
    and pathwise uniqueness holds among adapted solutions whose gap
    processes belong to $\mathcal P_{\kappa,\rho}^p$.
    Conversely, the gap process of any solution in this class, together
    with its zeroth coordinate, solves \eqref{eq gap,0th sde}.  The gap
    solution is pathwise unique in $\mathcal P_{\kappa,\rho}^p$.

    \item The canonical spatial truncations satisfy, almost surely,
    \begin{equation}\label{eq tilde x is finite}
       \lim_{n\to\infty}\sup_{0\leq t\leq T}
       \left|\widetilde{\sfY}^{\left[n\right]}_i\left(t\right)-\sfY_i\left(t\right)\right|=0,
       \qquad T>0,\quad i\in\mathbb Z.
    \end{equation}

    \item If $\varepsilon_k\downarrow0$, then, almost surely,
    \begin{equation}\label{eq:tsai-padded-limit}
       \sfL_v^{\vee\varepsilon_k}\left(t\right)\downarrow\sfL_v\left(t\right),
       \qquad v\in\mathbb V,\quad t\geq0.
    \end{equation}
    \end{enumerate}
\end{proposition}
Tail envelopes place a possibly nonmonotone sequence of initial gaps between two monotone sequences.
\begin{definition}\label{def sup inf gap}
Let $(\bl^\UN)_{N\geq1}\subset\mathcal L$.  Define the tail
envelopes first in $[0,\infty]^{\mathbb V}$ by
\begin{equation*}
 l_v^{\tsup,\UN}\defeq \sup_{n\geq N}l_v^{\left(n\right)},\qquad
 l_v^{\tinf,\UN}\defeq \inf_{n\geq N}l_v^{\left(n\right)},
 \qquad v\in\mathbb V.
\end{equation*}
\end{definition}
The upper tail envelopes have finite supports that increase to the full gap lattice. This allows us to compare their finite gap processes with a single infinite limiting process.
\begin{lemma}\label{lem:envelope-supports}
Suppose that each $\bl^\UN$ is finite and that
$l_v^\UN\to l_v\in(0,\infty)$ for every $v\in\mathbb V$.  Put
\begin{equation*}
 \mathcal V_N^{\mathrm{sup}}\defeq \left\{v:l_v^{\tsup,\UN}<\infty\right\}.
\end{equation*}
After deleting a finite initial segment if necessary,
$\mathcal V_N^{\mathrm{sup}}$ is a nonempty finite interval,
\begin{equation*}
 \mathcal V_N^{\mathrm{sup}}\subset\mathcal V_{N+1}^{\mathrm{sup}},
 \qquad \bigcup_{N\geq1}\mathcal V_N^{\mathrm{sup}}=\mathbb V,
\end{equation*}
and
\begin{equation*}
 l_v^{\tinf,\UN}\uparrow l_v,\qquad
 l_v^{\tsup,\UN}\downarrow l_v,\qquad v\in\mathbb V.
\end{equation*}
In particular, $\bl^{\tsup,\UN}$ is a finite element of $\mathcal L$.
Whenever $\bl^{\tinf,\UN}$ belongs to
$\mathcal L_{\kappa,\rho'}^p$ for some $\rho'>0$, it is an infinite
element of $\mathcal L$.
\end{lemma}
\begin{proof}
If $\mathcal V_n=\{v:l_v^{(n)}<\infty\}$, then
$\mathcal V_N^{\mathrm{sup}}=\bigcap_{n\geq N}\mathcal V_n$.  An intersection
of intervals is an interval, and it is finite because it is contained
in $\mathcal V_N$.  Since $l_v^{(n)}\to l_v<\infty$, every fixed $v$
belongs to all $\mathcal V_n$ for sufficiently large $n$, proving the
exhaustion.  The two coordinatewise limits are the elementary tail
infimum and tail supremum limits of a convergent real sequence.  The
last assertion follows from the definition of
$\mathcal L_{\kappa,\rho'}^p$, whose elements have all coordinates
finite.
\end{proof}
Fix $\beta\geq1$ and let $\bx\in\mathcal W$ satisfy
$\bl=\mathfrak q(\bx)\in\mathcal L_{\kappa,\rho}^p$ for some
$\kappa\in(0,1)$, $\rho>0$, and $p>1$.  Let
$(\bx^\UN)_{N\geq1}$ be arbitrary finite elements of $\mathcal W$ such
that $\bx^\UN\to\bx$ coordinatewise. After deleting finitely many terms,
we assume $x_0^\UN\in\mathbb R$ for every $N$, and put
\begin{equation*}
 \mathcal I_N\defeq \left\{i\in\mathbb Z:x_i^\UN\in\mathbb R\right\}
 =\left[i_N^-,i_N^+\right]\cap\mathbb Z,\qquad
 \mathcal V_N\defeq \left(i_N^-,i_N^+\right)\cap\mathbb V,\qquad
 \bl^\UN\defeq \mathfrak q\left(\bx^\UN\right).
\end{equation*}
Thus $0\in\mathcal I_N$ for every $N$.
Let $\bl^{\tsup,\UN}$ and $\bl^{\tinf,\UN}$ be the envelopes of
Definition \ref{def sup inf gap}.  For each $N$, let $(\sfX_i^\UN)_{i\in\mathcal I_N}$ be the finite
Dyson Brownian motion driven by $(\sfB_i)_{i\in\mathcal I_N}$:
\begin{equation}\label{eq finite DBM}
    \mathsf{X}_i^\UN\left(t\right) = x_i^\UN+\sfB_i\left(t\right)
    + \frac{\beta}{2}\int_0^t \mu_i\left(\boldsymbol{\sfX}^\UN\left(s\right)\right)\dif s,
    \qquad i\in\mathcal I_N,
\end{equation}
where the configuration is extended by $-\infty$ to the left of
$i_N^-$ and by $+\infty$ to the right of $i_N^+$.
The coordinates outside $\mathcal I_N$ remain the fixed formal values
$-\infty$ and $+\infty$, so the finite support of the gap process is
$\mathcal V_N$ at every time.

The next theorem gives convergence for finite approximations whose principal-value tails need not agree with that of the limiting configuration. The discrepancy produces a deterministic linear translation, while the limiting gaps are unchanged.
\begin{theorem}\label{theo sde theorem}
Under the preceding setup, choose $b>0$ avoiding
$\{|x_i|,|x_i^\UN|:i\in\mathbb Z,N\in\mathbb N\}$ and assume that
$\lim_{N\to\infty}\gamma_{1,b}(\bx^\UN)$ exists.  Define
 \begin{equation}\label{eq def of gamma}
     \gamma\defeq \lim_{N\to\infty}\gamma_{1,b}\left(\bx^\UN\right)-\gamma_{1,b}\left(\bx\right).
 \end{equation}
The value of $\gamma$ is independent of the admissible choice of $b$.
Suppose also that there are
$C<\infty$, $N_0\in\mathbb N$, and $\rho^\UN>0$ such that
    \begin{equation}\label{eq condition for L inf N}
    \begin{aligned}
    &\rho^\UN\longrightarrow\rho,\qquad
    \bl^{\tinf,\UN}\in\mathcal L_{\kappa,\rho^\UN}^p
       \quad\left(N\geq N_0\right),\\
    &\sup_{N\geq N_0}\mathfrak D_{\kappa,\rho^\UN}\left(\bl^{\tinf,\UN}\right)\leq C,
    \end{aligned}
    \end{equation}
and assume additionally that
    \begin{equation}\label{eq condition for L sup N}
       \sup_{N\in\mathbb N}\sup_{v\in\mathcal V_N}l_v^\UN\leq C.
    \end{equation}
Then there exists a process $\boldsymbol{\sfX}=(\sfX_i)_{i\in\mathbb Z}$
such that, almost surely,
    \begin{equation}\label{eq convergence of X_i}
        \lim_{N\to\infty}\sup_{t\in\left[0,T\right]}
        \left|\sfX_i^\UN\left(t\right)-\sfX_i\left(t\right)\right|=0,
        \qquad T>0,\quad i\in\mathbb Z,
    \end{equation}
where the expression is considered for all sufficiently large $N$ for
which $i\in\mathcal I_N$.  Moreover, $\boldsymbol{\sfX}$ is a strong
solution, with gap process in $\mathcal P_{\kappa,\rho}^p$, of
    \begin{equation}\label{eq equation for X}
        \mathsf{X}_i\left(t\right) = x_i+\sfB_i\left(t\right)
        +\frac{\beta}{2}\int_0^t
        \left\{\mu_i\left(\boldsymbol{\sfX}\left(s\right)\right)-\gamma\right\}\dif s,
        \qquad i\in\mathbb Z.
    \end{equation}
Pathwise uniqueness holds among adapted solutions in this class.\end{theorem}
\begin{remark}
Let $b'>b$ be another admissible cutoff.  Only finitely many limiting
particles lie in $(-b',-b)\cup(b,b')$.  Coordinatewise convergence and
strict ordering imply that, for all sufficiently large $N$, precisely
the corresponding labels of $\bx^\UN$ lie in this annulus.  Therefore
    \begin{equation*}
 \lim_{N\to\infty}\left\{\gamma_{1,b'}\left(\bx^\UN\right)-\gamma_{1,b'}\left(\bx\right)\right\}
 =\lim_{N\to\infty}\left\{\gamma_{1,b}\left(\bx^\UN\right)-\gamma_{1,b}\left(\bx\right)\right\}
 =\gamma,
    \end{equation*}
because the two finite annular sums converge term by term.
\end{remark}
\subsection{Comparison of gap processes}
We need comparison both between regular infinite systems and between systems with different finite boundaries. These comparisons will place the approximating gap processes between the two envelope processes.
\begin{lemma}\label{lm monotonicity of gap process}
The following comparisons hold for solutions driven by the same gap
Brownian motions.
\begin{enumerate}[label=\textnormal{(\roman*)}]
\item Let $\boldsymbol{\sfL}^{(1)}$ and
$\boldsymbol{\sfL}^{(2)}$ be the corresponding regular infinite gap
solutions started from regular configurations $\bl^{(1)}$ and
$\bl^{(2)}$. If $\bl^{(1)}\leq\bl^{(2)}$, then
\begin{equation}\label{eq almost surely monotone}
 \bsfL^{\left(1\right)}\left(t\right)\leq\bsfL^{\left(2\right)}\left(t\right),\qquad
 t\geq0,\quad\text{almost surely}.
\end{equation}
\item Let $I=[i_-,i_+]\cap\mathbb Z\subset J$ be intervals of particle
labels, with $I$ finite and $J$ either finite or equal to $\mathbb Z$,
and put $I^{\mathrm g}=(i_-,i_+)\cap\mathbb V$.  Let
$\bsfL^I$ be the finite gap process on $I$, and let
$\bsfL^J$ be the
gap process on $J$, regular when $J=\mathbb Z$. If
\begin{equation*}
 \sfL_v^I\left(0\right)\geq \sfL_v^J\left(0\right),\qquad v\in I^{\mathrm g},
\end{equation*}
then, almost surely,
\begin{equation}\label{eq:mixed-gap-comparison}
 \sfL_v^I\left(t\right)\geq \sfL_v^J\left(t\right),\qquad
 v\in I^{\mathrm g},\quad t\geq0.
\end{equation}
\end{enumerate}
\end{lemma}
\begin{proof}
Part (i) follows by applying Tsai's comparison to the monotone
constructions of Propositions 2.5--2.6 of \cite{MR3568040} and then
letting the padding parameter decrease to zero.

For part (ii), given positive gaps $\mathbf z$ on $I^{\mathrm g}$, reconstruct
an ordered particle vector $\mathbf x$ on $I$, up to an irrelevant common
translation, and define
\begin{equation*}
 \mu_j^I\left(\mathbf x\right)\defeq \sum_{\substack{k\in I\\k\neq j}}\frac1{x_j-x_k},
 \qquad
 \lambda_{j+1/2}^I\left(\mathbf z\right)\defeq \mu_{j+1}^I\left(\mathbf x\right)-\mu_j^I\left(\mathbf x\right),
 \qquad j,j+1\in I.
\end{equation*}
When the $J$-equation is restricted to $I^{\mathrm g}$, its drift is
\begin{equation}\label{eq:external-compression}
 \lambda_v^I\left(\bsfL^J\left(s\right)|_{I^{\mathrm g}}\right)
 -\chi_v^{J\setminus I}\left(\bsfL^J\left(s\right)\right),
 \qquad v=j+\tfrac12,
\end{equation}
where, for a gap configuration $\mathbf z$ on $J$ and any particle
reconstruction $\mathbf x^J$ of it, define
\begin{equation*}
 \chi_v^{J\setminus I}\left(\mathbf z\right)
 \defeq\sum_{k\in J\setminus I}
 \frac{z_v}
 {\left|x_{j+1}^J-x_k^J\right|\,\left|x_j^J-x_k^J\right|}\geq0.
\end{equation*}
Indeed, a particle $k$ outside $I$ contributes
$-z_v/(|x_{j+1}^J-x_k^J|\,|x_j^J-x_k^J|)$ to
$\mu_{j+1}^J-\mu_j^J$.  Thus the $I$-system has the first term in
\eqref{eq:external-compression} and zero exterior compression.
Tsai's finite external-force comparison, Lemma 3.7 of
\cite{MR3568040}, therefore gives \eqref{eq:mixed-gap-comparison} up to
the usual stopping times keeping the finitely many gaps away from zero
and infinity.  Noncollision and continuity remove the stopping.
When $J=\mathbb Z$, regularity makes the exterior compression series
locally uniformly convergent; equivalently one may first truncate that
series and pass to the limit by its nonnegative monotone convergence.
Taking a countable intersection over finite $I$, rational terminal
times, and then using continuity makes all the comparisons
simultaneous.
\end{proof}
\subsection{Convergence of the particle dynamics}
We now prove Theorem \ref{theo sde theorem}.  All processes below are
constructed on the common Brownian space.  Let $\bsfL$ be Tsai's
regular gap solution from $\bl$.  Let $\bsfL^\UN$ be the finite gap
process on $\mathcal V_N$ from $\bl^\UN$, extended by $+\infty$
outside $\mathcal V_N$.
Finally, let $\bsfL^{\tinf,\UN}$ be the
regular infinite solution from $\bl^{\tinf,\UN}$, and let
$\bsfL^{\tsup,\UN}$ be the finite solution from
$\bl^{\tsup,\UN}$ on $\mathcal V_N^{\mathrm{sup}}$, again extended by
$+\infty$ outside its finite support. The superscripts refer to the initial
envelopes, not to extrema over time of the resulting solutions.
Subtracting adjacent equations in \eqref{eq finite DBM} and using
finite-dimensional pathwise uniqueness gives
$\mathfrak q(\boldsymbol{\sfX}^{(N)}(t))=\bsfL^{(N)}(t)$
simultaneously for $t\geq0$, almost surely.
We now use the blocks appearing in Section~6 of Tsai's work \cite{MR3568040}. More explicitly,
for $i\in\mathbb{Z}$ and $k\in\mathbb{N}$, let
\begin{equation*}
    m_i\defeq 
    \sgn\left(i\right)\left\lfloor \left|i\right|^{1/\kappa}\right\rfloor,
    \quad\widetilde m_i^{\,k}\defeq m_{ki},
\end{equation*}
and, for $b\in\mathbb V$, define
\begin{equation*}
    A_{b,k}\defeq\left(\widetilde m_{b-1/2}^{\,k},
    \widetilde m_{b+1/2}^{\,k}\right)\cap\mathbb V.
\end{equation*}

We first identify the limits of the two envelope processes. Their convergence will imply convergence of every fixed gap in the original finite systems.
\begin{proposition}\label{prop convergence of gap}
Under the hypotheses of Theorem \ref{theo sde theorem}, almost surely,
    \begin{equation}\label{eq sup inf converge}
 \lim_{N\to\infty}
 \max_{\diamond\in\left\{\mathrm{inf},\mathrm{sup}\right\}}
 \sup_{0\leq t\leq T}\left|\sfL_v^{\diamond,\UN}\left(t\right)-\sfL_v\left(t\right)\right|=0
    \end{equation}
for every $T>0$ and $v\in\mathbb V$.
\end{proposition}
\begin{proof}
Lemma \ref{lem:envelope-supports} and Lemma
\ref{lm monotonicity of gap process} give, on one event of probability
one,
    \begin{equation}\label{eq inf sup monotone}
 \bsfL^{\tinf,\UN}\leq\bsfL^{\tinf,\left(N+1\right)}\leq\bsfL
 \leq\bsfL^{\tsup,\left(N+1\right)}\leq\bsfL^{\tsup,\UN}.
    \end{equation}
All inequalities are coordinatewise on their common finite coordinates,
with the extended value $+\infty$ elsewhere.  Define
\begin{equation*}
 \bsfL^{\mathrm{up}}\defeq \lim_{N\to\infty}\bsfL^{\tsup,\UN},
 \qquad
 \bsfJ\defeq \lim_{N\to\infty}\bsfL^{\tinf,\UN}.
\end{equation*}
We first identify the upper limit.  Fix $v\in\mathbb V$ and $T>0$,
and choose $N_v$ such that $v\in\mathcal V_{N_v}^{\mathrm{sup}}$.  For
$N\geq N_v$, set
\begin{equation*}
 \mathsf M_{v,T}\defeq \sup_{0\leq s\leq T}\sfL_v^{\tsup,\left(N_v\right)}\left(s\right)<\infty.
\end{equation*}
Since $\bsfL^{\mathrm{up}}\geq\bsfL$, its cumulative gaps have positive
lower density, uniformly on compact time intervals.  Hence
$\bsfL^{\mathrm{up}}(s)\in\underline{\mathcal L}$ and
$\lambda_v(\bsfL^{\mathrm{up}}(s))$ is defined by the absolutely
convergent series above.
For a positive gap vector $\mathbf z$, write
\begin{equation*}
 d_{v,i}\left(\mathbf z\right)\defeq 
 \frac{z_v}{z_{\left(v,v+i\right)}\left\{z_v+z_{\left(v,v+i\right)}\right\}},\qquad \left|i\right|>1.
\end{equation*}
By \eqref{eq inf sup monotone},
\begin{equation}\label{eq:upper-drift-domination}
 \frac2{\sfL_v^{\tsup,\UN}\left(s\right)}\leq\frac2{\sfL_v\left(s\right)},
 \qquad
 0\leq d_{v,i}\left(\bsfL^{\tsup,\UN}\left(s\right)\right)
 \leq\frac{\mathsf M_{v,T}}{\sfL_{\left(v,v+i\right)}\left(s\right)^2}.
\end{equation}
Since $\bsfL\in\mathcal P_{\kappa,\rho}^p$, there are random finite
$\mathsf K_{v,T}$ and $\mathsf c_{v,T}>0$ such that
\begin{equation*}
 \inf_{0\leq s\leq T}\sfL_{\left(v,v+i\right)}\left(s\right)
 \geq \mathsf c_{v,T}\left|i\right|,\qquad \left|i\right|\geq \mathsf K_{v,T}.
\end{equation*}
Consequently,
\begin{equation}\label{eq:upper-compression-tail}
 \sup_{N\geq N_v}\sup_{0\leq s\leq T}
 \sum_{\left|i\right|>K}d_{v,i}\left(\bsfL^{\tsup,\UN}\left(s\right)\right)
 \leq \frac{\mathsf C_{v,T}}K,\qquad K\geq \mathsf K_{v,T}.
\end{equation}
For fixed $K$, coordinatewise monotone convergence gives convergence
of the truncated drift.  Equations
\eqref{eq:upper-drift-domination}--\eqref{eq:upper-compression-tail}
then give convergence of the full drift in $L^1([0,T])$.  Passing to
the finite gap equations yields
\begin{equation*}
 \sfL_v^{\mathrm{up}}\left(t\right)=l_v+\sfG_v\left(t\right)
 +\frac\beta2\int_0^t\lambda_v\left(\bsfL^{\mathrm{up}}\left(s\right)\right)\dif s.
\end{equation*}
Thus $\bsfL^{\mathrm{up}}$ is adapted, has continuous strictly positive
coordinates, and solves the infinite gap ISDE.  Moreover,
$\bsfL^{\mathrm{up}}\geq\bsfL$, and so
\begin{equation*}
 \liminf_{\left|m\right|\to\infty}\inf_{0\leq s\leq T}
 \overline\Sigma_{\left(0,m\right)\cap\mathbb V}\left(\bsfL^{\mathrm{up}}\left(s\right)\right)\geq\rho.
\end{equation*}
In particular $\bsfL^{\mathrm{up}}\in\underline{\mathcal P}$.

Let $\varepsilon_k=2^{-k}$.  Since $\bsfL^{\mathrm{up}}(0)=\bl\leq
\bl\vee\varepsilon_k$, the greatest-solution assertion in Proposition
\ref{ISDE e and u} gives
\begin{equation*}
 \bsfL\leq\bsfL^{\mathrm{up}}\leq\bsfL^{\vee\varepsilon_k},
 \qquad k\geq1.
\end{equation*}
On the countable intersection over $v\in\mathbb V$, rational
$t\geq0$, and $k\geq1$, equation \eqref{eq:tsai-padded-limit} implies
$\sfL_v^{\mathrm{up}}(t)=\sfL_v(t)$ at every rational time.  Continuity extends the
identity to all times.  Finally, the continuous functions
$\sfL_v^{\tsup,\UN}$ decrease on $[0,T]$ to the continuous function
$\sfL_v$; Dini's theorem proves the upper-envelope convergence in
\eqref{eq sup inf converge}.
For the lower limit, put $\bsfJ^\UN=\bsfL^{\tinf,\UN}$.  Fix
$v\in\mathbb V$ and $T>0$.  For $N\geq N_0$, monotonicity gives
\begin{equation*}
 \frac2{\mathsf J_v^\UN\left(s\right)}\leq\frac2{\mathsf J_v^{\left(N_0\right)}\left(s\right)},
 \qquad
 d_{v,i}\left(\bsfJ^\UN\left(s\right)\right)
 \leq\frac{\sup_{u\leq T}\sfL_v\left(u\right)}
 {\mathsf J_{\left(v,v+i\right)}^{\left(N_0\right)}\left(s\right)^2}.
\end{equation*}
Moreover $\bsfJ\geq\bsfJ^{(N_0)}$, so
$\bsfJ(s)\in\underline{\mathcal L}$ and its drift is defined.
The regularity of $\bsfJ^{(N_0)}$ makes the second majorant summable,
uniformly on $[0,T]$.  The same truncated-drift and dominated-
convergence argument used above therefore shows that $\bsfJ$ is
adapted, has continuous strictly positive coordinates, starts from
$\bl$, and solves \eqref{eq gap sde}.  It remains to prove that
$\bsfJ\in\mathcal P_{\kappa,\rho}^p$.

We first verify the density part of this assertion. Since
$\rho^{(N)}\rightarrow\rho>0$, choose a deterministic $N_{*}$ sufficiently large that $\rho_{*}=\rho^{(N_{*})}>{3\rho}/{4}.$ By Proposition \ref{ISDE e and u}, $\bsfJ^{(N_{*})}\in\mathcal{P}_{\kappa,\rho_{*}}^{p}$. Hence, for every $T>0$, we let
\begin{equation*}
    \mathsf C_{*,T}\defeq \sup_{0\leq t\leq T}\mathfrak D_{\kappa,\rho_{*}}\left(\bsfJ^{\left(N_{*}\right)}\left(t\right)\right)<\infty.
\end{equation*}
The deterministic estimate in equation~(6.4) of \cite{MR3568040} applied with
$\alpha=\kappa$, gives a constant $c_{\kappa}<\infty$ such that, for all sufficiently large $k$,
\begin{equation*}
    \sup_{0\leq t\leq T}
    \sup_{b\in\mathbb V}\left|\overline{\Sigma}_{A_{b,k}}
    \left(\bsfJ^{\left(N_{*}\right)}\left(t\right)\right)
    -\rho_{*}\right|\leq\frac{c_{\kappa}\mathsf C_{*,T}}{k}.
\end{equation*}
Choose $\mathsf K_T=\mathsf K_T(\omega)$ so large that the blocks are nonempty and
$c_{\kappa}\mathsf C_{*,T}/\mathsf K_T\leq\rho/4$. Since
$\bsfJ^{(N_{*})}\leq\bsfJ$, for every $k\geq \mathsf K_T$ we obtain
\begin{equation*}
    \inf_{0\leq t\leq T}\inf_{b\in\mathbb V}
    \overline{\Sigma}_{A_{b,k}}\left(\bsfJ\left(t\right)\right)
    \geq\rho_{*}-\frac{\rho}{4}>\frac{\rho}{2}.
\end{equation*}
This is precisely the block lower bound required in condition~(6.26) of \cite{MR3568040}.

Comparison with the same regular lower process gives
\begin{equation*}
 \inf_{0\le t\le T}\overline\Sigma_{\left(0,m\right)\cap\mathbb V}
 \left(\bsfJ\left(t\right)\right)
 \geq\rho_*-\mathsf C_{*,T}|m|^{-\kappa}.
\end{equation*}
Together with the continuity and strict positivity already proved, this yields
\begin{equation}\label{eq bsfJ in Y_t}
 \liminf_{|m|\to\infty}\inf_{0\le t\le T}
 \overline\Sigma_{\left(0,m\right)\cap\mathbb V}\left(\bsfJ\left(t\right)\right)
 \geq\rho_*>\rho/2.
\end{equation}
We now apply Lemma~6.4 of Tsai's work with $\alpha=\kappa$ and
$\gamma=\rho/2$. In the present notation, its assumptions are that $\bsfJ$
solves the gap ISDE, that
$\mathfrak D_{\kappa,\rho}(\bsfJ(0))=\mathfrak D_{\kappa,\rho}(\bl)<\infty$, that
$\bsfJ$ satisfies \eqref{eq bsfJ in Y_t}, and that the preceding block lower bound holds on
every compact time interval. All four assertions have been verified above,
and the conclusion of that lemma is
\begin{equation*}
    \sup_{t\in\left[0,T\right]}\mathfrak D_{\kappa,\rho}\left(\bsfJ\left(t\right)\right)<\infty,
    \quad \forall T>0.
\end{equation*}
It remains only to verify the averaged $p$-moment condition. The comparison
inequality established at the beginning and
$\bsfL\in\mathcal{P}_{\kappa,\rho}^{p}$ give, for every $T>0$,
\begin{equation*}
\begin{aligned}
    \sup_{0\leq t\leq T}\mathfrak M_p\left(\bsfJ\left(t\right)\right)
    \leq\sup_{0\leq t\leq T}\mathfrak M_p\left(\bsfL\left(t\right)\right)<\infty.
\end{aligned}
\end{equation*}
Thus $\bsfJ\in\mathcal P_{\kappa,\rho}^p$.
Pathwise uniqueness in Proposition \ref{ISDE e and u} identifies
$\bsfJ=\bsfL$ indistinguishably.
Since $\sfL_v^{\tinf,\UN}$ increases pointwise to the continuous
function $\sfL_v$, Dini's theorem therefore gives
\begin{equation*}
 \lim_{N\to\infty}\sup_{0\leq t\leq T}
 \left|\sfL_v^{\tinf,\UN}\left(t\right)-\sfL_v\left(t\right)\right|=0,
 \qquad T>0,\quad v\in\mathbb V.
\end{equation*}
\end{proof}
The mixed comparison lemma also gives
\begin{equation*}
 \bsfL^{\tinf,\UN}\left(t\right)\leq\bsfL^\UN\left(t\right)
 \leq\bsfL^{\tsup,\UN}\left(t\right)
\end{equation*}
on all common finite coordinates, simultaneously for $t\geq0$ and
$N\geq N_0$.  
The comparison inequality therefore transfers the envelope convergence to the original finite gap processes. This gives the local convergence needed to identify the limiting particle drift.
\begin{cor}\label{cor convergence of L^N}
Under the hypotheses of Theorem \ref{theo sde theorem}, almost surely,
    \begin{equation}\label{eq convergence of L^N}
        \lim_{N\to\infty}\sup_{0\leq t\leq T}
        \left|\sfL_v^\UN\left(t\right)-\sfL_v\left(t\right)\right|=0,
        \qquad T>0,\quad v\in\mathbb V.
    \end{equation}
\end{cor}
We bound cumulative gap displacements uniformly over the finite systems. This supplies the control of the far interaction tails needed to identify the limiting drift.
\begin{lemma}\label{lem:uniform-density-error}
Assume the hypotheses and notation of Theorem \ref{theo sde theorem}. For
$m\in\mathbb Z\setminus\{0\}$ and $t\geq0$, set
\[
\mathsf D_m^{\left(N\right)}\left(t\right)\defeq \sum_{v\in\left(0,m\right)\cap\mathcal V_N}
\left(\mathsf L^{\left(N\right)}_v\left(t\right)-l^{\left(N\right)}_v\right).
\]
There is an event $\Omega_*$ of probability one on which, for every
$T>0$, there is a finite random variable $\mathsf C_T$, chosen increasing
in $T$, such that
\begin{equation}\label{eq:uniform-density-error}
\sup_{N\in\mathbb N}\sup_{0\leq t\leq T}
\sup_{m\in\mathbb Z\setminus\left\{0\right\}}
\left|m\right|^{\kappa-1}\left|\mathsf D_m^{\left(N\right)}\left(t\right)\right|\leq \mathsf C_T
\quad\text{almost surely}.
\end{equation}
\end{lemma}

\begin{proof}
We give the finite-volume details, since the endpoints of the finite
systems need not be nested compared to the proof in \cite{MR3568040}. The proof is arguably the most technical in the paper and proceeds in several steps. All processes below use the common Brownian
coupling fixed before Proposition \ref{ISDE e and u}.

\smallskip
\noindent\emph{Step 1: a common lower process and the Bessel controls.}
Initially consider $N\geq N_0$, and put
\[
\underline{\bl}\defeq \bl^{\mathrm{inf},\left(N_0\right)},\qquad \rho_*\defeq \rho^{\left(N_0\right)}>0,
\]
and let $\underline{\bsfL}$ be the gap-ISDE solution from
$\underline{\bl}$. By \eqref{eq condition for L inf N} and Proposition
\ref{ISDE e and u},
\begin{equation}\label{eq:ud-lower-regular}
\underline{\bsfL}\in\mathcal P^p_{\kappa,\rho_*}.
\end{equation}

Coordinatewise convergence
$\bx^{(N)}\to\bx$ implies that, for each $r<\infty$,
\begin{equation}\label{eq:ud-support-exhausts}
\left[-r,r\right]\cap\mathbb Z\subset\mathcal I_N
\quad\text{for all sufficiently large }N.
\end{equation}
Moreover, on one event of probability one,
\begin{equation}\label{eq:ud-mixed-comparison}
\underline{\mathsf L}_v\left(t\right)\leq\mathsf L^{\left(N\right)}_v\left(t\right),
\qquad v\in\mathcal V_N,\ t\geq0,\ N\geq N_0.
\end{equation}
This is the mixed finite--infinite case of Lemma
\ref{lm monotonicity of gap process}, since
$\underline l_v\leq l_v^{(N)}$ for every $v$.  Intersecting the
probability-one comparison events over the countable set of $N$ gives
\eqref{eq:ud-mixed-comparison} simultaneously.

For $v\in\mathbb V$ and $0\leq s\leq u$, let
$\mathsf Q_v^s(u)$ be the nonnegative solution
\begin{equation}\label{eq:ud-Q-definition}
\mathsf Q_v^s\left(u\right)=\sfG_v\left(u\right)-\sfG_v\left(s\right)
+\beta\int_s^u\frac{\dif r}{\mathsf Q_v^s\left(r\right)},
\qquad\mathsf Q_v^s\left(s\right)=0,
\end{equation}
and set $\mathsf Q^{s,t}_v=\sup_{s\leq u\leq t}\mathsf Q_v^s(u)$.
These are Tsai's processes in (3.3)--(3.5): $\mathsf Q_v^s/\sqrt2$ is a
Bessel process of dimension $\beta+1$ started from zero at time $s$.
For fixed $s$, the indices $v-1/2$ of either parity use disjoint Brownian
pairs, so each parity subfamily is iid. Since the compression in each
finite gap drift is nonnegative,
Tsai's (3.7)--(3.8) gives
\begin{equation}\label{eq:ud-Q-comparison}
-l_v^{\left(N\right)}\leq\mathsf L^{\left(N\right)}_v\left(u\right)-l_v^{\left(N\right)}
\leq\mathsf Q^{0,T}_v,
\qquad 0\leq u\leq T,\quad v\in\mathcal V_N,
\end{equation}
and, whenever $s\leq u\leq t$,
\begin{equation}\label{eq:ud-Q-increment}
\mathsf L^{\left(N\right)}_v\left(u\right)-\mathsf L^{\left(N\right)}_v\left(s\right)\leq\mathsf Q^{s,t}_v.
\end{equation}
We shall also use the backward form (3.7): for any gap process of this
type, in particular for $\underline{\bsfL}$,
\begin{equation}\label{eq:ud-Q-backward}
 \underline{\mathsf L}_v\left(t\right)-
 \inf_{s\leq u\leq t}\underline{\mathsf L}_v\left(u\right)
 \leq\mathsf Q_v^s\left(t\right)\leq\mathsf Q^{s,t}_v.
\end{equation}

\smallskip
\noindent\emph{Step 2: the common block event and one-sided lower bounds.}
Use the blocks $A_{b,k}$ defined before Proposition \ref{prop convergence of gap}.
For $k\geq2$ the blocks are nonempty and form an ordered
partition of $\mathbb{V}$, with
\begin{equation}\label{eq:ud-block-geometry}
\frac{\left|A_{b+1,k}\right|}{\left|A_{b,k}\right|}\vee
\frac{\left|A_{b,k}\right|}{\left|A_{b+1,k}\right|}\leq R,\qquad R\defeq 2^{1/\kappa+2}.
\end{equation}
Indeed, put $\vartheta=1/\kappa$ and
$d_j=k^{\vartheta}((j+1)^{\vartheta}-j^{\vartheta})$ for $j\geq0$. These unrounded positive-side
block lengths increase, satisfy $d_j\geq k^{\vartheta}>2$, and obey
$d_{j+1}/d_j\leq2^{\vartheta}$: for $j=0$ the ratio is $2^{\vartheta}-1$, while for
$j\geq1$ this follows by integrating
$(u+1)^{\vartheta-1}\leq2^{\vartheta-1}u^{\vartheta-1}$ over $[j,j+1]$.
The actual length $|A_{j+1/2,k}|$ differs from $d_j$ by less than one,
so it lies between $d_j/2$ and $3d_j/2$. Both adjacent ratios are
therefore at most $3\cdot2^{\vartheta}<R$; reflection handles the negative side,
and the two central blocks have equal length.
For fixed $k$, equations (6.2)--(6.3) of \cite{MR3568040} also show
that every block whose underlying real interval contains the integer
$r$ has size at most
\begin{equation}\label{eq:ud-block-size}
c_{\kappa,k}\left(1+\left|r\right|^{1-\kappa}\right).
\end{equation}

Fix $T>0$. Choose a deterministic partition
$0=t_0<t_1<\cdots<t_J=T$ sufficiently fine that
\begin{equation}\label{eq:ud-short-time}
2q\left(t_j-t_{j-1},1\right)\leq\frac{\rho_*}{4},
\qquad j=1,\ldots,J,
\end{equation}
where
\begin{equation*}
 q\left(h,r\right)\defeq\expt\left[\left(\mathsf Q_{1/2}^{0,h}\right)^r\right]
\end{equation*}
is finite and tends to zero with $h$ for every $r\geq1$; this is
Tsai's (3.5).  We claim that, on an event of
probability one, there is a finite random integer $\mathsf K=\mathsf K(T)$ for which,
for every $b\in\mathbb V$ and $j=1,\ldots,J$,
\begin{subequations}\label{eq:ud-common-block-event}
\begin{align}
\overline\Sigma_{A_{b,\mathsf K}}\left(\underline{\bsfL}\left(t_j\right)\right)
&\geq\frac{3\rho_*}{4},                                \label{eq:ud-block-lower-final}\\
\overline\Sigma_{A_{b,\mathsf K}}
\left(\boldsymbol{\mathsf Q}^{t_{j-1},t_j}\right)
&\leq\frac{\rho_*}{4},                                 \label{eq:ud-block-Q-small}\\
\overline\Sigma_{A_{b,\mathsf K}}\left(\boldsymbol{\mathsf Q}^{0,T}\right)
&\leq C_Q\left(T\right),                                           \label{eq:ud-block-Q-global}\\
\max_{i\in\left[\widetilde m_{b-1/2}^{\,\mathsf K},
\widetilde m_{b+1/2}^{\,\mathsf K}\right]\cap\mathbb Z}
\sup_{0\leq u\leq T}\left|\sfB_i\left(u\right)\right|
&\leq C_B\left(T\right)\left|A_{b,\mathsf K}\right|^{1/2}.                            \label{eq:ud-block-B}
\end{align}
\end{subequations}
Here $C_Q(T)$ and $C_B(T)$ may be taken deterministic and are independent
of $b$, $j$, and $N$. To prove the claim, use
\eqref{eq:ud-lower-regular} and Tsai's deterministic block estimate
(6.4): for all sufficiently large $k$,
\[
\inf_{0\leq u\leq T}\inf_{b\in\mathbb{V}}
\overline\Sigma_{A_{b,k}}\left(\underline{\bsfL}\left(u\right)\right)
\geq\rho_*-\frac{c_\kappa}{k}
\sup_{0\leq u\leq T}
\mathfrak D_{\kappa,\rho_*}\left(\underline{\bsfL}\left(u\right)\right).
\]
This gives \eqref{eq:ud-block-lower-final}. For the remaining bounds,
split the gap indices into the two independent parity classes.  The
Bessel suprema have a finite exponential moment, and Brownian reflection
gives the corresponding Gaussian tail.  Chernoff's inequality, applied
separately to the two parity classes (and hence accounting for the
one-dependence of adjacent gap noises), gives deterministic
$C_Q,C_B<\infty$ such that,
writing $n=|A_{b,k}|$,
\begin{equation*}
\begin{split}
 &\prob\left(
 \overline\Sigma_{A_{b,k}}\left(\boldsymbol{\mathsf Q}^{t_{j-1},t_j}\right)
 >\frac{\rho_*}{4}\right)
 +\prob\left(
 \overline\Sigma_{A_{b,k}}\left(\boldsymbol{\mathsf Q}^{0,T}\right)>C_Q\right)\\
 &\quad
 +\prob\left(
 \max_{i\in\left[\widetilde m_{b-1/2}^{\,k},
 \widetilde m_{b+1/2}^{\,k}\right]\cap\mathbb Z}
 \sup_{0\leq u\leq T}\left|\sfB_i\left(u\right)\right|>C_Bn^{1/2}\right)
 \leq c\exp\left(-n/c\right).
\end{split}
\end{equation*}
Here we also used \eqref{eq:ud-short-time}; the finitely many $j$'s
are absorbed into $c$.  In particular, the lower block-size bound
(6.2) of \cite{MR3568040} makes the sum of these probabilities over
$b\in\mathbb{V}$ and all sufficiently large $k$ finite.
Borel--Cantelli gives
\eqref{eq:ud-block-Q-small}--\eqref{eq:ud-block-B}; this is the argument
used for Tsai's (6.18)--(6.19). Taking the maximum of the finitely many
random thresholds and $2$ produces one $\mathsf K$ satisfying
\eqref{eq:ud-common-block-event}.
In the remainder of the proof, $\mathsf C_T$ denotes an almost surely finite
random constant, independent of $N$, $m$, and $j$, whose value may
increase from line to line.

For $j=1,\ldots,J$, define the coordinatewise time infimum, distinct
from the initial-data tail envelope,
\[
\mathsf L_v^{\min,j}\defeq \inf_{t_{j-1}\leq u\leq t_j}\underline{\mathsf L}_v\left(u\right).
\]
Equation \eqref{eq:ud-Q-backward} gives
$\underline{\mathsf L}_v(t_j)-\mathsf L_v^{\min,j}
\leq\mathsf Q^{t_{j-1},t_j}_v$. Hence
\begin{equation}\label{eq:ud-block-inf-lower}
\overline\Sigma_{A_{b,\mathsf K}}\left(\boldsymbol{\mathsf L}^{\min,j}\right)
\geq\frac{\rho_*}{2},
\qquad b\in\mathbb{V},\quad j=1,\ldots,J.
\end{equation}

Put $\delta=\rho_*/2$. For integers $i$ define
\[
h_{\left(i,+\infty\right)}\left(\mathbf y\right)\defeq \inf_{\substack{r\in\mathbb Z\\r>i}}\frac{y_{\left(i,r\right)}}{r-i},\qquad
h_{\left(i,-\infty\right)}\left(\mathbf y\right)\defeq \inf_{\substack{r\in\mathbb Z\\r<i}}\frac{y_{\left(i,r\right)}}{i-r}.
\]
Set $c_*=\delta/[2(1+R)]$. For each $b$ and $j$ there are integer
labels
\[
 \mathsf I^+_{b,j}\in\left[\widetilde m_{b-1/2}^{\,\mathsf K},
 \widetilde m_{b+1/2}^{\,\mathsf K}\right),\qquad
 \mathsf I^-_{b,j}\in\left(\widetilde m_{b-1/2}^{\,\mathsf K},
 \widetilde m_{b+1/2}^{\,\mathsf K}\right]
\]
such that
\begin{equation}\label{eq:ud-directed-seeds}
h_{\left(\mathsf I^+_{b,j},+\infty\right)}\left(\boldsymbol{\mathsf L}^{\min,j}\right)\geq c_*,
\qquad
h_{\left(\mathsf I^-_{b,j},-\infty\right)}\left(\boldsymbol{\mathsf L}^{\min,j}\right)\geq c_*.
\end{equation}
For completeness, consider the plus direction. A directed interval
beginning at the left endpoint of $A_{b,\mathsf K}$ and ending beyond its right
endpoint consists of complete blocks of total size $S$ and at most one
terminal block fragment of size $P$. By
\eqref{eq:ud-block-inf-lower}, the complete blocks have mass at least
$\delta S$, while \eqref{eq:ud-block-geometry} gives $P\leq RS$.
Thus every such directed average is at least
$\delta/(1+R)>c_*$. Lemma~4.4 of \cite{MR3568040}, with terminal
endpoint $+\infty$, gives $\mathsf I^+_{b,j}$. Reflection gives $\mathsf I^-_{b,j}$.
Choose the least admissible label in each of these finite blocks. The
conditions defining admissibility involve countably many measurable
averages, so $\mathsf I^{\pm}_{b,j}$ are random variables. These selected labels and
$c_*$ do not depend on $N$.

\smallskip
\noindent\emph{Step 3: the exact finite summed-gap inequalities.}
Let $i_N^-\leq r<s\leq i_N^+$ be integers and let
$\mathbf z=(z_v)_{v\in\mathcal V_N}$ be a positive finite gap
configuration. Define
\begin{align}
\eta^{\mathrm{up}}_{\left(r,s\right)}\left(\mathbf z\right)
&\defeq \frac12\sum_{i=r+1}^{s}\frac1{z_{\left(r,i\right)}}
+\frac12\sum_{i=r}^{s-1}\frac1{z_{\left(i,s\right)}},
                                                        \label{eq:ud-eta-up}\\
\eta^{\mathrm{lw},N}_{\left(r,s\right)}\left(\mathbf z\right)
&\defeq \frac12\sum_{i=i_N^-}^{r-1}
\frac{z_{\left(r,s\right)}}{z_{\left(i,r\right)}z_{\left(i,s\right)}}
+\frac12\sum_{i=s+1}^{i_N^+}
\frac{z_{\left(r,s\right)}}{z_{\left(r,i\right)}z_{\left(s,i\right)}}.                  \label{eq:ud-eta-lw}
\end{align}
Empty sums are zero. If
$\mathbf z=\mathfrak q(\mathbf x)|_{\mathcal V_N}$ for a finite
particle configuration with label set $\mathcal I_N$, then telescoping
the particle drifts gives
\[
 \mu_s\left(\mathbf x\right)-\mu_r\left(\mathbf x\right)
 =2\left\{\eta^{\mathrm{up}}_{\left(r,s\right)}\left(\mathfrak q\left(\mathbf x\right)\right)
 -\eta^{\mathrm{lw},N}_{\left(r,s\right)}\left(\mathfrak q\left(\mathbf x\right)\right)\right\}.
\]
Indeed, the terms indexed by $r<i<s$ have positive sign, while an
exterior particle $i<r$ contributes
$-z_{(r,s)}/(z_{(i,r)}z_{(i,s)})$ and an exterior particle $i>s$
contributes $-z_{(r,s)}/(z_{(r,i)}z_{(s,i)})$.
Consequently,
\begin{equation}\label{eq:ud-finite-summed-gap}
\begin{split}
\mathsf L^{\left(N\right)}_{\left(r,s\right)}\left(u\right)-\mathsf L^{\left(N\right)}_{\left(r,s\right)}\left(u_0\right)
&=\left(\sfB_s\left(u\right)-\sfB_s\left(u_0\right)\right)
-\left(\sfB_r\left(u\right)-\sfB_r\left(u_0\right)\right)\\
&\quad+\beta\int_{u_0}^u
\left\{\eta^{\mathrm{up}}_{\left(r,s\right)}\left(\bsfL^{\left(N\right)}\left(w\right)\right)
-\eta^{\mathrm{lw},N}_{\left(r,s\right)}\left(\bsfL^{\left(N\right)}\left(w\right)\right)\right\}
\dif w.
\end{split}
\end{equation}
Both interactions in \eqref{eq:ud-eta-up}--\eqref{eq:ud-eta-lw} are
nonnegative. If $r=i_N^-$, the first sum in \eqref{eq:ud-eta-lw} is
absent; if $s=i_N^+$, the second is absent. Thus a physical endpoint
removes nonnegative summands from $\eta^{\mathrm{lw},N}$ and improves
the lower estimate from \eqref{eq:ud-finite-summed-gap};
$\eta^{\mathrm{up}}$ is unchanged.

Let $\mathcal C\subset\mathcal A=(r,s)\cap\mathbb{V}
\subset\mathcal V_N$ be intervals and, for $u_0\leq u\leq u_1$, put
\[
\mathsf O_{\left(r,s\right)}\left(u_0,u_1\right)\defeq \sup_{u_0\leq w\leq u_1}\left|\sfB_r\left(w\right)-\sfB_r\left(u_0\right)\right|
+\sup_{u_0\leq w\leq u_1}\left|\sfB_s\left(w\right)-\sfB_s\left(u_0\right)\right|.
\]
If $\mathcal A=(r,s)\cap\mathbb{V}$, we abbreviate
$\mathsf O_{\mathcal A}=\mathsf O_{(r,s)}$ and use the analogous
abbreviation for the two $\eta$'s.
For a finite $\mathcal D\subset\mathbb V$, we also write
$\sfL^{(N)}_{\mathcal D}=\sum_{v\in\mathcal D}\sfL_v^{(N)}$.
Rearrange \eqref{eq:ud-finite-summed-gap}, use
\eqref{eq:ud-Q-increment} on $\mathcal A\setminus\mathcal C$ for the
lower bound, and use positivity on that buffer for the upper bound.
All these inequalities are taken on the common full-probability event
for the countable choices of $N,r,s$ and the fixed time grid; continuity
then covers every $u$ in the relevant time interval.
For every $u_0\leq u\leq u_1$ this gives
\begin{subequations}\label{eq:ud-finite-summed-inequalities}
\begin{align}
\mathsf L^{\left(N\right)}_{\mathcal C}\left(u\right)
&\geq\mathsf L^{\left(N\right)}_{\mathcal C}\left(u_0\right)
-\beta\int_{u_0}^u\eta^{\mathrm{lw},N}_{\left(r,s\right)}
\left(\bsfL^{\left(N\right)}\left(w\right)\right)\dif w
-\mathsf O_{\left(r,s\right)}\left(u_0,u_1\right)
-\sum_{v\in\mathcal A\setminus\mathcal C}\mathsf Q^{u_0,u_1}_v,
                                                        \label{eq:ud-finite-lower}\\
\mathsf L^{\left(N\right)}_{\mathcal C}\left(u\right)
&\leq\mathsf L^{\left(N\right)}_{\mathcal C}\left(u_0\right)
+\beta\int_{u_0}^u\eta^{\mathrm{up}}_{\left(r,s\right)}
\left(\bsfL^{\left(N\right)}\left(w\right)\right)\dif w
+\mathsf O_{\left(r,s\right)}\left(u_0,u_1\right)
+\sum_{v\in\mathcal A\setminus\mathcal C}\mathsf L^{\left(N\right)}_v\left(u_0\right).
                                                        \label{eq:ud-finite-upper}
\end{align}
\end{subequations}
These are the unnormalized finite-volume counterparts of Tsai's
(6.5)--(6.6), as used in his (6.32)--(6.35), with the
physical-boundary contributions explicit.

\smallskip
\noindent\emph{Step 4: the core--fringe decomposition.}
The core consists of complete blocks with both neighboring blocks in
the finite system; the fringe consists of the remaining boundary pieces.
We prove the result for $m>0$; the case $m<0$ follows by reflection.
By \eqref{eq:ud-support-exhausts}, after discarding finitely many $N$
we may assume that the four central blocks lie in the finite support:
\begin{equation}\label{eq:ud-central-buffer}
 A_{b,\mathsf K}\subset\mathcal V_N,\qquad
 b\in\left\{-3/2,-1/2,1/2,3/2\right\}.
\end{equation}
Since $\mathsf K<\infty$ almost surely, the number of discarded systems is
almost surely finite.

For $m>0$, put $m_N=m\wedge i_N^+$. If $m_N\leq0$, the sum is empty. If
$m_N>0$, then $\mathsf D_m^{(N)}=\mathsf D_{m_N}^{(N)}$. Inside
$(0,m_N)\cap\mathbb V$, let $\mathcal C_{N,m}$ be the union of all
complete blocks $A_{b,\mathsf K}$ such that both $A_{b-1,\mathsf K}$ and $A_{b+1,\mathsf K}$
are contained in $\mathcal V_N$, and put
\[
\mathcal E_{N,m}\defeq \left(\left(0,m_N\right)\cap\mathbb V\right)\setminus\mathcal C_{N,m}.
\]
Since $0$ is a block boundary and \eqref{eq:ud-central-buffer} supplies the left neighbor of
the first positive block, every complete block in $(0,m_N)\cap\mathbb V$,
except possibly the last one, has both neighbors contained in $V_N$.
These blocks therefore belong to $\mathcal{C}_{N,m}$. Thus the core is an interval,
possibly empty, and the fringe is covered by at most two blocks: the
last complete block, if excluded from the core, and the block containing
the terminal partial piece, if present. If the target contains no
complete block, it lies entirely within the first positive block.
Consequently,  \eqref{eq:ud-block-size} gives
\begin{equation}\label{eq:ud-fringe-size}
\left|\mathcal E_{N,m}\right|
\leq c_{\kappa,\mathsf K}\left(1+m_N^{1-\kappa}\right).
\end{equation}
In particular, when $m_N=i_N^+$, the block cut by the right physical
endpoint and the adjacent unbuffered complete block are part of the
fringe; no interaction beyond that endpoint is introduced.

By the uniform upper-gap hypothesis
$\sup_{N,v\in\mathcal V_N}l_v^{(N)}<\infty$, equations
\eqref{eq:ud-Q-comparison}, \eqref{eq:ud-common-block-event}, and
\eqref{eq:ud-fringe-size} imply
\begin{equation}\label{eq:ud-fringe-bound}
\sup_{0\leq u\leq T}
\left|\sum_{v\in\mathcal E_{N,m}}
\left(\mathsf L^{\left(N\right)}_v\left(u\right)-l_v^{\left(N\right)}\right)\right|
\leq \mathsf C_T\left(1+m_N^{1-\kappa}\right).
\end{equation}
Indeed, the left-hand side is at most
$\sum_{v\in\mathcal E_{N,m}}(l_v^{(N)}+\mathsf Q^{0,T}_v)$, and each
fringe piece lies in a block controlled by
\eqref{eq:ud-block-Q-global}.

Suppose the core is nonempty, and denote its first and last block
indices by $b_-$ and $b_+$. For $j=1,\ldots,J$, set
\[
\mathcal A^{\mathrm{lw}}_j\defeq \left(\mathsf I^-_{b_--1,j},\mathsf I^+_{b_++1,j}\right)\cap\mathbb V,\qquad
\mathcal A^{\mathrm{up}}_j\defeq \left(\mathsf I^+_{b_--1,j},\mathsf I^-_{b_++1,j}\right)\cap\mathbb V.
\]
Both contain $\mathcal C_{N,m}$ and are contained in $\mathcal V_N$.
Their complements of the core lie in the two neighboring buffer
blocks. Uniformly in $N,m,j$, the block bounds therefore give
\begin{align}
\left|\mathcal A^{\mathrm{lw}}_j\setminus\mathcal C_{N,m}\right|
+\left|\mathcal A^{\mathrm{up}}_j\setminus\mathcal C_{N,m}\right|
&\leq \mathsf C_T\left(1+m_N^{1-\kappa}\right),                            \label{eq:ud-buffer-size}\\
\mathsf O_{\mathcal A^{\mathrm{lw}}_j}\left(t_{j-1},t_j\right)
+\mathsf O_{\mathcal A^{\mathrm{up}}_j}\left(t_{j-1},t_j\right)
&\leq \mathsf C_T\left(1+m_N^{1-\kappa}\right),                            \label{eq:ud-buffer-B}\\
\sum_{v\in\mathcal A^{\mathrm{lw}}_j\setminus\mathcal C_{N,m}}
\mathsf Q^{t_{j-1},t_j}_v
+\sum_{v\in\mathcal A^{\mathrm{up}}_j\setminus\mathcal C_{N,m}}
\mathsf L^{\left(N\right)}_v\left(t_{j-1}\right)
&\leq \mathsf C_T\left(1+m_N^{1-\kappa}\right).                            \label{eq:ud-buffer-mass}
\end{align}
For the last inequality, also use
$\mathsf L^{(N)}_v(t_{j-1})\leq l_v^{(N)}+\mathsf Q^{0,T}_v$.

It remains to bound the interactions. If
$\mathcal A^{\mathrm{up}}_j=(r,s)\cap\mathbb V$, then
\eqref{eq:ud-mixed-comparison} and \eqref{eq:ud-directed-seeds} imply,
for $u\in[t_{j-1},t_j]$,
\[
\mathsf L^{\left(N\right)}_{\left(r,i\right)}\left(u\right)\geq c_*\left(i-r\right)\quad\left(r<i\leq s\right),
\qquad
\mathsf L^{\left(N\right)}_{\left(i,s\right)}\left(u\right)\geq c_*\left(s-i\right)\quad\left(r\leq i<s\right).
\]
Consequently,
\begin{equation}\label{eq:ud-upper-interaction}
\eta^{\mathrm{up}}_{\mathcal A^{\mathrm{up}}_j}
\left(\bsfL^{\left(N\right)}\left(u\right)\right)
\leq c_*^{-1}H_{s-r}\leq \mathsf C_T\log\left(2+m_N\right),
\end{equation}
where $H_n=\sum_{q=1}^nq^{-1}$. Here
$s-r=|\mathcal A_j^{\mathrm{up}}|\leq
m_N+\mathsf C_T(1+m_N^{1-\kappa})\leq \mathsf C_T(1+m_N)$ by
\eqref{eq:ud-buffer-size}, which justifies the last bound.

Next write $\mathcal A^{\mathrm{lw}}_j=(r,s)\cap\mathbb{V}$ and
$\mathsf Z=\mathsf L^{(N)}_{(r,s)}(u)$. The pointwise estimate
\eqref{eq:ud-Q-comparison}, the initial upper-gap bound, and
\eqref{eq:ud-block-Q-global} apply because
\[
 \left|\mathcal A^{\mathrm{lw}}_j\right|\leq \mathsf C_T\left(1+m_N\right).
\]
Indeed, this interval is contained between the two buffer blocks of
$(0,m_N)$; cover it by the intervening complete blocks and two terminal
fragments. It follows that
\begin{equation}\label{eq:ud-z-upper}
\mathsf Z\leq\sum_{v\in\left(r,s\right)}
\left(l_v^{\left(N\right)}+\mathsf Q^{0,T}_v\right)
\leq \mathsf C_T\left(1+m_N\right).
\end{equation}
The outward one-sided bounds and \eqref{eq:ud-mixed-comparison} give
\[
\mathsf L^{\left(N\right)}_{\left(i,r\right)}\left(u\right)\geq c_*\left(r-i\right)\quad\left(i_N^-\leq i<r\right),
\qquad
\mathsf L^{\left(N\right)}_{\left(s,i\right)}\left(u\right)\geq c_*\left(i-s\right)\quad\left(s<i\leq i_N^+\right).
\]
Insert these bounds in \eqref{eq:ud-eta-lw} and extend both finite
exterior sums to infinity. Then
\begin{align}
\eta^{\mathrm{lw},N}_{\mathcal A^{\mathrm{lw}}_j}
\left(\bsfL^{\left(N\right)}\left(u\right)\right)
&\leq\sum_{q=1}^{\infty}
\frac{\mathsf Z}{\left(\mathsf Z+c_*q\right)c_*q}
=\frac1{c_*}\sum_{q=1}^{\infty}
\left(\frac1q-\frac1{q+\mathsf Z/c_*}\right)                         \notag\\
&\leq\frac1{c_*}\left(1+\log\left(1+\mathsf Z/c_*\right)\right)
\leq \mathsf C_T\log\left(2+m_N\right).                         \label{eq:ud-lower-interaction}
\end{align}
The first inequality also makes explicit why either finite physical
endpoint can only improve this estimate.

Apply \eqref{eq:ud-finite-lower} with
$\mathcal A=\mathcal A^{\mathrm{lw}}_j$ and
\eqref{eq:ud-finite-upper} with
$\mathcal A=\mathcal A^{\mathrm{up}}_j$. Equations
\eqref{eq:ud-buffer-size}--\eqref{eq:ud-lower-interaction} yield
\[
\sup_{t_{j-1}\leq u\leq t_j}
\left|\mathsf L^{\left(N\right)}_{\mathcal C_{N,m}}\left(u\right)
-\mathsf L^{\left(N\right)}_{\mathcal C_{N,m}}\left(t_{j-1}\right)\right|
\leq \mathsf C_T\left(1+m_N^{1-\kappa}+\log\left(2+m_N\right)\right).
\]
The number $J$ depends only on $T,\beta,\rho_*$ and not on $N,m$.
Iterating over $j=1,\ldots,J$, and using
$\log(2+r)\leq C_\kappa(1+r^{1-\kappa})$, gives
\begin{equation}\label{eq:ud-core-bound}
\sup_{0\leq u\leq T}
\left|\sum_{v\in\mathcal C_{N,m}}
\left(\mathsf L^{\left(N\right)}_v\left(u\right)-l_v^{\left(N\right)}\right)\right|
\leq \mathsf C_T\left(1+m_N^{1-\kappa}\right).
\end{equation}
Together with \eqref{eq:ud-fringe-bound}, this proves
\begin{equation}\label{eq:ud-unscaled-final}
\sup_{0\leq u\leq T}\left|\mathsf D_m^{\left(N\right)}\left(u\right)\right|
\leq \mathsf C_T\left(1+m_N^{1-\kappa}\right),
\qquad m>0,\quad N\text{ sufficiently large}.
\end{equation}

If $m_N\geq1$, then $m_N\leq m$ and $\kappa-1<0$, so
\[
m^{\kappa-1}\left|\mathsf D_m^{\left(N\right)}\left(u\right)\right|
\leq \mathsf C_Tm_N^{\kappa-1}\left(1+m_N^{1-\kappa}\right)\leq2\mathsf C_T.
\]
For $m<0$, reflection gives the same bound with $m_N=m\vee i_N^-$
and absolute values. An empty sum contributes zero.

Finally, only almost surely finitely many $N$ were discarded. For each
such finite system there are only finitely many distinct partial gap
sums, all continuous on $[0,T]$; beyond a physical endpoint the partial
sum is constant while $|m|^{\kappa-1}$ decreases. Their contribution
to the left-hand side of \eqref{eq:uniform-density-error} is therefore
almost surely finite and can be absorbed into $\mathsf C_T$.
For $r\in\mathbb N$, let $\mathsf Z_r$ be the supremum on the left of
\eqref{eq:uniform-density-error} with $T=r$. Continuity permits the time
supremum to be restricted to rational times, so $\mathsf Z_r$ is measurable.
The preceding proof gives $\prob[\mathsf Z_r<\infty]=1$. Thus
\begin{equation*}
 \Omega_*\defeq\bigcap_{r\ge1}\left\{\mathsf Z_r<\infty\right\},\qquad
 \mathsf C_T\defeq1+\mathsf Z_{\lceil T\rceil}
\end{equation*}
give a common probability-one event and increasing constants as asserted.
\end{proof}
We now express the discrepancy of the initial particle tails in gap coordinates. This identifies the constant that survives when the finite interaction drifts converge.
For $M\in\mathbb N$, define
\begin{equation}\label{eq:def-psi-tail}
 \psi_{0,>M}\left(\bl\right)
 \defeq\lim_{R\to\infty}
 \left(
 \sum_{M<k\leq R}\frac{1}{l_{\left(0,-k\right)}}
 -\sum_{M<k\leq R}\frac{1}{l_{\left(0,k\right)}}
 \right),
\end{equation}
whenever the limit exists, and put
\begin{equation*}
 \psi_{0,\leq M}\left(\bl\right)\defeq\sum_{k=1}^M
 \left(\frac1{l_{\left(0,-k\right)}}-\frac1{l_{\left(0,k\right)}}\right).
\end{equation*}
\begin{proposition}\label{from inverse of point to inverse of gap}
Under the hypotheses of Theorem \ref{theo sde theorem},
\begin{equation}\label{eq:initial-tail-gamma}
 \lim_{M\to\infty}\limsup_{N\to\infty}
 \left|\psi_{0,>M}\left(\bl^{\left(N\right)}\right)+\gamma\right|=0.
\end{equation}
\end{proposition}
\begin{proof}
Coordinatewise convergence gives
\begin{equation*}
 C_0\defeq\sup_N\left|x_0^{\left(N\right)}\right|<\infty.
\end{equation*}
The lower-envelope hypothesis gives deterministic constants $a>0$
and $M_0\in\mathbb N$ such that, whenever the relevant cumulative gap
is finite,
\begin{equation}\label{eq lower bound for l^UN, l}
 l^{\left(N\right)}_{\left(0,\sigma k\right)}\geq ak,
 \qquad \sigma\in\left\{-1,1\right\},\quad k\geq M_0,
\end{equation}
uniformly in $N$; the same estimate holds for $\bl$.  Increasing
$M_0$ if necessary, we also have
$|x_{\sigma k}^{(N)}|\geq ak/2$ whenever $\sigma k\in\mathcal I_N$.
Consequently,
\begin{equation}\label{eq:psi-particle-tail}
 \psi_{0,>M}\left(\bl^{\left(N\right)}\right)
 =-\sum_{\substack{i\in\mathcal I_N\\\left|i\right|>M}}
   \frac{1}{x_i^{\left(N\right)}}+e_{N,M},
\end{equation}
where
\begin{equation*}
 e_{N,M}\defeq -\sum_{\substack{i\in\mathcal I_N\\\left|i\right|>M}}
 \left\{
 \frac{1}{x_i^{\left(N\right)}-x_0^{\left(N\right)}}-\frac{1}{x_i^{\left(N\right)}}
 \right\}
\end{equation*}
and, for $M\geq M_0$,
\begin{equation}\label{eq:psi-particle-error}
 \left|e_{N,M}\right|
 \leq
 \sum_{\substack{i\in\mathcal I_N\\\left|i\right|>M}}
 \frac{\left|x_0^{\left(N\right)}\right|}
 {\left|x_i^{\left(N\right)}\right|\,\left|x_i^{\left(N\right)}-x_0^{\left(N\right)}\right|}
 \leq \frac{C}{M}.
\end{equation}
For $M$ sufficiently large, uniformly for all $N$ after the discarded
initial segment, the particles with $|i|>M$ lie outside $[-b,b]$. Put
\begin{equation*}
 F_{N,M}\defeq \sum_{\substack{i\in\mathcal I_N,\ \left|i\right|\leq M\\
                         \left|x_i^{\left(N\right)}\right|>b}}\frac1{x_i^{\left(N\right)}},
 \qquad
 F_M\defeq \sum_{\substack{\left|i\right|\leq M\\\left|x_i\right|>b}}\frac1{x_i}.
\end{equation*}
Then the finite-sum algebra in \eqref{eq:psi-particle-tail}, and its
counterpart for $\bl$, reads
\begin{equation*}
 \psi_{0,>M}\left(\bl^{\left(N\right)}\right)
 =-\gamma_{1,b}\left(\bx^{\left(N\right)}\right)+F_{N,M}+e_{N,M},
 \qquad
 \psi_{0,>M}\left(\bl\right)=-\gamma_{1,b}\left(\bx\right)+F_M+e_M,
\end{equation*}
where $|e_{N,M}|+|e_M|\leq C/M$.  Since $b$ is admissible,
coordinatewise convergence gives $F_{N,M}\to F_M$ for each fixed
$M$, while \eqref{eq def of gamma} gives
$\gamma_{1,b}(\bx^{(N)})-\gamma_{1,b}(\bx)\to\gamma$. Therefore
\begin{equation*}
 \limsup_{N\to\infty}
 \left|\psi_{0,>M}\left(\bl^{\left(N\right)}\right)+\gamma-\psi_{0,>M}\left(\bl\right)\right|\leq\frac{C}{M},
 \qquad
 \limsup_{N\to\infty}\left|\psi_{0,>M}\left(\bl^{\left(N\right)}\right)+\gamma\right|
 \leq \left|\psi_{0,>M}\left(\bl\right)\right|+\frac{C}{M}.
\end{equation*}
The regular tail estimate
$|\psi_{0,>M}(\bl)|\leq C M^{-\kappa}$ now proves
\eqref{eq:initial-tail-gamma}.
\end{proof}
The next estimate shows that the far interaction tail changes little on a compact time interval, uniformly over the finite systems. Together with the initial-tail identity, it determines the limiting drift.
\begin{lemma}\label{lem:uniform-interaction-tails}
  Under the hypotheses and notation of Theorem \ref{theo sde theorem}, let
$\bsfL^\UN(\cdot)$ be the finite gap process starting from $\bl^\UN$.
For every $T>0$, there exist an almost surely finite random integer
$\mathsf M_T$ and an almost surely finite random variable $\mathsf C_T$ such that, almost surely,
\begin{equation}
\label{eq:uniform-interaction-tails}
\sup_{N\in\mathbb N}\sup_{t\in\left[0,T\right]}\left|
\psi_{0,>M}\left(\bsfL^\UN\left(t\right)\right)-
\psi_{0,>M}\left(\bl^\UN\right)\right|\leq \mathsf C_TM^{-\kappa},
\quad M\geq \mathsf M_T.
\end{equation}
\end{lemma}
\begin{proof}
Work on the common event of Lemma \ref{lem:uniform-density-error} and
fix $T>0$. For $\sigma\in\{-1,1\}$ and $k\ge1$, write
$g_{N,k}^{\sigma}=l_{(0,\sigma k)}^{(N)}$ and
$\mathsf G_{N,k}^{\sigma}(t)=\sfL_{(0,\sigma k)}^{(N)}(t)$.
By \eqref{eq lower bound for l^UN, l} and
\eqref{eq:uniform-density-error}, there are deterministic $a>0$, $M_0$,
and an almost surely finite $\mathsf A_T$ such that, whenever
$g_{N,k}^{\sigma}<\infty$ and $k\ge M_0$,
\begin{equation*}
 g_{N,k}^{\sigma}\ge ak,\qquad
 \left|\mathsf G_{N,k}^{\sigma}\left(t\right)-g_{N,k}^{\sigma}\right|
 \le\mathsf A_Tk^{1-\kappa},\qquad N\in\mathbb N,\quad t\in[0,T].
\end{equation*}
Choose
\begin{equation*}
 \mathsf M_T\defeq M_0\vee\left\lceil
 \left(2\mathsf A_T/a\right)^{1/\kappa}\right\rceil.
\end{equation*}
For $k\ge\mathsf M_T$, the same bounds give
\begin{equation*}
 \mathsf G_{N,k}^{\sigma}\left(t\right)\ge ak/2,\qquad
 \left|\frac1{\mathsf G_{N,k}^{\sigma}\left(t\right)}
       -\frac1{g_{N,k}^{\sigma}}\right|
 \le\frac{2\mathsf A_T}{a^2}k^{-1-\kappa}.
\end{equation*}
The initial and current gap configurations have the same finite support.
If $g_{N,k}^{\sigma}=\infty$, both reciprocals are therefore zero.
Summing over both signs and $k>M$, and using
$\sum_{k>M}k^{-1-\kappa}\le M^{-\kappa}/\kappa$, proves
\eqref{eq:uniform-interaction-tails} with
$\mathsf C_T=4\mathsf A_T/(a^2\kappa)$, uniformly in $N$ and $t\in[0,T]$.
\end{proof}
The gap limit determines the particle configuration up to a common translation. We identify that translation by proving convergence of the zeroth particle and its drift.
\begin{proposition}\label{prop equation for x_0}
    Under the hypotheses of Theorem \ref{theo sde theorem}, the finite
    zeroth coordinates converge to the process $\sfX_0$ given by
    \begin{equation}\label{eq equation for X_0}
        \sfX_0\left(t\right) = x_0+\sfB_0\left(t\right) + \frac{\beta}{2}\left(\int_0^t \psi_0\left(\bsfL\left(s\right)\right)\dif s - \gamma t\right).
    \end{equation}
    Moreover, for any $t>0$, we have, almost surely,
    \begin{equation}\label{eq convergence of X_0}
        \lim_{N\to\infty}\sup_{s\in\left[0,t\right]} \left|{\sfX}_0^\UN\left(s\right)-\sfX_0\left(s\right)\right| =0.
    \end{equation}
\end{proposition}
\begin{proof}
Fix $T>0$ and work on an event of probability one on which Corollary
\ref{cor convergence of L^N} and Lemma
\ref{lem:uniform-interaction-tails} hold on $[0,T]$.  Recall from
the gap identity at the beginning of this subsection that
$\mathfrak q(\boldsymbol{\sfX}^{(N)}(\cdot))=\bsfL^{(N)}(\cdot)$.
Consequently,
\begin{equation*}
    \sfX_0^{\left(N\right)}\left(t\right)=x_0^{\left(N\right)}+\sfB_0\left(t\right)+\frac{\beta}{2}\int_0^t \psi_0\left(\bsfL^{\left(N\right)}\left(s\right)\right)\,\dif s.
\end{equation*}

We first identify the limit of the drift, uniformly on $[0,T]$. For
$M\in\mathbb N$, the core--tail decomposition gives
\begin{equation*}
 \psi_0\left(\bsfL^{\left(N\right)}\left(s\right)\right)
 =\psi_{0,\leq M}\left(\bsfL^{\left(N\right)}\left(s\right)\right)
 +\psi_{0,>M}\left(\bsfL^{\left(N\right)}\left(s\right)\right),
\end{equation*}
and the analogous identity holds for $\bsfL$.  For each fixed $M$, Corollary \ref{cor convergence of L^N} gives uniform convergence of every cumulative gap
$\sfL_{(0,i)}^{(N)}$ with $0<|i|\leq M$.  Since the limiting cumulative
gaps are strictly positive and continuous on $[0,T]$, their minima on
this interval are positive.  Hence
\begin{equation*}
    \lim_{N\to\infty}
    \sup_{0\leq s\leq T}
    \left|\psi_{0,\leq M}\left(\bsfL^{\left(N\right)}\left(s\right)\right)-\psi_{0,\leq M}\left(\bsfL\left(s\right)\right)\right|=0.
\end{equation*}

The initial tail is controlled by \eqref{eq:initial-tail-gamma}.
Lemma \ref{lem:uniform-interaction-tails} gives almost surely finite $\mathsf M_T$ and
$\mathsf C_T$ such that, for $M\geq \mathsf M_T$,
\begin{equation*}
    \sup_{N\in\mathbb N}\sup_{0\leq s\leq T}
    \left|\psi_{0,>M}\left(\bsfL^{\left(N\right)}\left(s\right)\right)
    -\psi_{0,>M}\left(\bl^{\left(N\right)}\right)\right|\leq \mathsf C_TM^{-\kappa}.
\end{equation*}
It follows that
\begin{equation}\label{eq L^N+gamma is closed to l^N+gamma}
    \limsup_{N\to\infty}\sup_{0\leq s\leq T}\left|\psi_{0,>M}\left(\bsfL^{\left(N\right)}\left(s\right)\right)+\gamma\right|\leq \mathsf C_TM^{-\kappa}
    +\limsup_{N\to\infty}
    \left|\psi_{0,>M}\left(\bl^{\left(N\right)}\right)+\gamma\right|.
\end{equation}

We also need the corresponding tail estimate for the limiting gap process.  Put
\begin{equation*}
    \mathsf A_T\defeq \sup_{0\leq s\leq T}
    \mathfrak D_{\kappa,\rho}\left(\bsfL\left(s\right)\right)<\infty.
\end{equation*}
For all sufficiently large $k$, uniformly in $s\in[0,T]$,
\begin{equation*}
    \sfL_{\left(0,\pm k\right)}\left(s\right)\geq\frac{\rho k}{2},
    \quad
    \left|
    \frac{1}{\sfL_{\left(0,-k\right)}\left(s\right)}
    -\frac{1}{\sfL_{\left(0,k\right)}\left(s\right)}
    \right|
    \leq\frac{8\mathsf A_T}{\rho^2}k^{-1-\kappa}.
\end{equation*}
Therefore, for all sufficiently large $M$,
\begin{equation}\label{eq psi^M(L(s)) is small}
    \sup_{0\leq s\leq T}
    \left|\psi_{0,>M}\left(\bsfL\left(s\right)\right)\right|
    \leq \mathsf C_T M^{-\kappa}.
\end{equation}
In particular, $\psi_0(\bsfL(\cdot))$ is the uniform limit on $[0,T]$
of the continuous functions $\psi_{0,\leq M}(\bsfL(\cdot))$, and is therefore continuous.

Indeed, the decomposition and the triangle inequality give, for every
$s\leq T$,
\begin{equation*}
\begin{split}
 \left|\psi_0\left(\bsfL^{\left(N\right)}\left(s\right)\right)-\psi_0\left(\bsfL\left(s\right)\right)+\gamma\right|
 \leq{}&\left|\psi_{0,\leq M}\left(\bsfL^{\left(N\right)}\left(s\right)\right)-\psi_{0,\leq M}\left(\bsfL\left(s\right)\right)\right|\\
 &\quad+\left|\psi_{0,>M}\left(\bsfL^{\left(N\right)}\left(s\right)\right)+\gamma\right|
 +\left|\psi_{0,>M}\left(\bsfL\left(s\right)\right)\right|.
\end{split}
\end{equation*}
Taking the supremum in $s$, then the limsup in $N$, and finally letting
$M\to\infty$ in \eqref{eq L^N+gamma is closed to l^N+gamma},
\eqref{eq:initial-tail-gamma}, and
\eqref{eq psi^M(L(s)) is small}, we obtain
\begin{equation*}
    \lim_{N\to\infty}
    \sup_{0\leq s\leq T}
    \left|\psi_0\left(\bsfL^{\left(N\right)}\left(s\right)\right)
    -\psi_0\left(\bsfL\left(s\right)\right)+\gamma \right|=0.
\end{equation*}

Let $\sfX_0$ be the process in the statement.
Since $x_0^{(N)}\to x_0$, the finite-system equation and the uniform
drift convergence give
\begin{equation*}
    \begin{aligned}
    \sup_{0\leq t\leq T}
    \left|\sfX_0^{\left(N\right)}\left(t\right)-\sfX_0\left(t\right)\right|\leq \left|x_0^{\left(N\right)}-x_0\right|+\frac{\beta T}{2}\sup_{0\leq s\leq T}
    \left|\psi_0\left(\bsfL^{\left(N\right)}\left(s\right)\right)
    -\psi_0\left(\bsfL\left(s\right)\right)+\gamma
    \right|\longrightarrow0.
    \end{aligned}
\end{equation*}
Thus the limit exists, is uniquely determined by $\bsfL$ and $\sfB_0$,
and satisfies \eqref{eq equation for X_0}.  Taking a countable intersection over
$T\in\mathbb N$ proves \eqref{eq convergence of X_0} on every compact time interval.
\end{proof}
We now reconstruct the remaining particles from the limiting gaps and the zeroth coordinate, completing the proof of Theorem \ref{theo sde theorem}.
\begin{proof}[Proof of Theorem \ref{theo sde theorem}]
Define
\begin{equation*}
 \sfX_i\left(t\right)\defeq 
 \begin{cases}
  \sfX_0\left(t\right)+\sfL_{\left(0,i\right)}\left(t\right),&i>0,\\
  \sfX_0\left(t\right),&i=0,\\
  \sfX_0\left(t\right)-\sfL_{\left(0,i\right)}\left(t\right),&i<0.
 \end{cases}
\end{equation*}
Put
\begin{equation*}
 \sfY_i\left(t\right)\defeq \sfX_i\left(t\right)+\frac{\beta\gamma}{2}t.
\end{equation*}
Then $\mathfrak q(\boldsymbol{\sfY})=\bsfL$ and
\begin{equation*}
 \sfY_0\left(t\right)=x_0+\sfB_0\left(t\right)
 +\frac\beta2\int_0^t\psi_0\left(\bsfL\left(s\right)\right)\dif s.
\end{equation*}
The gap--particle equivalence in Proposition \ref{ISDE e and u}
therefore shows that $\boldsymbol{\sfY}$ is the unshifted Dyson ISDE
solution.  Since $\mu_i$ is invariant under common spatial
translations, $\boldsymbol{\sfX}$ solves \eqref{eq equation for X}.

If $\widetilde{\boldsymbol{\sfX}}$ is another adapted solution with
gap process in $\mathcal P_{\kappa,\rho}^p$, driven by the same
Brownian family, then
$\widetilde{\sfY}_i(t)=\widetilde{\sfX}_i(t)+\beta\gamma t/2$ is a
second regular solution of the unshifted ISDE.  Proposition
\ref{ISDE e and u} gives $\widetilde{\boldsymbol{\sfY}}
=\boldsymbol{\sfY}$ almost surely, and hence
$\widetilde{\boldsymbol{\sfX}}=\boldsymbol{\sfX}$.  Finally,
\eqref{eq convergence of X_0} and convergence of every fixed finite
collection of gaps in \eqref{eq convergence of L^N} prove
\eqref{eq convergence of X_i}.  When $\beta=2$, the relation is
$\sfX_i(t)=\sfY_i(t)-\gamma t$, in agreement with Proposition
\ref{gammadependence}.
\end{proof}
\subsection{Stieltjes-transform dynamics}
We next pass to the Stieltjes-transform formulation.  Recall that
$\mathbb H=\{z\in\mathbb C:\operatorname{Im}z>0\}$. For an open set
$D\subset\mathbb C$, let $\mathcal O(D)$ be the space of holomorphic
functions on $D$ with the topology of locally uniform convergence.
Thus $C([0,T];\mathcal O(D))$ denotes its continuous path space. Let
$\boldsymbol{\sfX}$ and $(\boldsymbol{\sfX}^{(N)})_{N\geq1}$ be as in
Theorem \ref{theo sde theorem}, and fix $T<\infty$.  The paired-series
estimate used repeatedly below is the following: for every compact
$\mathcal K\Subset\mathbb H$ and $A<\infty$, there is
$C_{\mathcal K,A}<\infty$ such that, whenever
$\varepsilon_i\in\mathbb R$ and
$|\varepsilon_i|\leq A|i|^{1-\kappa}$,,
\begin{equation*}
 \sup_{R\geq M}\sup_{z\in\mathcal K}
 \left|
 \sum_{\substack{i\in\mathbb Z\\M\leq\left|i\right|\leq R}}
 \frac{1}{\rho i+\varepsilon_i+z}
 \right|
 \leq C_{\mathcal K,A}M^{-\kappa}.
\end{equation*}
It follows by pairing $i$ with $-i$ and using the mean-value theorem.
The same argument has the following pole-free form. If
$x_i=x_0+\rho i+O(|i|^{1-\kappa})$, then, for every compact
$\mathcal K\Subset\mathbb C\setminus\{x_i:i\in\mathbb Z\}$, there is
$M_{\mathcal K}$ such that
\begin{equation}\label{eq:pole-free-paired-tail}
 \sup_{R\geq M}\sup_{z\in\mathcal K}
 \left|\sum_{M\leq\left|i\right|\leq R}\frac1{x_i-z}\right|
 \leq C_{\mathcal K}M^{-\kappa},\qquad M\geq M_{\mathcal K}.
\end{equation}
When the asymptotic bound and the distance from $\mathcal K$ to the
poles are uniform on a compact time interval, so are the constants in
\eqref{eq:pole-free-paired-tail}.
Since $\mathfrak q(\boldsymbol{\sfX}(\cdot))\in
\mathcal P_{\kappa,\rho}^p$, the following limit exists locally
uniformly on $[0,T]\times\mathbb H$ almost surely:
\begin{equation}\label{eq definition of S S^UN}
 \mathsf S_t\left(z\right)\defeq \lim_{R\to\infty}
 \sum_{\left|i\right|\leq R}\frac{1}{\sfX_i\left(t\right)-z},
 \qquad
 \mathsf S_t^{\left(N\right)}\left(z\right)\defeq 
 \sum_{i\in\mathcal I_N}\frac{1}{\sfX_i^{\left(N\right)}\left(t\right)-z}.
\end{equation}
The same tail discrepancy appears as an additive constant in the Stieltjes transform. Its locally uniform convergence will let us pass the finite transform equation and its spatial derivatives to the limit.
\begin{proposition}\label{prop s transform convergence}
    Let $\boldsymbol{\sfX}$, $(\boldsymbol{\sfX}^{(N)})_{N\geq1}$,
    and $\gamma$ be as in Theorem \ref{theo sde theorem}.  For every
    $T>0$ and compact $\mathcal K\Subset\mathbb H$, almost surely,
    \begin{equation}\label{eq s transform convergence}
       \lim_{N\to\infty} \sup_{t\in\left[0,T\right]}\sup_{z\in\mathcal K}
       \left|\mathsf S_t\left(z\right)+\gamma-\mathsf S^{\left(N\right)}_t\left(z\right)\right|=0.
    \end{equation}
\end{proposition}
\begin{proof}
Fix $\mathcal K\Subset\mathbb H$ and work on the common
probability-one event furnished above.  By Proposition
\ref{prop equation for x_0},
\begin{equation*}
 \mathsf C_{\mathcal K,T}\defeq \sup_N\sup_{0\leq t\leq T}\sup_{z\in\mathcal K}
 \left|\sfX_0^{\left(N\right)}\left(t\right)-z\right|<\infty.
\end{equation*}
For $N,M\in\mathbb N$, define
\begin{equation*}
 \mathsf R_{t,M}^{\left(N\right)}\left(z\right)\defeq \sum_{\substack{i\in\mathcal I_N\\\left|i\right|>M}}
 \left\{
 \frac1{\sfX_i^{\left(N\right)}\left(t\right)-z}
 -\frac1{\sfX_i^{\left(N\right)}\left(t\right)-\sfX_0^{\left(N\right)}\left(t\right)}
 \right\}.
\end{equation*}
The mixed comparison with the fixed regular lower process
$\bsfL^{\tinf,(N_0)}$ gives random $\mathsf c_T>0$ and $\mathsf M_T<\infty$ such
that, for every $N\geq N_0$, $t\leq T$, and finite coordinate
$i\in\mathcal I_N$ with $|i|\geq \mathsf M_T$,
\begin{equation*}
 \left|\sfX_i^{\left(N\right)}\left(t\right)-\sfX_0^{\left(N\right)}\left(t\right)\right|\geq \mathsf c_T\left|i\right|.
\end{equation*}
The finitely many systems $N<N_0$ can be absorbed by increasing
$\mathsf M_T$.  Increase it again until $\mathsf c_T|i|\geq2\mathsf C_{\mathcal K,T}$ for
$|i|\geq \mathsf M_T$.  The resolvent identity then gives
\begin{equation}\label{eq bd of ep^N_tMz}
 \sup_N\sup_{t\leq T}\sup_{z\in\mathcal K}
 \left|\mathsf R_{t,M}^{\left(N\right)}\left(z\right)\right|
 \leq 2\mathsf C_{\mathcal K,T}\mathsf c_T^{-2}
 \sum_{\left|i\right|>M}\left|i\right|^{-2}
 \leq\frac{\mathsf C'_{\mathcal K,T}}M.
\end{equation}
Here and below the displayed constants may be random.  Directly from
the definitions,
\begin{equation}\label{eq:finite-S-decomposition}
 \sfS_t^{\left(N\right)}\left(z\right)
 =\sum_{\substack{i\in\mathcal I_N\\\left|i\right|\leq M}}
 \frac1{\sfX_i^{\left(N\right)}\left(t\right)-z}
 -\psi_{0,>M}\left(\bsfL^{\left(N\right)}\left(t\right)\right)
 +\mathsf R_{t,M}^{\left(N\right)}\left(z\right).
\end{equation}
The same argument for the regular limiting process gives
\begin{equation*}
 \sfS_t\left(z\right)
 =\sum_{\left|i\right|\leq M}\frac1{\sfX_i\left(t\right)-z}
 -\psi_{0,>M}\left(\bsfL\left(t\right)\right)+\mathsf R_{t,M}\left(z\right),
 \qquad
 \sup_{t\leq T,z\in\mathcal K}\left|\mathsf R_{t,M}\left(z\right)\right|\leq\mathsf C/M.
\end{equation*}
For fixed $M$, every label $|i|\leq M$ belongs to $\mathcal I_N$ for
large $N$, and \eqref{eq convergence of X_i} gives uniform convergence
of the finite resolvent sums.  Equations
\eqref{eq:initial-tail-gamma},
\eqref{eq L^N+gamma is closed to l^N+gamma},
\eqref{eq psi^M(L(s)) is small}, and
\eqref{eq bd of ep^N_tMz} control the remaining terms.  Letting first
$N\to\infty$ and then $M\to\infty$ proves
\eqref{eq s transform convergence}.
\end{proof}
The same field extends to $\mathbb C\setminus\{\sfX_i(t):i\in\mathbb Z\}$ by
\begin{equation}\label{eq:meromorphic-S}
 \sfS_t\left(z\right)= \frac1{\sfX_0\left(t\right)-z}
 +\lim_{R\to\infty}\sum_{i=1}^{R}
 \left\{\frac1{\sfX_i\left(t\right)-z}+\frac1{\sfX_{-i}\left(t\right)-z}\right\}.
\end{equation}
The paired-series estimate gives locally uniform convergence on this
domain.  Hence $\sfS_t$, and therefore
$\sfS_{t,\gamma}=\sfS_t+\gamma$, is meromorphic with precisely the
simple poles $\sfX_i(t)$, each of residue $-1$.  The finite transforms
are rational functions.  We next derive their SDE.
For any holomorphic function $f$, put
\begin{equation}\label{eq:def-Qf}
 Q_f\left(z,w\right)\defeq 
 \begin{cases}
  \dfrac{f\left(z\right)-f\left(w\right)}{z-w},&z\neq w,\\[6pt]
  f'\left(z\right),&z=w.
 \end{cases}
\end{equation}
This is holomorphic in both variables, including on the diagonal, and
$Q_{f+\gamma}=Q_f$ for any constant $\gamma$.
We first compute the stochastic equation for the finite Stieltjes transform, including both complex covariations. These formulas identify the drift and covariance of the limiting field.
\begin{lemma}\label{lemma S^UN SDE}
    For $z\in\mathbb H$, the finite transform $\sfS_t^\UN(z)$
    satisfies
    \begin{equation}\label{eq SDE for S transform}
        \dif \sfS^\UN_t\left(z\right) = \dif \sfU^\UN_t\left(z\right) + \left(\frac{\beta}{4}\partial_z\left(\sfS^\UN_t\left(z\right)\right)^2+\frac{2-\beta}{4}\partial_z^2\sfS^\UN_t\left(z\right)\right)\dif t.
    \end{equation}
    Here, $\sfU^\UN_t(z)$ is a continuous martingale whose bilinear and
    Hermitian covariations are given by
    \begin{equation}\label{eq quadratic variation of sfU}
        \dif\left\langle\sfU^\UN\left(z\right),\sfU^\UN\left(z'\right)\right\rangle_t
        =\partial_z\partial_{z'}Q_{\sfS_t^\UN}\left(z,z'\right)\dif t,
    \end{equation}
    including $z=z'$ by \eqref{eq:def-Qf}, and
    \begin{equation}\label{eq:finite-hermitian-bracket}
      \dif\left\langle\sfU^\UN\left(z\right),\overline{\sfU^\UN\left(w\right)}\right\rangle_t
      =\left.\partial_z\partial_\zeta
      Q_{\sfS_t^\UN}\left(z,\zeta\right)\right|_{\zeta=\overline w}\dif t,
    \end{equation}
    where the finite transform in the second variable is continued to
    the lower half-plane by Schwarz reflection.
\end{lemma}
\begin{proof}
Write $I=\mathcal I_N$ and suppress $t$ temporarily.  It\^o's
formula gives
\begin{equation}\label{eq ito for S^UN_t}
\begin{aligned}
 \dif\sfS_t^{\left(N\right)}\left(z\right)
 &=\left\{
 \sum_{i\in I}\left(\sfX_i^{\left(N\right)}-z\right)^{-3}
 -\frac\beta2\sum_{\substack{i,j\in I\\i\neq j}}
 \frac{\left(\sfX_i^{\left(N\right)}-z\right)^{-2}}
      {\sfX_i^{\left(N\right)}-\sfX_j^{\left(N\right)}}
 \right\}\dif t-\sum_{i\in I}\left(\sfX_i^{\left(N\right)}-z\right)^{-2}\dif\sfB_i\left(t\right).
\end{aligned}
\end{equation}
Symmetrizing the finite double sum yields
\begin{equation*}
 \sum_{\substack{i,j\in I\\i\neq j}}
 \frac{\left(\sfX_i^{\left(N\right)}-z\right)^{-2}}
      {\sfX_i^{\left(N\right)}-\sfX_j^{\left(N\right)}}
 =-\frac12\partial_z\left(\sfS_t^{\left(N\right)}\left(z\right)\right)^2
 +\sum_{i\in I}\left(\sfX_i^{\left(N\right)}-z\right)^{-3},
\end{equation*}
where
$\sum_{i\in I}(\sfX_i^{(N)}-z)^{-3}
=\frac12\partial_z^2\sfS_t^{(N)}(z)$.  Substitution gives
\eqref{eq SDE for S transform} with
\begin{equation}\label{eq:def-finite-U}
 \sfU_t^{\left(N\right)}\left(z\right)\defeq -\int_0^t\sum_{i\in\mathcal I_N}
 \left(\sfX_i^{\left(N\right)}\left(s\right)-z\right)^{-2}\dif\sfB_i\left(s\right).
\end{equation}
Taking the bilinear and Hermitian quadratic covariations in
\eqref{eq:def-finite-U} and using the elementary resolvent identity
gives \eqref{eq quadratic variation of sfU} and
\eqref{eq:finite-hermitian-bracket}.
\end{proof}
The following common localization and fourth-moment bound give uniform integrability for the finite martingale fields and their derivatives. This will retain the martingale identities in the limit.
\begin{lemma}\label{lem:field-ui}
Let $\mathcal K\Subset\operatorname{int}\mathcal K_1$ with
$\mathcal K_1\Subset\mathbb H$, and put
\begin{equation*}
 \mathsf A_t\defeq \sup_N\sup_{\substack{0\leq s\leq t\\z\in\mathcal K_1}}
 \left|\sfS_s^{\left(N\right)}\left(z\right)\right|,
 \qquad
 \mathsf T_R^{\mathrm{loc}}\defeq \inf\left\{t\leq T:\mathsf A_t\geq R\right\}\wedge T.
\end{equation*}
Then $\mathsf T_R^{\mathrm{loc}}$ is a stopping time,
$\mathsf T_R^{\mathrm{loc}}\uparrow T$ almost surely, and
\begin{equation}\label{eq:field-ui}
 \sup_N\expt\left[
  \sup_{0\leq t\leq T}\sup_{z\in\mathcal K}
  \left|\sfU_{t\wedge\mathsf T_R^{\mathrm{loc}}}^{\left(N\right)}\left(z\right)\right|^4
 \right]<\infty.
\end{equation}
The same conclusion holds for every fixed finite number of
$z$-derivatives on a smaller compact set.
\end{lemma}
\begin{proof}
Proposition \ref{prop s transform convergence} gives almost sure uniform
convergence of $(\sfS^{(N)})_N$ on $[0,T]\times\mathcal K_1$. The limit
is continuous, so the family is equicontinuous: uniform convergence
controls the tail, and the finitely many remaining continuous functions
preserve equicontinuity. Hence $\mathsf A_T<\infty$ and $\mathsf A$ is
continuous. Taking the suprema over fixed countable dense subsets
shows that $\mathsf A$ is adapted. Therefore
$\mathsf T_R^{\mathrm{loc}}$ is a stopping time and increases to $T$.
Put $\eta=\inf_{z\in\mathcal K_1}\operatorname{Im}z>0$.  Before
$\mathsf T_R^{\mathrm{loc}}$,
\begin{equation*}
 \sum_{k\in\mathcal I_N}
 \frac1{\left|\sfX_k^{\left(N\right)}\left(s\right)-z\right|^4}
 \leq\eta^{-3}\operatorname{Im}\sfS_s^{\left(N\right)}\left(z\right)
 \leq\eta^{-3}R.
\end{equation*}
The Burkholder--Davis--Gundy inequality therefore gives, uniformly in
$N$ and $z\in\mathcal K_1$,
\begin{equation*}
 \expt\left[
  \sup_{t\leq T}\left|\sfU_{t\wedge\mathsf T_R^{\mathrm{loc}}}^{\left(N\right)}\left(z\right)\right|^4
 \right]
 \leq C\left(\eta^{-3}RT\right)^2.
\end{equation*}
Cover $\mathcal K$ by finitely many discs whose doubled closed discs
lie in $\operatorname{int}\mathcal K_1$, and denote the boundary
circles of the doubled discs by $\Gamma_1,\ldots,\Gamma_q$.  Cauchy's
formula and H\"older's inequality give
\begin{equation*}
 \sup_{z\in\mathcal K}\left|\sfU_{t\wedge\mathsf T_R^{\mathrm{loc}}}^{\left(N\right)}\left(z\right)\right|^4
 \leq C_{\mathcal K,\mathcal K_1}
 \sum_{r=1}^q\int_{\Gamma_r}
 \left|\sfU_{t\wedge\mathsf T_R^{\mathrm{loc}}}^{\left(N\right)}\left(\zeta\right)\right|^4\left|\dif\zeta\right|.
\end{equation*}
Taking the supremum in $t$, then expectation, proves
\eqref{eq:field-ui}.  Cauchy's derivative formula proves the final
assertion.
\end{proof}
Recall that $\sfS_{t,\gamma}=\sfS_t+\gamma$. Proposition
\ref{prop s transform convergence} gives
\begin{equation}\label{eq S_t,gamma}
\sfS_{t,\gamma}\left(z\right)= \lim_{N\to\infty}\sfS^\UN_t\left(z\right)
=\sfS_t\left(z\right)+\gamma,\qquad t\in\left[0,T\right],\quad z\in\mathbb H.
\end{equation}
We introduce the notation and stopped-domain convention used when
passing to the limit in \eqref{eq SDE for S transform}.  For a
holomorphic function $f$, define
\begin{equation*}
    \mathcal A\left(f\right)\defeq \frac{\beta}{4}\partial_z\left(f^2\right)
    +\frac{2-\beta}{4}\partial_z^2f,
\end{equation*}
and use $Q_f$ from \eqref{eq:def-Qf}.
For complex local martingales, $\langle \mathsf{M},\mathsf{N}\rangle$ denotes bilinear
covariation and $\langle \mathsf{M},\overline {\mathsf{N}}\rangle$ denotes Hermitian
covariation. In a Hermitian covariance formula, the Stieltjes
transform is extended to the lower half-plane by $\sfS_t(\zeta)=\overline{\sfS_t(\overline\zeta)}.$ 

\noindent\textbf{Stopped pole-free continuation.}
We require the martingale identities to remain valid on domains that stay
away from the moving poles, including domains selected measurably at a
stopping time. This permits contour integration around individual particles.

Let $\mathsf T_{\mathrm{in}}\leq\mathsf T_{\mathrm{out}}\leq T$ be bounded stopping times, let
$(\mathcal D_j)_{j\geq1}$ be a fixed countable family of relatively compact open
subsets of $\mathbb C$, every connected component of which meets
$\mathbb H$, and let $\mathsf J$ be an $\mathcal F_{\mathsf T_{\mathrm{in}}}$-measurable
$\mathbb N$-valued random variable. Put $D=\mathcal D_{\mathsf J}$ and
$E_j=\{\mathsf J=j\}$. Suppose that, on $E_j$,
\begin{equation*}
    \inf_{\substack{\mathsf T_{\mathrm{in}}\leq s\leq\mathsf T_{\mathrm{out}},\ z\in\overline{\mathcal D_j}\\k\in\mathbb Z}}
    \left|z-\sfX_k\left(s\right)\right|>0.
\end{equation*}

Using the symmetric paired continuation of $\sfS$, define on $E_j$
and for $z\in \mathcal D_j$
\begin{equation*}
    \mathsf V_t^{\mathsf T_{\mathrm{in}},\mathsf T_{\mathrm{out}}}\left(z\right)\defeq 
    \begin{cases}
        0, & t\leq\mathsf T_{\mathrm{in}},\\[3pt]
        \sfS_{t\wedge\mathsf T_{\mathrm{out}},\gamma}\left(z\right)-\sfS_{\mathsf T_{\mathrm{in}},\gamma}\left(z\right)
        -\displaystyle\int_{\mathsf T_{\mathrm{in}}}^{t\wedge\mathsf T_{\mathrm{out}}}
          \mathcal A\left(\sfS_{s,\gamma}\right)\left(z\right)\dif s,
        & t>\mathsf T_{\mathrm{in}}.
    \end{cases}
\end{equation*}
Let $\mathsf W^j=\mathbf1_{E_j}\mathsf V^{\mathsf T_{\mathrm{in}},\mathsf T_{\mathrm{out}}}$ on $E_j$ and set $\mathsf W^j=0$
on $E_j^c$. We say that $(\sfS_{\cdot,\gamma},\sfU)$ has the
\emph{stopped pole-free continuation property} $(\mathbf{SC})$ on
$[0,T]$ if, for every such choice and every $j$, $\mathsf W^j$ is an adapted
continuous local-martingale field in
$C([0,T];\mathcal O(\mathcal D_j))$: for every $\mathcal K\Subset \mathcal D_j$, there is a
common localization under which the stopped field is an integrable
$C(\mathcal K)$-valued martingale. It also satisfies
\begin{equation*}
    \mathsf W_t^j\left(z\right)
    =
    \mathbf1_{E_j}
    \left\{\sfU_{t\wedge\mathsf T_{\mathrm{out}}}\left(z\right)-\sfU_{t\wedge\mathsf T_{\mathrm{in}}}\left(z\right)\right\},
    \quad z\in \mathcal D_j\cap\mathbb H.
\end{equation*}

Moreover, on $E_j$, for $z,w\in \mathcal D_j$,
\begin{equation*}
\begin{aligned}
    \left\langle \mathsf W^j\left(z\right),\mathsf W^j\left(w\right)\right\rangle_t
    &=\int_{t\wedge\mathsf T_{\mathrm{in}}}^{t\wedge\mathsf T_{\mathrm{out}}}
      \partial_z\partial_wQ_{\sfS_s}\left(z,w\right)\dif s,\\
    \left\langle \mathsf W^j\left(z\right),\overline{\mathsf W^j\left(w\right)}\right\rangle_t
    &=\int_{t\wedge\mathsf T_{\mathrm{in}}}^{t\wedge\mathsf T_{\mathrm{out}}}
      \left.
      \partial_z\partial_\zeta Q_{\sfS_s}\left(z,\zeta\right)
      \right|_{\zeta=\overline w}\dif s.
\end{aligned}
\end{equation*}
On $E_j^c$, $\mathsf W^j$ and both of its covariations vanish.
We can now pass the finite transform equation to the limit, retaining its martingale field and covariations. The stopped continuation property will allow us to recover the particle motions by contour integration.
\begin{proposition}
\label{prop:stieltjes-spde}
Fix $\beta\geq1$, $T>0$. Work under the
hypotheses and notation of Theorem \ref{theo sde theorem} and Proposition \ref{prop s transform convergence}, on the
common filtered probability space carrying the Brownian motions
$(\sfB_i)_{i\in\mathbb Z}$. Let $\sfS_t$ be the symmetric Stieltjes
transform of $\boldsymbol{\sfX}(t)$ on $\mathbb H$.

There exists an adapted field
$\sfU\in C([0,T];\mathcal O(\mathbb H))$ such that $\sfU_0=0$ and
$\sfU(z)$ is a continuous complex local martingale for every
$z\in\mathbb H$. Moreover,
\begin{equation}\label{eq:stieltjes-spde}
    \sfS_{t,\gamma}\left(z\right)
    =
    \sfS_{0,\gamma}\left(z\right)+\sfU_t\left(z\right)
    +\int_0^t\mathcal A\left(\sfS_{s,\gamma}\right)\left(z\right)\,\dif s.
\end{equation}
For $z,w\in\mathbb H$, its bilinear and Hermitian covariations are
\begin{equation}\label{eq:stieltjes-bilinear-bracket}
    \left\langle\sfU\left(z\right),\sfU\left(w\right)\right\rangle_t
    =
    \int_0^t\partial_z\partial_wQ_{\sfS_s}\left(z,w\right)\,\dif s
\end{equation}
and
\begin{equation}\label{eq:stieltjes-hermitian-bracket}
    \left\langle\sfU\left(z\right),\overline{\sfU\left(w\right)}\right\rangle_t
    =
    \int_0^t
    \left.
    \partial_z\partial_\zeta Q_{\sfS_s}\left(z,\zeta\right)
    \right|_{\zeta=\overline w}\dif s.
\end{equation}
Finally,
$(\sfS_{\cdot,\gamma},\sfU)$ satisfies $(\mathbf{SC})$.
\end{proposition}
\begin{proof}
Recall the finite martingale field from \eqref{eq:def-finite-U}.
Using a countable exhaustion of $\mathbb H$, Proposition \ref{prop s transform convergence} and
Cauchy's integral formula give, on one event of probability one,
\begin{equation*}
    \partial_z^r\sfS^\UN
    \longrightarrow
    \partial_z^r\sfS_{\cdot,\gamma},
    \quad r=0,1,2,3,
\end{equation*}
locally uniformly on $[0,T]\times\mathbb H$. The finite-system SPDE
therefore yields
\begin{equation*}
    \sfU_t\left(z\right)\defeq 
    \sfS_{t,\gamma}\left(z\right)-\sfS_{0,\gamma}\left(z\right)
    -\int_0^t\mathcal A\left(\sfS_{s,\gamma}\right)\left(z\right)\,\dif s
\end{equation*}
and $\sfU^\UN\to\sfU$ in
$C([0,T];\mathcal O(\mathbb H))$.

Use the common localization $\mathsf T_R^{\mathrm{loc}}\uparrow T$
of Lemma \ref{lem:field-ui}, for
$\mathcal K\Subset\operatorname{int}\mathcal K_1$ and
$\mathcal K_1\Subset\mathbb H$. Its fourth-moment bound, together with
$\sfU^{(N)}\to\sfU$ in
$C([0,T]\times\mathcal K)$, permits passage to conditional
expectations and proves that the stopped limit is a
$C(\mathcal K)$-valued martingale.  Its fourth-moment bound also gives
uniform integrability of the bilinear and Hermitian compensated
products.  Cauchy's formula for the holomorphic divided differences
gives locally uniform convergence of
$\partial_z\partial_wQ_{\sfS_s^{(N)}}$ to
$\partial_z\partial_wQ_{\sfS_s}$, including on the diagonal; the
Hermitian kernels converge in the same way.  Passing to the limit on
a fixed countable dense subset and then using continuity proves that
$\sfU$ is a local-martingale field and gives
\eqref{eq:stieltjes-bilinear-bracket} and
\eqref{eq:stieltjes-hermitian-bracket}.

For the explicit densities, subtract the symmetric partial sums in
the definition of $\sfS_s$. Since the constant $\gamma$ cancels,
\begin{equation*}
    Q_{\sfS_s}\left(z,w\right)
    =
    \sum_{k\in\mathbb Z}
    \frac{1}{\left(\sfX_k\left(s\right)-z\right)\left(\sfX_k\left(s\right)-w\right)}.
\end{equation*}
Membership of $\mathfrak q(\boldsymbol{\sfX})$ in $\tsaipro$ implies, uniformly on
compact time intervals,
\begin{equation*}
    \sfX_k\left(s\right)
    =
    \sfX_0\left(s\right)+\rho k+O_T\left(\left|k\right|^{1-\kappa}\right).
\end{equation*}
The series above is therefore absolutely and locally uniformly
convergent away from the poles. Termwise differentiation gives the
bilinear and Hermitian covariance densities
\begin{equation*}
    \sum_{k\in\mathbb Z}
    \frac{1}{\left(\sfX_k\left(s\right)-z\right)^2\left(\sfX_k\left(s\right)-w\right)^2},\quad
    \sum_{k\in\mathbb Z}
    \frac{1}{\left(\sfX_k\left(s\right)-z\right)^2\left(\sfX_k\left(s\right)-\overline w\right)^2},
\end{equation*}
respectively. These series also converge absolutely and locally uniformly
away from the poles.

It remains to verify $(\mathbf{SC})$. Fix $\mathcal D_j$ and
$E_j=\{\mathsf J=j\}$. The same particle asymptotics show that the paired
summands in the Stieltjes series are $O_T(|k|^{-1-\kappa})$, uniformly
on compact time intervals and compact pole-free sets. The series for
its first two derivatives are absolutely and locally uniformly
convergent. Consequently, the selected field $\mathsf W^j$ in the definition of $(\mathbf{SC})$
is adapted, belongs to $C([0,T];\mathcal O(\mathcal D_j))$, and, on
$E_j$, extends continuously to
$[0,T]\times\overline{\mathcal D_j}$. On $\mathcal D_j\cap\mathbb H$,
\eqref{eq:stieltjes-spde} identifies it with
\begin{equation*}
    \mathbf1_{E_j}
    \left\{\sfU_{t\wedge\mathsf T_{\mathrm{out}}}-\sfU_{t\wedge\mathsf T_{\mathrm{in}}}\right\}.
\end{equation*}
This is a local martingale because
$\mathbf1_{E_j}\mathbf1_{(\mathsf T_{\mathrm{in}},\mathsf T_{\mathrm{out}}]}$ is predictable.

On $E_j$, put $\mathsf H_t^j=0$ for $t<\mathsf T_{\mathrm{in}}$, and for $t\geq\mathsf T_{\mathrm{in}}$ set
\begin{equation*}
    \mathsf H_t^j\defeq 
    \sup_{\mathsf T_{\mathrm{in}}\leq u\leq t\wedge\mathsf T_{\mathrm{out}}}
    \sup_{z\in\overline{\mathcal D_j}}
    \left\{
        \left|\mathsf W_u^j\left(z\right)\right|
        +\sum_{k\in\mathbb Z}\left|\sfX_k\left(u\right)-z\right|^{-4}
    \right\}.
\end{equation*}
Define
\begin{equation*}
    \mathsf T_m^{j,\mathrm{loc}}\defeq \inf\left\{t:\mathsf H_t^j>m\right\}\wedge T.
\end{equation*}
On $E_j^c$, put $\mathsf H^j=0$ and $\mathsf T_m^{j,\mathrm{loc}}=T$. Because
$E_j\in\mathcal F_{\mathsf T_{\mathrm{in}}}$, and because the suprema may be taken over
countable dense sets, $\mathsf T_m^{j,\mathrm{loc}}$ is a stopping time. The
pole-separation assumption and the particle asymptotics give
$\mathsf H_T^j<\infty$, so $\mathsf T_m^{j,\mathrm{loc}}\uparrow T$. For
$0\leq r\leq t\leq T$ and
$A\in\mathcal F_r$, the function
\begin{equation*}
    z\longmapsto
    \expt\left[
        \mathbf1_A
        \left\{
            \mathsf W^j_{t\wedge\mathsf T_m^{j,\mathrm{loc}}}\left(z\right)
            -\mathsf W^j_{r\wedge\mathsf T_m^{j,\mathrm{loc}}}\left(z\right)
        \right\}
    \right]
\end{equation*}
is holomorphic by Morera's theorem and dominated convergence. It
vanishes on $\mathcal D_j\cap\mathbb H$, and hence on every component of $\mathcal D_j$
by the identity theorem. Boundedness in the stopped
$C(\mathcal K)$-norm and separability of $C(\mathcal K)$ then give the
$C(\mathcal K)$-valued martingale property for every
$\mathcal K\Subset \mathcal D_j$. Thus $\mathsf W^j$ is a local-martingale field.
Applying the same argument on $\mathcal D_j\times \mathcal D_j$ to the bilinear and
Hermitian compensated products gives the two stopped covariance
formulas; in the Hermitian case the second variable is
anti-holomorphic. Countability of the domains and their components
allows all versions to be chosen simultaneously. This proves
$(\mathbf{SC})$.
\end{proof}

For regular particle paths, the paired-series estimate gives a meromorphic
Stieltjes transform with simple poles at the particles, each of residue
$-1$, simultaneously on every compact time interval. Indeed, regularity
gives $\sfX_i(t)=\sfX_0(t)+\rho i+O_T(|i|^{1-\kappa})$ on one event for
all integer time horizons, so \eqref{eq:pole-free-paired-tail} gives
uniform convergence on compact sets staying away from the poles. Near
any fixed particle its single summand supplies the residue, while strict
ordering and the paired-tail bound make the remaining sum holomorphic.

Conversely, the following result recovers the shifted Dyson ISDE from the
Stieltjes equation. The stopped continuation property permits contour
integration around the moving poles and recovers their driving Brownian motions.
\begin{proposition}
\label{prop:pole-recovery}
Fix $\beta\geq1$, $\gamma\in\mathbb R$, $p>1$,
$\kappa\in(0,1)$, and $\rho>0$. Let
$(\Omega,\mathcal F,(\mathcal F_t)_{t\geq0},\prob)$ satisfy the
usual conditions, and let $\bx\in\mathcal W$ be deterministic with
$\mathfrak q(\bx)\in\mathcal L_{\kappa,\rho}^{p}$. Let
$\boldsymbol{\sfX}=(\sfX_i)_{i\in\mathbb Z}$ be an adapted, coordinatewise
continuous process such that, almost surely,
\begin{equation*}
    \boldsymbol{\sfX}\left(0\right)=\bx,
    \quad
    \sfX_i\left(t\right)<\sfX_{i+1}\left(t\right),
    \quad i\in\mathbb Z,\ t\geq0,
    \quad
    \mathfrak q\left(\boldsymbol{\sfX}\right)\in\tsaipro.
\end{equation*}
Let $\sfS_t$ be the symmetric Stieltjes transform of $\boldsymbol{\sfX}(t)$.
Suppose there is a single adapted continuous
local-martingale field $\sfU$ on $\mathbb H$ such that
\eqref{eq:stieltjes-spde},
\eqref{eq:stieltjes-bilinear-bracket},
\eqref{eq:stieltjes-hermitian-bracket}, and $(\mathbf{SC})$ hold on
every compact time interval.

Then there exists an independent family
$(\sfB_i)_{i\in\mathbb Z}$ of standard real
$(\mathcal F_t)_{t\geq0}$-Brownian motions such that
\begin{equation*}
    \sfX_i\left(t\right)
    =
    x_i+\sfB_i\left(t\right)
    +\frac{\beta}{2}\int_0^t
    \left\{\mu_i\left(\boldsymbol{\sfX}\left(s\right)\right)-\gamma\right\}\dif s,
    \quad i\in\mathbb Z,\ t\geq0.
\end{equation*}
Moreover, $\boldsymbol{\sfX}$ is the unique strong solution of this
shifted ISDE with initial condition $\bx$ among processes whose gap
process belongs to $\tsaipro$.
\end{proposition}
\begin{proof}
    Fix $T>0$ and $i\in\mathbb Z$. Membership in $\tsaipro$ gives,
uniformly for $s\in[0,T]$,
\begin{equation*}
    \sfX_k\left(s\right)
    =
    \sfX_0\left(s\right)+\rho k+O_T\left(\left|k\right|^{1-\kappa}\right).
\end{equation*}
    Hence the paired summands defining $\mu_i(\boldsymbol{\sfX}(s))$ are
$O_T(|k|^{-1-\kappa})$. Thus the principal value exists locally
uniformly in $s$ and is continuous.  For every compact set
$\mathcal K\Subset\mathbb C\setminus
\bigcup_{0\leq s\leq T}\{\sfX_k(s):k\in\mathbb Z\}$,
the same estimate gives
\begin{equation*}
    \sup_{0\leq s\leq T}\sup_{z\in\mathcal K}
    \sum_{k\in\mathbb Z}\left|\sfX_k\left(s\right)-z\right|^{-4}<\infty.
\end{equation*}

Put $\mathsf T_0=0$. If $\mathsf T_n<T$, let
\begin{equation*}
    \mathsf g_n\defeq 
    \min\left\{
        \sfX_i\left(\mathsf T_n\right)-\sfX_{i-1}\left(\mathsf T_n\right),
        \sfX_{i+1}\left(\mathsf T_n\right)-\sfX_i\left(\mathsf T_n\right)
    \right\}.
\end{equation*}
From a fixed enumeration of
$\mathbb Q\times(\mathbb Q_{>0})^2$, select the first
$(\mathsf c_n,\mathsf r_n,\mathsf d_n)$ satisfying the following inequalities, and denote
its enumeration index by $\mathsf J_n$:
\begin{equation*}
    \left|\mathsf c_n-\sfX_i\left(\mathsf T_n\right)\right|<\frac{\mathsf g_n}{64},
    \quad
    \frac{\mathsf g_n}{4}<\mathsf r_n<\frac{9\mathsf g_n}{32},
    \quad
    \frac{\mathsf g_n}{256}<\mathsf d_n<\frac{\mathsf g_n}{128}.
\end{equation*}
The choice and $\mathsf J_n$ are $\mathcal F_{\mathsf T_n}$-measurable. Define
\begin{equation*}
\begin{aligned}
    \mathsf T_{n+1}\defeq {}&T\wedge\inf\left\{t\geq\mathsf T_n:
        \left|\sfX_i\left(t\right)-\mathsf c_n\right|\geq \mathsf r_n-2\mathsf d_n
        \text{ or }
        \min_{\varepsilon\in\left\{-1,1\right\}}
        \left|\sfX_{i+\varepsilon}\left(t\right)-\mathsf c_n\right|
        \leq \mathsf r_n+2\mathsf d_n
    \right\},
\end{aligned}
\end{equation*}
and
\begin{equation*}
    \Gamma_n\defeq \left\{z:\left|z-\mathsf c_n\right|=\mathsf r_n\right\},
    \quad
    D_n\defeq \left\{z:\mathsf r_n-\mathsf d_n<\left|z-\mathsf c_n\right|<\mathsf r_n+\mathsf d_n\right\},
\end{equation*}
where $\Gamma_n$ is counterclockwise oriented. Throughout
$[\mathsf T_n,\mathsf T_{n+1}]$, $\Gamma_n$ encloses $\sfX_i$ and no other
pole, and
\begin{equation*}
    \inf_{\substack{\mathsf T_n\leq s\leq\mathsf T_{n+1},\ z\in\overline D_n\\
                    k\in\mathbb Z}}
    \left|z-\sfX_k\left(s\right)\right|\geq\mathsf d_n.
\end{equation*}

The construction reaches $T$ after finitely many steps almost surely.
Indeed,
\begin{equation*}
    \mathsf g_*\defeq 
    \min_{0\leq s\leq T}
    \min\left\{
        \sfX_i\left(s\right)-\sfX_{i-1}\left(s\right),
        \sfX_{i+1}\left(s\right)-\sfX_i\left(s\right)
    \right\}>0.
\end{equation*}
At a nonterminal exit, one of
$\sfX_{i-1},\sfX_i,\sfX_{i+1}$ moves by at least
$7\mathsf g_n/32\geq7\mathsf g_*/32$. Uniform continuity therefore prevents
accumulation of the stopping times. Set $\mathsf T_n=T$ after the first
time that $T$ is reached.

Apply $(\mathbf{SC})$ to the fixed countable family of rational annuli
generated by the preceding enumeration, with $\mathsf T_{\mathrm{in}}=\mathsf T_n$,
$\mathsf T_{\mathrm{out}}=\mathsf T_{n+1}$, and $\mathsf J=\mathsf J_n$ (so $\mathcal D_{\mathsf J_n}=D_n$), and write
$\mathsf V^{(n)}=\mathsf V^{\mathsf T_n,\mathsf T_{n+1}}$. On every event selecting a fixed
rational annulus, contour integration is a continuous linear
functional on $\mathcal O(D_n)$. Hence
\begin{equation*}
    \mathsf M_i^{\left(n\right)}\left(t\right)\defeq 
    -\frac{1}{2\pi\mi }
    \oint_{\Gamma_n}z\mathsf V_t^{\left(n\right)}\left(z\right)\dif z
\end{equation*}
is a continuous local martingale, stopped on
$[\mathsf T_n,\mathsf T_{n+1}]$. For the random annulus, first localize to the
first $m$ alternatives in the fixed enumeration and use a common localizer for them; letting $m\to\infty$ gives the assertion.

For $\mathsf T_n\leq t\leq\mathsf T_{n+1}$, the residue theorem and the defining
identity for $\mathsf V^{(n)}$ give
\begin{equation*}
\begin{aligned}
    \sfX_i\left(t\right)-\sfX_i\left(\mathsf T_n\right)
    ={}&\mathsf M_i^{\left(n\right)}\left(t\right)-\frac{1}{2\pi\mi }
      \int_{\mathsf T_n}^t\oint_{\Gamma_n}
      z\,\mathcal A\left(\sfS_{s,\gamma}\right)\left(z\right)\dif z\dif s.
\end{aligned}
\end{equation*}
The uniform pole separation justifies interchanging the time and
contour integrals. Near $\sfX_i(s)$,
\begin{equation*}
    \sfS_{s,\gamma}\left(z\right)
    =
    -\frac{1}{z-\sfX_i\left(s\right)}+\mathsf h_i\left(s,z\right),
    \quad
    \mathsf h_i\left(s,\sfX_i\left(s\right)\right)=\gamma-\mu_i\left(\boldsymbol{\sfX}\left(s\right)\right).
\end{equation*}
The second identity follows from the locally uniform paired
convergence: the cutoffs centered at $0$ and at $i$ differ by
$2|i|$ boundary terms, each of order $O_T(K^{-1})$. Consequently,
\begin{equation*}
    \frac{1}{2\pi\mi }
    \oint_{\Gamma_n}\sfS_{s,\gamma}\left(z\right)^2\dif z
    =2\left\{\mu_i\left(\boldsymbol{\sfX}\left(s\right)\right)-\gamma\right\},
    \quad
    \oint_{\Gamma_n}z\,\partial_z^2\sfS_{s,\gamma}\left(z\right)\dif z=0.
\end{equation*}
Using
$\oint_{\Gamma_n}z\,\partial_zf(z)\,\dif z
=-\oint_{\Gamma_n}f(z)\,\dif z$, we obtain
\begin{equation*}
    \sfX_i\left(t\right)-\sfX_i\left(\mathsf T_n\right)
    =
    \mathsf M_i^{\left(n\right)}\left(t\right)
    +\frac{\beta}{2}\int_{\mathsf T_n}^t
    \left\{\mu_i\left(\boldsymbol{\sfX}\left(s\right)\right)-\gamma\right\}\dif s.
\end{equation*}

Define
\begin{equation*}
 \sfB_i\left(t\right)\defeq 
    \sfX_i\left(t\right)-x_i
    -\frac{\beta}{2}\int_0^t
    \left\{\mu_i\left(\boldsymbol{\sfX}\left(s\right)\right)-\gamma\right\}\dif s.
\end{equation*}
For each $m$, the process stopped at $\mathsf T_m$ is a finite pasting of
continuous local martingales. Since $\mathsf T_m\uparrow T$ almost surely,
$\sfB_i$ is a real continuous local martingale on $[0,T]$.

For $i,j\in\mathbb Z$, take a common refinement of their stopping
covers and apply $(\mathbf{SC})$ to the union of the two active annuli.
All unions of two rational annuli form a fixed countable family, and
the selected union is measurable at the left endpoint of each refined
interval.
On a refined interval denote the two active contours by
$\Gamma^{(i)}$ and $\Gamma^{(j)}$.  The bilinear covariance formula
gives
\begin{equation*}
\begin{aligned}
    \frac{\dif}{\dif s}\left\langle\sfB_i,\sfB_j\right\rangle_s
    ={}&
    \sum_{k\in\mathbb Z}
    \left(
        \frac{1}{2\pi\mi }
        \oint_{\Gamma^{\left(i\right)}}
        \frac{z\dif z}{\left(\sfX_k\left(s\right)-z\right)^2}
    \right)
    \left(
        \frac{1}{2\pi\mi }
        \oint_{\Gamma^{\left(j\right)}}
        \frac{w\dif w}{\left(\sfX_k\left(s\right)-w\right)^2}
    \right)
    =\delta_{ij}.
\end{aligned}
\end{equation*}
The fourth-power bound and Cauchy--Schwarz justify the interchange of
the series and the contour integrals; each contour factor is the
indicator that it encloses $\sfX_k(s)$. It follows that
\begin{equation*}
    \left\langle\sfB_i,\sfB_j\right\rangle_t=\delta_{ij}t.
\end{equation*}
Taking a countable intersection over $T\in\mathbb N$ and
$i,j\in\mathbb Z$ makes these identities simultaneous for all indices
and times. Multidimensional L\'evy's characterization, applied to
every finite subfamily, shows that $(\sfB_i)_{i\in\mathbb Z}$ is an
independent family of standard real Brownian motions.

Set $\sfY_i(t)=\sfX_i(t)+\beta\gamma t/2$. Since the gaps and
interaction drift are invariant under a common translation,
$\boldsymbol{\sfY}$ is an adapted regular solution of
\eqref{eq ISDE definition} driven by the recovered Brownian family.
Proposition \ref{ISDE e and u} supplies a strong solution driven by that
family, and pathwise uniqueness among adapted regular solutions identifies
it with $\boldsymbol{\sfY}$. Translating back proves strong existence and
uniqueness for $\boldsymbol{\sfX}$.
\end{proof}
\bibliographystyle{acm}
\bibliography{ReferencesDBM}

\bigskip

\noindent{\sc School of Mathematics, University of Edinburgh, James Clerk Maxwell Building, Peter Guthrie Tait Rd, Edinburgh EH9 3FD, U.K.}\newline
\href{mailto:theo.assiotis@ed.ac.uk}{\small theo.assiotis@ed.ac.uk}

\bigskip

\noindent{\sc School of Mathematics, University of Edinburgh, James Clerk Maxwell Building, Peter Guthrie Tait Rd, Edinburgh EH9 3FD, U.K.}\newline
\href{mailto:F.Li-31@sms.ed.ac.uk}{\small F.Li-31@sms.ed.ac.uk}

\end{document}